\documentclass[leqno,english]{smfart}
\usepackage{amssymb}
\usepackage{amsmath}
\usepackage{amsrefs}
\usepackage{smfthm}
\usepackage{verbatim}
\usepackage{enumitem}
\theoremstyle{plain}
\newtheorem{theorem}[equation]{Theorem}
\newtheorem{corollary}[equation]{Corollary}
\newtheorem{lemma}[equation]{Lemma}

\newtheorem{proposition}[equation]{Proposition}

\theoremstyle{definition}
\theoremstyle{remark}
\newtheorem{example}[equation]{Example}

\newtheoremstyle{indenteddefinition}{\topsep}{\topsep}{\addtolength{\leftskip}{2.0em}}{-0em}{\bfseries}{.}{
}{} %{-2.5em}{}{}{}{}
\theoremstyle{indenteddefinition}
\newtheorem{definition}[equation]{Definition}

\font\temporary=manfnt
\def\dbend{{\temporary\char127}} % dangerous bend sign
\def\danger{\begin{trivlist}\item[]\noindent%
\begingroup\hangindent=3pc\hangafter=-2%\clubpenalty=10000%
\def\par{\endgraf\endgroup}%
\hbox to0pt{\hskip-\hangindent\dbend\hfill}\ignorespaces}
\def\enddanger{\par\end{trivlist}}

\DeclareMathOperator\Lie{Lie}
\DeclareMathOperator\rank{rank}
\DeclareMathOperator\Ad{Ad}
\DeclareMathOperator\Aut{Aut}
\DeclareMathOperator\Hom{Hom}
\DeclareMathOperator\Map{Map}
\DeclareMathOperator\Id{Id}
\DeclareMathOperator\Ind{Ind}
\DeclareMathOperator\End{End}
\DeclareMathOperator\Sesq{Sesq}

\DeclareMathOperator\quo{quo}
\DeclareMathOperator\sub{sub}
\DeclareMathOperator\cont{cont}
\DeclareMathOperator\weak{weak}
\DeclareMathOperator\tr{tr}
\DeclareMathOperator\gr{gr}
\DeclareMathOperator\RE{Re}
\DeclareMathOperator\CRE{RE}

\DeclareMathOperator\abs{abs}
\DeclareMathOperator\im{im}
\DeclareMathOperator\re{re}
\DeclareMathOperator\mult{mult}
\DeclareMathOperator\POS{pos}
\DeclareMathOperator\NEG{neg}
\DeclareMathOperator\RAD{rad}
\DeclareMathOperator\for{for}
\DeclareMathOperator\sig{sig}
\DeclareMathOperator\sgn{sgn}

\renewcommand{\theequation}{\thesection.\arabic{equation}}
\newcommand{\Cite}[1]{\hspace{-.2ex}\cite{#1}} %USE \cite after (
\newcommand{\tildehatk}[2]{\big(\widetilde{#1}\big)_{#2}^{\widehat{}}}
\newcommand{\tildehatR}[1]{\big(\widetilde{#1}\big)_1^{\widehat{}}}
\newcommand{\LERP}{\le_{\text{RP}}}

\usepackage{tabularx}
\usepackage{ltablex}
\newcommand{\nind}[3]{#1 & #2; #3.} % entry, location, description

\author[J.~D.~Adams]{Jeffrey D. Adams}% \thanks{All of the authors were
\address{Department of Mathematics \\ University of  Maryland}
\email{jda@math.umd.edu}
\urladdr{http://www.math.umd.edu/~jda/}
\author[M. van Leeuwen]{Marc van Leeuwen}
\address{Laboratoire de Math\'ematiques et Applications
\\Universit\'e de Poitiers}
\email{Marc.van-Leeuwen@math.univ-poitiers.fr}
\urladdr{http://wwwmathlabo.univ-poitiers.fr/~maavl/}
\author[P.~E.~Trapa]{Peter E. Trapa} %\thanks{The third author was supported in
\address{Department of Mathematics \\ University of Utah}
\email{ptrapa@math.utah.edu}
\urladdr{http://math.utah.edu/~ptrapa}
\author[D.~A.~Vogan,~Jr.]{David A. Vogan, Jr.} % \thanks{The fourth author
\address{Department of Mathematics \\ Massachusetts 
Institute of Technology}
\email{dav@math.mit.edu} 
\urladdr{http://www-math.mit.edu/~dav} 
\begin{document}
\hyphenation{Grothen-dieck}
\frontmatter
\title{Unitary representations of real reductive groups}  
\alttitle{Repr\'esentations unitaires des groupes de Lie r\'eductifs} 

\begin{abstract}
We present an algorithm for computing the irreducible unitary
representations of a real reductive group $G$.  The Langlands
classification, as formulated by Knapp and Zuckerman, exhibits any
representation with an invariant Hermitian form as a deformation of a
unitary representation from % Harish-Chandra's
the Plancherel formula. The behavior of
these deformations was in part determined in the Kazhdan-Lusztig
analysis of irreducible characters; more complete information comes
from the Beilinson-Bernstein proof of the Jantzen conjectures.

Our algorithm traces the signature of the form through this
deformation, counting changes at reducibility points.  An important
tool is Weyl's ``unitary trick:'' replacing the classical invariant
Hermitian form (where $\mathop{\rm Lie}(G)$ acts by skew-adjoint operators) by a
new one (where a compact form of $\mathop{\rm Lie}(G)$ acts by skew-adjoint
operators).
\end{abstract}
\begin{altabstract}
Nous pr{\'e}sentons un algorithme pour le calcul des
repr{\'e}sentations unitaires irr{\'e}ductibles d'un groupe de Lie
r{\'e}ductif r{\'e}el~$G$. La classification de Langlands, dans sa
formulation par Knapp et Zuckerman, pr{\'e}sente toute
repr{\'e}sentation hermitienne comme {\'e}tant la d{\'e}formation
d'une repr{\'e}sentation unitaire intervenant dans la formule de
Plancherel. % de Harish-Chandra.
Le comportement de ces d{\'e}formations est en partie
d{\'e}termin{\'e} par l'analyse de Kazhdan-Lusztig des caract{\`e}res
irr{\'e}ductibles; une information plus compl{\`e}te provient de la
preuve par Beilinson-Bernstein des conjectures de Jantzen.

Notre algorithme trace {\`a} travers
cette d{\'e}formation les changements de la signature de la forme qui peuvent
intervenir aux points de r{\'e}ductibilit{\'e}. Un
outil important est ``l'astuce unitaire'' de Weyl: on remplace
la forme hermitienne classique (pour laquelle $\mathop{\rm Lie}(G)$ agit par
des op{\'e}rateurs antisym{\'e}triques) par une forme hermitienne nouvelle
(pour laquelle c'est une forme compacte de $\mathop{\rm Lie}(G)$ qui
agit par des op{\'e}rateurs antisym{\'e}triques).
\end{altabstract}
\subjclass{22E46, 20G05, 17B15}
\keywords{unitary representation, Kazhdan-Lusztig polynomial, Hermitian form}
\altkeywords{representation unitaire, polyn\^omes de Kazhdan-Lusztig,
  forme hermitienne}
\thanks{All of the authors were supported in part by NSF grant
  DMS-0968275. The first author was supported in part by NSF grant
  DMS-0967566. The third author was supported in part by NSF grant
  DMS-0968060. The fourth author was supported in part by NSF grant
  DMS-0967272.}
\maketitle
\vfill\eject
\mainmatter
{\linespread{.95}
\tableofcontents}
\vfill\eject

\section*{Index of notation}

\begin{tabularx}{\linewidth}{l X}
\nind{$(A,\nu)$}
{continuous parameter}
{Definition \ref{def:dLP}}
%\\
%\nind{$D$}
%{Weyl Denominator}
%{Definition \ref{def:Weylden}}
\\
\nind{${}^\delta G(\mathbb C), {}^\delta G, {}^\delta K$}
{extended groups}
{Definition \ref{def:extgrp}}
\\
\nind{$\widehat{G}_u$}
{unitary dual of $G$}
{\eqref{eq:Ghatu}}
\\ 
\nind{$\widehat{G}$}
{nonunitary dual of $G$}
{Definition \ref{def:HCmod}}
\\ 
\nind{$G(\mathbb R, \sigma)$}
{real points of a real form of complex Lie group $G(\mathbb C)$}
{\eqref{eq:Grealpoints}}
\\
%\nind{$G(\mathbb R, \sigma_c)$}
%{real points of the compact real 
%form of a complex connected reductive Lie group}
%{Example \ref{ex:compactform}}
%\\
\nind{$\mathcal G(\mathfrak{h},L(\mathbb C))^\sigma$}
{Grothendieck group of finite length admissible
$(\mathfrak{h},L(\mathbb C))$ modules with nondegenerate
$\sigma$-invariant %Hermitian 
forms}
{Definition \ref{def:hermgroth}}
\\
\nind{$\Gamma=(H,\gamma,R^+_{i \mathbb R})$}
{(continued) Langlands parameter}
{Definitions \ref{def:langlandsparam},  \ref{def:genparam}}
\\
\nind{$\Gamma_1=({}^1 H,\gamma,R^+_{i \mathbb R})$}
{extended Langlands parameter}
{Definition \ref{def:extLP}}
\\
\nind{$\Gamma^{h,\sigma_0}, \Gamma^{h,\sigma_c}$}
{Hermitian dual (or $c$-Hermitian dual) of Langlands parameter}
{Definitions \ref{def:dualparam}, \ref{def:cdualparam}}
\\
\nind{$H \simeq T \times A$}
{Cartan decomposition of real torus}
{Proposition \ref{prop:toruschars}}
\\
\nind{${}^1H = \langle H,\delta_1\rangle$}
{extended maximal torus}
{Definition \ref{def:exttorus}}
\\
\nind{$(\mathfrak{h}, {}^\delta L(\mathbb C))$}
{(extended) pair}
{Definitions \ref{def:pair}, \ref{def:extpair}}
\\
\nind{$\theta$}
{Cartan involution}
{\eqref{eq:cartaninvolution}}
\\
\nind{$\Theta_\pi$}
{distribution character of $\pi$}
{Definition \ref{def:distnchar}}
\\
\nind{$I_{\mathrm{quo}}(\Gamma), I_{\mathrm{sub}}(\Gamma)$}
{standard quotient-type and sub-type modules}
{Theorem \ref{thm:Lrealiz}}
\\
%\nind{$K(\mathbb C), L(\mathbb C)$}
%{complexification of the compact Lie group $K$ or $L$}
%{\eqref{se:cptred}}
%\\
\nind{$\ell(\Gamma)$}
{integral length of $\Gamma$}
{Definition \ref{def:length}}
\\
\nind{$\ell_o(\Gamma)$}
{orientation number of $\Gamma$}
{Definition \ref{def:ornumber}}
\\
\nind{$\Lambda=(T,\lambda,R^+_{i \mathbb R})$}
{discrete Langlands parameter}
{Definition \ref{def:dLP}}
\\ 
%\nind{$m_\alpha$}
%{element of order two in $T$}
%{Definition \ref{def:rhoim}}
%\\
\nind
{$m_{\Xi,\Gamma}, M_{\Gamma, \Psi}$}
{entries of multiplicity matrices for standard modules (or character
  formulas for irreducible modules)}
{\eqref{se:charformulas}}
\\
\nind{$(\POS_V,\NEG_V,\RAD_V)$}
{signature character of
$\sigma$-invariant Hermitian form on
admissible $(\mathfrak{h},L(\mathbb C))$ module $V$}
{Proposition \ref{prop:sigchar}(5)}
\\
\nind{$P_{\Gamma, \Psi}$}
{character polynomial}
{Definition \ref{def:multpoly}}
\\
\nind{$P^c_{\Gamma, \Psi}$}
{signature character polynomial}
{Definition \ref{def:Wpoly}}
\\
\nind{$Q_{\Gamma, \Psi}$}
{multiplicity polynomial}
{Definition \ref{def:charpoly}}
\\
\nind{$Q^c_{\Gamma, \Psi}$}
{signature multiplicity polynomial}
{Definition \ref{def:wpoly}}
\\
\nind{$R_{\mathbb R}, R_{i\mathbb R}, R_{\mathbb C}$}
{real, imaginary, and complex roots in $R$}
{Definition \ref{def:rhoim}}
\\
\nind{$2\rho_{\abs}$}
{character of $H$}
{Lemma \ref{lemma:rhoimcover}}
\\
\nind{$\sigma$}
{real form of complex algebraic variety or group}
{\eqref{se:realform}, \eqref{eq:Grealform}}
\\
\nind{$\sigma_c$}
{compact real form of complex reductive group}
{Theorem \ref{thm:KC}}%, Example \ref{ex:compactform}}
\\
% \nind{$\tau(\Gamma)$}
% {$\tau$-invariant of $\Gamma$}
% {Definition \ref{def:tauinv}}
% \\
% \nind{$t_1$, $t_s$}
% {extended Weyl group elements for extended tori ${}^1 H$, ${}^1 H_s$}
% {Definition \ref{def:exttorus}, Proposition \ref{prop:findimlextHswts}}
% \\
%\nind{$V^h$}
%{Hermitian dual of an $(\mathfrak{h}, L(\mathbb C))$-module $V$}
%{\eqref{eq:Kinf}, \eqref{eq:silentKfin}}
%\\
\nind{$(\pi^{h,\sigma},V^{h,\sigma})$}
{$\sigma$-Hermitian dual of $(\mathfrak{h}, L(\mathbb C))$-module $V$}
{\eqref{se:hermdual}}
\\
\nind{$\mathbb W$}
{signature ring}
{Definition \ref{def:witt}}
\\
\nind{$W^\Lambda$}
{stabilizer of $\Lambda$ in real Weyl group}
{Proposition \ref{prop:LCshape}(1)}
\\
\nind{$W(G,H)$}
{real Weyl group}
{Definition \ref{def:rhoim}, Proposition \ref{prop:realweyl}(4)}
\\
\nind{${}^\delta W(G,H)$}
{extended real Weyl group}
{Definition \ref{def:exttorus}}
\\ 

\nind{$X(\mathbb R,\sigma)$}
{real points for  real form $\sigma$ of variety $X(\mathbb  C)$}
{\eqref{se:realform}}
\\
% \nind{$X^*,X_*$}
% {character and cocharacter lattices of torus}
% {Definition \ref{def:torus}}
% \\
\end{tabularx}
\vfill\eject

\section{First introduction}\label{sec:firstintro}
\setcounter{equation}{0}
The purpose of this paper is to give a finite algorithm for computing
the set of irreducible unitary representations of a real reductive Lie
group $G$. Before explaining the nature of the algorithm, it is
worth recalling why this is an interesting question.  A serious
historical survey would go back at least to the work of Fourier (which
can be understood in terms of the irreducible unitary representations
of the circle).  

\begin{subequations}\label{se:gelfand}
Since we are not serious historians, we will begin instead with a
formulation of ``abstract harmonic analysis'' arising from the work
of Gelfand beginning in the 1930s.  In Gelfand's formulation, one
begins with a topological group $G$ acting on a topological space
$X$.  A reasonable example to keep in mind is $G=GL(n,{\mathbb R})$
acting on the space $X$ of lattices in ${\mathbb R}^n$. What makes
such spaces difficult to study is that there is little scope for
using algebra.  

The first step in Gelfand's program is therefore to find a nice
Hilbert space ${\mathcal H}$ (often of functions on $X$); for example,
if $G$ preserves a measure on $X$, one can take
\begin{equation*} {\mathcal H} =L^2(X) \end{equation*}
If choices are made well (as in the case of an invariant measure, for
example) then $G$ acts on the Hilbert space ${\mathcal H}$ by unitary
operators: 
\begin{equation*} \pi\colon G \rightarrow U({\mathcal H}). \end{equation*}
Such a group homomorphism (assumed to be continuous, in the sense that
the map
\begin{equation*} G \times {\mathcal H} \rightarrow {\mathcal H},
  \qquad (g,v) \mapsto \pi(g)v \end{equation*}
is continuous) is called a {\em unitary representation of $G$}.
Gelfand's program says that questions about the action of $G$ on $X$
should be recast as questions about the unitary representation of $G$
on ${\mathcal H}$, where one can bring to bear tools of linear
algebra.

One of the most powerful tools of linear algebra is the theory of
eigenvalues and eigenvectors, which allow some problems about linear
transformations to be reduced to the case of dimension one, and the
arithmetic of complex numbers.  An eigenvector of a linear
transformation is a one-dimensional subspace preserved by the linear
transformation. In unitary representation theory the
analogue of an eigenvector is an {\em irreducible unitary
  representation}: a nonzero unitary representation having no proper
closed subspaces invariant under $\pi(G)$. Just as a
finite-dimensional complex vector space is a direct sum of eigenspaces
of any (nice enough) linear transformation, so any (nice enough) unitary
representation is something like a direct sum of irreducible unitary
representations. 

The assumption that we are looking at a {\em unitary} representation
avoids the difficulties (like nilpotent matrices) attached to
eigenvalue decompositions in the finite-dimensional case; but allowing
infinite-dimensional Hilbert spaces introduces complications of other
kinds. First, one must allow not only direct sums but also ``direct
integrals'' of irreducible representations. This complication appears
already in the case of the action of ${\mathbb R}$ on $L^2({\mathbb
  R})$ by translation. The decomposition into one-dimensional
irreducible representations is accomplished by the Fourier transform,
and so involves integrals rather than sums.

For general groups there are more serious difficulties, described by
von Neumann's theory of ``types.'' But one of Harish-Chandra's
fundamental theorems (\cite{HCI}*{Theorem 7}) is that real reductive
Lie groups are ``type I,'' and therefore that any unitary
representation of a reductive group may be written uniquely as a
direct integral of irreducible unitary representations. The second
step in Gelfand's program is to recast questions about the (reducible)
unitary representation $\pi$ into questions about the irreducible
representations into which it is decomposed.

The third step in Gelfand's program is to describe all of the
irreducible unitary representations of $G$. This is the problem of
``finding the unitary dual''
\begin{equation}\label{eq:Ghatu} \widehat{G}_u =_{\text{def}}
  \{\text{equiv.~classes of irreducible unitary representations of
    $G$}\} \end{equation}  
It is this problem for which we offer a solution (for real reductive
$G$) in this paper.  It is far from a completely satisfactory solution
for Gelfand's program; for of course what Gelfand's program asks is
that one should be able to answer interesting questions about all
irreducible unitary representations.  (Then these answers can be
assembled into answers to the questions about the reducible
representation $\pi$, and finally translated into answers to the
original questions about the topological space $X$ on which $G$ acts.)
We offer not a list of unitary representations but a method to
calculate the list. To answer general questions about unitary
representations in this way, one would need to study how the questions
interact with our algorithm.

%All of
Which is to say that we may continue to write papers after this
one.

Here is an outline of the algorithm for identifying unitary group
representations. To describe it, we will use a number of ideas to be
introduced and explained only later. What we will actually describe is
an algorithm for computing the signature of an invariant Hermitian
form on any irreducible representation that admits one. The algorithm
for calculating signatures of invariant Hermitian forms in the
analogous setting of highest weight modules is treated by Yee in
\cites{Yee,Yee-new}. %PT
We follow her ideas closely.

The Langlands classification Theorem
\ref{thm:LC} realizes each irreducible representation
$\overline{\pi}_1$ of $G$ as the unique irreducible quotient of an induced
representation $\pi_1$. This induced representation is part of an
analytic one-parameter family $\{\pi_t\mid t\in
[0,\infty)\}$ of representations. (One simply replaces the real
part $\nu_{\RE}$ of the continuous parameter (cf.~Proposition
\ref{prop:toruschars}) by $t\nu_{\RE}$.) If $\overline{\pi}_1$ is tempered
(and therefore unitary), then $\nu_{\RE}=0$ and $\pi_t=\pi_1
=\overline{\pi}_1$ is a constant family.  In all cases $\pi_0$ is
tempered, and therefore unitary. The representation $\pi_t$ is irreducible
except for a discrete set of values of $t$.

In case $\overline{\pi}_1$ admits an invariant Hermitian form
$\langle,\rangle_1$, we can extend it to an analytic one-parameter
family $\langle,\rangle_t$ of invariant Hermitian forms for the
representations $\pi_t$ (Proposition \ref{prop:meromorphic}). This
result of Knapp and Stein is one of the earliest general tools for
studying non-tempered unitary representations. Because of the
analyticity, the signature of the form $\langle,\rangle_t$ is {\em
  locally constant} on the open set where $\pi_t$ is irreducible.
Because $\pi_0$ is unitary, the form $\langle,\rangle_0$ is definite.

Our algorithm calculates the signature of $\langle,\rangle_t$
beginning with the definiteness at $t=0$, and then calculating the
{\em change} in the signature at each reducibility point $t' \in
[0,1]$. This idea, taken from \Cite{Vu}, is Corollary
\ref{cor:formchange}. The answer is described in terms of forms on
subquotients of the {\em Jantzen filtration} of $\pi_{t'}$ (defined in
\eqref{se:jantzenfilt}).

We have therefore described the signature of the invariant form
$\langle,\rangle_1$ in terms of various signatures of forms on
subquotients $\pi'$ of $\pi_{t'}$, for $t' \le 1$. The
reason this leads to an effective algorithm is that (as is evident
from Langlands' original proof of his classification in \Cite{LC}) the
real part $\RE\nu'$ of the continuous parameter attached to $\pi'$
satisfies
$$\| \RE\nu' \| < \| \RE\nu \|. $$

In order to carry out this computation, we first need to identify
explicitly the irreducible representations $\pi'$ appearing in the
various levels of the Jantzen filtrations. This is done by the Jantzen
Conjecture (Theorem \ref{thm:BBpol}, proved by Beilinson and Bernstein
in \Cite{BB}). The answer is phrased in terms of the Kazhdan-Lusztig
polynomials defined and calculated in \Cite{LV83}

All of these ideas were in place in the 1980s. The great difficulty is
that each irreducible representation $\pi'$ with an invariant
Hermitian form actually has {\em two} inequivalent forms, one the
negative of the other; and there is no way to specify a preferred
form. (A convincing example of this phenomenon is provided by the
two-dimensional irreducible representation $\overline{\pi}_1$ of
$G=SU(1,1)$. As is evident from the name of the group, there is an
invariant Hermitian form of signature $(1,1)$ on
$\overline{\pi}_1$. This form and its negative are exchanged by the
action of the outer automorphism group of $G$.)

When $\pi'$ appears in a Jantzen filtration
(and therefore carries a form defined by Jantzen, related to our
orginal $\langle,\rangle_1$ by analytic continuation) we must decide
which of the two forms on $\pi'$ is the Jantzen form. This
is the content of Theorem \ref{thm:KLsigpolQ}.

The main idea in the proof of Theorem \ref{thm:KLsigpolQ} is evident
in its statement: the theorem refers to Hermitian forms preserved not
by the real Lie algebra ${\mathfrak g}_0$, but rather by a compact
real form of ${\mathfrak g}_0$. We call these forms {\em $c$-invariant
  forms}. Their existence (which is quite easy) is Proposition
\ref{prop:cdualstd}. The great advantage of $c$-invariant forms is
that they are automatically definite on the lowest $K$-types (see
Proposition \ref{prop:cdualstd}(5)). Consequently there is a natural
choice of $c$-invariant form, the one which is {\em positive} on all the
lowest $K$-types.

Of course we are finally interested in ordinary invariant Hermitian
forms rather than $c$-invariant forms, so we need to be able to
translate signature calculations from one to the other.  This is done
by Theorem \ref{thm:ctoinv}. The translation depends on extending the
representation to include an action of the Cartan involution. The
Cartan involution is inner exactly when $\rank(G) = \rank(K)$; this
setting has been called the {\em equal rank case} since the work of
Harish-Chandra on discrete series representations in the 1960s. In the
equal rank case, the translation of signatures between invariant and
$c$-invariant forms is elementary (Theorem \ref{thm:ctoinveqrk}).

When the Cartan involution is {\em not} inner (the {\em unequal rank
  case}), we need to extend the representations we study to include an
action of the Cartan involution. Formally this is a straightforward
version of ``Clifford theory''; the result is Theorem
\ref{thm:extLC}. But we also need to extend the definition and
calculation of Kazhdan-Lusztig polynomials to these extended groups,
and this is a serious matter. (The outer automorphisms act on the
perverse sheaves calculated in \Cite{LV83}, and we need to know the
$+1$ and $-1$ eigenspaces of this action on each cohomology stalk.)
The mathematical ideas are in \Cite{LV12}, and the tracking of signs
is done in \Cite{AV}. One formulation of the answer is Corollary
\ref{cor:sigform}.

We will offer more details about the algorithm in Section
\ref{sec:secondintro}, after recalling the results of Harish-Chandra,
Langlands, and others in terms of which the algorithm is formulated.

% For the moment we can say that the algorithm
% calculates not only the set of unitary representations, but more
% generally the signature of the invariant Hermitian form on any
% irreducible representation admitting one.  What the algorithm uses is
% the Kazhdan-Lusztig polynomials describing irreducible characters of
% $G$, and their interpretation (established by Beilinson and Bernstein
%in \Cite{BB}) in terms of Jantzen filtrations.

Section \ref{sec:exsl2C} carries out the algorithm for
the complex reductive group $SL(2,{\mathbb C})$, computing all
invariant Hermitian forms and in particular identifying the unitary
representations (first described in \Cite{GN}).

The algorithm described in this paper has been implemented in the {\tt
  atlas} software package \Cite{atlas}. There it has been tested on
thousands of known unitary (and non-unitary) representations of groups
of rank up to about eight. The software in its present form is {\em
  not} fast enough to determine the unitarity of an arbitrary
representation of our favorite example, the split real form of $E_8$.

We are deeply grateful to George Lusztig for producing the mathematics
in \Cite{LV12}, without which we would have been unable to complete this
work. We thank Annegret Paul, whose expert reading of a draft of this
paper led to many emendations and improvements. Finally, we thank
Wai Ling Yee, whose work first showed us that an analysis of unitary
representations along these lines might be possible; indeed we hoped
to persuade her to be a coauthor.

\end{subequations}

\section{Nonunitary representations} \label{sec:nonunit}
\setcounter{equation}{0}
In this section we will recall Harish-Chandra's framework for the
study of representation theory for real reductive groups.
His first big idea is
% therefore Harish-Chandra's classical one,
to provide a category of possibly nonunitary representations for which
functional analysis works nearly as well as in the unitary case.
In order to formulate Harish-Chandra's definition of this category, we
need some general terminology about groups and representations.

\begin{subequations}\label{se:reps}
Suppose $H$ is a real Lie
  group. The real Lie algebra of $H$ and its complexification are
  written
\begin{equation*}{\mathfrak h}_0 = \Lie(H), \qquad {\mathfrak h} = {\mathfrak
  h}_0\otimes_{\mathbb R} {\mathbb C}.\end{equation*}
The universal enveloping algebra is written $U({\mathfrak h})$. The
group $H$ acts on ${\mathfrak h}$ and on $U({\mathfrak h})$ by
automorphisms
\begin{equation*}\Ad \colon H \rightarrow \Aut(U({\mathfrak
    h}));\end{equation*}
we write
\begin{equation*} \label{eq:center} {\mathfrak Z}({\mathfrak h}) = U({\mathfrak
    h})^{\Ad(H)}\end{equation*}
for the algebra of invariants. If $H$ is connected, or if $H$ is the group of
real points of a connected complex group, then ${\mathfrak
  Z}({\mathfrak h})$ is precisely the center of $U({\mathfrak h})$,
but in general we can say only
\begin{equation*}
{\mathfrak Z}({\mathfrak h}) \subset \text{center of $U({\mathfrak h})$}.
\end{equation*}

A {\em continuous representation} of $H$ is a complete locally convex
topological vector space $V$ equipped with a continuous action of $H$
by linear transformations. We will sometimes write representations
with module notation, and sometimes with a named homomorphism
$$\pi\colon H \rightarrow \Aut(V).$$
In either case the continuity assumption means that the map
\begin{equation*} H\times V \rightarrow V, \qquad (h,v) \mapsto h\cdot
  v = \pi(h)v\end{equation*}
is continuous.  % We mention again that the representation is called
% {\em unitary} if $V$ is a Hilbert space, and the action of $H$
% preserves the inner product.
An {\em invariant subspace} is a closed subspace
$W\subset V$ preserved by the action of $H$. We say that $V$ is {\em
  irreducible} if it has exactly two invariant subspaces (namely $0$
and $V$, which are required to be distinct). 

The space of {\em smooth vectors in $V$} is
\begin{equation}\label{eq:smooth} V^\infty = \{ v\in V \mid \text{the
    map $h\mapsto h\cdot v$ from $H$ to $V$ is smooth}\}. \end{equation}
There is a natural complete locally convex topology on $V^\infty$, and
the inclusion in $V$ is continuous with dense image. The space
$V^\infty$ is preserved by the action of $H$, in this way becoming a
continuous representation in its own right. It carries at the same time
a natural (continuous) real Lie algebra representation of ${\mathfrak h}_0$ (equivalently, a
complex algebra action of $U({\mathfrak h})$) obtained by differentiating the
action of $H$. These two representations are related by
\begin{equation*}
h\cdot(X\cdot v) =
  (\Ad(h)X)\cdot(h\cdot v) \qquad (h \in H, X \in {\mathfrak h}_0, v
  \in V^\infty);
\end{equation*}
or in terms of the enveloping algebra
\begin{equation*} % \label{eq:compatible} 
h\cdot(u\cdot v) =
  (\Ad(h)u)\cdot(h\cdot v) \qquad (h \in H, u \in  U({\mathfrak h}), v
  \in V^\infty). \end{equation*}
\end{subequations}

Here is Harish-Chandra's notion of ``reasonable'' nonunitary
representations.  (The definition of ``reductive group'' is of 
course logically even more fundamental than that of ``irreducible
representation;'' but it is less important for understanding the
results, and so we postpone it to Section \ref{sec:realred}.)  

\begin{definition}\label{def:quasisimple}(Harish-Chandra
  \Cite{HCI}*{page 225}). Suppose $G$ is a a real reductive
  Lie group. A continuous representation $(\pi,V)$ of $G$ is called
  {\em quasisimple} if every $z\in {\mathfrak Z}({\mathfrak g})$
  (see \eqref{eq:center}) acts by scalars on $V^\infty$.  In this case the
  algebra homomorphism
$$\chi_\pi\colon {\mathfrak Z}\rightarrow {\mathbb C}, \quad \pi(z) =
\chi_\pi(z) \Id_{V^\infty}$$
is called the {\em infinitesimal character of $\pi$}.
\end{definition}

Harish-Chandra's ``good'' nonunitary representations are the
irreducible quasisimple representations.   In order to state his basic
results about these representations, and about how the unitary
representations appear among them, we need one more bit of
technology. The difficulty is that these nonunitary representations do
not yet form a nice {\em category}: there are not enough intertwining
operators.

Assume again that $H$ is a real Lie group, and fix an
arbitrary choice of 
\begin{equation*} L = \text{compact subgroup of
    $H$}. \end{equation*} 
If $(\pi,V)$ is any representation of $H$, we define
\begin{equation}\label{eq:Kfin} V_L = \{v\in V \mid \dim(\langle \pi(L)v \rangle
  ) < \infty\}; \end{equation}
here $\langle \pi(L)v \rangle$ denotes the span of the vectors
$\pi(l)v$ as $l$ runs over $L$. This is the smallest $L$-invariant
subspace of $V$ containing $v$, so we call $V_L$ the space of {\em
  $L$-finite vectors of $\pi$}.  Let us write 
\begin{equation*} \widehat L =
%  \left\{\text{\parbox{.45\textwidth}{equivalence classes of irreducible
%      representations of $L$}}.
\left\{\text{equivalence classes of irreducible representations of $L$} 
%\text{equiv classes of irr
%   representations of $L$} 
\right\}\end{equation*}
By elementary representation theory for compact groups,
\begin{equation*} V_L = \sum_{\delta \in \widehat L}
  V(\delta),\qquad  V(\delta) = \text{largest sum of copies of
    $\delta$ in $V$;}\end{equation*} 
% with
% $$V(\delta) = \text{largest sum of copies of $\delta$ in $V$},$$
$V(\delta)$ is called the space of {\em $\delta$-isotypic vectors} in
$V$. There is a 
natural complete locally convex topology on $V_L$, making it the 
algebraic direct sum of the subspaces $V(\delta)$; each $V(\delta)$ is a
closed subspace of $V$. 

\begin{lemma}[\cite{HCI}]\label{lemma:gK} Suppose $V$ is a
  representation of a real Lie group $H$, and $L$ is a compact
  subgroup of $H$. Then the space $V_L^\infty$ of smooth $L$-finite
  vectors (\eqref{eq:smooth}, \eqref{eq:Kfin}) is preserved by the
  actions of $L$ and of\, ${\mathfrak h}_0$ (equivalently, of
  $U({\mathfrak h})$). These two actions satisfy the
  conditions
\begin{enumerate}
\item the representation of $L$ is a direct sum of finite-dimensional
  irreducible representations;
\item the differential of the action of $L$ is equal to the
  restriction to ${\mathfrak l}_0$ of the action of ${\mathfrak h}_0$;
  and
\item $l\cdot(X\cdot v) =
  (\Ad(l)X)\cdot(l\cdot v)$, $l \in L$, $X \in {\mathfrak h}_0$, $v
  \in V^\infty_L)$; or equivalently
\item $l\cdot(u\cdot v) =
  (\Ad(l)u)\cdot(l\cdot v)$, $l \in L$, $u \in  U({\mathfrak h})$, $v
  \in V^\infty_L)$.
\end{enumerate} 
\end{lemma}
% The Lemma is easy to prove.

\begin{definition}\label{def:gK} Suppose $L$ is a compact subgroup of
  a real Lie group $H$. An {\em $({\mathfrak h}_0,L)$-module} is a
  complex vector space $V$ that is at the same time a representation
  of the group $L$ 
  and of the real Lie algebra ${\mathfrak h}_0$, subject to the
  conditions in Lemma \ref{lemma:gK}.  A {\em morphism of $({\mathfrak
      h}_0,L)$-modules} is a linear map respecting the actions of $L$
  and of ${\mathfrak h}_0$ separately.
\end{definition}

\begin{lemma}\label{lemma:gKfunctor} Suppose $L$ is a compact subgroup
  of a real Lie group $H$. Passage to $L$-finite smooth vectors
$$V \rightarrow V_L^\infty$$
defines a faithful functor from the category of representations to the
category of $({\mathfrak h}_0,L)$-modules.  In particular, any
(nonzero) intertwining operator between representations induces a
(nonzero) morphism between the corresponding $({\mathfrak
  h}_0,L)$-modules.
\end{lemma}
This is very easy.  What is harder (and requires more
hypotheses) is proving results in the other direction: that
$({\mathfrak h}_0,L)$-morphisms induce maps between group
representations.  For that, we return to the setting of a real
reductive Lie group $G$.  

Fix once and for all a choice of 
\begin{equation*} K = \text{maximal compact subgroup of
    $G$}. \end{equation*} 
By a theorem of E.~Cartan (see Theorem \ref{thm:realforms} below),
this choice is unique up to conjugation. 
% One maximal compact subgroup of $GL(n,{\mathbb R})$ is
% the orthogonal group $O(n)$.  If $G$ is any real reductive subgroup of
% $GL(n,{\mathbb R})$ that is preserved by transpose, then we may choose
% $K = G\cap O(n)$ as a maximal compact subgroup of $G$. (This fact
% applies to the standard matrix realizations of all the classical
% groups.  One can find a proof for example in \Cite{KOv}*{pages 3--4}.)
Here are some of Harish-Chandra's fundamental results about
quasisimple representations.

\begin{theorem}\label{theorem:quasisimple} Suppose $G$ is a real
  reductive Lie group and $K$ is a maximal compact subgroup.
\begin{enumerate}
\item\label{item:segal} Suppose $V$ is an irreducible unitary representation of
  $G$. Then $V$ is quasisimple.
\item\label{item:HC1} Suppose $V$ is an irreducible quasisimple
  representation of $G$ 
  (Definition \ref{def:quasisimple}). Then the associated $({\mathfrak
    g}_0,K)$-module $V^\infty_K$ (Lemma \ref{lemma:gK}) is irreducible.
\item\label{item:HC2} Suppose $V_1$ and $V_2$ are irreducible unitary
  representations 
  of $G$, and that $V_{1,K}^\infty \simeq V_{2,K}^\infty$ as
  $({\mathfrak g}_0,K)$-modules. Then $V_1\simeq V_2$ as unitary
  representations. 
\item\label{item:HC3} Suppose $X$ is an irreducible $({\mathfrak
    g}_0,K)$-module. Then 
  there is an irreducible quasisimple representation $V$ of $G$ so
  that $X \simeq V^\infty_K$ as $({\mathfrak g}_0,K)$-modules.
\item\label{item:HC4} Suppose $X$ is an irreducible $({\mathfrak
    g}_0,K)$-module. Then 
  $X$ is associated to an irreducible unitary representation of $G$ if
  and only if $X$ admits a positive definite invariant Hermitian form
  $\langle,\rangle_X$.  Here ``invariant'' means
$$\begin{aligned}
\langle k\cdot v, w \rangle_X = \langle v,k^{-1}\cdot w\rangle_X,\quad 
&\langle X\cdot v, w \rangle_X = -\langle v,X\cdot w\rangle_X\\
&(k \in K,\ X\in {\mathfrak g}_0,\ v, w\in X)\end{aligned}
$$ 
\end{enumerate}\end{theorem}

\begin{proof} Part \ref{item:segal} is due to Segal \Cite{Segal}. Part
  \ref{item:HC1} is
proved in Harish-Chandra \Cite{HCI}, although it is difficult to point
to a precise 
reference; the result is a consequence of Lemmas 33 and 34 and Theorem
5 (on pages 227--228). Part \ref{item:HC2} is \Cite{HCI}*{Theorem 8}.
Part \ref{item:HC3} is Harish-Chandra
\Cite{HCII}*{Theorem 4}. Part
\ref{item:HC4} is \Cite{HCI}*{Theorem 9}\end{proof}

\begin{definition}\label{def:HCmod}
Suppose $G$ is a real reductive Lie group and $K$ is a maximal compact
subgroup. Two representations $V_1$ and $V_2$ of $G$ are said to be
{\em infinitesimally equivalent} if the corresponding $({\mathfrak
  g}_0,K)$-modules $V^\infty_{1,K}$ and $V^\infty_{2,K}$ (Lemma
\ref{lemma:gK}) are equivalent.  We define
\begin{align*} \widehat{G} &=_{\text{def}} \{ \text{infinitesimal
%    equiv.~classes of irr.~quasisimple reps of $G$} \} \\
    equiv.~classes of irreducible quasisimple reps of $G$} \} \\
&= \{ \text{equivalence classes of irreducible $({\mathfrak
    g}_0,K)$-modules}\}, \end{align*}
and call this set the ``nonunitary dual of $G$.''  The $({\mathfrak
  g}_0,K)$-module $V_{1,K}^\infty$ is called the {\em Harish-Chandra
  module of $V_1$}. 
\end{definition}

Harish-Chandra's theorem \ref{theorem:quasisimple} provides a natural
containment of the unitary dual \eqref{eq:Ghatu} in the nonunitary
dual
$$\widehat{G}_u \subset \widehat{G},$$
always assuming that $G$ is real reductive.  In case $G=K$ is compact,
then every irreducible representation is finite-dimensional, and the
representation space can be given a $K$-invariant Hilbert space
structure; so every representation is equivalent to a unitary one, and
$$\widehat{K}_u = \widehat{K}.$$

We will recall in Section \ref{sec:langlands} results of
Langlands and of Knapp-Zuckerman that provide a complete and explicit
parametrization of the nonunitary dual $\widehat G$.  The rest of the
paper will be devoted to an algorithm for identifying $\widehat{G}_u$
as a subset of these parameters for $\widehat G$.

\begin{subequations}\label{se:cptred}
It is well known that in order to discuss real-linear complex
representations of a real Lie algebra ${\mathfrak h}_0$, it is
equivalent (and often preferable) to discuss complex-linear
representations of the complexified Lie algebra 
\begin{equation*}
{\mathfrak h} =_{\text{def}} {\mathfrak h}_0 \otimes_{\mathbb R}
{\mathbb C}.
\end{equation*}
We will conclude this section with an equally useful language for
representations of 
\begin{equation}
L = \text{compact Lie group}.
\end{equation}
For that, we use the algebra
\begin{equation}
R(L) =_{\text{def}} \text{span of matrix coefficients of
  finite-dimensional reps of $L$.} 
% R(L) =_{\text{def}} \text{\parbox{.5\textwidth}{span of matrix coefficients of
%  finite-dimensional representations of $L$.}}
\end{equation}
The algebra $R(L)$ is finitely generated and has no nilpotent
elements, so it is the algebra of regular functions on a complex
affine algebraic variety $L({\mathbb C})$. The closed points of
$L({\mathbb C})$ are the maximal ideals of $R(L)$, and these include
the ideals of functions vanishing at one point of $L$; so
% This is an algebra of functions on $L$ (closed under multiplication)
% because of the operation of tensor product of representations. It is
% an elementary exercise in Lie theory that $L$ has a finite-dimensional
% faithful self-dual representation $\rho$. Write $E_\rho$ for the
% finite-dimensional space of matrix coefficients of $\rho$. The
% Stone-Weierstrass Theorem guarantees that the algebra generated by
% $E_\rho$ is dense in the continuous functions on $L$. The Schur
% orthogonality relations then imply that $E_\rho$ generates
% $R(L)$. As an algebra of smooth functions on a compact manifold,
% $R(L)$ has no nilpotent elements; so $R(L)$ is the algebra of regular
% functions on a complex affine algebraic variety $L({\mathbb C})$. The
% closed points of $L({\mathbb C})$ are the maximal ideals of $R(L)$,
% and these include the ideals of functions vanishing at one point of
% $L$; so
\begin{equation}
L \subset L({\mathbb C})
\end{equation}
We call $L({\mathbb C})$ the {\em complexification of $L$}.
\end{subequations}

For the formulation of the next theorem, we need to recall (see for
example \Cite{LAG}*{\S\S 11--14}, or \Cite{Spr}*{Chapter 11}) the notion
of a {\em real form} of a complex affine algebraic variety $X({\mathbb
  C})$.  Write $R(X)$ for the (complex) algebra of regular
functions on $X({\mathbb C})$. Because ${\mathbb C}/{\mathbb R}$ is a
Galois extension, we can define a real form as a certain kind of
action of the Galois group on $R(X)$. Since the Galois group has only
one nontrivial element, a {\em real form of $X({\mathbb C})$} is given
  by a single map (the action of complex conjugation)
\begin{subequations}\label{se:realform}
\begin{equation}
\sigma^*\colon R(X) \rightarrow R(X)
\end{equation}
The map $\sigma^*$ has the characteristic properties
\begin{enumerate}
\item $\sigma^*$ is a ring automorphism of $R(X)$:
\begin{equation*}
\sigma^*(fg) = \sigma^*(f) \sigma^*(g), \quad \sigma^*(f + g) =
\sigma^*(f) + \sigma^*(g) \qquad (f, g\in R(X)).
\end{equation*}
\item $\sigma^*$ has order $2$.
\item $\sigma^*$ is conjugate linear:
\begin{equation*}
\sigma^*(zf) = {\overline{z}}\sigma^*(f) \qquad (z\in {\mathbb C},
f\in R(X)).
\end{equation*}
\end{enumerate}
% Conversely, from such an automorphism $\sigma^*$ of $R(X)$, we can
% recover the real form $R(X)_0$ as the algebra of fixed points:
% \begin{equation}
% R(X)_0 = \{f\in R(X) \mid \sigma^*(f) = f \}.
% \end{equation}  
Then $\sigma^*$ defines an automorphism $\sigma$ of order $2$ of the set
$X({\mathbb C})$ of maximal ideals in $R(X)$, and the (closed) real points
$X({\mathbb R},\sigma)$ are the maximal ideals fixed by $\sigma$.
Because the ring of functions on a variety has no nilpotent elements,
we can recover $\sigma^*$ from $\sigma$: regarding elements of $R(X)$
as functions on $X({\mathbb C})$, we have
\begin{equation}
(\sigma^*f)(x) = \overline{f(\sigma(x))}.
\end{equation}
 
\end{subequations}

\begin{theorem}[Chevalley \Cite{Chev}*{\S\S VIII--XII of Chapter
    VI}] \label{thm:KC} Suppose $L$ is a compact Lie group.  
\begin{enumerate}
\item The construction of \eqref{se:cptred} enlarges $L$ to a complex
  reductive algebraic group $L({\mathbb C})$. It is a covariant functor in $L$.
\item Every locally finite continuous representation $(\pi,V)$ of $L$ extends
  uniquely to an algebraic representation $(\pi({\mathbb C}),V)$ of
  $L({\mathbb C})$; and every algebraic representation of $L({\mathbb C})$
  restricts to a locally finite continuous representation of $L$.
\item There is unique real form $\sigma_c$ of $L({\mathbb C})$ with
  $L({\mathbb R},\sigma_c) = L$.
% ; if $(\pi,V, \langle,\rangle_V)$ is a
% finite-dimensional unitary representation of $L$, and we write
% $$f_{v,w} = \langle \pi(g^{-1})v,w\rangle_V \in R(L),$$
% then
% $$\sigma_c(f_{v,w}) = $$
% Need to realize V^* on the space of V with opposite complex
% structure, same operators; lambda_w(v) = <v,w>. Not worth writing
% out?
\end{enumerate}
Conversely, suppose $L({\mathbb C})$ is a complex reductive algebraic
group. There is a real form $\sigma_c$ of $L({\mathbb C})$,
unique up to conjugation by $L({\mathbb C})$, whose group of
real points $L = L({\mathbb R},\sigma_c)$ is compact, and meets every
component of $L({\mathbb C})$. Then
$L({\mathbb C})$ is the complexification of $L$.
\end{theorem}

In the setting of the theorem, suppose $(\pi,V)$ is a
finite-dimensional unitary representation of $L$ with inner product
$\langle, \rangle$. If $u,v \in V$, then the matrix coefficient
$$f_{u,v}(\ell) =_{\text{def}} \langle \pi(\ell)u,v\rangle$$
is one of the spanning functions used to define $R(L)$.  The real form
in the theorem is
$$(\sigma^*f_{u,v})(\ell) = \langle \pi(\ell^{-1})v,u\rangle.$$
From this one can see that 
$$L = \{\ell \in L({\mathbb C}) \mid \pi({\mathbb C})(\ell) \text{\ is
  unitary, all\ } \pi \in \widehat{L}_u\}.$$
It follows in particular that $L$ is a maximal compact subgroup of
$L({\mathbb C})$.

\begin{corollary}\label{cor:HCC}
Suppose $L$ is a compact subgroup of a real Lie group $H$, and
$L({\mathbb C})$ is the complexification of $L$
(cf.~\eqref{se:cptred}). Define an $({\mathfrak h},L({\mathbb
  C}))$-module by analogy with Definition \ref{def:gK}, requiring the
representation of ${\mathfrak h}$ to be complex linear and the action
of $L({\mathbb C})$ to be algebraic. Then any $({\mathfrak 
  h}_0,L)$-module extends uniquely to an $({\mathfrak h},L({\mathbb
  C}))$-module; and this extension defines an equivalence of
categories from $({\mathfrak h}_0,L)$-modules to $({\mathfrak
  h},L({\mathbb C}))$-modules.
\end{corollary}

\section{Real reductive groups}\label{sec:realred}
\setcounter{equation}{0}
Before embarking on more representation theory, we must be more
precise about the definition of ``reductive group,'' since the
detailed description of the nonunitary dual is sensitive to details of
this definition.

\begin{subequations}\label{se:reductive}
We begin with a complex connected reductive algebraic group 
\begin{equation} G({\mathbb C}). \end{equation}
Recall that this means a subgroup of the group of $n\times n$
invertible matrices, specified as the zero locus of a collection of
polynomial equations in the matrix entries and the inverse of the
determinant, with two additional properties:
\begin{enumerate}
\item $G({\mathbb C})$ has no nontrivial normal subgroup consisting of
  unipotent matrices, and
\item $G({\mathbb C})$ is connected as a Lie group.
\end{enumerate}
The words preceding the conditions are the definition of ``complex
algebraic;'' the first condition is the definition of ``reductive;''
and the second condition is (equivalent to) the definition of
``connected.'' The first condition may be replaced by
\begin{enumerate}
\item[$1'.$] After an appropriate change of basis in ${\mathbb C}^n$, the
  group $G({\mathbb C})$ is preserved by the automorphism $\sigma_c$,
\begin{equation}\label{eq:Unbasis} \sigma_c(g) = (g^*)^{-1} \end{equation}
(inverse conjugate transpose) of $GL(n,{\mathbb C})$.
\end{enumerate}

A {\em real form} of a complex Lie group $G({\mathbb C})$ is an
antiholomorphic Lie group automorphism $\sigma$ of order 2:
\begin{equation} 
\label{eq:Grealform} %PT
\sigma\colon G({\mathbb C}) \rightarrow G({\mathbb
    C}), \qquad \sigma^2 = \Id. \end{equation}
(Essentially we are repeating the discussion of real forms from
\eqref{se:realform}, but now in the category of analytic manifolds
rather than of varieties.  For reductive algebraic groups, the two
categories give rise to exactly the same real forms.)
Here ``antiholomorphic'' means that if $f$ is a (locally defined)
holomorphic function on $G({\mathbb C})$, then 
\begin{equation*}
(\sigma^*f)(g) = \overline{f(\sigma(g))}
\end{equation*}
is also holomorphic. Clearly the differential at the
identity of a real form of $G({\mathbb C})$ is a real form of the
complex Lie algebra ${\mathfrak g}$: that is, a real Lie algebra
automorphism of order two carrying multiplication by $i$ to
multiplication by $-i$.  (But not every real form of the Lie algebra
must exponentiate to the group.) 

Given a real form $\sigma$, the {\em
  group of real points} is
\begin{equation} 
\label{eq:Grealpoints} %PT
G = G({\mathbb R},\sigma) = G({\mathbb
    C})^\sigma, \end{equation}
the group of fixed points of $\sigma$ on $G({\mathbb C})$. This is a
real Lie group of real dimension equal to the complex dimension of
$G({\mathbb C})$.  A {\em real reductive algebraic group} is a group
of real points $G = G({\mathbb R},\sigma)$ for a complex connected
reductive algebraic group.

This is the class of groups with which we will work in this paper.  It
is fairly to common to work with a somewhat larger class of groups: to
start with $G({\mathbb R},\sigma)$ as above, and to allow a real Lie
group $\widetilde{G}$ endowed with a homomorphism
\begin{equation}\label{eq:nonlinearG} \pi\colon \widetilde{G} \rightarrow
  G({\mathbb R},\sigma) \end{equation}
having finite kernel and open image.  For such a group $\widetilde{G}$, it is
still possible to formulate the very precise form of the Langlands
classification theorem that we need.  But the unitarity algorithm will
use also the Kazhdan-Lusztig theory of irreducible characters, and
this is not available in the setting \eqref{eq:nonlinearG}. (A number
of interesting special cases have been treated, for example in
\Cite{RenTr}.)  Once this Kazhdan-Lusztig theory is available, the
rest of the proof of the unitarity algorithm can proceed as in the
linear case.  But we will not discuss such extensions further.
\end{subequations}

We conclude with a little structure theory for real reductive groups,
which will be vital for the formulation of the Langlands
classification.
\begin{theorem}[Cartan; see for example
  \Cite{KBey}*{Proposition 1.143}, or \Cite{OV}*{Theorem
    5.1.4}]\label{thm:realforms}  
% to number equations within a definition, theorem, etc:
\addtocounter{equation}{-1}
\begin{subequations}\label{se:thmrealforms}

 Suppose $G({\mathbb C})$ is a
  complex connected reductive algebraic group
  (\eqref{se:reductive}). 
\begin{enumerate}
\item Suppose $\sigma_0$ is a real form of $G({\mathbb C})$. Then there
  is a compact real form $\sigma_c$ of $G({\mathbb C})$ that commutes with
$\sigma_0$. This compact real form is unique up to conjugation by
$G = G({\mathbb R},\sigma_0)$.  The composition
\begin{equation} %PT
\label{eq:cartaninvolution}
\theta = \sigma_0\circ \sigma_c
\end{equation}
is an algebraic automorphism of $G({\mathbb C})$ of order two; it is
called a {\em Cartan involution for the real form $\sigma_0$}.
\item Suppose $\theta$ is an involutive algebraic automorphism of
  $G({\mathbb C})$. Then there is a compact real form $\sigma_c$ of
  $G({\mathbb C})$ that commutes with $\theta$. This compact real form
is unique up to conjugation by $K({\mathbb C}) = G({\mathbb
  C})^\theta$. The composition
$$\sigma_0 = \theta\circ \sigma_c$$
is a real form of $G({\mathbb C})$, called a {\em real form for
  the involution $\theta$}.
\item The constructions above define a bijection
\begin{align*} &\left\{\text{real forms of $G({\mathbb
        C})$}\right\}/\left(\text{conjugation by $G({\mathbb
      C})$}\right) \\ \longleftrightarrow &\left\{\text{algebraic
    involutions of $G({\mathbb C})$}\right\}/\left(\text{conjugation
    by $G({\mathbb C})$}\right) \end{align*}
\item The group $K = G^\theta$ is
  maximally compact in $G$. Its
  complexification is the reductive algebraic group
$$K({\mathbb C}) = G({\mathbb C})^\theta.$$
\item Write 
$${\mathfrak g}_0 = {\mathfrak k}_0 + {\mathfrak s}_0$$
for the decomposition of the real Lie algebra ${\mathfrak g}_0 =
\Lie(G({\mathbb R},\sigma_0))$ into the $+1$ and $-1$ eigenspaces of
the Cartan involution $\theta$.  (In the setting \eqref{eq:Unbasis},
these are the skew-hermitian and 
hermitian matrices in the Lie algebra, respectively.) Then there is a
($K$-equivariant under conjugation) diffeomorphism
$$K  \times {\mathfrak s}_0 \rightarrow G, \qquad (k,X) \mapsto
k\cdot \exp(X).$$ 
\end{enumerate}
\end{subequations} %{se:thmrealforms}
\end{theorem}

The last assertion is a generalization (and also a consequence) of the
polar decomposition for invertible matrices.  The fact that
$K$ is a {\em maximal} compact subgroup follows immediately
from this decomposition.

\section{Maximal tori}\label{sec:maxtor}
\setcounter{equation}{0}

Our next serious goal is Langlands' classification of the irreducible
(possibly nonunitary) representations of a real reductive group. To a
first approximation, it says that the irreducible representations of
such a group are indexed by orbits of the (real) Weyl groups on characters of
(real points of) maximal tori. What requires a little care is the
understanding of ``maximal tori;'' that is the subject of the present section.

\begin{definition}\label{def:torus} Suppose $G({\mathbb C})$ is a
  complex connected 
  reductive algebraic group (\eqref{se:reductive}).  A {\it
    (complex) maximal torus} in $G({\mathbb C})$ is a maximal
  connected abelian subgroup consisting of diagonalizable matrices.
  Automatically such a 
  subgroup $H({\mathbb C})$ is again complex connected reductive
  algebraic.  Its {\em character lattice} is the group of algebraic
  group homomorphisms
$$X^* =_{\text{def}} \Hom_{\text{alg}}(H({\mathbb C}), {\mathbb
  C}^\times).$$
This is a lattice (a finitely-generated torsion free abelian group) of
rank equal to the complex dimension of $H({\mathbb C})$.  The dual
lattice 
$$X_* =_{\text{def}} \Hom(X^*,{\mathbb Z})$$
is naturally isomorphic to
$$X_* \simeq \Hom_{\text{alg}}({\mathbb C}^\times,H({\mathbb C})),$$
the {\em lattice of one-parameter subgroups}. 

Suppose now that $\sigma$ is a real form of $G({\mathbb C})$, so that
$G({\mathbb R},\sigma)$ is a real reductive algebraic group. We say
that $H({\mathbb C})$ is {\em defined over ${\mathbb R}$} (with
respect to $\sigma$) if it is preserved by $\sigma$. In that case the
group of real points of $H({\mathbb C})$ is
$$H = H({\mathbb R},\sigma) = G({\mathbb R},\sigma) \cap H({\mathbb
  C}) = G\cap H({\mathbb C}),$$
which we call a {\em (real) maximal torus of $G$}.
\end{definition}

All complex maximal tori in $G({\mathbb C})$ are conjugate, and all
are isomorphic to $({\mathbb C}^\times)^n$. There can be more than one
conjugacy class of real maximal tori in $G$, and in
that case they are not all isomorphic.  For example, in the real form
$SL(2,{\mathbb R})$ of $SL(2,{\mathbb C})$ there are two real forms of
the maximal torus ${\mathbb C}^\times$. One is $\sigma(z) = \overline
z$, with real points ${\mathbb R}^\times$; and the other is
$\sigma_c(z) = (\overline z)^{-1}$, with real points the (compact) circle group.

Because lattices are built from copies of ${\mathbb Z}$, it is
traditional and natural to write the group law in $X^*$ as $+$. But
the group law in $X^*$ arises from multiplication of the (nonzero
complex) values of algebraic characters, so there is always a danger
of confusion in the notation.  We will try to offer warnings about
some instances of this confusion.

The lattice $X^*$ is a fundamental invariant of the complex reductive
group. Our first goal is to describe all of the (continuous)
characters---the possibly nonunitary representations---of a real torus
$H$ in terms of this lattice of (algebraic) characters of
$H({\mathbb C})$. (This is a way of thinking about the Langlands
classification theorem for the abelian reductive group $H$.)  It will
be instructive at the same time to identify the unitary characters
within these nonunitary representations, since that is the simplest
case of the classification of unitary representations at which we are
aiming. 

\begin{proposition}\label{prop:realtorus} Suppose $H({\mathbb C})$ is a
  connected abelian complex
  reductive algebraic group (a complex torus), and $X^*$ is the
  lattice of algebraic characters (Definition \ref{def:torus}). % , and
  % $\sigma$ is a real form of $H({\mathbb C})$.
\begin{enumerate}
\item We can recover $H({\mathbb C})$ from $X^*$ by the contravariant functor
$$H({\mathbb C}) \simeq \Hom(X^*,{\mathbb C}^\times).$$
The complex Lie algebra ${\mathfrak h}$ is 
$${\mathfrak h} \simeq X_*\otimes_{\mathbb Z} {\mathbb C},$$
and therefore is in a natural way defined over ${\mathbb Z}$.
\item Any real form $\sigma$ of $H({\mathbb C})$ gives rise to an
  automorphism $\theta$ of order $2$ of $X^*$, by the requirement
$$\theta(\lambda)(h) = \overline{\lambda(\sigma h)}^{-1}.$$
This defines a bijection between real forms and automorphisms of order
two of $X^*$.
\item An automorphism $\theta$ of $X^*$ of order two induces an
  automorphism (still called $\theta$) of $H({\mathbb C})$ of order
  two, by means of the functor in (1).  This automorphism is the
  Cartan involution corresponding to $\sigma$ (Theorem
  \ref{thm:realforms}).  In particular, $\theta$ is trivial if and
  only if $H({\mathbb R},\sigma)$ is compact.
\end{enumerate}
\end{proposition}

\begin{proposition}\label{prop:toruschars} Suppose $H({\mathbb C})$ is a
  connected abelian complex reductive algebraic group (a complex
  torus), $X^*$ is the lattice of algebraic characters (Definition
  \ref{def:torus}), and $\sigma$ is a real form of $H({\mathbb
    C})$. Write $\theta$ for the corresponding Cartan involution
  (Proposition \ref{prop:realtorus}), and $H = H({\mathbb R},\sigma)$
  for the group of real points.
\begin{enumerate}
\item The group 
$$T = H^\theta$$
is the (unique) maximal compact subgroup of $H$. Its
complexification is the (possibly disconnected) reductive algebraic
group
$$T({\mathbb C}) = H({\mathbb C})^\theta.$$
\item Write 
$${\mathfrak h}_0 = {\mathfrak t}_0 + {\mathfrak a}_0$$
for the decomposition of the real Lie algebra of $H$ into
$+1$ and $-1$ eigenspaces of $\theta$.  Then
$${\mathfrak t}_0 \simeq \Hom_{\mathbb Z}((X^*)^\theta, i{\mathbb
  R}),$$
$${\mathfrak a}_0 \simeq \Hom_{\mathbb Z}((X^*)^{-\theta}, {\mathbb
  R}).$$
\item Write
$$A = \exp({\mathfrak a}_0),$$
the identity component of the maximal split torus of $H$: the identity
component of the subgroup of elements $h$ with $\theta h = h^{-1}$.
The group $A$ is isomorphic to its Lie algebra, and
$$H\simeq T\times A,$$
a direct product of abelian Lie groups.
\item The group of continuous characters of $T$ may be
identified with the group of algebraic characters of $T({\mathbb C})$,
and so (by restriction) with a quotient of $X^*$:
$$\widehat{T} \simeq \Hom_{\text{alg}}(T({\mathbb
  C}),{\mathbb C}^\times) \simeq X^*/(1-\theta)X^*.$$
\item The group of continuous characters of $A$ is
$$\widehat A \simeq {\mathfrak a}^* \simeq (X^*)^{-\theta}
\otimes_{\mathbb Z} {\mathbb C},$$
a complex vector space naturally defined over ${\mathbb Z}$.
Therefore 
$$\widehat{H} \simeq \widehat{T} \times
\widehat{A} \simeq [X^*/(1-\theta)X^*] \times [(X^*)^{-\theta}
\otimes_{\mathbb Z} {\mathbb C}],$$
a direct product of a finitely generated discrete abelian group and a
complex vector space defined over ${\mathbb Z}$.
\item Suppose $\gamma = (\lambda,\nu)\in {\widehat T} \times {\widehat
    A}$ is a parameter for a character of $H$. Then the character is
  unitary if and only if 
  $\nu$ is purely imaginary; that is,
$$\nu \in (X^*)^{-\theta} \otimes_{\mathbb Z} i{\mathbb R}.$$
\end{enumerate}
\end{proposition}

There is an essential asymmetry between the two factors $T$ and $A$ in part (3).  The first is
the group of real points of an algebraic group.  The symmetrically
defined algebraic group
$$A({\mathbb C}) = H({\mathbb C})^{-\theta}$$
is also defined over ${\mathbb R}$, and its group of real points is
$$A({\mathbb R}) = \{h \in H \mid \theta(h) = h^{-1}\}.$$
This group is often disconnected; and
$$A = \text{identity component of\ } A({\mathbb R}).$$
(This notation is inconsistent with our general practice of writing
unadorned Roman letters for groups of real points; we will not make use
of the algebraic group $A({\mathbb C})$, so no confusion should arise.)

\begin{proof} Parts (1)--(3) are just the Cartan decomposition of
  Theorem \ref{thm:realforms} for the abelian case.  The first
  assertion in part (4) (identifying continuous representations of a
  compact Lie group with algebraic representations of its
  complexification) is very general, and in some sense amounts to the
  way the complexification is defined.  The second part of (4) is part
  of the general fact that taking algebraic characters is 
  a contravariant equivalence of categories from complex reductive
  abelian groups to finitely generated abelian groups.

  For (5), the characters of a vector group are the same thing as
  complex-valued linear functionals on the Lie algebra; this can be
  computed from the description of the Lie algebra in (2).  The linear
  functional $\nu$ corresponds to the character $e^\nu$ defined by
$$e^\nu(\exp X) = e^{\nu(X)} \qquad (X \in {\mathfrak a}_0).$$
This formula makes it plain that the character $e^\nu$ is unitary if and only
if $\nu$ takes purely imaginary values.  (Any character of the compact abelian
group $T$ is automatically unitary.)
\end{proof}

We see therefore that a parameter for a character of $H$ has a
discrete part (the restriction to $T$, given by an element of the
lattice quotient $X^*/(1-\theta)X^*$) and a continuous part (the
restriction to $A$, given by an element of the complex vector space
${\mathfrak a}^*$). The complex vector space is naturally defined over
${\mathbb R}$ (with real points ${\mathfrak a}_0^*$); and the
character is unitary if and only if the real part of the continuous
parameter is equal to zero.  These statements are the ones we will
generalize to all real reductive groups.

\section{Coverings of tori}\label{sec:covtor}
\setcounter{equation}{0}

What enters most naturally in the Langlands classification are
characters not of the real tori $H$, but of certain
double covers
$$1 \rightarrow \{1,\epsilon \} \rightarrow \widetilde{H}
\rightarrow H \rightarrow 1.$$
The reason is that we are interested in versions of the Weyl character
formula, and the Weyl denominator is a function not on the torus but
rather on such a double cover.  Here is the construction we need.

\begin{definition}\label{def:ncover} Suppose $H$ is a Lie group, $n$
  is a positive integer, and
$$\gamma \in  \Hom(H,{\mathbb C}^\times),$$
an element of the group of (one-dimensional) continuous characters of
$H$.  The {\em $\gamma/n$ cover of $H$} is the Lie group
$$\widetilde H = \{(h,z) \in H \times {\mathbb C}^\times \mid
\gamma(h) = z^n\}.$$
(Here $\gamma/n$ is for the moment a formal symbol, recording the
defining character $\gamma$ and the level $n$ of the cover. Despite the absence of $\gamma$ and $n$ in the notation, the group
$\widetilde H$ depends very much on both of these values; we write
$\widetilde H(\gamma/n)$ when necessary.) Projection
on the first factor defines a surjective homomorphism to 
$H$; the kernel is isomorphic (by projection on the second factor) to the group
$$\langle \zeta_n \rangle = \{1,\zeta_n,
\zeta_n^2,\ldots,\zeta_n^{n-1}\}$$
(with $\zeta_n = e^{2\pi i/n}$) of $n$th roots of $1$ in ${\mathbb
  C}^\times$. Therefore $\widetilde H$ is an $n$-fold cover of $H$, in
the sense that there is a natural short exact sequence
$$1 \rightarrow \langle \zeta_n\rangle \rightarrow \widetilde{H}
\rightarrow H \rightarrow 1.$$
The subgroup $\langle \zeta_n\rangle$ is central. Projection on the
second factor defines a group character 
$$(\gamma/n) \colon \widetilde{H} \rightarrow {\mathbb C}^\times,
\qquad (\gamma/n)(h,z) = z.$$
The $n$th power of this character descends to $H$, and is equal to
$\gamma$; that is why we call $\widetilde H$ the ``$\gamma/n$ cover.''

Any (reasonable; for example, if $H$ is reductive, ``quasisimple'' is
a strong enough hypothesis) irreducible representation $\pi$ of
$\widetilde H$ must send $\zeta_n$ to $\zeta_n^k$ for 
some $k \in {\mathbb Z}$; we say that such a representation has {\em level
  $k$}. It is easy to check that all
level $k$ representations are uniquely of the form $\tau \otimes
k(\gamma/n)$ for 
some $\tau \in \widehat{H}$:
$$ \tildehatk{H}{k} %\widehat{\widetilde H}_k 
=_{\text{def}} \{\text{level $k$ irreducibles of
    $\widetilde H$}\} \leftrightarrow k(\gamma/n) \otimes \widehat{H}.$$

Different characters can define isomorphic covers. An expression
$\gamma_2 = \gamma_1 + n\phi$ for one-dimensional 
characters of $H$ gives rise to a natural isomorphism
$$\widetilde{H}(\gamma_1/n) \simeq_{\phi} \widetilde{H}(\gamma_2/n),
\qquad (h,z) \mapsto (h,z\phi(h)).$$
\end{definition}

In order to introduce the coverings we need, we recall now the
fundamental structure theory of complex reductive groups.
\begin{definition}\label{def:posroots} Suppose $G({\mathbb C})$ is a
  complex connected reductive algebraic group, $H({\mathbb C})$ is
  a maximal torus, and $X^*$ is the lattice of algebraic characters of
  $H({\mathbb C})$.  Write
$$R = R(G({\mathbb C}),H({\mathbb C})) \subset X^*$$
for the finite set of nonzero weights of the adjoint representation
of $H({\mathbb C})$ on ${\mathfrak g}$, the {\em roots of $H({\mathbb
    C})$ in $G({\mathbb C})$}.  Attached to each root $\alpha$ there
is a {\em three-dimensional subgroup}
$$\phi_\alpha \colon SL(2,{\mathbb C}) \rightarrow G({\mathbb C}),$$
an algebraic group homomorphism with finite kernel satisfying
$$\phi_\alpha(\text{diagonal matrices}) \subset H({\mathbb C}), \qquad
0 \ne d\phi_\alpha\begin{pmatrix} 0 & 1 \\ 0 & 0\end{pmatrix} \in
{\mathfrak g}_\alpha.$$
These conditions characterize $\phi_\alpha$ uniquely up to conjugation
by diagonal matrices in $SL(2,{\mathbb C})$, or by $H({\mathbb
  C})$. In particular, they characterize uniquely the {\em coroot
  $\alpha^\vee$ corresponding to $\alpha$}, 
$$\alpha^\vee \colon {\mathbb C}^\times \rightarrow H({\mathbb C}),
\qquad \alpha^\vee(z) =_{\text{def}} \phi_\alpha \left(\begin{pmatrix} z& 0 \\ 0&
    z^{-1} \end{pmatrix} \right).$$
Write
$$R^\vee = R^\vee(G({\mathbb C}),H({\mathbb C})) \subset X_*$$
(notation as in Definition \ref{def:torus}) for the set of coroots.

Attached to the root $\alpha$ is the {\em root reflection}
$$s_\alpha \in \Aut(X^*), \qquad s_\alpha(\lambda) = \lambda -
\langle\lambda,\alpha^\vee\rangle \alpha.$$
This reflection acts also (by transpose) on $X_*$, and by extension of
scalars on ${\mathfrak h}^* = X^*\otimes_{\mathbb Z} {\mathbb C}$ and
${\mathfrak h}$ and $H({\mathbb C})$.  The group generated by all
reflections is the {\em complex Weyl group}
$$W(G({\mathbb C}),H({\mathbb C})) = \langle s_\alpha \mid \alpha \in
R(G,H) \rangle \subset \Aut(X^*).$$ 
This group may also be described as
$$W(G({\mathbb C}),H({\mathbb C})) = N_{G({\mathbb C})}(H({\mathbb
  C}))/H({\mathbb C}) \subset \Aut(H({\mathbb C})),$$
the group of automorphisms of $H({\mathbb C})$ arising from
$G({\mathbb C})$. 

Suppose $R^+ \subset R$ is a choice of positive roots (see
\Cite{Hum}*{10.1}). Define
$$2\rho = 2\rho(R^+) = \sum_{\alpha \in R^+}\alpha \in X^*,$$
Write $\Pi = \Pi(R^+)$ for the corresponding set of simple roots
(\cite{Hum}*{10.1}).  Then (\cite{Hum}*{Lemma 13.3A})
$$\Pi = \{\alpha \in R \mid \langle 2\rho, \alpha^\vee \rangle = 2
\}.$$
Consequently all coroots take nonzero even integer values on $2\rho$;
the positive coroots are those taking positive values on $2\rho$; and
the simple coroots are those taking the value $2$ on $2\rho$.
\end{definition}

We pause here for a moment to recall Harish-Chandra's description of
the center of the enveloping algebra ${\mathfrak Z}({\mathfrak g})$
\eqref{se:reps}. 

\begin{definition}\label{def:HChom} In the setting of Definition
  \ref{def:posroots}, suppose $R^+$ is a choice of positive
  roots. Write
$$\begin{aligned} 
{\mathfrak n} &= {\mathfrak n}(R^+) = \sum_{\alpha\in R^+} {\mathfrak
  g}_\alpha,\qquad 
{\mathfrak n}^- &= {\mathfrak n}(-R^+) = \sum_{\alpha\in R^+} {\mathfrak
  g}_{-\alpha},\\
{\mathfrak b} &= {\mathfrak b}(R^+) = {\mathfrak h} + {\mathfrak n},
\qquad {\mathfrak b}^- &= {\mathfrak b}(-R^+) = {\mathfrak h} + {\mathfrak n}^-
\end{aligned}$$
for the corresponding maximal nilpotent and Borel subalgebras, so that
$${\mathfrak g} = {\mathfrak n}^- \oplus {\mathfrak h} \oplus
{\mathfrak n}$$
% is a triangular decomposition.  The Poincar\'e-Birkhoff-Witt theorem
% then decomposes the universal enveloping algebra as a vector space
$$U({\mathfrak g}) = U({\mathfrak n}^-) \otimes U({\mathfrak h})
\otimes U({\mathfrak n}), \qquad U({\mathfrak g}) = U({\mathfrak b}^-)
  \oplus U({\mathfrak g}){\mathfrak n}.$$
Write
$$\tilde\xi\colon U({\mathfrak g}) \rightarrow U({\mathfrak b}^-)$$
for the (linear) projection corresponding to the last direct sum
decomposition. Then Harish-Chandra observes that % the direct sum is
% preserved by the adjoint action of ${\mathfrak h}$, and therefore
$$\tilde\xi\colon U({\mathfrak g})^{\mathfrak h} \rightarrow
U({\mathfrak b}^-)^{\mathfrak h} = U({\mathfrak h}) \simeq
S({\mathfrak h}) \simeq \text{polynomial functions on ${\mathfrak h}^*$}$$
% Furthermore this restriction 
is an algebra homomomorphism, the {\em
  unnormalized Harish-Chandra homomorphism}. 
%  Since
% ${\mathfrak h})$ is an abelian Lie algebra,
% $$U({\mathfrak h}) \simeq S({\mathfrak h}) \simeq \text{polynomial
%  functions on ${\mathfrak h}^*$;}$$
Every $\sigma\in {\mathfrak h}^*$ gives rise to an algebra
automorphism of $S({\mathfrak h})$
$$\tau_\sigma = \text{translation by $\sigma$,} \qquad \tau_\sigma(Z)
= Z - \sigma(Z) \quad (Z\in {\mathfrak h}).$$
The {\it normalized Harish-Chandra homomorphism} is
$$\xi\colon U({\mathfrak g})^{\mathfrak h} \rightarrow S({\mathfrak
  h}),\qquad  \xi(z) = \tau_\rho(\tilde\xi(z)).$$
Here $\rho \in {\mathfrak h}^*$ (Definition \ref{def:posroots}) is
half the sum of the roots in $R^+$. For any $\lambda\in {\mathfrak
  h}^*$, the {\em infinitesimal character $\lambda$} is the algebra
homomorphism 
$$\xi_\lambda\colon {\mathfrak Z}({\mathfrak g}) \rightarrow {\mathbb
  C}, \qquad \xi_\lambda(z) = \xi(z)(\lambda) =
\tilde\xi(z)(\lambda-\rho).$$
% Here we identify $U({\mathfrak h})$ with polynomial functions on
% ${\mathfrak h}^*$.
\end{definition}

\begin{theorem}[Harish-Chandra; see \Cite{Hum}*{130--134}]
  \label{thm:HChom}. Suppose $H({\mathbb C}) \subset G({\mathbb C})$
  is a maximal torus in a connected reductive algebraic group; use the
  notation of Definition \ref{def:HChom}. The Harish-Chandra
  homomorphism $\xi$ is an algebra isomorphism
$$\xi\colon {\mathfrak Z}({\mathfrak g}) \rightarrow S({\mathfrak
  h})^{W(G({\mathbb C}),H({\mathbb C}))}.$$
Consequently
\begin{enumerate}
\item the homomorphism $\xi$ is independent of the choice of positive
  root system; % in its definition; 
\item every algebra homomorphism from ${\mathfrak Z}({\mathfrak g})$
  to ${\mathbb C}$ is % of the form 
$\xi_\lambda$ for some $\lambda\in {\mathfrak h}^*$; and
\item $\xi_\lambda = \xi_{\lambda'}$ if and only if $\lambda' \in
  W(G({\mathbb C}),H({\mathbb C}))\cdot \lambda$.
\end{enumerate}
\end{theorem}

Algebraically the most complicated questions in representation theory
concern the integral weights $X^* \subset {\mathfrak h}^*$. The key to
the arguments of this paper is deformation: understanding the behavior
of representation theory in one-parameter families. Because of these
two facts, we will make a great deal of use of 

\begin{definition}[\Cite{Vgreen}*{Definition 5.4.11}]
  \label{def:realinf} Suppose we are in the setting of 
  Definition \ref{def:posroots}. Recall from Proposition 4.2 the
  identifications 
$${\mathfrak h}\simeq X_*\otimes_{\mathbb Z} {\mathbb C}, \qquad
{\mathfrak h}^*\simeq X^*\otimes_{\mathbb Z} {\mathbb C}$$
exhibiting ${\mathfrak h}$ as canonically defined over ${\mathbb
  Z}$. This means in particular that ${\mathfrak h}$ is canonically
defined over ${\mathbb R}$; we write
$${\mathfrak h}_{\CRE} = X_*\otimes_{\mathbb Z}{\mathbb R}, \qquad {\mathfrak
  h}^*_{\CRE} = X^*\otimes_{\mathbb Z}{\mathbb R},$$
the {\em canonical real form of ${\mathfrak h}$}.  (This real form is
{\em not} the real form defined by a real maximal torus unless that
torus is split.) Clearly these spaces
are preserved by the action of the Weyl group. The infinitesimal
character $\xi_\lambda$ is said to be {\em real} if $\lambda \in
{\mathfrak h}^*_{\CRE}$. By Theorem \ref{thm:HChom}(3), this property
is independent of the choice of $\lambda$ representing the
infinitesimal character.
\end{definition}

A real infinitesimal character must indeed be real-valued on the form of
${\mathfrak Z}({\mathfrak g})$ defined by the split real form of
$G({\mathbb C})$, but this is neither a sufficient condition to be
real nor (as far as we know) an interesting one.  

\begin{lemma}\label{lemma:wrho} In the setting of Definition
  \ref{def:posroots}, suppose $R^+$ and $(R^+)'$ are two choices of
  positive root system, with corresponding sums of positive roots
  $2\rho$ and $2\rho'$. 
\begin{enumerate}
\item There is a homomorphism
$$\epsilon \colon W(G({\mathbb C}),H({\mathbb C})) \rightarrow \{\pm
1\}, \quad \epsilon(s_\alpha) = -1 \quad (\alpha \in R).$$
\item There is a unique element $w \in W(G,H)$ (depending on $R^+$ and
  $(R^+)'$) such that $w(R^+) = (R^+)'$.
\item Define 
$$S = \{\alpha \in R^+ \mid \alpha \notin (R^+)'\}, \qquad 2\rho(S) =
\sum_{\alpha \in S} \alpha \in X^*.$$
Then $2\rho' = 2\rho - 4\rho(S)$.
\item The $\rho$ and $\rho'$ double covers of $H({\mathbb C})$
  (Definition \ref{def:ncover}) are canonically isomorphic.  We may
  therefore speak of ``the'' $\rho$ double cover of $H({\mathbb C})$
  without reference to a particular choice of positive root system.
\end{enumerate}
\end{lemma}
The notation introduced in Definition \ref{def:ncover} says that we
should call the coverings in (4) the ``$2\rho/2$" and ``$2\rho'/2$''
covers. We prefer either to think of $\rho$ as an abbreviation for
the formal symbol $2\rho/2$, or else simply to abuse notation in this
way. Consistent with this convention, we will also write $\rho$ for
the character of this double cover that is called $2\rho/2$ in
Definition \ref{def:ncover}.

\begin{proof}
Part (1) is \Cite{Hum}*{Exercise 6 on page 54}. Part (2) is \Cite{Hum}*{Theorem 10.3}.  (Notice that this says that $W(G,H)$ acts in a simply
transitive way on positive root systems for $R$.) Part (3) is almost
obvious: what matters is that the roots appear in pairs $\{\pm
\alpha\}$, and that a positive system picks exactly one root from each
pair. Part (4) follows from (3) and from the discussion at the end of
Definition \ref{def:ncover}.
\end{proof}

For real groups, the character $\rho$ is not quite the most useful
one. In Mackey's definition of unitary induction there appears the
square root of the absolute value of the determinant of an action of a
reductive subgroup (on the cotangent space of a real flag
variety). This determinant is essentially a sum of positive roots, so
the square root gives something close to $\rho$; but the absolute
value has no parallel in what we have done so far.  In \Cite{ABV}
there is a systematic use of the ``$\rho$ double cover'' of maximal
tori. The absence of Mackey's absolute value requires a complicated
correction (see \Cite{ABV}*{Definition 11.6}). In this paper we will
adopt a slightly different approach, introducing a slightly
complicated modification of $\rho$ to define the covering, and as a
reward getting a simpler formulation of the Langlands classification
in the next section.

\begin{definition}\label{def:rhoim} Suppose $G({\mathbb C})$ is complex
  connected reductive algebraic group endowed with a real form
  $\sigma_0$ (cf. \eqref{se:reductive}) and a Cartan involution
  $\theta$ (Theorem \ref{thm:realforms}). Suppose that
  $H({\mathbb C})$ is a maximal torus defined over ${\mathbb
    R}$. After conjugating by $G$, we may assume that $H$
  is also preserved by $\theta$; the shorthand is that ``$H$ is a
  $\theta$-stable maximal torus in $G$.'' Define
$$W(G,H) = N_G(H)/H \simeq N_K(H)/H\cap K \subset \Aut(H),$$
the {\em real Weyl group of $H$ in $G$}. Because of the second
description of the complex Weyl group in Definition \ref{def:posroots},
it is clear that $W(G,H) \subset W(G({\mathbb C}),H({\mathbb C}))$.

  Write
$$R = R(G({\mathbb C}),H({\mathbb C})) \subset X^*$$
for the roots (Definition \ref{def:posroots}). These roots fall into
three classes (each of which is preserved by the real Weyl group
$W(G,H)$).
\begin{enumerate}[label=\alph*)]
\item the {\em real roots} $R_{\mathbb R}$, those defining real-valued
  characters $\alpha$ of $H$ (or equivalently of the Lie
  algebra ${\mathfrak h}_0$). Equivalent conditions are
$$\sigma_0(\alpha) = \alpha, \qquad \theta(\alpha) = -\alpha.$$
The real roots are a root subsystem, a Levi factor of $R$.
\item the {\em imaginary roots} $R_{i\mathbb R}$, those defining unitary
  characters $\beta$ of $H$ (or equivalently, purely
  imaginary-valued characters of the Lie
  algebra ${\mathfrak h}_0$). Equivalent conditions are
$$\sigma_0(\beta) = -\beta, \qquad \theta(\beta) = \beta.$$
The imaginary roots are a root subsystem, a Levi factor of
$R$.
\item the {\em complex roots} $R_{\mathbb C}$, those which are neither
  real-valued nor unitary. Equivalently, these are the roots taking
  neither real nor purely imaginary values on the Lie algebra
  ${\mathfrak h}_0$. Complex roots appear in four-tuples
  of distinct roots
$$\{\delta, \sigma_0(\delta), -\delta, -\sigma_0(\delta)\} = \{\delta,
-\theta(\delta), -\delta, \theta(\delta)\}.$$
\end{enumerate}

Recall from Definition \ref{def:posroots} the algebraic group homomorphism
$$\phi_\alpha\colon SL(2,{\mathbb C}) \rightarrow G({\mathbb
  C}).$$
If $\alpha$ is real or imaginary, then the real form $\sigma_0$ and the
Cartan involution $\theta$ both 
preserve the image of $\phi_\alpha$. It follows easily that these
automorphisms pull back to a real form $\sigma_\alpha$ of
$SL(2,{\mathbb C})$ and a corresponding Cartan involution
$\theta_\alpha$:
$$\phi_\alpha(\sigma_\alpha(g)) = \sigma_0(\phi_\alpha(g)), \qquad
\phi_\alpha(\theta_\alpha(g)) = \theta(\phi_\alpha(g)) \qquad (g \in
SL(2,{\mathbb C}).$$
If $\alpha$ is imaginary, then the diagonal torus in $SL(2,{\mathbb
  C})$ is compact for $\sigma_\alpha$.  We may conjugate $\phi_\alpha$
so that $\sigma_\alpha$ defines either $SU(1,1)$ (in which case we say that
$\alpha$ is {\em noncompact}) or $SU(2)$ (in which case $\alpha$ is
called {\em compact}). If $\alpha$ is real, then the diagonal torus is
split, and we may conjugate $\phi_\alpha$ so that $\sigma_\alpha$
defines the real form $SL(2,{\mathbb R})$.

% The dichotomy between compact and noncompact imaginary roots has a
% parallel for real roots. 
If $\alpha$ is a real or imaginary root, define
% could insert \phantom{-} before 0s in matrix, but I don't like that either.
$$m_\alpha = \phi_\alpha\begin{pmatrix} -1 & 0 \\ 0 &
  -1\end{pmatrix} = \alpha^\vee(-1),$$
an element of order (one or) two in $T \subset H$. 

If $\delta$ is complex, the sum $\delta + \sigma_0(\delta)$ takes
  positive real values on $H$, and therefore has a
  distinguished positive-valued square root
$$\frac{1}{2}(\delta+\sigma_0(\delta)) \in \widehat{H}.$$

Similarly, the difference $\delta - \sigma_0(\delta)$ is a unitary
character of $H$, and has a distinguished unitary square
root
$$\frac{1}{2}(\delta-\sigma_0(\delta)) =_{\text{def}} \delta -
\left(\frac{1}{2}(\delta +\sigma_0(\delta)) \right) \in
\widehat{H}.$$
\end{definition}

We will eventually need Knapp's detailed description of the real Weyl
group $W(G,H)$.  The notation can be a little confusing: the complex
roots $R_{\mathbb C}$ of Definition \ref{def:rhoim} are not a root
system, but in this proposition we will consider a (small) subset
$R_C$ that {\em is} a root system.

\begin{proposition}[Knapp; see \Cite{IC4}*{Proposition
  4.16}]\label{prop:realweyl} With notation as in Definition
  \ref{def:rhoim}, fix positive root systems $R^+_{\mathbb R}$ and
  $R^+_{i{\mathbb R}}$ for the real and imaginary roots, with
    corresponding half sums $\rho_{\mathbb R}$ and $\rho_{i{\mathbb
        R}}$ in ${\mathfrak h}^*$. Define
$$R_{C} = \{\alpha \in R \mid \rho_{\mathbb R}(\alpha^\vee) =
\rho_{i{\mathbb R}}(\alpha^\vee) = 0\}.$$
\begin{enumerate}
\item The root system $R_{C}$ is $\theta$-stable and complex
  (\cite{IC4}*{Definition 3.10}). That is, it is the sum of two root
  systems $R^L_{C}$ and $R^R_{C} = \theta(R^L_{C})$ that are
  interchanged by $\theta$. 
\item We have 
$$W(R_{C})^\theta = \{(w,\theta w) \mid w \in W(R^L_{C}) \} \subset
W(R^L_{C}) \times W(R^R_{C});$$
this description shows that
$$ W(R_{C})^\theta \simeq W(R^L_{C}).$$
\item
$$W(G({\mathbb C}),H({\mathbb C}))^\theta = W(R_{C})^\theta
\ltimes [W(R_{\mathbb R}) \times W(R_{i{\mathbb R}})],$$
a semidirect product with the subgroup in brackets normal. 
\item The real Weyl group $W(G,H)$ is a subgroup of $W(G({\mathbb
    C}),H({\mathbb C}))^\theta$. The first two factors in (3) are
  contained in $W(G,H)$. Consequently
$$W(G,H) = W(R_{C})^\theta\ltimes [W(R_{\mathbb R}) \times W(G^A,H)].$$
Here
$$W(G^A,H) = W(R_{i{\mathbb R}}) \cap W(G,H)$$
is the subgroup of the imaginary Weyl group having representatives in
$G$ (or equivalently in $K$).
\end{enumerate}
\end{proposition}

In general we follow a long tradition and write the group structure on
one-dimensional characters {\em additively}, even though values of the
character are being multiplied. In Lie theory this is particularly
attractive when we are identifying certain characters (like roots)
with linear functionals on the Lie algebra; we write addition on
either side of this identification, even though there is an
exponential map involved.  But in the following lemma we are taking
absolute values of some roots, so the tradition would require us to
write addition of absolute values when we mean multiplication. This
exceeds our (admittedly large) capacity to tolerate abuse of
notation. We will therefore write group characters multiplicatively in
this lemma. 
\begin{lemma}\label{lemma:rhoimcover} In the setting of Definition
  \ref{def:rhoim}, suppose $R^+$ is a choice of
  positive root system.
\begin{enumerate}
\item Define a character
$2\rho_{\abs} = 2\rho_{\abs}(R^+)$ of $H$ by
$$2\rho_{\abs}(h) = \left(\prod_{\alpha \in R^+_{\mathbb R}}|\alpha(h)|\right)
\left(\prod_{\beta \in R^+_{i\mathbb R}}\beta(h)\right) \left(\prod_{\delta \in
    R^+_{\mathbb C}} \delta(h)\right)$$
This character differs only by factors of $\pm 1$ from $2\rho$; in
particular, it has the same differential as $2\rho$. 
\item In the definition of the character $2\rho_{\abs}$, each factor in
  the first product has a distinguished square root $|\alpha|^{1/2}$;
  and the factors in the third product fall into pairs
  $\{\delta,\sigma_0(\delta)\}$ or $\{\delta, -\sigma_0(\delta)\}$. The
  product of each such pair of factors has a distinguished square root
  (Definition \ref{def:rhoim}).  Consequently the quotient
  $[2\rho_{\abs}][2\rho_{i\mathbb R}]^{-1}$ (the second factor being
  the inverse of the product of
  the positive imaginary roots) has a distinguished square root.
\item The {\em $\rho_{\abs}$-double cover of $H$} (Definition
  \ref{def:ncover}) exists, and is naturally isomorphic to 
the $\rho_{i\mathbb R}$-cover defined by the ``half sum'' (that is, the
square root of the product) of positive
imaginary roots.  

\item There is a homomorphism
$$\epsilon_{i{\mathbb R}} \colon W\left(G,H\right)
\rightarrow \{\pm 1\},$$
$$\epsilon_{i{\mathbb R}}(w) = (-1)^{\#\{\text{imaginary roots changing
    sign under $w$}\}}.$$
\item Suppose $(R^+)'$ is a second choice of positive roots. Define 
$$S = \{\alpha \in R^+ \mid \alpha \notin (R^+)'\} = S_{\mathbb R}
\cup S_{i\mathbb R} \cup S_{\mathbb C} $$
$$2\rho_{\abs}(S) = \prod_{\alpha \in S_{\mathbb R}} |\alpha| \cdot
\prod_{\beta \in S_{i\mathbb R} \cup S_{\mathbb C}} \beta \quad \in \ \ \widehat{H}.$$
Then $2\rho'_{\abs} = 2\rho_{\abs}\cdot[2\rho_{\abs}(S)]^{-2}$.
\item The $\rho_{\abs}$ and $\rho'_{\abs}$ double covers of $H$
  (Definition \ref{def:ncover}) are canonically isomorphic.  We may
  therefore speak of ``the'' $\rho_{\abs}$ double cover of $H$
  without reference to a particular choice of positive root system.
\end{enumerate}
\end{lemma}

% The definition of the character $2\rho_{\abs}(S)$ is an instance where the
% additive notation for characters may be confusing: the values of this
% character are obtained by {\em multiplying} roots (or their absolute
% values, in the case of real roots).

The proof of Lemma \ref{lemma:rhoimcover} is parallel to that of Lemma
\ref{lemma:wrho}.  We omit the details.

\begin{definition} \label{def:Weylden} Suppose we are in the setting
  of Definition \ref{def:rhoim}. Write $\widetilde H$ for the
  $\rho_{\abs}$-double cover of $H$,
$$1 \rightarrow \{1,-1\} \rightarrow \widetilde{H}
\rightarrow H \rightarrow 1$$
as in Lemma \ref{lemma:rhoimcover}.  Fix a positive root system $R^+$ for
$H$ in ${\mathfrak g}$, and therefore a character
$$\rho_{\abs} \colon \widetilde{H} \rightarrow {\mathbb C}^\times.$$
A {\em Weyl denominator function} is the level one (that is, the
central element $-1$ acts by $-1$) function 
% alternating sum of characters
% $$D(\widetilde{h}) = \sum_{w \in W(G,H)} \epsilon(w)
% (w\rho)(\widetilde{h});$$ 
% all of these characters may be regarded as characters of the same
% covering group by Lemma \ref{lemma:wrho}.  Notice that changing the
% choice of positive root system (say to $(R^+)' = w'(R^+)$) changes
% $D$ only by the sign $\epsilon(w') = \pm 1$. So there are just
% two possible Weyl denominator functions, differing by sign.  This
% formula shows clearly that the denominator function is skew under the
% Weyl group:
% $$D(w\widetilde{h}) = \epsilon(w)D(\widetilde{h}) \qquad (w
% \in W(G,H)).$$
% Another formula for the denominator function is
% $$D(\widetilde{h}) = \rho_{\abs}(\widetilde{h})\cdot \sum_{S
% \subset R^+} (-1)^{|S|} 2\rho(S)(h^{-1}).$$
% Here $h$ is the image of $\widetilde{h}$ in $H({\mathbb C})$.  
% This uses part 3 of Lemma \ref{lemma:wrho}.  (The subsets $S$ that
% are {\it not} related to a second positive root system $(R^+)'$
% appear in pairs of opposite parity and the same $2\rho(S)$, which
% cancel in the sum.)
% This formula is easily seen to be equivalent to
$$D(\widetilde{h}) = \rho_{\abs}(\widetilde{h})\cdot \left(\prod_{\alpha \in
  R_{\mathbb R}^+} |1 - \alpha(h)^{-1}|\right)\left(\prod_{\beta\in R^+
 - R^+_{\mathbb R}}(1-\beta(h)^{-1})\right).$$
(Here $h\in H$ is the image of $\widetilde{h}$ under the covering
map.) From this formula it is clear that {\em $D(\widetilde{h}) \ne 0$
  if and only if $\widetilde{h}$ is regular}; that is, if and only if
$\alpha(h) \ne 1$ for all $\alpha \in R$.
\end{definition}

An easy calculation shows that replacing $R^+$ by a different positive
system $(R^+)'$ changes $D$ by a sign (namely $-1$ to the number of
non-real roots that change sign). Consequently there are at most two possible
Weyl denominator functions on each Cartan subgroup, differing by a sign.  

The absolute values appearing in the product above are in some sense
not so important. We will be using the Weyl denominator to describe
character formulas, and these formulas take their nicest form on the subset
$$H^+ = \{h \in H \mid\ |\alpha(h)| \ge 1 \ \ 
(\alpha \in R^+_{\mathbb R})\}.$$
(Every $W(R_{\mathbb R})$-orbit in $H$
must meet $H^+$.)  On $H^+$, all of the factors $1 - \alpha(h)^{-1}$ (for
$\alpha\in R^+_{\mathbb R}$) are nonnegative; so the absolute value
could be omitted there.

With the Weyl denominator functions in hand, we can begin to
understand the nature of the parameters appearing in the Langlands
classification of irreducible representations.  Perhaps the easiest
way to describe the Langlands classification (if not the easiest way
to prove it) is by 
Harish-Chandra's theory of distribution characters.  In order to talk
about distributions on manifolds, where there is no natural
identification of functions with smooth measures, it is useful to be a
little careful. Recall first of
all that a {\em test density} on a manifold $M$ is a compactly
supported complex-valued smooth measure $T$ on $M$. A {\it
  generalized function on $M$} is a continuous linear functional on
the space of test densities.  For example, any locally $L^1$
measurable function $f$ on $M$ defines a generalized function, by the
rule
$$f(T) = \int_M f(m)dT(m).$$
% (We can cover the compact support of $T$ by a finite number of the
% open sets on which $f$ is in $L^1$; then the integral is finite on
% each of these finitely many open sets.)

\begin{definition} 
\label{def:distnchar} %PT
Suppose $(\pi, V)$ is a continuous representation
  of a Lie group $G$.  Every test density $T$ on $G$ defines a
  continuous linear operator
$$\pi(T) = \int_G \pi(g) dT(g) \in \End(V).$$
These operators are typically much ``nicer'' (in terms of decay
properties) than the operators $\pi(g)$, because the integration
against $T$ is a kind of smoothing.  We say that $\pi$ has a {\em
  distribution character} if the operators $\pi(T)$ are all trace
class, and the linear map
$$\Theta_\pi(T) = \tr(\pi(T))$$
is continuous on test densities. (This requires that $V$ be a space on
which there is a reasonable theory of traces of endomorphisms; for our
purposes Hilbert spaces will suffice.)  In this case the generalized
function $\Theta_\pi$ is called the {\em distribution character} of
$\pi$. 
\end{definition}
Notice that if $\pi$ is finite-dimensional, then $\Theta_\pi$
is the ordinary continuous function $\tr(\pi(g))$ on $G$ . If $\pi$ is
infinite-dimensional and unitary, then the unitary operators $\pi(g)$
are never of trace class, so $\tr(\pi(g))$ is not defined for any
$g$. That is why the following theorem of Harish-Chandra is
remarkable: it attaches a meaning to $\tr(\pi(g))$ for most $g$.

\begin{theorem}[Harish-Chandra and others]\label{thm:distnchar}
  Suppose $G$ is a 
  real reductive Lie group, and $(\pi,V)$ is a quasisimple
  representation of finite length on a Hilbert space $V$.
\begin{enumerate}
\item For every test density $T$, the operator $\pi(T)$ is of trace
  class. The resulting linear functional on test densities is
  continuous, so that the character
$$\Theta_\pi(T) =_{\text{def}} \tr(\pi(T))$$
is a well-defined generalized function on $G$.
\item The generalized function $\Theta_\pi$ is given by a locally
  $L^1$ function on $G$, which is analytic on the open subset of
  regular semisimple elements.
\item Suppose $H \subset G$ is a Cartan subgroup. Fix a set
  $R^+_{\mathbb R}$ of positive roots for the system $R_{\mathbb R}$
  of real roots of $H$ in ${\mathfrak g}$.  Since each real root takes
  real values on $H$, we can define
$$H^+ = \{h \in H \mid |\alpha(h)| \ge 1 \ (\alpha \in R^+_{\mathbb
  R}) \}.$$
Write
  $\widetilde{H}$ for the $\rho_{\abs}$ double cover of $H$ (Lemma
  \ref{lemma:rhoimcover}), and $D$ for a Weyl denominator function on
  $\widetilde H$ (Definition \ref{def:Weylden}). Then the product
$D\cdot \Theta_\pi|_H$ (an analytic function on the set of
regular elements of $\widetilde{H}$) is equal {\em on
  $\widetilde{H}^+_{\text{reg}}$} to a finite integer combination of
characters of ${\widetilde{H}}$:
$$\Theta_\pi(\tilde h)= \Big(\sum_{\gamma \in \tildehatk{H}{1}}
%\widehat{\widetilde H}_1}
  a_\pi(\gamma) \gamma(\tilde h)\Big)\Big/D(\tilde h) \qquad (\tilde h \in
{\widetilde H}^+_{\text{reg}}).$$ 
\end{enumerate}
\end{theorem}
The subscript $1$ on $\tildehatk{H}{1}$ %$\widehat{\widetilde H}_1$ 
(Definition \ref{def:ncover}) is ``level $1$,'' meaning that the
characters take the value $-1$ on the central element $-1$ of the
cover. Therefore both the numerator and denominator are ``level $1$''
functions, so the quotient is ``level $0$'' and defines a function on
$H$.
\begin{proof}
  Part (1) is in \Cite{HCIII}*{pages 243--245}. The analyticity in part
  (2) is the main result of \Cite{HCchar}; now that distribution
  theory is a familiar tool, (2) is a routine ``elliptic regularity''
  argument. But the proof that $\Theta_\pi$ actually coincides with
  integration against this function is much deeper, and was achieved
  only in \Cite{HClocL1}. 

Part (3) is difficult to attribute properly. The fact that the
numerator looks locally like a trigonometric polynomial on $\widetilde
H$ is due to Harish-Chandra, and is proved in \Cite{HCchar}.
Harish-Chandra's arguments, which proceed by solving differential
equations, allow also terms with a polynomial dependence on
$\widetilde h$. That these polynomials are always constant was first
proved in \Cite{FS}.

Proving that the coefficients in (3) are actually integers is best
accomplished by realizing the numerator as the trace of the action of
$\tilde h$ on a virtual representation of $\widetilde H$.  To write
down such a virtual representation, fix first of all a Cartan
involution $\theta$ on $G$ preserving $H$. Let ${\mathfrak n}^- \subset
{\mathfrak g}$ be the sum of the {\em negative} root spaces for some
positive root system $R^+ \supset R^+_{\mathbb R}$. Then the Lie
algebra cohomology spaces $H^p({\mathfrak n}^-,V_K^\infty)$ with
coefficients in the Harish-Chandra module of $V$ (Definition
\ref{def:HCmod}) are finite-dimensional $({\mathfrak h},H\cap
K)$-modules. (The finite-dimensionality requires proof; this is
provided in \Cite{HS} and \Cite{IC2}.)  According to Theorem
\ref{theorem:quasisimple}, finite-dimensional $({\mathfrak h},H\cap
K)$-modules may be identified with representations of $H$.  The
numerator of the character formula in (3) is the character of the
virtual representation of $\widetilde H$
$${\mathbb C}_{\rho_{\abs}^-} \otimes\left(\sum_p (-1)^p H^p({\mathfrak
    n}^-,V_K^\infty) \right).$$
Here ${\mathbb C}_{\rho_{\abs}^-}$ is the space of the defining
character $\rho_{\abs}^-$ of $\widetilde H$ (Lemma
\ref{lemma:rhoimcover}), with respect to the positive root system
consisting of the roots $R^-$ in ${\mathfrak n}^-$.

The first very general result of this nature is the main theorem of
\Cite{HS}; that result allows one to reduce (3) to the
case of a compact Cartan subgroup $H$.  A parallel result
\Cite{IC2}*{Theorem 8.1} allows reduction of (3) to the
case of a split Cartan subgroup. Applying these two results in
succession (in either order!) proves (3).
\end{proof}

% insert another "interlude" section on infinitesimal characters?

\section{Langlands classification and the nonunitary
  dual}\label{sec:langlands} 
\setcounter{equation}{0} 

We can now give a precise formulation of the classification of the
nonunitary dual $\widehat G$.  This is extremely important, since we
are going to explain how to compute the {\em unitary} dual as a subset
of the {\em nonunitary} dual. The most familiar case of the Langlands
classification is for compact connected Lie groups: in that case the
irreducible representations are in one-to-one correspondence with
orbits of the Weyl group on the lattice of characters of a maximal
torus. (The correspondence sends an irreducible representation to its
set of extremal weights.)  We will use a slightly different version of
this correspondence, in which an irreducible representation of a
compact group corresponds to the (Weyl group orbit) of terms in the
numerator of the Weyl character formula. This modified correspondence
extends to real reductive groups fairly easily and completely.  Just
as for compact groups, it says (approximately) that the irreducible
representations are indexed by orbits of the Weyl groups on characters
of (the $\rho$-coverings of) maximal tori.  The implementation
(approximately) says that a representation corresponds to some character
appearing in the numerator of the Weyl character formula (Theorem
\ref{thm:distnchar}) for the representation. 

What is needed to make this precise is a way to pick out a
``preferred'' term in the Weyl character formula. (In the case of
compact $G$, all the terms belong to the same Weyl group orbit, so no
choice is necessary.) Langlands' idea in the noncompact case was to
pick a term with largest growth at infinity as the preferred term.
His original version in \Cite{LC} made a clean and precise connection
between growth and the classification; but at the same time lumped
together several different representations.  Subsequent work (notably
\Cite{KZ}) refined this into a precise classification of
representations, but at the expense of obscuring slightly the
connection with character formulas. Here is a statement.  There is a
closely related result on page 234 of \Cite{Ross}; but the emphasis
there is on reducibility of standard tempered representations, so it
is difficult to extract exactly what we want. The best reference we
know is \Cite{ABV}; but the formulation here is a bit different. The
main change is that we have modified the covering of $H$.  This change
means that we no longer need to keep track of a choice of positive
real roots; the definition of Langlands parameter here is accordingly
simpler than in \Cite{ABV}.
\begin{theorem}[Langlands classification; see \Cite{ABV}*{Chapter
    11}]\label{thm:LC} Suppose
  $G$ is the group of real points of a connected complex
  reductive algebraic group (cf.~\eqref{se:reductive}). Then there
  is a one-to-one correspondence between infinitesimal equivalence
  classes of irreducible quasisimple representations of $G$
  (Definition \ref{def:HCmod}) and $G$-conjugacy 
  classes of triples (``Langlands parameters'')
$$\Gamma = (H, \gamma, R^+_{i\mathbb R}),$$
subject to the following requirements.
\begin{enumerate}
\item The group $H$ is a Cartan subgroup of $G$: the
  group of real points of a maximal torus of $G({\mathbb C})$ defined
  over ${\mathbb R}$.
\item The character $\gamma$ is level one character of the $\rho_{\abs}$
  double cover of $H$ (Definition \ref{def:ncover}, Lemma
  \ref{lemma:rhoimcover}). Write $d\gamma \in {\mathfrak h}^*$ for
  its differential.%
\item The roots $R^+_{i\mathbb R}$ are a positive system for the
  imaginary roots of $H$ in ${\mathfrak g}$. 
\item The weight $d\gamma$ is weakly dominant for $R^+_{i\mathbb R}$.
\item If $\alpha^\vee$ is real and $\langle d\gamma,
  \alpha^\vee\rangle= 0$, then $\gamma_{\mathfrak q}(m_{\alpha}) = +1$
  (Definition \ref{def:rhoim}).  (Here $\gamma_{\mathfrak{q}}$ is a
  ``$\rho$-shift'' of $\gamma$ defined in \eqref{eq:gammaq} below.)
\item If $\beta$ is simple for $R^+_{i\mathbb R}$ and
  $\langle d\gamma,\beta^\vee\rangle = 0$, then $\beta$ is noncompact.
\end{enumerate}
This one-to-one correspondence may be realized in the following
way. Attached to each (equivalence class) of Langlands parameters
$\Gamma$ there is a ``standard (limit) $({\mathfrak g}_0,K)$-module''
$I(\Gamma) = I_{\quo}(\Gamma)$, which has finite length. This
module has a unique irreducible quotient module $J(\Gamma)$; the
correspondence is $\Gamma \leftrightarrow J(\Gamma)$.  (Alternatively,
there is a standard module $I_{\sub}(\Gamma)$ having $J(\Gamma)$
as its unique irreducible submodule. The two standard modules
$I_{\quo}(\Gamma)$ and $I_{\sub}(\Gamma)$ have the same
composition factors and multiplicities.)

The modules $J(\Gamma)$ and $I(\Gamma)$ both
have infinitesimal character $d\gamma \in {\mathfrak h}^*$.
% In case the infinitesimal character $d\gamma$ is
% regular\dots
% Was meant to be statement that then $\gamma$ appears in the
% numerator formula of Theorem \ref{thm:distnchar} for
% $\Theta_{J(\Gamma)}. Maybe later?

\end{theorem}

\begin{definition}\label{def:langlandsparam}
In the setting of Theorem \ref{thm:LC}, the set of $G$-conjugacy
classes of  triples $\Gamma =
(H,\gamma,R^+_{i{\mathbb R}})$ satisfying conditions (1)--(6) is
called the set of {\em Langlands parameters for $G$}, and written $\Pi(G)$.
\end{definition}
Occasionally we will need some generalizations of the Langlands
parameters in Theorem \ref{thm:LC}. Conditions (1)--(3) are
essential to make the construction of $I(\Gamma)$ described below;
conditions (4)--(6) are imposed so that $I(\Gamma)$ has additional
nice properties.  We record here some basic facts about these
generalizations.

\begin{definition}\label{def:genparam} 
In the setting of Theorem \ref{thm:LC}, a triple $\Gamma =
(H,\gamma,R^+_{i{\mathbb R}})$ satisfying conditions (1)--(3) is
called a {\em continued Langlands parameter}.  The set of equivalence
classes of such parameters is written $\Pi_{\cont}(G)$.

Attached to each $\Gamma \in \Pi_{\cont}(G)$ there is a {\em continued
  standard (virtual) $({\mathfrak g}_0,K)$-module $I(\Gamma)$} of
  finite length.  (The construction of $I(\Gamma)$ involves some
  cohomological functors, and in this setting the functors may be
  nonzero in several different degrees. The virtual representation
  $I(\Gamma)$ is the alternating sum of these functors.)

  If in addition $\Gamma$ satisfies condition (4), then $\Gamma$ is
  called a {\em weak Langlands parameter}. The set of weak parameters
  is written $\Pi_{\weak}(G)$.

Under condition (4) there is a vanishing theorem, and one can
construct {\em weak standard $({\mathfrak g}_0,K)$-modules
  $I_{\quo}(\Gamma)$ and $I_{\sub}(\Gamma)$}.  A weak standard module
is always a direct sum of a finite number (possibly zero) of standard
modules. For any weak Langlands parameter $\Gamma$, we can therefore
define {\em the Langlands quotient representation} $J(\Gamma)$ to be
the corresponding direct sum of irreducible representations; this is
the cosocle of $I_{\quo}(\Gamma)$, or equivalently the socle of
$I_{\sub}(\Gamma)$.

If $\Gamma$ does not satisfy condition (4), then we do not know any
reasonable way to define $J(\Gamma)$.
\end{definition}

\begin{subequations}\label{se:ds}
In order to understand where the restrictions on $\Gamma$ come from,
we recall Langlands' construction of $I(\Gamma)$.  Write $MA$
for the Langlands decomposition of the centralizer of $A$ in
$G$.  (This centralizer is the group of real points of a
connected complex reductive algebraic group, but the factors $M$ and
$A$ need not be.  Our chosen perspective would therefore prefer not to
make this factorization, but it is traditional and some of us are
nothing if not hidebound.) The compact group $T$ is a
compact Cartan subgroup of $M$, and the parameters
\begin{equation} \Lambda = (T, \gamma|_{T},
  R^+_{i\mathbb R}) \end{equation}
are (according to conditions (1--4) of Theorem \ref{thm:LC})
Harish-Chandra parameters for a limit of discrete series 
representation $D(\Lambda) \in \widehat{M}$.  A construction of
$D(\Lambda)$ by Zuckerman's cohomological induction functors is
a special case of Theorem \ref{thm:VZrealiz}. If ``weakly
dominant'' is strengthened to ``strictly dominant'' in condition (4), we get a
discrete series representation.  If condition 4 is dropped, then the
vanishing-except-in-one-degree statement of Theorem \ref{thm:VZrealiz}
can fail, so $D(\Lambda)$ becomes only a virtual representation
(see Definition \ref{def:genparam}).

For a weak parameter $\Gamma$ (that is, assuming condition (4))
condition (6) is exactly what is required to make 
$D(\Lambda)$ nonzero. (The fact that $D(\Lambda) = 0$ if
$d\lambda$ vanishes on a compact simple coroot is the simpler
of the eponymous character identities of \Cite{Sch2}.) So condition (6)
is needed to get a nonzero standard representation.

Now write
\begin{equation*}
\nu = \gamma|_A \in \widehat A
\end{equation*}
for the (possibly nonunitary) character of $A$ defined by $\Gamma$.
(The character $2\rho_{i{\mathbb R}}$ (sum of the positive imaginary
roots) is trivial on $A$. According to Lemma
\ref{lemma:rhoimcover}(3), it follows that the level one character 
$\gamma \in \tildehatR{H}$ 
% \widehat{\widetilde {H({\mathbb R})}}_1$ 
may be regarded as a pair $(\lambda,\nu)$, with $\lambda \in \tildehatR{T}$
% \big(\widetilde{T({\mathbb R})}\big)_1^{\widehat{}}
and $\nu \in \widehat{A}$.) 

Choose a parabolic subgroup $P = MAN$ of
$G$ making (the real part of) $\nu$ weakly dominant for
the weights of ${\mathfrak a}$ in ${\mathfrak n}$. Then the standard
representation of Theorem \ref{thm:LC} may be realized as
\begin{equation*}
I_{\quo}(\Gamma) = \Ind_{P}^{G} D(\Lambda)
\otimes \nu \otimes 1
\end{equation*}
(The dominance requirement does not specify $P$
uniquely, but all choices give isomorphic standard representations.
Replacing $P$ by the opposite parabolic subgroup $P^{\text{op}}$
replaces $I_{\quo}(\Gamma)$ by $I_{\sub}(\Gamma)$.)
If condition (5) fails in Theorem \ref{thm:LC}, then there is a real
coroot $\alpha^\vee$ with $\langle
d{\gamma},\alpha^\vee\rangle =0$; equivalently, $\langle\nu,
\alpha^\vee\rangle = 0$.  Possibly after modifying $P$, this coroot
defines a parabolic subgroup 
\begin{equation*}
P_\alpha = M_\alpha A_\alpha N_{\alpha}
\supset P.
\end{equation*}
Induction by stages realizes the standard representation $I(\Gamma)$
as
\begin{equation*}
I_{\quo}(\Gamma) = \Ind_{P_\alpha}^{G}
\left[\Ind_{M(A\cap M_\alpha)(N\cap M_\alpha)}D(\Lambda) \otimes 1
    % \nu|_{A\cap M_\alpha} 
\otimes 1\right] \otimes
  \nu|_{A_\alpha} \otimes 1.
\end{equation*}
The fact that the inner induction involves the trivial character $1$
of $A\cap M_\alpha$ is precisely the condition $\langle\nu,\alpha^\vee\rangle
= 0$. The inner induction to $M_\alpha$ is therefore unitary
induction. We claim that {\em when 
condition (5) above fails (that is, when $\gamma_{\mathfrak
  q}(m_\alpha) = -1$) then this 
induced representation may be written (using the first character
identity of \Cite{Sch2}) as a sum of one or two standard
representations attached to the (more compact) Cartan subgroup
obtained from $H$ by Cayley transform through the real
root $\alpha$}. In case $G = SL(2)$, this claim is just the familiar fact
that the nonspherical principal series representation with continuous
parameter zero is a sum of two limits of discrete series.  More generally, when
the Cartan subgroup $H$ is split, the group $M_\alpha$ is
locally isomorphic to $SL(2,{\mathbb R})$, and the claim again reduces
to $SL(2)$. In general, one can realize the standard representation
$I(\Gamma)$ by cohomological induction from a split Levi subgroup of
$G$ (see Section \ref{sec:stdmods}) and in this way reduce to the case
when $H$ is split. 

The conclusion is that when condition (5) fails, the standard
representation $I(\Gamma)$ may be realized as a direct sum of one or two
representations attached to weak parameters $\Gamma'$ on a Cartan
subgroup with strictly larger compact part.  That is, imposing
condition (5) seeks to make a canonical realization of $I(\Gamma)$: 
one attached to the most compact Cartan subgroup possible. These
remarks do not prove that the realization {\em is} canonical, but this
turns out to be the case.
\end{subequations}

This discussion of the realization of $I(\Gamma)$ suggests the
following definitions; we will use them constantly in analyzing
unitary representations.
\begin{definition}\label{def:dLP} Suppose $G$ is the group of real
  points of a connected complex reductive algebraic group
  (cf.~\eqref{se:reductive}). A {\em discrete Langlands parameter}
  is a triple $\Lambda = (T, \lambda,R^+_{i{\mathbb R}})$,
  subject to the following conditions.
\begin{enumerate}
\item The group $T$ is the maximal compact subgroup of a
  Cartan subgroup $H$ of $G$ (Proposition
  \ref{prop:realtorus}).
\item The character $\lambda$ is a level one character of the $\rho_{i{\mathbb R}}$
  double cover of $T$ (Definition \ref{def:ncover}, Lemma
  \ref{lemma:rhoimcover}). (Write $d\lambda \in {\mathfrak t}^*$ for
  the differential of $\lambda$.)
\item The roots $R^+_{i\mathbb R}$ are positive roots for the
  imaginary roots of $H$ in ${\mathfrak g}$. 
\item The weight $d \lambda$ is weakly dominant for $R^+_{i\mathbb R}$.
\item If $\beta$ is a simple root of $R^+_{i\mathbb R}$ and
  $\langle d\lambda, \beta^\vee\rangle =0$, then $\beta$ is noncompact.
\end{enumerate}
Two discrete Langlands parameters are called {\em equivalent} if they
are conjugate by $G$.  The set of equivalence classes of discrete
Langlands parameters is written
$$\Pi_{\text{disc}}(G).$$

Alternatively, one can define discrete Langlands parameters up to
$K({\mathbb C})$ conjugacy using the $\theta$-fixed subgroup
$T({\mathbb C})$ of a $\theta$-stable maximal torus $H({\mathbb
  C})$. The $K({\mathbb C})$-equivalence classes in this alternate
definition are in one-to-correspondence with the $G$-equivalence
classes above, the equivalence being implemented by 
choice of a $\theta$-stable representative in the $G$-class.

A discrete parameter $\Lambda$ is called {\em final} if in addition it
satisfies
\begin{enumerate}[resume]
\item If $\alpha^\vee$ is real, then $\lambda_{\mathfrak q}(m_\alpha)
  = 1$ (see \eqref{eq:gammaq}).
\end{enumerate}
We write $\Pi_{\text{fin,disc}}(G)$ for the set of final discrete parameters.
Suppose now that $\Lambda = (T, \lambda,R^+_{i{\mathbb
    R}})$ is a discrete Langlands parameter for $G$. A
{\em continuous parameter for $\Lambda$} is a pair $(A,\nu)$, subject
to the following conditions: 

\begin{enumerate}
\item The group $A$ is the vector part of a real Cartan subgroup
  $H$ with maximal compact $T$ (Proposition
  \ref{prop:realtorus}).
\item $\nu$ is a (possibly nonunitary) character of $A$;
  equivalently, $\nu \in \widehat{A} \simeq {\mathfrak a}^*$.
\item If $\alpha^\vee$ is real and $\langle\nu,\alpha^\vee\rangle = 0$,
  then $\lambda_{\mathfrak q}(m_{\alpha}) = +1$ (\eqref{eq:gammaq}).
\end{enumerate}
Two such continuous parameters are called {\em equivalent} if they are
conjugate by the stabilizer $N(\Lambda)$ of $\Lambda$ in $G({\mathbb
  R})$. 
\end{definition}

\begin{proposition}\label{prop:LCshape}
Fix a discrete Langlands parameter $\Lambda = (T,
\lambda,R^+_{i{\mathbb R}})$, and a real Cartan subgroup $H({\mathbb
  R}) = TA$.  Define
$$W^\Lambda = \{w \in W(G, H) \mid w\cdot
\Lambda = \Lambda\},$$
the stabilizer of $\Lambda$ in the real Weyl group (Definition
\ref{def:rhoim}).
\begin{enumerate}
\item Introduce a ${\mathbb Z}/2{\mathbb Z}$-grading on the real
  coroots $R^\vee_{\mathbb R}$ by (notation \eqref{eq:gammaq})
$$\text{$\alpha^\vee$ is odd} \iff \Lambda_{\mathfrak q}(m_\alpha) =
-1.$$
Then the set of equivalence classes of continuous parameters for
$\Lambda$ is
$$\left({\mathfrak a}^* - \bigcup_{\alpha^\vee\ \text{real odd}}
\ker(\alpha^\vee)\right)/W^\Lambda.$$
\item Write ${\mathcal L}(\Lambda)$ for the set of lowest $K$-types
  (\cite{Vgreen}*{Definition 5.4.18}) of
  the standard representations $I(\Lambda,\nu)$ (with $\nu$ a
  continuous parameter for $\Lambda$).  These $K$-types all have
  multiplicity one in $I(\Lambda,\nu)$, and they all appear in the
  Langlands quotients $J(\Lambda,\nu)$.
\item The element $0 \in \widehat A$ is a continuous parameter for
  $\Lambda$ if and only if $\Lambda$ is final. In this case ${\mathcal
    L}(\Lambda)$ consists of exactly one irreducible representation
  $\mu(\Lambda) \in \widehat K$.
\item The Langlands quotient $J(\Lambda,\nu)$ is tempered if and only
  if $\nu \in \widehat{A}_u \simeq i{\mathfrak a}_0^*$ is a unitary
  character. In this case (and still assuming $\nu$ is a continuous
  parameter for $\Lambda$) the standard representation
  $I(\Lambda,\nu)$ is unitary and irreducible.
\item The infinitesimal character of the Langlands quotient
  $J(\Lambda,\nu)$ is {\em real} (Definition \ref{def:realinf}) % \cite{Vgreen} % *{Definition 5.4.11}) 
if and only if $\nu  \in {\mathfrak a}_0^*$ is a real-valued character.
\item The standard representation $I(\Lambda,\nu)$ is reducible only
  if either
\begin{enumerate}[label=\alph*)]
\item there is a real root $\alpha$ such that 
$$\langle \nu,\alpha^\vee\rangle \in {\mathbb Z}\backslash \{0\},
\qquad \Lambda_{\mathfrak q}(m_\alpha) = (-1)^{\langle
  \nu,\alpha^\vee\rangle + 1};$$ 
or
\item there is a complex root $\delta$ such that $\langle
  (d{\lambda}, \nu),\delta^\vee\rangle \in {\mathbb Z}$,
  and 
$$\langle \nu, \delta^\vee\rangle > |\langle d{\lambda},
  \delta^\vee \rangle|.$$
\end{enumerate}
\end{enumerate}
\end{proposition} 
\begin{proof} The description of continuous parameters in (1) is just a
  reformulation of Theorem \ref{thm:LC}. The statement about lowest
  $K$-types in (2) largely comes from Theorem 6.5.12 in \Cite{Vgreen},
  although some translation of language is needed to get to this
  formulation.  The first assertion in (3) is immediate from (1), and
  the rest can be deduced from Theorem 6.5.12 in \Cite{Vgreen}. (This
  is the central idea in \Cite{VK}.) 

For (4), that the Langlands quotient is tempered if and only if the
continuous parameter is unitary is part of Langlands' original
result. The fact that the standard  representation is irreducible
(under the assumption that the unitary continuous parameter does not
vanish on any real odd root) is a very special case of (6). Part (5)
is immediate from (and in fact motivates) the definition of real
infinitesimal character. Finally
(6) is the main theorem of \Cite{SpehV}.
\end{proof}

The point of the proposition is to describe the shape of the
nonunitary dual of $G$.  It is a countable union (over
the set of equivalence classes of discrete Langlands parameters
$\Lambda$) of connected pieces. Each connected piece is a complex
vector space ${\mathfrak a}^*$, minus a finite collection of
hyperplanes through the origin (defined by the odd real coroots),
modulo the action of a finite group $W^\Lambda$.  (At the missing
hyperplanes, the standard representations can be rewritten as a direct
sum of one or two standard representations attached to a more compact
Cartan subgroup. So the ``missing'' representations still exist; they
are just attached to a different discrete Langlands parameter, and may
be reducible.)

The complex vector space ${\mathfrak a}^*$ is naturally defined over
${\mathbb R}$; the real subspace ${\mathfrak a}_0^*$ corresponds to
characters of $A$ taking positive real values. Indeed it is defined
over the integers: we can take as integer points
$$(X^*)^{-\theta} \simeq \text{$(-\theta)$-fixed algebraic characters of
  $H$ restricted to  $A$}.$$

\section{Second introduction: the shape of the unitary
  dual}\label{sec:secondintro}
\setcounter{equation}{0}
We now have enough machinery in place to explain the kind of description
we are going to find for the unitary dual of a real
reductive group $G$ (cf.~\eqref{se:reductive}).
\begin{theorem}[Knapp-Zuckerman; see \Cite{KOv}*{Chapter
  16}]\label{thm:unitarydual2}
Suppose $\Gamma = (\Lambda,\nu)$ is a Langlands parameter for an
irreducible quasisimple representation $J(\Gamma)$ for the real
reductive group $G$ (Theorem \ref{thm:LC}), corresponding
to the Cartan subgroup $H = TA$.  Define
$$W^\Lambda \subset W(G, H)$$
as in Proposition \ref{prop:LCshape}.

Write the continuous parameter $\nu$ as
$$\nu = \nu_{\re} + i\nu_{\im} \qquad (\nu_{\re},
\nu_{\im} \in {\mathfrak a}_0^*).$$
Define
$$L_{\im} = G^{\nu_{\im}},$$
the real Levi subgroup of $G$ corresponding to coroots vanishing on the
imaginary part of $\nu$.
\begin{enumerate}
\item The irreducible representation $J(\Gamma)$ admits an invariant
  Hermitian form if and only if there is a $w\in W^\Lambda$ such that
$$w\cdot \nu = -\overline{\nu};$$
equivalently
$$w\cdot\nu_{\im} = \nu_{\im}, \qquad
w\cdot\nu_{\re} = -\nu_{\re}.$$
The first condition here is equivalent to $w\in W_{L_{\im}}^{\Lambda}$; so we
conclude that $J(\Gamma)$ admits an invariant Hermitian form if and
only if the irreducible $J_{L_{\im}}(\Gamma)$ for $L_{\im}$ admits an
invariant Hermitian form.
\item The irreducible representation $J(\Gamma)$ is irreducibly
  induced from the representation $J_{L_{\im}}(\Gamma)$ on
  $L_{\im}$. In the presence of invariant Hermitian forms, this
  induction is unitary, so $J(\Gamma)$ is unitary if and only if
  $J_{L_{\im}}(\Gamma)$ is unitary.  This in turn is the case if
  and only if $J_{L_{\im}}((\Lambda,\nu_{\im}))$  is unitary.  
\end{enumerate}
\end{theorem}

We will explain thoroughly and precisely what an ``invariant Hermitian
form'' is beginning in Section \ref{sec:forms} (see especially
\eqref{se:gKclassicalinv}).  At that point we will explain some of the
proof of this theorem.

This theorem exhibits every unitary representation as unitarily
induced from a unitary representation with real continuous parameter;
equivalently, with ``real infinitesimal character'' (Definition
\ref{def:realinf}), in a precise and simple way. The question we need 
to address is {\em classification of irreducible unitary
  representations with real infinitesimal character}; that is,
determining {\em which representations with real continuous parameter are
unitary}.  

\begin{corollary} Fix a discrete Langlands parameter $\Lambda$,
  attached to a Cartan subgroup $H = TA$,
  and $W^\Lambda$ the stabilizer in the real Weyl group (Proposition
  \ref{prop:LCshape}).  The real continuous parameters attached to
  $\Lambda$ are
$$\left({\mathfrak a}_0^* - \bigcup_{\alpha^\vee\ \text{real odd}}
\ker(\alpha^\vee)\right)/W^\Lambda$$
(Proposition \ref{prop:LCshape}). 
For each $w\in W^\Lambda$, consider the $-1$ eigenspace
$${\mathfrak a}^*_w =_{\text{def}} \{\nu \in {\mathfrak a}^* \mid
w\cdot\nu = -\nu\},$$
and its real form ${\mathfrak a}^*_{w,0}$.
\begin{enumerate}
\item The Hermitian representations
$J(\Lambda,\nu)$ with real continuous parameter are precisely those with
$\nu \in {\mathfrak a}^*_{w,0}$, for some $w\in W^\Lambda$.
\item Fix $w\in W^\Lambda$.  Each of the hyperplanes of possible reducibility
  described in Proposition \ref{prop:LCshape} meets ${\mathfrak
  a}^*_{w,0}$ in a hyperplane, or not at all.  In this way we get a
locally finite hyperplane arrangement in ${\mathfrak a}^*_{w,0}$,
partitioning this vector space into a locally finite collection of
facets; each facet is characterized by various equalities
$$\langle(d{\lambda},\nu),\delta^\vee \rangle = m_\delta$$
and inequalities
$$\langle(d{\lambda},\nu), (\delta')^\vee \rangle <
\ell_{\delta'}$$
for roots $\delta,\delta'$ and integers $m_\delta$, $\ell_{\delta'}$.
\item The signature of the Hermitian form on $J(\Lambda,\nu)$ (for
  $\nu \in {\mathfrak a}^*_{w,0}$) is constant on each of the facets
  described in (2). 
\item For each $w\in W^\Lambda$ there is a compact rational polyhedron
  $P_w \subset {\mathfrak a}^*_{0,w}$ (a finite union of some of the facets of
  (2)) so that if $\nu$ is a continuous real parameter attached to
  $\Lambda$, then $J(\Lambda,\nu)$ is unitary if and only if $\nu$
  belongs to some $P_w$.
\end{enumerate}
\end{corollary}
Again, this corollary describes only what was known about the
structure of the unitary dual in the 1980s.

The main result of this paper is {\em an algorithm to calculate
  the signature of the invariant Hermitian form on $J(\Lambda,\nu)$
  for $\nu$ belonging to any specified facet of one of the vector
  spaces ${\mathfrak a}^*_{0,w}$}. (Exactly what is meant by
``calculate the signature'' will be explained in Section
\ref{sec:sigcharsc}.)  In particular, this algorithm decides whether or
not the facet consists of unitary representations.  As was explained
already in \Cite{Vu}, there is some predictable repetition in these
calculations; so doing a finite number of them is enough to determine
the unitary dual of $G$.

% \subsection{Invariant Hermitian forms}
% \setcounter{equation}{0}
\section{Hermitian forms on $({\mathfrak
    h},L({\mathbb C}))$-modules}\label{sec:forms} 
\setcounter{equation}{0}
We turn now to the study of invariant forms on
representations. The work that we do depends fundamentally on
generalizing and modifying the notion of ``invariant form.'' We will
therefore look very carefully and slowly at this notion.  A convenient
setting (allowing us to consider both group and Lie algebra
representations) is that of a ``pair.''

\begin{definition}[\Cite{KV}*{Definition 1.64}]\label{def:pair} A {\em
    pair} is $({\mathfrak h}, L({\mathbb C}))$, together with 
$$\Ad\colon L({\mathbb C}) \rightarrow \Aut({\mathfrak h}), \qquad i
\colon {\mathfrak l} \hookrightarrow {\mathfrak h},$$ 
subject to
\begin{enumerate}[label=\alph*)] 
\item ${\mathfrak h}$ is a complex Lie algebra
\item $L({\mathbb C})$ is a complex algebraic group, with (complex) Lie algebra
  ${\mathfrak l}$
\item $i$ is an inclusion of complex Lie algebras
\item $\Ad$ is an algebraic action by automorphisms of ${\mathfrak
    h}$, extending the adjoint action of $L({\mathbb C})$ 
  on ${\mathfrak l}$, and 
\item the differential of $\Ad$ is the adjoint
  action (by Lie bracket) of ${\mathfrak l}$ on ${\mathfrak h}$.
\end{enumerate}

A {\em representation} of the pair $({\mathfrak h},L({\mathbb C}))$
(cf.~Definition \ref{def:gK}) is a complex vector space $V$ endowed
with a complex-linear representation of the Lie algebra ${\mathfrak
  h}$ and an algebraic representation of the algebraic group
$L({\mathbb C})$, subject to the requirements
\begin{enumerate}[label=\alph*)]
\item the differential of the action of $L({\mathbb C})$ (a
  representation of ${\mathfrak l}$) is equal to the composition of $i$ with
  the action of ${\mathfrak h}$; and 
\item $$\ell \cdot(X\cdot v) = (\Ad(\ell)X)\cdot(\ell \cdot v), \qquad
  (\ell \in L({\mathbb C}), X\in {\mathfrak h}, v\in V).$$
\end{enumerate}

A {\em real structure} on the pair consists of two maps % (denoted by the
% same letter) 
% \begin{equation}
$$\sigma\colon {\mathfrak h} \rightarrow {\mathfrak h}, \qquad \sigma
\colon L({\mathbb C}) \rightarrow L({\mathbb C}),$$
% \end{equation}
both ``conjugate linear'' in the following sense. First, the map on
${\mathfrak h}$ is a {\em conjugate linear Lie algebra homomorphism}:
% \begin{equation}
$$\sigma(zX) = \overline{z}\sigma(X), \quad \sigma([X,Y]) =
[\sigma(X),\sigma(Y)] \qquad(z\in {\mathbb C}, \ X,Y \in {\mathfrak h}).$$
% \end{equation}
These conditions are all that is needed to define invariant
sesquilinear forms on representations; but to make sense of invariant
{\it Hermitian} forms, we must assume also that $\sigma$ is a
{\em complex structure} on ${\mathfrak h}$, meaning that
$$\sigma^2 = \Id.$$
Second, the map on $L({\mathbb C})$ 
$$\sigma\colon L({\mathbb C}) \rightarrow L({\mathbb C})$$
is a group homomorphism so that the corresponding map on functions
$$(\sigma^*f)(\ell) = \overline{f(\sigma(\ell))}$$
preserves regular functions on 
$L({\mathbb C})$:
$$\sigma^*\colon R(L({\mathbb C})) \rightarrow \overline{R(L({\mathbb
    C}))}.$$
(By construction $\sigma^*$ is a conjugate-linear homomorphism of
algebras.) Again this assumption is all that is needed to define
invariant sesquilinear forms, but to get invariant Hermitian forms we
must assume also that $\sigma$ is a real form of the complex algebraic
group $L({\mathbb C})$; that is, that
$$\sigma^2 = \Id.$$
\end{definition}

We will be interested most of all in the pair $({\mathfrak
  g},K({\mathbb C}))$ coming from a real
reductive algebraic group $G$, described in Theorem
\ref{thm:realforms}. We include more general pairs for the same
reason that they appear in \Cite{KV}: as a tool in the construction
of Harish-Chandra modules for $G$.

% \begin{subequations}\label{se:gKherm}

\begin{subequations}\label{se:hermform}
Suppose $V$ is a complex vector space (or a representation of a
complex Lie algebra ${\mathfrak h}$, or an algebraic representation
of an algebraic group $L({\mathbb C})$, or an $({\mathfrak
  h},L({\mathbb C}))$-module (Definition \ref{def:pair})).  Recall
the {\em Hermitian dual of $V$}
\begin{equation}\label{eq:Kinf} V^h = \{\xi \colon V \rightarrow {\mathbb C} \mid
  \xi(u+v) = \xi(u) + \xi(v), \xi(z\cdot v) = \overline{z}\cdot v\} 
\end{equation}
(We are going to modify this definition a little in the presence of an
algebraic group representation; see \eqref{eq:silentKfin} below.) 
It is clear that the Hermitian dual of $V$ consists of complex
conjugates of the linear functionals on $V$.
Taking Hermitian dual is a contravariant functor: if $T\in \Hom(V,W)$,
then the {\em Hermitian transpose of $T$} is
\begin{equation}
T^h \in \Hom(W^h,V^h), \quad T^h(\xi)(v) = \xi(T(v)) \qquad (v\in V,
\xi \in W^h).
\end{equation}
What makes everything complicated and interesting is that the
Hermitian transpose operation is {\em not} complex linear:
\begin{equation*}
(S\circ T)^h = T^h\circ S^h, \qquad (zT)^h = \overline{z}(T^h).
\end{equation*}

To see why this is a complication, suppose that $\rho\colon {\mathfrak
  h} \rightarrow \End(V)$ is a representation of the complex Lie
algebra ${\mathfrak h}$.  This means that $\rho$ is a complex-linear
Lie algebra homomorphism.  If we try to take Hermitian transpose, we
find
\begin{equation*}
\rho^h\colon {\mathfrak h} \rightarrow \End(V^h), \qquad \rho^h(X) =
-[\rho(X)]^h
\end{equation*}
is additive and respects the Lie bracket; but $\rho^h(zX) =
\overline{z}\rho^h(X)$, so $\rho^h$ is not complex linear.  

There is a parallel problem in case $V$ carries an algebraic
representation $\pi$ of an algebraic group $L({\mathbb C})$.  The group
homomorphism $\pi\colon L({\mathbb C}) \rightarrow GL(V)$ is assumed to be
algebraic; this means in particular that every matrix coefficient of
the homomorphism is an algebraic function on $L({\mathbb C})$. We can
get a group homomorphism
\begin{equation*}
\pi^h\colon L({\mathbb C}) \rightarrow GL(V^h), \qquad \pi^h(k) =
[\pi(k^{-1})]^h. 
\end{equation*}
But the matrix coefficients of $\pi^h$ include complex conjugates of
the matrix coefficients of $\pi$. Since the complex conjugates of
nonconstant algebraic functions are not algebraic, $\pi^h$ is almost
never an algebraic representation.

There is a smaller technical problem here. Part of what it means for
$V$ to be an algebraic representation of $L({\mathbb C})$ is that the
action is locally finite:
$$\dim \langle \pi(L({\mathbb C}))v \rangle < \infty \qquad (v\in V).$$
This condition is not inherited by the Hermitian dual as defined in
\eqref{eq:Kinf} above.  Whenever we have in mind an algebraic
representation of $L({\mathbb C})$ on $V$, it is therefore convenient
to modify that definition to the smaller space
\begin{equation}\label{eq:silentKfin}\begin{aligned} V^h =
    \{\xi \colon V 
    \rightarrow {\mathbb C} &\mid 
  \xi(u+v) = \xi(u) + \xi(v),\ \xi(z\cdot v) = \overline{z}\cdot v,\ \\
&\ \dim \langle \pi^h(L({\mathbb C}))\xi \rangle < \infty\}  
\end{aligned}\end{equation}
It would be more natural to retain the definition
\eqref{eq:Kinf} for $V^h$, and to denote this subspace by
$V^h_{L({\mathbb C})}$; but for the purposes of this paper, imposing
$L({\mathbb C})$-finiteness silently is more convenient.
\end{subequations}

\begin{subequations}\label{se:sesq}
We can now discuss forms. Recall that if $V$ and $W$ are complex
vector spaces, a {\em sesquilinear pairing between $V$ and $W$} is
\begin{equation}
\langle,\rangle \colon V \times W \rightarrow {\mathbb C},
\quad\text{linear in $V$, conjugate linear in $W$}.
\end{equation}
We write
\begin{equation}
\Sesq(V,W) = \{\text{sesquilinear pairings between $V$ and $W$}\}.
\end{equation}
A sesquilinear pairing on one vector space $V$ is called {\em
  Hermitian} if it satisfies
\begin{equation}
\langle v,v' \rangle = \overline{\langle v',v \rangle} \qquad (v,v'
\in V).
\end{equation}
\end{subequations}

\begin{lemma}\label{lem:formdic}
Suppose $V$ and $W$ are complex vector spaces.
\begin{enumerate}
\item Sesquilinear pairings between $V$ and $W$ may be naturally
  identified with linear maps from $V$ to the Hermitian dual $W^h$
  (cf.~\eqref{eq:Kinf}):
$$\Hom(V,W^h) \simeq \Sesq(V,W), \qquad T \leftrightarrow \langle
v,w\rangle_T =_{\text{def}} (Tv)(w).$$
\item Complex conjugation defines a (conjugate linear) isomorphism
$$\Sesq(V,W) \simeq \Sesq(W,V).$$
The corresponding (conjugate linear) isomorphism on linear maps is
$$\Hom(V,W^h) \simeq \Hom(W,V^h), \qquad T\mapsto T^h.$$
\item A sesquilinear form $\langle,\rangle_T$ on $V$ is Hermitian if
  and only if $T^h = T$.
\end{enumerate}
\end{lemma}

The statements of the lemma are just versions of the standard
dictionary between bilinear forms and linear maps, and the proofs are
short and straightforward. In (2), the Hermitian transpose $T^h$ was
defined as a map from $(W^h)^h$ to $V^h$.  But there is a natural
inclusion $W\hookrightarrow (W^h)^h$, and it is the corresponding
restriction of $T^h$ that is intended in (2).

\begin{subequations}\label{se:hermdual}  
We fix now a pair $({\mathfrak h},L({\mathbb C}))$ with a real
structure $\sigma$ (Definition \ref{def:pair}); we wish to define the
{\em $\sigma$-Hermitian dual}
\begin{equation}
V^{h,\sigma}
\end{equation}
of an $({\mathfrak h},L({\mathbb C}))$-module $V$. The space is just
the space $V^h$ defined in \eqref{eq:silentKfin}, of complex conjugates
of $L({\mathbb C})$-finite linear functionals on $V$.
To define the Lie algebra representation on $V^{h,\sigma}$, we write
$\rho$ for the complex Lie algebra representation on $V$:
\begin{equation*}
\rho\colon {\mathfrak h} \rightarrow \End(V).
\end{equation*}
In \eqref{se:hermform}, we have seen how to define a (complex
conjugate linear) Lie algebra representation 
\begin{equation*}
\rho^h\colon {\mathfrak h} \rightarrow \End(V^h).
\end{equation*}
(That this representation preserves the subspace of $L({\mathbb
  C})$-finite vectors is an elementary exercise.) In order to get a
complex linear Lie algebra representation, we simply 
compose with the conjugate linear Lie algebra homomorphism $\sigma$,
defining the {\em $\sigma$-Hermitian dual Lie algebra representation}
\begin{equation}
\rho^{h,\sigma}\colon {\mathfrak h} \rightarrow \End(V^h), \qquad
\rho^{h,\sigma}(X) = -[\rho(\sigma(X))]^h.
\end{equation}

The construction of the Hermitian dual of the algebraic group
representation
\begin{equation*}
\pi\colon L({\mathbb C}) \rightarrow GL(V).
\end{equation*}
is exactly parallel. We have introduced in \eqref{se:hermform} a
(nonalgebraic) group representation
\begin{equation*}
\pi^h\colon L({\mathbb C}) \rightarrow GL(V^h),
\end{equation*}
which is locally finite on $V^h$
(because of the redefinition in \eqref{eq:silentKfin}). The matrix
coefficients of $\pi^h$ are complex 
conjugates of matrix coefficients of $\rho$ (composed with the
inversion map on $L({\mathbb C})$), so they are not algebraic.  In
order to make them algebraic, we need to compose them with our fixed conjugate
linear algebraic group homomorphism
\begin{equation*}
\sigma\colon L({\mathbb C}) \rightarrow L({\mathbb C}).
\end{equation*} 
The {\em $\sigma$-Hermitian dual representation} is
\begin{equation}\label{eq:dualrep}
\pi^{h,\sigma}\colon L({\mathbb C}) \rightarrow GL(V^h), \qquad
\pi^{h,\sigma}(k) = [\pi(\sigma(k^{-1}))]^h.
\end{equation}
This is an algebraic representation of $L({\mathbb C})$. (The
hypothesis on $\sigma$ makes some nice matrix coefficients of
$\pi^{h,\sigma}$ into regular functions. The restriction to
$L({\mathbb C})$-finite vectors provides the remaining formal
properties making an algebraic representation.)

Of course we can apply this process twice, obtaining a
$\sigma$-Hermitian double dual representation on
$(V^{h,\sigma})^{h,\sigma}$. The representation on
$(V^{h,\sigma})^{h,\sigma}$ preserves the subspace $V$ (see the
remarks after Lemma \ref{lem:formdic}) and acts there by the original
representation {\em twisted by $\sigma^2$}. Because of our assumption
in Definition \ref{def:pair} that $\sigma^2= \Id$, it follows that the
inclusion $V\subset (V^{h,\sigma})^{h,\sigma}$ is an inclusion of
$({\mathfrak h},L({\mathbb C}))$-modules.

We define {\em $\sigma$-invariant sesquilinear pairings} between
$({\mathfrak h},L({\mathbb C}))$-modules $V$ and $W$: in addition to
the sesquilinearity condition in \eqref{se:sesq}, we require
\begin{equation*}\begin{aligned}
\langle X\cdot v,w\rangle &= \langle v, -\sigma(X)\cdot w \rangle,
\quad (X \in {\mathfrak h},\ \ell\in L({\mathbb C}),\  v\in V,\ w\in W)\\
\langle \ell\cdot v,w\rangle &= \langle v, \sigma(\ell^{-1})\cdot w \rangle
\quad (\ell \in L({\mathbb C}),\  v\in V,\ w\in W).
\end{aligned}
\end{equation*}
The dictionary in Lemma \ref{lem:formdic} extends immediately to 
\begin{equation}\label{eq:pairinv}
\Sesq^\sigma_{{\mathfrak h},L({\mathbb C})}(V,W) \simeq
\Hom_{{\mathfrak h}, L({\mathbb C})}(V,W^{h,\sigma}) 
% = \Hom_{{\mathfrak h},L({\mathbb C})}(V,W^{h,\sigma}).
\end{equation}
 (Since $V$ is locally $L({\mathbb C})$-finite, an intertwining
 operator necessarily takes values in $L({\mathbb C})$-finite
 vectors; this uses our assumed surjectivity of $\sigma$.) 
\end{subequations}

We can now define the most general invariant Hermitian forms that we
need.
% in the most general way
% relevant to this paper. 
\begin{definition}\label{def:invherm} Suppose $({\mathfrak
    h},L({\mathbb C}))$ is a pair with a real structure $\sigma$
  (Definition \ref{def:pair}), and $V$ is an $({\mathfrak
    h},L({\mathbb C}))$-module. A {\em $\sigma$-invariant Hermitian
    form on $V$} is a Hermitian pairing $\langle,\rangle$ on $V$,
  satisfying
$$\langle X\cdot v,w\rangle = \langle v, -\sigma(X)\cdot w \rangle
\qquad (X \in {\mathfrak h},\  v,w\in V),$$
and
$$\langle \ell\cdot v,w\rangle = \langle v, \sigma(\ell^{-1})\cdot w \rangle
\qquad (\ell \in L({\mathbb C}),\  v,w\in V).$$
Such a form may be identified with an intertwining operator
$$T\in \Hom_{{\mathfrak h},L({\mathbb C})}(V,V^{h,\sigma})$$
satisfying the Hermitian condition 
$$T^h = T.$$
\end{definition}

\begin{subequations}\label{se:classicalinvform}
Many examples of pairs arise from a real Lie group $H$ with a compact
subgroup $L$. In this case the Lie algebra ${\mathfrak h}$ is the
complexification of ${\mathfrak h}_0 = \Lie(H)$, so we have a
complex structure $\sigma_0$ on ${\mathfrak h}$:
\begin{equation}
\sigma_0(X+iY) = X - iY \qquad (X, Y \in {\mathfrak h}_0).
\end{equation}
At the same time $L({\mathbb C})$ is the complexification of $L$, and
this complexification construction defines the compact real form 
\begin{equation}
\sigma_0\colon L({\mathbb C}) \rightarrow L({\mathbb C}), \qquad
L({\mathbb R}, \sigma_0) = L 
\end{equation}
(Theorem \ref{thm:KC}).  These two maps $\sigma_0$ are a natural real
structure for the pair $({\mathfrak h},L({\mathbb C}))$. In this
setting, a $\sigma_0$-invariant sesquilinear pairing
$\langle,\rangle$ between $({\mathfrak h},L({\mathbb C}))$-modules $V$
and $W$ may be characterized by the simple requirements
\begin{equation}\begin{aligned}
\langle X\cdot v,w \rangle = \langle v, -X\cdot w\rangle, \qquad
&\langle k\cdot v,w \rangle = \langle v, k^{-1}\cdot w\rangle \\
&(X\in {\mathfrak h}_0,\ k \in L,\ v \in V,\ w \in W).\end{aligned}
\end{equation}
The reason we did not use this as the definition of an
invariant form is that it obscures the possibility of generalization
by changing $\sigma_0$, which will be crucial for us.
\end{subequations}

\begin{subequations}\label{se:multspace}
Suppose now that $V$ is an $({\mathfrak h},L({\mathbb C}))$-module {\em with
  $L({\mathbb C})$ reductive}. Then %  In this setting, if $V$ is
% an $({\mathfrak  h},L({\mathbb C}))$-module, 
\begin{equation*}
V = \sum_{\mu \in \widehat {L({\mathbb C})}} V(\mu),
\end{equation*}
with $V(\mu)$ the largest sum of copies of $\mu$ in $V$.  It
will be useful to refine this further. So fix a model $E_\mu$ for
each irreducible representation of $L({\mathbb C})$.  Define
\begin{equation}
V^\mu = \Hom_{L({\mathbb C})}(E_\mu,V).
\end{equation}
Then there is a natural isomorphism
\begin{equation}
V(\mu) \simeq V^\mu \otimes E_\mu, \qquad T\otimes e \mapsto T(e).
\end{equation}
The vector space $V^\mu$ is called the {\em multiplicity space of
  $\mu$ in $V$}, because of the obvious relation
\begin{equation}
\dim V^\mu = \text{multiplicity of $\mu$ in $V$} =_{\text{def}}
\mult_V(\mu).
\end{equation}

The $({\mathfrak h},L({\mathbb C}))$-module $V$ is called {\em admissible} if
$\mult_V(\mu) < \infty$ for all $\mu$.

In terms of the multiplicity spaces, we can calculate Hermitian dual spaces:
\begin{equation*}
V(\mu)^h \simeq V^{\mu,h} \otimes E_\mu^h.
\end{equation*}

We fix now a real structure $\sigma$ on the pair (Definition
\ref{def:pair}). The real structure $\sigma$ on $L({\mathbb C})$ defines via
\eqref{eq:dualrep} an algebraic representation
$\sigma(\mu)$ on $E_\mu^h$.  The assumed surjectivity of $\sigma$ on
$L({\mathbb C})$ implies that $\sigma(\mu)$ is irreducible. The map
$\mu \mapsto \sigma(\mu)$ of irreducible representations of
$L({\mathbb C})$ is bijective, and $\sigma^2(\mu)$ is naturally
isomorphic to $\mu$. We therefore find
\begin{equation*}
V^h = \sum_{\mu \in \widehat{L({\mathbb C})}} V^{\mu,h} \otimes E_{\sigma(\mu)}.
\end{equation*}
(This uses the redefinition of Hermitian dual in
\eqref{eq:silentKfin}. Without that, we would instead have the direct {\em
  product} over $\mu$ as the Hermitian dual of the direct sum.)

We will often (but not always) be interested in the special case that
\begin{equation}\label{eq:Lcpt}
\text{$\sigma$ is a compact real form of $L({\mathbb C})$}
\end{equation}
with (compact) group of real points $L$.  In this special case we fix
also an $L$-invariant unitary structure---a positive definite 
$L$-invariant Hermitian form---on each space $E_\mu$. This amounts
to a choice of isomorphism 
$$E_{\sigma(\mu)} \simeq E_\mu.$$
In particular, we must have $\sigma(\mu) \simeq \mu$ whenever
\eqref{eq:Lcpt} holds. Since $E_\mu$ and $E_{\sigma(\mu)}$ are irreducible,
the space of such 
isomorphisms has complex dimension one. Requiring the form to be
Hermitian defines a real line in that complex line, and the positivity
defines a real half line.
\end{subequations}

\begin{proposition}\label{prop:sigchar} Suppose $({\mathfrak
    h},L({\mathbb C}))$ is a pair 
  with $L({\mathbb C})$ reductive, and $\sigma$ is a real
  structure (Definition \ref{def:pair}).   
\begin{enumerate}
\item Suppose
  $V$ is an admissible $({\mathfrak h},L({\mathbb
    C}))$-module (see \eqref{se:multspace}).  Then the
  $\sigma$-Hermitian dual module $V^{h,\sigma}$ is again admissible,
  with multiplicity spaces 
$$(V^{h,\sigma})^{\sigma(\mu)} \simeq (V^\mu)^h.$$
\item The space of $\sigma$-invariant sesquilinear forms $\langle,\rangle_T$ on
  $V$ may be identified (by \eqref{eq:pairinv}) with the % complex vector
  space of intertwining operators
  $\Hom_{{\mathfrak h}, L({\mathbb C})}(V,V^{h,\sigma})$.
This complex vector space has a real structure
$\overline{T} =_{\text{def}} T^h$; the subspace of real points
consists of the Hermitian $\sigma$-invariant forms.
\item Suppose $J$ is an admissible {\em irreducible} $({\mathfrak h},L({\mathbb
    C}))$-module (see \eqref{se:multspace}).  Then the
  $\sigma$-Hermitian dual module $J^{h,\sigma}$ is again admissible
  and irreducible. Consequently $J$ admits a nonzero
invariant sesquilinear form if and only if $J\simeq J^{h,\sigma}$.  In
this case the form is necessarily nondegenerate 
% (so that $\RAD_J = 0$) 
and unique up to a nonzero complex scalar multiple. 
\end{enumerate}
Assume from now on, as in \eqref{eq:Lcpt}, that $\sigma$ defines a compact real
form of $L$.
\begin{enumerate}[resume]
\item Fix a $\sigma$-invariant Hermitian form $\langle,\rangle_V$ on
  $V$. The decomposition
$$V = \sum_{\mu\in \widehat{L}} V^\mu\otimes E_\mu$$ 
(cf.~\eqref{se:multspace}) is orthogonal. The chosen
unitary structure $\langle,\rangle_\mu$ on each irreducible
representation $\mu$ of $L$ gives rise to a natural Hermitian form
$\langle,\rangle^\mu_V$ on each multiplicity space $V^\mu$ so that  
$$(V,\langle,\rangle_V) = \sum_{\mu\in \widehat{L}}
(V^\mu,\langle,\rangle_V^\mu) \otimes
(E_\mu, \langle,\rangle_\mu).$$ 
\item Define
\begin{equation*}\begin{aligned}
\POS_V(\mu) &= \text{dimension of maximal positive definite subspace
  of $V^\mu$}\\
\NEG_V(\mu) &= \text{dimension of maximal negative definite subspace
  of $V^\mu$}\\
\RAD_V(\mu) &= \text{dimension of radical of $V^\mu$.}\\
\end{aligned}\end{equation*}
Then
$$(\POS_V,\NEG_V,\RAD_V)\colon \widehat{L} \rightarrow {\mathbb
  N}\times {\mathbb N}\times {\mathbb N}, \quad \POS_V + \NEG_V
+\RAD_V = \mult_V.$$  
Here $\mult_V(\mu)$ is the multiplicity of $\mu$ in $V$
(cf.~\eqref{se:multspace}).
\item Assume again that $J$ is irreducible and that $J\simeq
  J^{h,\sigma}$, so that there is a nondegenerate invariant Hermitian
  form $\langle,\rangle_J$. Multiplying this form by a negative scalar
  interchanges  $\POS_J$ and $\NEG_J$. The unordered pair of functions
  $\{\POS_J,\NEG_J\}$ is independent of the choice of form.
\end{enumerate}
\end{proposition}

All of this is straightforward and elementary.

We call the triple of functions $(\POS_V,\NEG_V,\RAD_V)$ the {\em
  signature character of $(V,\langle,\rangle_V)$}. The form is
nondegenerate if and only if  $\RAD_V=0$; in that case we will often
omit it from the notation. 

\begin{corollary}
In the setting of Proposition \ref{prop:sigchar}, the admissible
irreducible $({\mathfrak h},L({\mathbb C}))$-module $J$ admits a
positive definite $\sigma$-invariant Hermitian form if and only if
\begin{enumerate}
\item $J$ is equivalent to the Hermitian dual irreducible
  $J^{h,\sigma}$, and
\item one of the functions $\POS_J$ or $\NEG_J$ is identically zero.
\end{enumerate}
\end{corollary}

This corollary begins to suggest why one can hope to calculate the
unitary dual: the problem is not whether some matrix is positive
definite (which sounds analytic, difficult, and subject to subtle
error analysis) but whether some nonnegative integers are zero
(which sounds combinatorial).  In Section \ref{sec:sigcharsc} we will describe
techniques from \Cite{Vu} for reducing further to a computation of
{\em finitely many} integers.

We will need to consider Hermitian forms invariant under several
different real structures on the same pair.  Here is some machinery
for doing that.

\begin{definition}\label{def:pairaut} Suppose $({\mathfrak
    h},L({\mathbb C}))$ is a pair (Definition \ref{def:pair}). An {\it
    automorphism $\delta$ of the pair} consists of a complex Lie algebra
  automorphism
$$\delta\colon{\mathfrak h} \rightarrow {\mathfrak h}$$
and an algebraic group
  automorphism 
$$\delta\colon L({\mathbb C}) \rightarrow L({\mathbb C}),$$
subject to the requirement that the differential of the group
automorphism is equal to the restriction to ${\mathfrak l}$ of the Lie
algebra automorphism.

If $V$ is a representation of the pair (Definition \ref{def:pair}),
then the {\it twist of $V$ by $\delta$} is the representation
$V^\delta$ on the same vector space $V$ defined by
$$X\cdot_{\delta} v = \delta(X) \cdot v, \quad \ell\cdot_{\delta} v =
\delta(\ell) \cdot v \quad (X \in {\mathfrak h},\ \ell \in L({\mathbb
  C})).$$

Twisting defines a {\it right} action of the group of automorphisms on
representations:
$$(V^{\delta_1})^{\delta_2} = V^{\delta_1\circ \delta_2}.$$

A {\it $\delta$-form of $V$} is an isomorphism
$$D\colon V \rightarrow V^\delta.$$
That is, $D$ is a linear automorphism of the vector space $V$
satisfying
$$D(X\cdot v) = \delta(X)\cdot D(v), \quad D(\ell \cdot v) =
\delta(\ell)\cdot D(v) \quad (X \in {\mathfrak h},\ \ell \in L({\mathbb
  C})).$$

This last definition in particular may look more transparent if we use
representation notation: writing $\pi$
for the representations of ${\mathfrak h}$ and $L({\mathbb C})$ on
$V$, and $\pi^\delta$ for their twist by $\delta$, the condition is
$$D\pi = \pi^\delta D;$$
that is, that $D$ intertwines the representations $\pi$ and
$\pi^\delta$. 
\end{definition}

\begin{definition}\label{def:extpair} Suppose $({\mathfrak
    h},L({\mathbb C}))$ is a pair, and $\delta$ is an automorphism
with the property that 
$$\delta^m = \Ad(\lambda) \qquad (\lambda \in L({\mathbb C}))$$
is an inner automorphism (Definition \ref{def:pairaut}). The
  corresponding {\it extended group} is the extension 
$$1\rightarrow L({\mathbb C}) \rightarrow {}^\delta L({\mathbb C}) \rightarrow {\mathbb
  Z}/m{\mathbb Z} \rightarrow 1$$
with a specified generator $\delta_1$ mapping to $1\in {\mathbb
  Z}/m{\mathbb Z}$, and subject to the relations 
$$\delta_1\ell \delta_1^{-1} = \delta(\ell) \qquad (\ell \in L({\mathbb
  C})),$$
and
$$\delta_1^m = \lambda.$$
This extended group acts by automorphisms on ${\mathfrak h}$, and we
call $({\mathfrak h}, {}^\delta L({\mathbb C}))$ an {\it extended pair}.

It is convenient to abuse notation by writing $\delta$ for the
distinguished generator $\delta_1$ %of the ${\mathbb Z}/m{\mathbb Z}$
                                % factor, so that 
$${}^\delta L({\mathbb C}) = L({\mathbb C})\{e,\delta,
\delta^2,\dots,\delta^{m-1}\},$$
a union of $m$ cosets of $L({\mathbb C})$.

We write $\chi$ for the one-dimensional $({\mathfrak h}, {}^\delta
L({\mathbb C}))$-module ${\mathbb C}_\chi$ defined by
$$X\cdot z = 0, \quad \ell\cdot z = z, \quad \delta\cdot z = e^{2\pi
  i/m}z \qquad (X\in {\mathfrak h},\ \ell\in L({\mathbb C}),\ z\in
{\mathbb C}_\chi).$$
\end{definition}

\begin{proposition}\label{prop:extpairrep} Suppose we are in the setting of
  Definition \ref{def:extpair}.
\begin{enumerate}
\item Twisting by $\delta$ defines a permutation of order (dividing)
  $m$ of the set of irreducible $({\mathfrak h},L({\mathbb
    C}))$-modules. Tensoring with $\chi$ defines a permutation of
  order (dividing) $m$ of the set of irreducible $({\mathfrak
    h},{}^\delta L({\mathbb C}))$-modules. 
\end{enumerate}
Now fix an irreducible $({\mathfrak h},L({\mathbb
    C}))$-module $V$, and write $\Lambda \in \Aut(V)$ for the action
  of $\lambda = \delta^m$.
\begin{enumerate}[resume]
\item The module $V$ is
  fixed by $\delta$ if and only if there is a $\delta$-form $D$ of $V$
  (Definition \ref{def:pairaut}); that is, an intertwining operator
  from $V$ to $V^\delta$. 
\item If $D$ exists for $V$, then it is unique up to multiplication by
  a complex scalar; and $D^m$ must be a scalar multiple of $\Lambda$.
\item If $D$ exists for $V$, then we can arrange $D^m = \Lambda$. With
  this constraint, $D$ is unique up to multiplication by an $m$th root
  of $1$.
\item The irreducible $({\mathfrak h},L({\mathbb C}))$-module $V$
  extends to $({\mathfrak h},{}^\delta L({\mathbb C}))$ if and only if
  $V\simeq V^\delta$. In this case there are exactly $m$ extensions,
  which are cyclically permuted by tensoring with the character
  $\chi$ of Definition \ref{def:extpair}. The actions of $\delta$
  are given by the $m$ $\delta$-forms $D$ described in (4). 
\item The irreducible $({\mathfrak h},L({\mathbb C}))$-module $V$
  induces irreducibly to $({\mathfrak h},{}^\delta L({\mathbb C}))$ if
  and only if the $m$ twists $V, V^\delta,\ldots,V^{\delta^{m-1}}$ are
  all inequivalent. In this case
$$\Ind_{({\mathfrak h},L({\mathbb C}))}^{({\mathfrak
      h},{}^\delta L ({\mathbb C}))}(V) = V\oplus V^\delta\oplus
  \cdots\oplus V^{\delta^{m-1}}$$
is the unique irreducible $({\mathfrak h},{}^\delta L({\mathbb
  C}))$-module containing $V$. This induced module is isomorphic to
its tensor product with the character $\chi$ of Definition
\ref{def:extpair}. 
\end{enumerate}
\end{proposition}

This analysis of representations of cyclic group extensions is often
referred to as ``Clifford theory'' (\cite{Cliff}*{pages 547--548}),
although much of it goes back at least to Frobenius.  

\begin{definition}\label{def:typeext}
In the setting of Definition \ref{def:extpair}, the
% $({\mathfrak h}, {}^\delta L({\mathbb C}))$-
modules in (5) are called {\em type $1$}, and those in
(6) {\em type $m$}, for the multiplicity on restriction to
$({\mathfrak h}, L({\mathbb C}))$. In the same way, we call an
irreducible $({\mathfrak h}, L({\mathbb C}))$-module {\em type one} if
its equivalence class is fixed by $\delta$, and {\em type $m$} if the
orbit under twisting has order $m$.

There is no difficulty in analyzing the cases % between (5) and (6), 
when the equivalence class of $V$ is fixed by some power
$\delta^{m_0}$ (with $m_0$ a proper divisor of $m$). Since we will be
interested in the case $m=2$, we have omitted these {\em type $m_0$}
cases.
\end{definition}

We turn now to an analysis of invariant forms for
extended groups.

\begin{lemma}\label{lem:changeform} Suppose $({\mathfrak
    h},L({\mathbb C}))$ is a pair endowed with a real structure
  $\sigma$ (Definition \ref{def:pair}); we do not require that
  $\sigma^2=1$. Suppose also that
\begin{enumerate}[label=\alph*)]
\item $\epsilon$ and $\delta$ are two automorphisms of the pair
  (Definition \ref{def:pairaut});
\item $V$, $W$, $V'$, and $W'$ are four $({\mathfrak h},L({\mathbb
    C}))$-modules;
\item $\langle \cdot,\cdot\rangle\colon V\times W \rightarrow {\mathbb
    C}$ is a $\sigma$-invariant sesquilinear pairing
  (\eqref{se:hermform}); and
\item $D\colon V' \rightarrow V^\delta$ and $E\colon W'\rightarrow
  W^\epsilon$ are $({\mathfrak h},L({\mathbb C}))$-maps to the
  indicated twisted modules (Definition \ref{def:pairaut}).
\end{enumerate}
Then the form
$$\langle\cdot,\cdot\rangle'\colon V' \times W' \rightarrow {\mathbb
  C}, \quad\langle v',w'\rangle' =_{\text{def}} \langle Dv',Ew'\rangle
\quad (v' \in V', \ w' \in W')$$
is an $\epsilon^{-1}\sigma\delta$-invariant sesquilinear pairing.
\end{lemma}

This is immediate from the definitions.

\begin{proposition}\label{prop:changeform} Suppose $({\mathfrak
    h},L({\mathbb C}))$ is a pair endowed with two real structures
  $\sigma_1$ and $\sigma_2$ (Definition \ref{def:pair}); we assume
  that $\sigma_i^2 = 1$. Suppose $V$ is an admissible $({\mathfrak
    h},L({\mathbb C}))$-module, and that $V$ is endowed with a
  nondegenerate $\sigma_1$-invariant Hermitian form
  $\langle,\rangle^{\sigma_1}$ that is unique up to real multiple. Then 
\begin{enumerate}
\item The maps
$$ \delta = \sigma_1^{-1}\sigma_2, \qquad \epsilon =
\sigma_2\sigma_1^{-1} $$% \delta^1_2 = \sigma_1^{-1}\sigma_2$$
define automorphisms of the pair (Definition \ref{def:pairaut}). 
\item Since $\sigma_i^2 = 1$, $\delta$ and $\epsilon$ are inverse to
  each other. It is also the case that they are conjugate by
  $\sigma_1$: $\sigma_1\delta\sigma_1^{-1} = \epsilon$.
\item The $\sigma_1$- and $\sigma_2$-hermitian duals of $V$ differ by
  twisting by $\delta$ or $\epsilon$:
$$V^{h,\sigma_2}\simeq [V^{h,\sigma_1}]^{\delta}, \qquad
[V^{\epsilon}]^{h,\sigma_1} \simeq V^{h,\sigma_2}.$$
(notation \eqref{eq:dualrep} and Definition \ref{def:pairaut}).
\item The following conditions are equivalent.
\begin{enumerate}
\item There is a nondegenerate $\sigma_2$-invariant form
  $\langle,\rangle^{\sigma_2}$ on $V$. In this case the form is
  unique up to a real multiple.
\item $V$ is equivalent to its twist by $\delta$, by an
  isomorphism 
$$D\colon V \rightarrow V^{\delta}$$
(cf. Definition \ref{def:pairaut}). In this case the isomorphism $D$
is unique up to a complex multiple.
\item $V$ is equivalent
to its twist by $\epsilon$, by an isomorphism 
$$E \colon V \rightarrow V^{\epsilon}.$$
 In this case the isomorphism $E$ is unique up to a complex
 multiple. One candidate for $E$ is $D^{-1}$.
\end{enumerate}
Assume now that $\delta^m = \Ad(\lambda)$ ($\lambda\in
L({\mathbb C})$) as in Definition \ref{def:extpair}; equivalently,
that $\epsilon^m = \Ad(\lambda^{-1}) = \Ad(\sigma_1(\lambda))$.  Write
$\Lambda$ for the 
action on $V$ of $\lambda$, and $\Lambda_1$ for the action of
$\sigma_1(\lambda)$. Then conditions (a)--(c) are also
equivalent to the following.
\begin{enumerate}[resume]
\item The $({\mathfrak h},L({\mathbb C}))$-module $V$ extends to an
  $({\mathfrak h},{}^{\delta}L({\mathbb C}))$-module, with
  $\delta$ acting by an isomorphism
$$D\colon V \rightarrow V^{\delta}$$
as in (b), subject to the additional requirement $D^m = \Lambda$. If
any such extension exists then there are exactly $m$, the
operators $D$ differing by multiplication by an $m$th root of 1.
\item The $({\mathfrak h},L({\mathbb C}))$-module $V$ extends to an
  $({\mathfrak h},{}^{\epsilon}L({\mathbb C}))$-module, with
  $\epsilon$ acting by an isomorphism
$$E\colon V \rightarrow V^{\epsilon}$$
as in (b), subject to the additional requirement $E^m =
\Lambda_1$. If any such extension exists then there are exactly $m$, the
operators $E$ differing by multiplication by an $m$th root of 1. One
candidate for $E$ is the inverse Hermitian transpose of $D$ (defined using
$\langle,\rangle^{\sigma_1}$).
\end{enumerate}
\item If the automorphism $\delta = \Ad(d)$ is inner (that is, given by the
  adjoint action of an element $d\in L({\mathbb C})$ with $d^m =
  \lambda$) then the conditions of (4) are automatically satisfied;
  the intertwining operator $D$ may be taken as the action of $d$, and
  $E$ as the action of $\sigma_1(d)$.
\end{enumerate}
Assume now that the conditions of (4) are satisfied, that
$\delta^m = \Ad(\lambda)$, and that the operators $D$ and $E$
are chosen as in (4)(d) and (4)(e).
\begin{enumerate}[resume]
\item There is a nonzero complex number $\xi(D,E)$ so that
$$E^{-1} = \xi(D,E)D.$$
The scalar $\xi(D,E)$ is an $m$th root of 1.
\item The form
$$\langle v, w\rangle' =_{\text{def}} \langle
Dv,Ew\rangle^{\sigma_1}$$
is a $\sigma_1$-invariant sesquilinear form on $V$
(\eqref{eq:pairinv}), and by hypothesis in (4) is therefore a multiple
of $\langle,\rangle^{\sigma_1}$:
$$\langle Dv,Ew\rangle^{\sigma_1} = \omega(D,E)\langle
v,w\rangle^{\sigma_1}.$$
Equivalently,
$$\langle Dv,w\rangle^{\sigma_1} = \omega(D,E)\langle
v,E^{-1}w\rangle^{\sigma_1} = \omega(D,E)\overline{\xi(D,E)} \langle
v,Dw\rangle^{\sigma_1}.$$ 
\item The scalar $\omega(D,E)$ satisfies
$$\omega(D,E)^m=1, \quad \omega(\tau_1 D,\tau_2E) =
\tau_1\tau_2^{-1}\omega(D,E) \quad
(\tau_i \in {\mathbb C}^\times, \ \tau_i^m = 1).$$
\item If $\zeta$ is any square root of $\omega(D,E)\overline{\xi(D,E)}$, then
$$\langle v,w\rangle^{\sigma_2} =_{\text{def}} \zeta^{-1}\langle
Dv,w\rangle^{\sigma_1} = \zeta \langle
v,Dw\rangle^{\sigma_1} $$
is a $\sigma_2$-invariant Hermitian form on $V$.
\end{enumerate}
\end{proposition}

\begin{proof}
Parts (1), (2), and (3) are immediate from the definitions. The equivalence
of (a)--(c) in (4) is also easy, and then (d)--(e) is a consequence of
Proposition \ref{prop:extpairrep}. 

The existence of the scalar in part (6) follows from the
description of $E$ in (4)(c). Raising this equation to the $m$th power
gives
$$\Lambda_1^{-1} = \xi(D,E)^m \Lambda.$$
The operators $\Lambda$ and $\Lambda_1^{-1}$ are Hermitian
transposes of each other with respect to the form
$\langle,\rangle^{\sigma_1}$. Taking Hermitian transpose of the last
equation gives
$$\Lambda = \overline{\xi(D,E)}^m \Lambda_1^{-1}.$$
Combining these two equations gives $|\xi(D,E)|^{2m} = 1$, which is the
last claim in (6).

Part (7) follows from Lemma \ref{lem:changeform}.

For (8), applying the formula in (7) $m$ times gives
$$\langle D^m v,w\rangle^{\sigma_1} = \omega(D,E)^m \langle
v,E^{-m}w\rangle^{\sigma_1}.$$
Again using the fact that the operators $D^m =\Lambda$ and
$E^{-m}=\Lambda_1^{-1}$ are Hermitian transposes of each other, it follows that
$\omega(D,E)^m = 1$. The formula with $\tau_i$ is easy.

For (9), that $\langle v,w\rangle^{\sigma_2} = \langle
Dv,w\rangle^{\sigma_1}$ is $\sigma_2$-invariant and sesquilinear
follows from Lemma \ref{lem:changeform}. That the indicated multiple
is Hermitian is an easy consequence of the last formula
$$\zeta^{-1}\langle Dv,w\rangle^{\sigma_1} = \zeta\langle
v,Dw\rangle^{\sigma_1} = \overline{\zeta^{-1}\langle
  Dw,v\rangle^{\sigma_1}}$$ 
in (7).  (We need to use that $\overline\zeta = \zeta^{-1}$, which
follows from (6) and (8).)
\end{proof}

\section{Interlude: realizing standard modules}\label{sec:stdmods}
\setcounter{equation}{0}
We described briefly in \eqref{se:ds} Langlands' original realization
of the standard $({\mathfrak g}_0,K)$-modules using parabolic induction
from discrete series representations. From time to time we are going
to need other realizations of the standard modules, notably those
found by Zuckerman and described in \Cite{Vgreen}.  The purpose of
this section is to collect some general facts about these other
realizations, which require some time to state precisely. The reader
will probably prefer to pass over this section, and consult it only as
necessary later on.

The main theorems of this section have their origins in
\Cite{Vgreen}*{Theorem 6.6.15} (which is not proven there) and
\Cite{IC3}*{Theorem 1.13} (which considers only integral infinitesimal
character).  Results 
in the generality that we need, with proofs, may be found in
\Cite{KV}*{Chapter 11}; but translating the formulation there into what
we want is still a bit of an exercise, most of which is silently left
to the reader.

\begin{subequations}\label{se:stdLparam}
We begin as in \eqref{se:ds} with a $\theta$-stable Cartan subgroup
\begin{equation}
H = TA \subset G, \qquad T = H^\theta = H\cap K, \qquad A =
\exp({\mathfrak h_0}^{-\theta}).
\end{equation}
and a weak Langlands parameter (Definition \ref{def:genparam})
\begin{equation}
\Gamma = (H,\gamma, R^+_{i{\mathbb R}}) = (\Lambda,\nu)
\end{equation}
% It is useful to work with slightly weaker hypotheses than in
% \eqref{se:ds}: we will call $\Gamma$ a {\em standard limit parameter} if
% 
% \begin{enumerate}
% \item The character $\gamma$ is level one character of the $\rho_{\abs}$
%   double cover of $H$ (Definition \ref{def:ncover}, Lemma
%   \ref{lemma:rhoimcover}). Write $d\gamma \in {\mathfrak h}^*$ for
%   its differential.
% \item The roots $R^+_{i\mathbb R}$ are positive roots for the
%   imaginary roots of $H$ in ${\mathfrak g}$. 
% \item The weight $d\gamma$ is weakly dominant for $R^+_{i\mathbb R}$.
% \end{enumerate}
% We are therefore omitting hypotheses (5) and (6) from the definition
% of Langlands parameter in Theorem \ref{thm:LC}.
Just as in \eqref{se:ds}, we consider the real Levi subgroup
\begin{equation} 
MA = \text{centralizer in $G$ of $A$} \supset H.
\end{equation}
The roots of $H$ in $MA$ are precisely the imaginary roots (Definition
\ref{def:rhoim}); so what is left are the non-imaginary roots. A real
parabolic subgroup $P=MAN$ of $G$ is called {\em type L
  with respect to $\Gamma$} (see \Cite{KV}*{page 704}) if whenever
$\alpha$ is a (necessarily 
non-imaginary) root of ${\mathfrak h}$ in ${\mathfrak g}$ such that
\begin{equation}\label{eq:typeL}
\langle d\gamma,\alpha^\vee\rangle \in {\mathbb Z}^{>0} \quad\text{and}\quad
\langle d\gamma, -\theta\alpha^\vee \rangle \ge 0, \end{equation}
then $\alpha$ is a root of $H$ in ${\mathfrak n}$.
The ``L'' stands for ``Langlands,'' since we will use this hypothesis
to make Langlands' construction of standard modules. Notice that
``type L'' is a weaker hypothesis than ``$\RE\nu$ weakly
dominant,'' which we assumed in \eqref{se:ds}. Notice also that the
parabolic subgroup $P=MAN$ is pointwise fixed by the real form $\sigma_0$, and
that the Cartan involution $\theta$ carries $P$ to the opposite
parabolic subgroup
\begin{equation}\label{eq:Pop}
P^{\text{op}} = MAN^{\text{op}} = MA \theta(N) = \theta(P).
\end{equation}
\end{subequations}

\begin{theorem}[\cite{KV}*{Theorem 11.129}] \label{thm:Lrealiz} Suppose
  $\Gamma = (\Lambda,\nu)$ is a weak Langlands parameter, and  
  $P = MAN$ is a parabolic subgroup of type $L$ with respect to
  $\Gamma$ (\eqref{se:stdLparam}). Then the standard quotient-type
  module 
$$I_{\quo}(\Gamma) = \Ind_{MAN}^G (D(\Lambda) \otimes \nu
\otimes 1)$$
is independent of the choice of $P$ (of type $L$). (Here we write
$D(\Lambda)$ for the limit of discrete series representation of $M$
introduced in \eqref{se:ds}.) It is a direct sum 
of distinct Langlands standard quotient-type modules (Theorem
\ref{thm:LC}), so its largest completely reducible quotient
$J(\Gamma)$ is the corresponding direct sum of distinct irreducible
modules.

In the same way, the standard sub-type module
$$I_{\sub}(\Gamma) = \Ind_{MAN^{\text{op}}}^G D(\Lambda) \otimes \nu
\otimes 1)$$ is independent of the choice of $P$ (of type $L$). It is
a direct sum of distinct Langlands standard sub-type modules (Theorem
\ref{thm:LC}), so its socle $J(\Gamma)$ is the corresponding direct
sum of distinct irreducible modules.
\end{theorem}

The two keys to making this construction are Mackey's notion of
induction and Harish-Chandra's description of the limit of discrete
series representation $D(\Lambda)$ for $M$.  The representation
$D(\Lambda)$ can be constructed using Zuckerman's ``cohomological
induction'' functor for the Borel subalgebra of ${\mathfrak m}$
determined by $R^+_{i{\mathbb R}}$.  (This is a consequence of Theorem
\ref{thm:VZrealiz} specialized to the case when $H$ is compact; the
result, due to Zuckerman, was one of his original
motivations for introducing cohomological induction.)

\begin{subequations}\label{se:stdVZparam}
In an exactly parallel way, still referring to a weak Langlands 
parameter $\Gamma$ as above, we can consider the subgroup
\begin{equation}
L = \text{centralizer in $G$ of $T_0$} \supset H.
\end{equation}
The roots of $H$ in $L$ are precisely the real roots (Definition
\ref{def:rhoim}); so what is left are the non-real roots.
A $\theta$-stable parabolic subalgebra ${\mathfrak q} = {\mathfrak l}
+ {\mathfrak u}$ of ${\mathfrak g}$ is called {\em type VZ with
  respect to $\Gamma$} (see \Cite{KV}*{page 706}) if (first) every root
in $R^+_{i{\mathbb R}}$ 
is a root in ${\mathfrak u}$; and (second) whenever $\alpha$ is a
(necessarily non-real) root of ${\mathfrak h}$ in ${\mathfrak g}$ such that 
\begin{equation}\label{eq:typeVZ}
\langle d\gamma,\alpha^\vee\rangle \in {\mathbb Z}^{>0} \quad\text{and}\quad
\langle d\gamma, \theta\alpha^\vee \rangle \ge 0, \end{equation}
then $\alpha$ is a root of $H$ in ${\mathfrak u}$.
The ``VZ'' stands for ``Zuckerman,'' since we will use this hypothesis
to make Zuckerman's construction of standard modules. 

Notice that ${\mathfrak l}$ is preserved by both the real form
$\sigma_0$ and by $\theta$. The parabolic ${\mathfrak q}$ is preserved
by $\theta$, but the real form $\sigma_0$ sends ${\mathfrak q}$ to the
opposite parabolic
\begin{equation}\label{eq:qop}
{\mathfrak q}^{\text{op}} = {\mathfrak l} + {\mathfrak u}^{\text{op}}
= {\mathfrak l} + \overline{\mathfrak u} = \overline{\mathfrak q}.
\end{equation}
These properties are complementary to those of the % real
parabolic $P$ described in \eqref{eq:Pop}.

For Zuckerman's construction, we first construct a
weak Langlands parameter 
\begin{equation}\label{eq:gammaq1} \Gamma_{\mathfrak q} = (H,\gamma_{\mathfrak q},
  \emptyset) = (\Lambda_{\mathfrak q},\nu)  \in \widehat T \times
  {\mathfrak a}^* \end{equation}  
for $L$ with differential
\begin{equation*} d\gamma_{\mathfrak q} = d\gamma - \rho({\mathfrak u})
  \in {\mathfrak h}^*; \end{equation*}
here $\rho({\mathfrak u})$ is half the sum of the roots of $H$ in
${\mathfrak u}$.  We can write
\begin{equation*} \Delta({\mathfrak u},H) =
  R^+_{i{\mathbb R}} \cup 
  \bigcup_{\text{complex pairs in ${\mathfrak u}$}} \{\alpha,
  \theta\alpha\}.\end{equation*} 
Recall that $\gamma$ is a character of the
$\rho_{i{\mathbb R}}$ cover of $H$, and this is the same as a character
$\gamma\otimes \rho_{i{\mathbb R}}^{-1}$ of $H$. We define
\begin{equation}\label{eq:gammaq}  \gamma_{\mathfrak q}|_T = \gamma\otimes
  \rho_{i{\mathbb R}}^{-1} - \sum_{\text{complex pairs $\alpha$,
      $\theta\alpha$ in ${\mathfrak u}$}} \alpha|_T, \qquad
  \gamma_{\mathfrak q}|_A = \gamma|_A = \nu;\end{equation}
since the two roots $\alpha$ and $\theta\alpha$ have the same
restriction to $T$, the formula in \eqref{eq:gammaq} is well defined
(independent of the choice of one root in each complex pair)
and has differential $d\gamma - \rho({\mathfrak u})$.  

Now we can form the standard representation
$I_{\quo}(\Gamma_{\mathfrak q})$ of $L$. Because it is a
fundamental building block in Zuckerman's construction, we recall how
this representation of $L$ is defined. By construction of $L$, the
roots of $H$ in $L$ are all real; so by choosing a system of positive
real roots, we can construct a Borel subgroup
\begin{equation}
B_L = TAN_L
\end{equation}
of $L$. According to \eqref{eq:typeL}, we can and do require that $B_L$ be
of ``type L.'' This means that $d\gamma_{\mathfrak q}$ should be
``integrally dominant'': for every real 
root $\alpha$ of $H$ in ${\mathfrak g}$, we require that {\em if}
\begin{equation}
\langle d\gamma_{\mathfrak q}, \alpha^\vee \rangle = \langle \nu,
\alpha^\vee \rangle \in {\mathbb Z}^{>0},
\end{equation}
{\em then} $\alpha$ is a root of $H$ in ${\mathfrak n}_L$. 
If we have already a parabolic subgroup $P$ of type L with respect to
$\Gamma$, then $B_L = P\cap L$ is of type L with respect to
$\Gamma_{\mathfrak q}$.

With such a choice of $B_L$, we have
\begin{equation}
I_{\quo}(\Gamma_{\mathfrak q}) = \Ind_{TAN_L}^L
(\gamma_{\mathfrak q}|_T \otimes \nu \otimes 1),
\end{equation}
an ordinary principal series representation of $L$. Similarly,
\begin{equation}
I_{\sub}(\Gamma_{\mathfrak q}) = \Ind_{TAN^{\text{op}}_L}^L
(\gamma_{\mathfrak q}|_T \otimes \nu \otimes 1),
\end{equation}

Because of the realization of the standard module
$I_{\quo}(\Gamma)$ in Theorem \ref{thm:VZrealiz}, certain simple
properties of this standard module correspond to simple properties of
the minimal principal series representation
$I_{\quo}(\Gamma_{\mathfrak q})$ for the split group $L$. By
induction by stages, these properties in turn come down to properties
of principal series representations for the real root $SL(2,{\mathbb
  R})$ subgroups.  For this reason, many statements (especially around
the Kazhdan-Lusztig conjectures) will involve the relationship of real roots
to the parameter $\Gamma_{\mathfrak q}$ of \eqref{eq:gammaq1} above.  

The hypothesis \eqref{eq:typeVZ} says that
$I_{\quo}(\Gamma_{\mathfrak q})$ is somewhat ``dominant'' with
respect to the parabolic ${\mathfrak q}$, and so somewhat
``negative'' with respect to $\overline{\mathfrak q}$.  The modules
% two ``highest weight'' modules
\begin{equation*}
M_{{\mathfrak q},\quo}(\Gamma) = U({\mathfrak g})
\otimes_{\overline{\mathfrak q}} [I_{\quo}(\Gamma_{\mathfrak
  q})\otimes \wedge^{\text{top}}({\mathfrak u})]
\end{equation*}
and
\begin{equation*}
M_{{\mathfrak q},\sub}(\Gamma) = \Hom_{\mathfrak q}(U({\mathfrak
  g}),[I_{\sub}(\Gamma_{\mathfrak
  q})\otimes \wedge^{\text{top}}({\mathfrak u})])_{\text{$L\cap K$-finite}}
\end{equation*}
are $({\mathfrak g},L\cap K)$-modules of infinitesimal character
$d\gamma$. (``Somewhat dominant'' means that these modules are trying
to be irreducible, like a Verma module with an antidominant highest
weight. If the hypothesis \eqref{eq:typeVZ} were strengthened to
\begin{equation}\label{eq:irrVZ}
\langle d\gamma,\alpha^\vee\rangle \notin {\mathbb Z}^{<0} \qquad
(\alpha \in \Delta({\mathfrak u},{\mathfrak h}))
\end{equation}
then $M_{{\mathfrak q},\quo}(\Gamma)$ and $M_{{\mathfrak
    q},\sub}(\Gamma)$ would be (isomorphic) irreducible
modules. The advantage of \eqref{eq:typeVZ} is that for any $\Gamma$ we
can choose ${\mathfrak q}$ of type VZ; but there are many $\Gamma$ for
which {\em no} choice of ${\mathfrak q}$ satisfies \eqref{eq:irrVZ}.)

The two modules $M_{{\mathfrak q},\quo}(\Gamma)$ and
$M_{{\mathfrak q},\sub}(\Gamma)$ have the same
restriction to $L\cap K$, and in fact (as is elementary to see)
the same irreducible composition factors. 
\end{subequations}

Here is the realization of the standard modules using Zuckerman's
functors.

\begin{theorem}[\cite{KV}*{Theorem 11.129}] \label{thm:VZrealiz}
  Suppose $\Gamma = 
  (\Lambda,\nu)$ is a weak Langlands parameter, and 
  ${\mathfrak q} = {\mathfrak l}+{\mathfrak u}$ is a $\theta$-stable
  parabolic subalgebra  of type $VZ$ with respect to
  $\Gamma$ (\eqref{se:stdVZparam}). Define $s = \dim({\mathfrak u}\cap
  {\mathfrak k})$. Then the standard quotient-type
  module of Theorem \ref{thm:Lrealiz} may also be realized by
  cohomological induction from $\overline{\mathfrak q}$:
$$I_{\quo}(\Gamma) = \left({\mathcal L}_{\overline{\mathfrak q},L\cap
  K}^{{\mathfrak g},K}\right)^s (I_{\quo}(\Gamma_{\mathfrak q}))
=_{\text{def}} \left({\Pi}_{{\mathfrak g},L\cap K}^{{\mathfrak
      g},K}\right)_s M_{{\mathfrak q},\quo}(\Gamma).$$ 
The derived functors vanish in all other degrees. This module is
therefore a direct sum of distinct Langlands standard 
quotient-type modules (Theorem \ref{thm:LC}), so its largest
completely reducible quotient $J(\Gamma)$ is the corresponding direct
sum of distinct irreducible modules.

In the same way, the standard sub-type module may be realized as
$$I_{\sub}(\Gamma) = \left({\mathcal R}_{{\mathfrak q},L\cap
  K}^{{\mathfrak g},K}\right)^s(I_{\sub}(\Gamma_{\mathfrak q}))
=_{\text{def}} \left({\Gamma}_{{\mathfrak g},L\cap K}^{{\mathfrak
      g},K}\right)^s M_{{\mathfrak q},\sub}(\Gamma).$$ 
The derived functors vanish in all other degrees. This is therefore
a direct sum of distinct Langlands standard sub-type modules (Theorem
\ref{thm:LC}), so its socle $J(\Gamma)$ is the corresponding direct
sum of distinct irreducible modules.
\end{theorem}

The functor ${\Gamma}_{{\mathfrak g},L\cap K}^{{\mathfrak g},K}$
appearing here is Zuckerman's ``largest $K$-finite submodule'' functor
(\cite{Vgreen}*{Definition 6.2.9}, or \Cite{KV}*{pages 101--107}), and
the superscripts are its right derived functors. The functor
${\Pi}_{{\mathfrak g},L\cap K}^{{\mathfrak g},K}$ is ``largest
$K$-finite quotient'' (\cite{KV}*{pages 101--107}), sometimes called a
``Bernstein functor'' because of lectures given by Bernstein in 1983
featuring a version of this functor.  The subscripts are left derived
functors.

The two keys to making this construction are Zuckerman's cohomological
parabolic induction functors, and (Mackey's) construction of the
ordinary principal series representations $I(\Gamma_{\mathfrak q})$
for $L$.

In order to describe Kazhdan-Lusztig polynomials in Sections
\ref{sec:klv} and \ref{sec:extklv}, it will be helpful to recall one
of the technical tools used to prove Theorem
\ref{thm:VZrealiz}. (Indeed this is the tool used in \Cite{Vgreen} to
prove Theorem \ref{thm:LC}.)  That is Lie algebra cohomology.

\begin{proposition}\label{prop:ubarcohom} Suppose $\Gamma =
  (TA,\gamma, R^+_{i{\mathbb R}})$ is a weak Langlands parameter, $L$ is
  the centralizer in $G$ of $T_0$, and 
$${\mathfrak q} = {\mathfrak l} + {\mathfrak u}$$
is a $\theta$-stable parabolic subalgebra of type VZ with respect to
$\Gamma$ (see \eqref{eq:typeVZ}). Define $\Gamma_{\mathfrak q}$, a
weak Langlands parameter for $L$ (that is, a character of $H$) as in
\eqref{eq:gammaq}. Suppose $X$ is a $({\mathfrak g},K)$-module of
finite length. 
\begin{enumerate}
\item Each Lie algebra cohomology space $H^p(\overline{\mathfrak u},X)$ is an
$({\mathfrak l},L\cap K)$-module of finite length.
\item The irreducible $({\mathfrak l},L\cap K)$-module
  $J(\Gamma_{\mathfrak q})$ can appear in $H^p(\overline{\mathfrak
    u},X)$ only if 
$$p \ge s =_{\text{def}} \dim \overline{\mathfrak u} \cap
  {\mathfrak k}.$$
\item The occurrence of  $J(\Gamma_{\mathfrak q})$ in Lie algebra
  cohomology $H^s(\overline{\mathfrak u},X)$ corresponds to the
  occurrence of $X$ as a quotient of $I_{\quo}(\Gamma)$:
$$\Hom_{{\mathfrak l},L\cap K}(J_{\quo}(\Gamma_{\mathfrak
  q}),H^s(\overline{\mathfrak u},X)) \simeq \Hom_{{\mathfrak
    g},K}(I_{\quo}(\Gamma),X).$$
\item % (under some regularity condition on gamma???)
  The coefficient
  $a_X(\Gamma)$ of $\Gamma$ in the numerator of the character formula
  for $X$ (\ref{thm:distnchar}) is equal to the alternating sum over
  $p$ of the multiplicities of $J(\Gamma_{\mathfrak q})$ in
  $H^p(\overline{\mathfrak u},X)$. (Here the Weyl denominator function
  should be chosen using a positive root system containing the roots of
  $\overline{\mathfrak u}$.)
\end{enumerate}
\end{proposition}
\begin{proof}%[Sketch of proof]
  Part (1) is \Cite{Vgreen}*{Corollary
  5.2.4}. Part (2) is \Cite{Vgreen}*{Theorem 6.5.9(f) and
  Corollary 6.4.13}. (To be precise, the references concern
    ${\mathfrak u}$ cohomology in degree $m - p$, with
    $m=\dim{\mathfrak u}$. One can deduce the
  statements needed here by a Poincar\'e duality
  relating $H^p(\overline{\mathfrak u},X)^h$ (any of the Hermitian
  dual functors introduced in Section \ref{sec:forms} will do) to
  $H^{m - p}({\mathfrak u},X^h\otimes \wedge^m({\mathfrak u}))$
  (\cite{KV}*{Corollary III.3.6 and page 409}). Alternatively, one can
  observe that the proof in \Cite{Vgreen} can be modified in a very
  straightforward way to give (2).) Part (3) is \Cite{Vgreen}*{Theorem
  6.5.9(g)} (with
  a parallel {\it caveat} about duality). Part (4) is
  \Cite{IC2}*{Theorem 8.1}. % check regularity hypothesis!
\end{proof}

\begin{subequations}\label{se:VZLKT}

Finally, it will be useful (for the unitarity algorithm) to have an
explicit description of the lowest $K$-types of standard modules.  We
begin as in \eqref{se:stdVZparam} with a weak Langlands parameter
\begin{equation} \Gamma = (\Lambda,\nu) \end{equation}
for a $\theta$-stable Cartan subgroup $H=TA$, and the Levi subgroup
\begin{equation} L = \text{centralizer in $G$ of $T_0$} \supset H,
\end{equation}
corresponding to the real roots of $H$ in $G$. Now we {\em
strengthen} the hypothesis \eqref{eq:typeVZ} to
\begin{equation}\label{eq:typeVZLKT}
\langle d\Lambda,\alpha^\vee\rangle \ge 0 \qquad (\alpha \in
\Delta({\mathfrak u},{\mathfrak h}));
\end{equation}
we say that ${\mathfrak q}$ is of {\em type VZLKT} with respect to
$\Gamma$, since in Theorem \ref{thm:VZLKT} below we will construct the
lowest $K$-types of standard modules using Zuckerman's cohomological
induction from ${\mathfrak q}$.  Using this special choice of
${\mathfrak q}$, we construct the character 
\begin{equation}
\Gamma_{\mathfrak q} = (\Lambda_{\mathfrak q},\nu)  \in \widehat T \times
  {\mathfrak a}^* \end{equation}
as in \eqref{eq:gammaq1} above.  The Harish-Chandra modules for $L$
used in the construction of $I(\Gamma)$ in Theorem \ref{thm:VZrealiz}
satisfy 
\begin{equation}
I_{\quo}(\Gamma_{\mathfrak q})|_{L\cap K} =
I_{\sub}(\Gamma_{\mathfrak q})|_{L\cap K} = \Ind_T^{L\cap
  K}(\Lambda_{\mathfrak q}).
\end{equation}
  
Choose a maximal (connected) torus
\begin{equation}
T_f \subset L\cap K.
\end{equation}
Necessarily $T_f \supset T_0$, and $T_f$ is a maximal torus in $K$. We
will discuss representations of $K$ in using $T_f$-highest weights,
bearing in mind that $K$ may be disconnected.  We therefore choose a
set of positive roots
\begin{equation*}
\Delta^+({\mathfrak l}\cap {\mathfrak k},{\mathfrak t}_f),
\end{equation*}
and extend it to
\begin{equation*}
\Delta^+({\mathfrak k},{\mathfrak t}_f) =_{\text{def}}
\Delta^+({\mathfrak l}\cap {\mathfrak k},{\mathfrak t}_f) \cup
\Delta({\mathfrak u\cap {\mathfrak k}},{\mathfrak t}_f).
\end{equation*}
Write
\begin{equation*}
{\mathfrak u} = \left({\mathfrak u}\cap {\mathfrak k}\right) \oplus
\left({\mathfrak u} \cap {\mathfrak s}\right)
\end{equation*}
for the decomposition of ${\mathfrak u}$ into the $+1$ and $-1$
eigenspaces of $\theta$.  We will need the one-dimensional character
\begin{equation}
2\rho({\mathfrak u}) \in \widehat{L}, \qquad 2\rho({\mathfrak u})(\ell)
= \det\left( \Ad(\ell)|_{\mathfrak u}\right).
\end{equation}
The restriction of this character to $L\cap K$ factors as
\begin{equation}
2\rho({\mathfrak u})|_{L\cap K} = 2\rho({\mathfrak u}\cap {\mathfrak
  k}) + 2\rho({\mathfrak u}\cap {\mathfrak s}). 
\end{equation}
\end{subequations} %se:VZLKT

\begin{theorem}[Cohomological induction of lowest $K$-types;
  \Cite{KV}*{Theorem 10.44}]\label{thm:VZLKT} Suppose $\Gamma = (\Lambda,\nu)$ is a
  weak Langlands parameter, and ${\mathfrak q}={\mathfrak l} +
  {\mathfrak u}$ is a $\theta$-stable parabolic subalgebra of type
  VZLKT (see \eqref{eq:typeVZLKT}). Use the notation of \eqref{se:VZLKT} above.

List the lowest $L\cap K$-types of 
$$I(\Gamma_{\mathfrak q})|_{L\cap K} = \Ind_T^{L\cap
  K}(\Lambda_{\mathfrak q})$$ 
as 
$$\mu_1^{L\cap K}, \mu_2^{L\cap K},\ldots,\mu_r^{L\cap K} \in
\widehat{L\cap K}.$$
For each $\mu_i^{L\cap K}$, choose a highest $T_f$-weight
$\phi_i^{L\cap K} \in \widehat{T_f}$, and define
$$\phi_i = \phi_i^{L\cap K} + 2\rho({\mathfrak u}\cap {\mathfrak s})
\in \widehat{T_f}.$$
\begin{enumerate}
\item For each $i$, the representation of $K$
$$\mu_i = \left({\Pi}_{L\cap K}^K\right)_s\left[U(\mathfrak
k)\otimes_{\overline{\mathfrak q}\cap {\mathfrak k}} \left(\mu_i^{L\cap
  K}\otimes 2\rho({\mathfrak u})\right)\right]$$
is irreducible or zero, depending on whether or not the weight
$\phi_i$ is dominant for $\Delta^+({\mathfrak k},{\mathfrak t}_f)$. 

\item In the dominant case, $\phi_i$ is a highest weight of $\mu_i$, and
$\mu_i$ may be characterized as the unique irreducible representation
of $K$ containing ${\mathfrak u}\cap {\mathfrak k}$-invariant vectors
transforming under the representation 
$$\mu_i^{L\cap K} \otimes 2\rho({\mathfrak u}\cap{\mathfrak s})\in
\widehat{L\cap K}.$$ 
Briefly, we say that $\mu_i$ has ``highest weight $\mu_i^{L\cap K} \otimes
2\rho({\mathfrak u}\cap{\mathfrak s})$.''
\item The irreducible representations $\mu_i$ constructed in (1)
  exhaust the lowest $K$-types of $I_{\quo}(\Gamma)$, by means of the
  ($K$-equivariant) bottom layer map
$$\mu_i \rightarrow I_{\quo}(\Gamma)$$
evident from the realization in Theorem \ref{thm:VZrealiz}.
\item Suppose that $\Lambda$ is final (Definition \ref{def:dLP}; that
  is, that $\Lambda_{\mathfrak q}(m_\alpha) = 1$ for every real root
  $\alpha$). Then there is a {\em unique} lowest $L\cap K$-type
  $\mu_1^{L\cap K}$ of $\Ind_T^{L\cap K}(\Lambda_{\mathfrak q})$. It
  has dimension one, and may be characterized by the two requirements
$$\mu_1^{L\cap K}|_T = \Lambda_{\mathfrak q}, \qquad \mu_1^{L\cap
  K}|_{[L_0,L_0] \cap K} \text{\ is trivial.}$$
The (unique) highest weight of $\mu_1^{L\cap K}$ may be characterized
by
$$\phi_1^{L\cap K}|_{T\cap T_f} = \Lambda_{\mathfrak q}|_{T\cap T_f},
\qquad \langle \phi_1^{L\cap K},\alpha^\vee\rangle = 0 \quad(\alpha \in
\Delta({\mathfrak l},{\mathfrak t}_f)),$$

\end{enumerate}
\end{theorem}
\begin{proof}%[Sketch of proof]
  Parts (1)--(3) are stated as
  \Cite{KV}*{Theorem 10.44}. Unfortunately that reference does not
  provide a
  complete proof; but what is needed are just the facts established in
  \Cite{Vgreen}*{\S 6.5}.  For (4), the ``final'' hypothesis means
  exactly that $\Lambda_{\mathfrak q}$ is trivial on $[L_0,L_0]\cap
  T$, which is generated by the elements $m_\alpha$ for real roots
  $\alpha$. It follows that any lowest $L\cap K$-type must be trivial
  on $[L_0,L_0]\cap K$.  Because (as one easily checks) $L\cap K$ is
  generated by $T$ and $[L_0,L_0]\cap K$, the remaining assertions in (4)
  follow immediately.
\end{proof}
\section{Invariant forms on irreducible representations}\label{sec:invirr}
\setcounter{equation}{0}
\begin{subequations}\label{se:gKrealstr}
Suppose $G$ is the group of real points of a connected complex
reductive algebraic group, $K$ is a maximal compact subgroup of $G$,
and $\theta$ is a Cartan involution (notation as in
\eqref{se:reductive} and Theorem \ref{thm:realforms}). 
Write 
\begin{equation}
{\mathfrak g} = \Lie(G)\otimes_{\mathbb R}{\mathbb C},
\end{equation}
\begin{equation}
K({\mathbb C}) = \text{complexification of $K$},
\end{equation}
so that $({\mathfrak g},K({\mathbb C}))$ is a pair (Definition
\ref{def:pair}). We are interested in two real structures (Definition
\ref{def:pair}) on this
pair: the classical real structure $\sigma_0$ given by the real form
$G$, and the real structure $\sigma_c$ given by the compact real form
\begin{equation}
\sigma_c = \sigma_0\circ \theta
\end{equation}
(Theorem \ref{thm:realforms}).  These two structures are the same on
$K({\mathbb C})$ (where they both define the compact form $K$), and
therefore they agree on ${\mathfrak k}$; but they differ (by a factor
of $-1$) on the $-1$ eigenspace ${\mathfrak s}$ of $\theta$ on
${\mathfrak g}$ (see Theorem \ref{thm:realforms}).
\end{subequations}
\begin{subequations}\label{se:gKclassicalinv}
Suppose now that $V$ is a $({\mathfrak g},K({\mathbb
  C}))$-module. Classically, an invariant Hermitian form on $V$ is a
Hermitian form $\langle,\rangle^0$ on $V$ with the properties
\begin{equation}\begin{aligned}
\langle k\cdot v, w \rangle^0 = \langle v,k^{-1}\cdot w\rangle^0,\quad 
&\langle X\cdot v, w \rangle^0 = -\langle v,X\cdot w\rangle^0\\
&(k \in K,\ X\in {\mathfrak g}_0,\ v, w\in V).\end{aligned}
\end{equation} 
(This is the definition that we left unstated in Theorem
\ref{thm:unitarydual2}.) 
The first requirement is that the compact group $K$ act by ``unitary'' 
(that is, preserving the Hermitian form)
operators; and the second is that the real Lie algebra ${\mathfrak
  g}_0$ act by skew-Hermitian operators. This in turn is the differentiated
version of the condition that $G$ act by unitary operators. The
assumption that ${\mathfrak k}_0$ acts by skew-Hermitian operators
guarantees that $K_0$ acts by unitary operators; so we need the first
hypothesis only for a single element $k$ in each connected component
of $K$.

It is equivalent to require either less
\begin{equation}\begin{aligned}
\langle k\cdot v, w \rangle^0 = \langle v,k^{-1}\cdot w\rangle^0,\quad 
&\langle X\cdot v, w \rangle^0 = -\langle v,X\cdot w\rangle^0\\
&(k \in K,\ X\in {\mathfrak s}_0,\ v, w\in V)\end{aligned}
\end{equation} 
or more
\begin{equation}\begin{aligned}
\langle k\cdot v, w \rangle^0 = \langle v,\sigma_0(k)^{-1}\cdot w\rangle^0,\quad 
&\langle X\cdot v, w \rangle^0 = -\langle v,\sigma_0(X)\cdot w\rangle^0\\
&(k \in K({\mathbb C}),\ X\in {\mathfrak g},\ v,w\in V)\end{aligned}
\end{equation} 
Of course these are precisely the properties defining a
$\sigma_0$-invariant Hermitian form on $V$ (see
\eqref{se:classicalinvform}).
\end{subequations}

In Section \ref{sec:forms} we related the existence of invariant forms
to the notion of Hermitian dual of any $({\mathfrak g},K({\mathbb
  C}))$-module. We therefore need to compute Hermitian duals in
terms of the Langlands classification.
\begin{definition}\label{def:dualparam} 
Suppose $\Gamma =
  (H,\gamma,R^+_{i\mathbb R})$ is a Langlands parameter (Theorem
  \ref{thm:LC}). The {\em Hermitian dual of $\Gamma$} is
\addtocounter{equation}{-1}
\begin{subequations}\label{se:dualparam}
\begin{equation}\Gamma^{h,\sigma_0} =
  (H,-\overline\gamma,R^+_{i\mathbb R}).\end{equation}
As usual the additive notation for characters may be a bit confusing;
the character of the $\rho_{\abs}$ (equivalently, $\rho_{i{\mathbb
    R}}$ cover) $\widetilde H$ is
$$(-\overline\gamma)(\widetilde h) =_{\text{def}}
\overline{\gamma(\widetilde h^{-1})}.$$

Now write $H=TA$ as in \ref{prop:toruschars}, and 
\begin{equation}\Gamma = (\Lambda,\nu) = (\Lambda, \nu_{\re} +
i\nu_{\im})\end{equation}
as in \ref{def:dLP}, with $\Lambda$ a 
discrete Langlands parameter and $\nu \in {\mathfrak a}^*$.
Since $\gamma$ takes values in the unit circle on the maximal
compact subgroup $\widetilde T$ of $\widetilde H$, we find
\begin{equation}\label{eq:dualform} \Gamma^{h,\sigma_0} =
  (\Lambda,-\overline{\nu}) = (\Lambda,\theta\overline\nu) = (\Lambda,
  -\nu_{\re} +  i\nu_{\im}).\end{equation}

The most important feature of this formula is that if the continuous
parameter $\nu$ is imaginary (Definition \ref{def:dLP}, corresponding
to a unitary character of $A$) then $\Gamma^{h,\sigma_0}$ is equal to
$\Gamma$. We will see (Proposition \ref{prop:dualstd}) that this means
that the corresponding representation admits an invariant Hermitian
form. These are the tempered irreducible representations of $G$, each
of which has a natural Hilbert space realization, and therefore a
positive-definite invariant Hermitian form.
\end{subequations} \end{definition}

\begin{proposition}[Knapp-Zuckerman; see \Cite{KOv}*{Chapter
  16}]\label{prop:dualstd} Suppose $\Gamma = (\Lambda,\nu)$ is a
  Langlands parameter for the real reductive group $G$, and
  $\Gamma^{h,\sigma_0} = (\Lambda,-\overline{\nu})$ is the Hermitian
  dual parameter.
\begin{enumerate}
\item The Hermitian dual of the standard representation with a
  Langlands quotient is the standard representation with a Langlands
  submodule:
$$[I_{\quo}(\Gamma)]^{h,\sigma_0} \simeq
I_{\sub}(\Gamma^{h,\sigma_0}),$$
and similarly with quotient and submodule reversed.
\item The Hermitian dual of the irreducible module $J(\Gamma)$ is
$$[J(\Gamma)]^{h,\sigma_0} \simeq J(\Gamma^{h,\sigma_0}).$$
\item The irreducible module $J(\Gamma)$ admits a nonzero
invariant Hermitian form if and only if the Langlands parameter
$\Gamma$ is equivalent to $\Gamma^{h,\sigma_0}$.  In this case there
is a nonzero invariant Hermitian form on $I_{\quo}(\Gamma)$,
unique up to a real scalar. This form has radical equal to the maximal
proper submodule $I_1(\Gamma)$, and factors to a nondegenerate
invariant Hermitian form on $J(\Gamma)$.
\item Write the parameter $\Gamma$ as 
$$\Gamma = (\Lambda,\nu)$$ 
as in \ref{def:dLP}. Then $J(\Gamma)$ admits a nonzero invariant Hermitian
form if and only if there is an element $w \in W^\Lambda$ (Proposition
\ref{prop:LCshape}) such that $w\cdot \nu = - \overline{\nu}$. (If
such an element $w$ exists, it may be chosen to have order $2$.) 
\item If $J(\Gamma)$ admits a nonzero invariant Hermitian form, then this
  form is unique up to a nonzero real scalar multiple.  Consequently
  the signature function
$$(\POS^{\sigma_0}_{J(\Gamma)},\NEG^{\sigma_0}_{J(\Gamma)}) \colon \widehat
K \rightarrow {\mathbb N} \times {\mathbb N}$$
 (cf.~Proposition \ref{prop:sigchar}) is well-defined up to
 interchanging the factors $\POS$ and $\NEG$.
\end{enumerate}
\end{proposition}

\begin{subequations}\label{se:gKcinv}
Having recalled the classical theory of invariant Hermitian forms, we
now turn to our main technical tool: the notion of $c$-invariant
Hermitian forms.

Suppose again that $V$ is a $({\mathfrak g},K({\mathbb
  C}))$-module. A {\em $c$-invariant Hermitian form on $V$} is a
Hermitian form $\langle,\rangle^c$ on $V$ with the properties
\begin{equation}\begin{aligned}
\langle k\cdot v, w \rangle^c = \langle v,k^{-1}\cdot w\rangle^c,\quad 
&\langle X\cdot v, w \rangle^c = -\langle v,X\cdot w\rangle^c\\
&(k \in K,\ X\in {\mathfrak g}^{\sigma_c},\ v, w\in V).\end{aligned}
\end{equation} 
Just as for a classical Hermitian form, we require first that the
compact group $K$ act by unitary  (that is, preserving the Hermitian form)
operators.  What is new is that we require not the real Lie algebra ${\mathfrak
  g}_0 = {\mathfrak g}^{\sigma_0}$, but rather the {\em compact} Lie
algebra ${\mathfrak g}^{\sigma_c}$, to act by skew-Hermitian
operators. Because the 
action of ${\mathfrak g}^{\sigma_c}$ usually does not exponentiate to
the compact real form of $G$, this new requirement is no longer the
derivative of a unitarity requirement on a group.  Nevertheless the
definition still makes sense.

It is equivalent to require either less
\begin{equation}\begin{aligned}
\langle k\cdot v, w \rangle^c = \langle v,k^{-1}\cdot w\rangle^c,\quad 
&\langle X\cdot v, w \rangle^c = -\langle v,X\cdot w\rangle^c\\
&(k \in K,\ X\in i{\mathfrak s}_0,\ v, w\in V)\end{aligned}
\end{equation} 
or more
\begin{equation}\begin{aligned}
\langle k\cdot v, w \rangle^c = \langle v,\sigma_c(k)^{-1}\cdot w\rangle^c,\quad 
&\langle X\cdot v, w \rangle^c = -\langle v,\sigma_c(X)\cdot w\rangle^c\\
&(k \in K({\mathbb C}),\ X\in {\mathfrak g},\ v,w\in V)\end{aligned}
\end{equation} 
These are precisely the properties defining a
$\sigma_c$-invariant Hermitian form on $V$ (see Definition
\ref{def:invherm}).
\end{subequations}

We need to compute $c$-Hermitian duals in terms of the Langlands classification.
\begin{definition}\label{def:cdualparam} Suppose $\Gamma =
  (H,\gamma,R^+_{i\mathbb R})$ is a Langlands parameter (Theorem
  \ref{thm:LC}). The {\em $c$-Hermitian dual of $\Gamma$} is
\addtocounter{equation}{-1}
\begin{subequations}\label{se:cdualparam}
\begin{equation} \Gamma^{h,\sigma_c} =
  (H,-\overline\gamma\circ\theta,R^+_{i\mathbb R}).\end{equation}
This is the same as the Hermitian dual, except that we have twisted
the character by the Cartan involution $\theta$. Write $H=TA$ as in
\ref{prop:toruschars}, and  
\begin{equation}\Gamma = (\Lambda,\nu) = (\Lambda, \nu_{\text{re}} +
i\nu_{\im})\end{equation} 
as in \ref{def:dLP}, with $\Lambda$ a 
discrete Langlands parameter and $\nu \in {\mathfrak a}^*$.
Since $\theta$ acts by the identity on $T$ and by inversion on $A$,
comparison with Definition \ref{def:dualparam} shows that 
\begin{equation}\label{eq:cdualform}
\Gamma^{h,\sigma_c} = (\Lambda,\overline{\nu}) = (\Lambda,
\nu_{\re} - i\nu_{\im}).\end{equation}
In particular, {\em if the continuous parameter of $\Gamma$ is real
  (Definition \ref{def:dLP}) then $\Gamma^{h,\sigma_c}$ is equal to
  $\Gamma$.}  We will see (Proposition \ref{prop:cdualstd}) that this means
that the corresponding representation admits a $c$-invariant Hermitian
form. 
\end{subequations}
\end{definition}

\begin{proposition}\label{prop:cdualstd} Suppose $\Gamma = (\Lambda,\nu)$ is a
  Langlands parameter for the real reductive group $G$, and
  $\Gamma^{h,\sigma_c} = (\Lambda,\overline{\nu})$ is the $c$-Hermitian
  dual parameter.
\begin{enumerate}
\item The $c$-Hermitian dual of the standard representation with a
  Langlands quotient is a standard representation with a Langlands
  submodule:
$$[I_{\quo}(\Gamma)]^{h,\sigma_c} \simeq
I_{\sub}(\Gamma^{h,\sigma_c}),$$
and similarly with quotient and submodule reversed.
\item The $c$-Hermitian dual of the irreducible module $J(\Gamma)$ is
$$[J(\Gamma)]^{h,\sigma_c} \simeq J(\Gamma^{h,\sigma_c}).$$
\item The irreducible module $J(\Gamma)$ admits a nonzero
$c$-invariant Hermitian form if and only if the Langlands parameter
$\Gamma$ is equivalent to $\Gamma^{h,\sigma_c}$.  In this case there
is a nonzero $c$-invariant Hermitian form on $I_{\quo}(\Gamma)$,
unique up to a real scalar. This form has radical equal to the maximal
proper submodule $I_1(\Gamma)$, and factors to a nondegenerate
$c$-invariant Hermitian form on $J(\Gamma)$.
\item Write the parameter $\Gamma$ as 
$$\Gamma = (\Lambda,\nu)$$ 
as in \ref{def:dLP}. Then $J(\Gamma)$ admits a nonzero $c$-invariant Hermitian
form if and only if there is an element $w \in W^\Lambda$ (Proposition
\ref{prop:LCshape}) such that $w\cdot \nu = \overline{\nu}$. In
particular, if $\nu$ is real, then a $c$-invariant form always exists.
% (If
% such an element $w$ exists, it may be chosen to have order $2$.) 
\item Suppose that the continuous parameter $\nu$ is real. Then any
  $c$-invariant Hermitian form on $J(\Gamma)$ has the same sign on
  every lowest $K$-type.  In particular, the form may be chosen to be
  positive definite on every lowest $K$-type, and is characterized up
  to a positive real scalar multiple by this requirement.  Consequently
  the signature function
$$(\POS^c_{J(\Gamma)},\NEG^c_{J(\Gamma)}) \colon \widehat
K \rightarrow {\mathbb N}\times {\mathbb N}$$
 (cf.~Proposition \ref{prop:sigchar}) is uniquely defined by the
 characteristic property
$$(\POS^c_{J(\Gamma)}(\mu),\NEG^c_{J(\Gamma)}(\mu)) = (1,0)
\qquad\text{$\mu$ any lowest $K$-type of $J(\Gamma)$}.$$
\end{enumerate}
\end{proposition}

\begin{proof}
% \begin{subequations}\label{se:cdualproof}
We use the proof of Proposition \ref{prop:dualstd}. Knapp and
Zuckerman began with the realization
\begin{equation*} I_{\quo}(\Gamma) = \Ind_{MAN}^G
  (D(\Lambda) \otimes \nu \otimes 1) \end{equation*}
(with $P$ chosen of type L with respect to $\Gamma$) and found a
natural isomorphism 
\begin{equation*}[I_{\quo}(\Gamma)]^{h,\sigma_0} \simeq \Ind_{MAN}^G
(D(\Lambda) \otimes -\overline{\nu} \otimes 1).\end{equation*}
Because of the minus sign on the right, it is easy to check that
$P^{\text{op}}$ is of type L with respect to $\Gamma^{h,\sigma_0}$; so
what appears on the right is the realization of
$I_{\sub}(\Gamma^{h,\sigma_0})$ from Theorem \ref{thm:Lrealiz}.
% Getting the ``quotient'' realization requires choosing $P$ to make
% $\RE \nu$ nonnegative. A little more precisely: realizing the spaces
% on the left and the right 
% as certain functions on $K$ with values in the 
% Hilbert space of the discrete series representation $D(\Lambda)$,
% the natural isomorphism defines a pairing between left and right by
% integration over $K$ of inner products in the Hilbert space.)
For the $c$-Hermitian dual, we are again looking at the (same) Hermitian dual
of the space of the standard module; what has changed is that the
representation of ${\mathfrak g}$ on this dual space has been twisted by
the Cartan involution $\theta$. Now $\theta$ carries $P=MAN$ to
$P^{\text{op}} = MAN^{\text{op}}$, so it is clear that (writing a
superscript for the operation of twisting the Lie algebra action by an
automorphism) 
\begin{equation*}[\Ind_{MAN}^G
(D(\Lambda) \otimes -\overline{\nu} \otimes 1)]^\theta
\simeq \Ind_{{MAN}^{\text{op}}}^G
(D(\Lambda)^\theta \otimes \overline{\nu} \otimes 1)\end{equation*}
Harish-Chandra characterized each discrete series representation in
terms of its distribution character on the $\theta$-fixed subgroup
$K$, so $D(\Lambda)^\theta \simeq D(\Lambda)$ whenever
$\Lambda$ is a discrete series representation. The construction of
limits of discrete series by the translation principle inherits this
property.  Assembling
all of this, we find
\begin{equation*} [I_{\quo}(\Gamma)]^{h,\sigma_c} \simeq
  \Ind_{MAN^{\text{op}}}^G (D(\Lambda) \otimes \overline{\nu}
  \otimes 1).\end{equation*}  
Because the minus sign on $\nu$ has disappeared, it is easy to check
that $P$ is of type L with respect to $\Gamma^{h,\sigma_c}$. According
to Theorem \ref{thm:Lrealiz}, the 
right side realizes $I_{\sub}(\Gamma^{h,\sigma_c})$. This proves
(1).

Parts (2) and (3) are a formal consequence of (1) and the Langlands
classification.  Part (4) is a restatement of (3), using Proposition
\ref{prop:LCshape}.
% \end{subequations}

For (5), we will use

\begin{theorem}[Signature Theorem; \Cite{KV}*{Theorem
  6.34}]\label{thm:sigthm} Suppose ${\mathfrak q} = {\mathfrak l} +
  {\mathfrak u}$ is a $\theta$-stable parabolic subalgebra of
  ${\mathfrak g}$, with Levi subgroup $L \subset
  G$. Write $s = \dim({\mathfrak u}\cap {\mathfrak k})$, and write
  ${\mathcal L}_s$ for Zuckerman's cohomological parabolic induction
  functor from $({\mathfrak l},L\cap K)$-modules to $({\mathfrak
    g},K)$-modules (\cite{KV}*{pages 327--328}).
\begin{enumerate}
\item Suppose that $Z$ is an $({\mathfrak l},L\cap K)$-module of
  finite length admitting an invariant Hermitian form
  $\langle,\rangle_L^0$. Then there is an induced invariant Hermitian form
  $\langle,\rangle_G^0$ on ${\mathcal L}_s(Z)$. The signature of
  $\langle,\rangle_G^0$ on the bottom layer of $K$-types matches the
  signature of $\langle,\rangle_L^0$ on the corresponding $L\cap
  K$-types.
\item Suppose that $Z$ is an $({\mathfrak l},L\cap K)$-module of
  finite length admitting a $c$-invariant Hermitian form
  $\langle,\rangle_L^c$. Then there is an induced $c$-invariant Hermitian form
  $\langle,\rangle_G^c$ on ${\mathcal L}_s(Z)$. The signature of
  $\langle,\rangle_G^c$ on the bottom layer of $K$-types matches the
  signature of $\langle,\rangle_L^c$ on the corresponding $L\cap
  K$-types.
\end{enumerate}
\end{theorem}

\begin{proof} Part (1) is \Cite{KV}*{Theorem 6.34}. Part (2)
can be proved by repeating almost exactly the same words, replacing
``invariant Hermitian'' by ``$c$-invariant Hermitian.''
\end{proof}

We continue now with the proof of Proposition \ref{prop:cdualstd}(5). 
We apply Theorem \ref{thm:sigthm} to the realization of the standard
module $I(\Gamma)$ in Theorem \ref{thm:VZrealiz}, and to the
description of the lowest $K$-types  in Theorem \ref{thm:VZLKT}.
% \begin{subequations}\label{se:splitcinv}
What this accomplishes is a reduction of
(5) to the special case
\begin{equation*}\text{$H$ is split modulo center of $G$;}\end{equation*}
that is, that all roots of $H$ are real, so that
\begin{equation*}P = MAN = TAN\end{equation*}
is a Borel subgroup of $G$, and the standard representations are
ordinary principal series.  We assume $P$ is
chosen to make $\nu = \RE\nu$ weakly dominant:
\begin{equation*}
\langle \nu, \alpha^\vee\rangle \ge 0 \qquad (\alpha \in
\Delta({\mathfrak n},{\mathfrak h});
\end{equation*}
this is a stronger hypothesis than ``type L'' \eqref{eq:typeL}.
Consider the parabolic subgroup $P_1 = M_1A_1N_1 \supset P$ defined by
either of the equivalent requirements
% \begin{equation}
% \langle \nu, \alpha^\vee\rangle = 0 \qquad (\alpha \in
% \Delta({\mathfrak m}_1,{\mathfrak h}),
% \end{equation}
% or
% \begin{equation}
% \langle \nu, \alpha^\vee\rangle > 0 \qquad (\alpha \in
% \Delta({\mathfrak n}_1,{\mathfrak h})).
% \end{equation}
\begin{equation*}
\langle \nu, \alpha^\vee\rangle = 0 \quad (\alpha \in
\Delta({\mathfrak m}_1,{\mathfrak h}) \quad\text{or}\quad \langle
\nu, \alpha^\vee\rangle > 0 \qquad (\alpha \in 
\Delta({\mathfrak n}_1,{\mathfrak h})).
\end{equation*}
Define
\begin{equation*}
H_1 = M_1 \cap H, \quad \Gamma_1 = \Gamma|_{H_1}, \quad \nu_1 =
\nu|{A_1}, \quad I_1 = I(\Gamma_1);
\end{equation*}
this last representation is a tempered unitary principal series representation
of $M_1$, with continuous parameter equal to zero. Induction by stages
says
\begin{equation*}
I_{\quo}(\Gamma) = \Ind_{M_1A_1N_1}^G (I_1 \otimes \nu_1 \otimes 1).
\end{equation*}
Here finally we have written the standard module precisely as
Langlands wrote it in \Cite{LC}: as induced from a tempered
representation $I_1$ twisted by a character $\nu$ with $\RE\nu$
strictly positive.  Because $\Gamma$ is assumed to be a Langlands
parameter, this standard module has a {\em unique} irreducible
quotient $J(\Gamma)$. It follows that the unitary representation $I_1$
must be irreducible. A principal series representation with continuous
parameter equal to zero is irreducible if and only if it has a unique
lowest $K$ type; so we conclude that
\begin{equation*}
\text{$I_1$ is irreducible, with unique lowest $M_1\cap K$-type $\mu_1$.}
\end{equation*}
We may therefore fix a nondegenerate $c$-invariant Hermitian form
$\langle,\rangle_1$ on $I_1$ with the property that
\begin{equation}\label{eq:tempsplcpos}
\text{$\langle,\rangle_1$ is positive definite on the lowest $M_1\cap
  K$-type $\mu_1$;}
\end{equation}
this determines the form up to a positive scalar multiple.

The $c$-hermitian dual of $I_{\quo}(\Gamma)$ is
\begin{equation*}
[I_{\quo}(\Gamma)]^{h,\sigma_c} = \Ind_{M_1A_1N_1^{\text{op}}}^G
  (I_1^{h,\sigma_c} \otimes \nu_1 \otimes 1) \simeq \Ind_{M_1A_1N_1^{\text{op}}}^G
  (I_1 \otimes \nu_1 \otimes 1). 
\end{equation*}
The last identification here uses the pairing $\langle,\rangle_1$ just
defined. 

A vector in $I_{\quo}(\Gamma)$ is a function $f_{\quo}$ on
$G$ with values in $I_1$, transforming in a certain way under the
action of $P_1$ on the right. A vector in the $c$-Hermitian dual is a
function $f_{\sub}$ on $G$, taking values in $I_1$, and transforming
under the action of $P_1^{\text{op}}$ on the right.  The pairing
between two such functions is
\begin{equation*}
\int_K \langle f_{\quo}(k),f_{\sub}(k)\rangle_1 dk.
\end{equation*}

A $c$-invariant Hermitian pairing on $I_{\quo}(\Gamma)$ is
therefore the same thing as an intertwining operator
\begin{equation*}
{\mathcal L}\colon \Ind_{M_1A_1N_1}^G (I_1 \otimes \nu_1 \otimes 1)
\rightarrow \Ind_{M_1A_1N_1^{\text{op}}}^G (I_1 \otimes \nu_1 \otimes
  1); 
\end{equation*}
the corresponding $c$-invariant Hermitian form is
\begin{equation}\label{eq:cinvtsplit}
\langle f_{\quo}, f'_{\quo} \rangle_{\mathcal L} = \int_K
\left\langle f_{\quo}(k),({\mathcal L}f'_{\quo})(k)\right\rangle_1 dk.
\end{equation}
In this formula we can use the integral long intertwining operator
\begin{equation*}
({\mathcal L}f)(g) = \int_{N_1^{\text{op}}} f(gx) dx;
\end{equation*}
this integral over $N_1^{\text{op}}$ is absolutely convergent because
$\RE\nu_1$ is strictly positive on the roots of $A_1$, just as in \Cite{LC}.

The lowest $K$-types of $I_{\quo}(\Gamma)$ are the {\em fine}
representations of $K$ (\cite{Vgreen}*{Definition 4.3.9}) that contain
the (necessarily fine) representation $\mu_1$ of $M_1\cap K$.  Each of
them contains $\mu_1$ with multiplicity one (\cite{Vgreen}*{Theorem
4.3.16}), so each lowest $K$-type $\mu$ has multiplicity one in
$I_{\quo}(\Gamma)$. Consequently 
\begin{equation*}
\text{${\mathcal L}$ acts by a scalar ${\mathcal L}(\mu)$ on each
  lowest $K$-type $\mu$.}
\end{equation*}
In light of the formula \eqref{eq:cinvtsplit} for the $c$-invariant
form, and the positivity \eqref{eq:tempsplcpos}, the proof of (5) comes
down to proving
\begin{equation*}
\text{\parbox{.55\textwidth}{${\mathcal L}$ acts by a strictly
    positive scalar on each fine representation $\mu$ containing $\mu_1$.}}
\end{equation*}

In order to compute the scalar ${\mathcal L}(\mu)$, we use the product
formula of Gindikin-Karpelevi\v c (see Schiffmann
\Cite{Schiff} or \Cite{Vgreen}*{Theorem 4.2.22}) for the intertwining operator
${\mathcal L}$. What this product formula shows is that the scalar in
question is a product (over the coroots of $H$ taking strictly
positive values on $\nu$) of a corresponding scalar computed in the
three-dimensional subgroup $\phi_\alpha(SL(2,{\mathbb R}))$.  There are
two cases: the ``even case,'' in which the character $\delta$ is
trivial on the element $m_\alpha$ of order two (Definition
\ref{def:rhoim}); and the ``odd case,'' in which 
$\delta(m_\alpha) = -1$.  

In the even case, the only fine
representation of $SO(2)$ containing $\delta$ is the trivial one, so
we must begin with the constant function $1$ on $SO(2)$ and extend it
to an element of the principal series representation
\begin{equation*}
f_{\text{even}}\left(\begin{pmatrix}\cos\theta & \sin\theta \\ -\sin\theta &
  \cos\theta \end{pmatrix} \begin{pmatrix}e^t & 0 \\ 0 & e^{-t}\end{pmatrix}
\begin{pmatrix} 1 & x \\ 0 & 1 \end{pmatrix} \right) = e^{-t(v+1)}.
\end{equation*}
Here $v = \langle \nu,\alpha^\vee\rangle $ is a strictly positive
parameter. Necessarily ${\mathcal L}(f_{\text{even}})$ is constant on
$SO(2)$. To compute the constant, we can evaluate at $1$, getting
\begin{equation*}
\int_{-\infty}^{\infty} f_{\text{even}}\left( \begin{pmatrix} 1 & 0\\
    y&1 \end{pmatrix}\right) dy = \int_{-\infty}^{\infty}
(1+y^2)^{-(v+1)/2} dy.
\end{equation*}
(Here we have computed the Iwasawa decomposition of the matrix in the
integrand, then used the formula above.) Convergence and positivity of
this integral are elementary. (Explicit evaluation, to
$\pi^{1/2}\Gamma(v/2)/\Gamma(v/2 + 1/2)$, is also easy and well known,
but we will not need this.)

In the odd case, there are two fine representations of $SO(2)$
containing $\delta$, corresponding to the functions $e^{\pm i\theta}$
on $SO(2)$.  The corresponding functions in the principal series are
\begin{equation*}
f_{\text{odd,}\pm}\left(\begin{pmatrix}\cos\theta & \sin\theta \\ -\sin\theta &
  \cos\theta \end{pmatrix} \begin{pmatrix}e^t & 0 \\ 0 & e^{-t}\end{pmatrix}
\begin{pmatrix} 1 & x \\ 0 & 1 \end{pmatrix} \right) = e^{\pm i\theta}
e^{-t(v+1)}.
\end{equation*}
Again applying ${\mathcal L}$ has to give a multiple of $e^{\pm
  i\theta}$ on $SO(2)$, and we can calculate the multiple by
evaluating at the identity:
\begin{equation*}
\int_{-\infty}^{\infty} f_{\text{odd},\pm}\left( \begin{pmatrix} 1 & 0\\
    y&1 \end{pmatrix}\right) dy = \int_{-\infty}^{\infty} [(1 \mp
iy)(1+y^2)^{-1/2}] (1+y^2)^{-(v+1)/2} dy
\end{equation*}
The imaginary part of the integrand is (integrable and) odd, so
contributes zero; we are left with
\begin{equation*}
\int_{-\infty}^{\infty} (1+y^2)^{-(v+2)/2} dy,
\end{equation*}
which is clearly positive. (The value is $\pi^{1/2}\Gamma(v/2 +
1/2)/\Gamma(v/2 + 1)$.) 
% \end{subequations}
\end{proof}

\section{Standard and $c$-invariant forms in the equal rank
  case}\label{sec:eqrk} 
\setcounter{equation}{0}

Our plan (roughly speaking) is (first) to calculate signatures of
$c$-invariant forms, and (second) to relate these to the signatures of
ordinary invariant forms that we care most about.  The first step will
occupy most of the rest of the paper. The second step is in most cases
much easier, and we can do it now. But in some cases this easy step
will fail. In order to repair it, we will introduce (in Section
\ref{sec:theta}) a slightly different category of representations.
% twist on the notion of $c$-invariant forms.

\begin{definition}\label{def:eqrk} Suppose that $G$ is a real
  reductive algebraic group as in \eqref{se:reductive}, and that
  $\theta$ is a Cartan involution as in Theorem
  \ref{thm:realforms}, with $K=G^\theta$. We say that $G$ is {\em
    equal rank} if $G$ and $K$ have the same rank; equivalently, if
  the automorphism $\theta$ of $G$ is inner.  In this case a {\em
    strong involution for $G$} is an element $x\in G$ such that
$$\Ad(x) = \theta, \qquad K=G^x$$
It follows that $x\in Z(K)$, and that
$$x^2 = z \in Z(G)\cap K.$$

Fix $x$ and $z$ as above. For every $\zeta \in {\mathbb C}^\times$,
define
$$\widehat G_\zeta = \{ V\in \widehat G \mid z\cdot v = \zeta v \ \
(v\in V)\},$$
the irreducible $({\mathfrak g},K({\mathbb C}))$-modules in which
$z$ acts by $\zeta$. Similarly define $\widehat K_\zeta$.  

Fix a square root $\zeta^{1/2}$ of $\zeta$. On every $\mu\in \widehat
K_\zeta$, $x$ must act by some square root of $\zeta$; so there is a
sign
$$\epsilon(\mu) \in \{\pm 1\}, \qquad \mu(x) =
\epsilon(\mu)\zeta^{1/2} \qquad (\mu \in \widehat K_\zeta).$$
Of course $\epsilon(\mu)$ depends on the choice of square root
$\zeta^{1/2}$. 
\end{definition}

\begin{theorem}\label{thm:ctoinveqrk} Suppose $G$ is an equal rank
  real reductive algebraic group with strong involution $x\in Z(K)$
  (Definition \ref{def:eqrk}). Write $z = x^2 \in Z(G)$. Suppose $J$
  is an irreducible $({\mathfrak g},K({\mathbb C}))$-module admitting
  a $c$-invariant Hermitian form $\langle,\rangle^c_J$, in which $z$
  acts by the scalar $\zeta\in {\mathbb C}^\times$. Fix a square root
  $\zeta^{1/2}$ of $\zeta$. Then $J$ admits an invariant Hermitian
  form $\langle,\rangle^0_J$ defined by the formula
$$\langle v, w\rangle^0_J = \zeta^{-1/2}\langle x\cdot v, w\rangle^c_J.$$ 
For each $\mu \in \widehat K_\zeta$, the forms on the $\mu$
multiplicity spaces (Proposition \ref{prop:sigchar}) are therefore
related by
$$\langle,\rangle_{J^\mu}^0 = \epsilon(\mu)
\langle,\rangle_{J^\mu}^c;$$
the sign $\epsilon(\mu) = \pm 1$ is defined in Definition
\ref{def:eqrk}. Consequently the signature functions are related by
$$(\POS_J^0(\mu),\NEG_J^0(\mu)) = \begin{cases}
(\POS_J^c(\mu),\NEG_J^c(\mu)) & (\epsilon(\mu) = +1)\\
(\NEG_J^c(\mu),\POS_J^c(\mu)) & (\epsilon(\mu) = -1)\\
\end{cases}$$

\end{theorem}
\begin{proof} 
The formula for the invariant Hermitian form is a
  special case of Proposition \ref{prop:changeform}; the two
  automorphisms $\delta$ and $\epsilon$ of that proposition are both
  $\theta$, and we can choose $\lambda = x$. The statements about
  signatures are immediate. \end{proof}

Theorem \ref{thm:ctoinveqrk} provides a very complete, explicit, and
computable passage from the signature of a $c$-invariant Hermitian
form to the signature of a classical invariant Hermitian form, {\em
  always assuming that $G$ is equal rank}. If $G$ is not equal rank,
there is really no parallel result, as one can see by investigating
the signatures of invariant Hermitian forms on finite-dimensional
representations of $SL(2,{\mathbb C})$ and $SL(3,{\mathbb R})$.  The
difficulty is that there is no element of $G$ whose adjoint action is
the Cartan involution $\theta$.  In the next section we will
address this problem by enlarging $G$ slightly to an {\em extended
  group} ${}^\delta G$ (containing $G$ as a subgroup of index
two). The key property of the extended group is that the Cartan
involution is inner, and therefore we get an analogue of Theorem
\ref{thm:ctoinveqrk}.  On the other hand, Clifford theory provides a
very close and precise relationship between representation theory for
$G$ and for ${}^\delta G$. In particular, information about signatures
of invariant Hermitian forms can be passed between the two groups.

\section{Twisting by the Cartan involution}\label{sec:theta} 
% \section{Twisting by $\theta$}\label{sec:theta} 
\setcounter{equation}{0}

We begin as in Theorem \ref{thm:realforms} with a real reductive algebraic 
\begin{subequations}\label{se:delta}
\begin{equation}
G= G({\mathbb R},\sigma) \subset G({\mathbb C}),
\end{equation}
and a maximal compact subgroup
\begin{equation}
K = G^\theta.
\end{equation}
Choose a maximal torus (a maximal connected abelian subgroup)
\begin{equation}
T_f \subset K,
\end{equation}
and define
\begin{equation*}
H_f = Z_G(T_f),
\end{equation*}
a {\em fundamental maximal torus in $G$}. It turns out that $H_f$ is the
group of real points of
\begin{equation}
H_f({\mathbb C}) = Z_{G({\mathbb C})}(T_f),
\end{equation}
and that this group is a complex maximal torus. Clearly $H_f$ is
uniquely defined up to conjugation by $K$. Recall from Definition
\ref{def:torus} the dual lattices
\begin{equation}
X^* =_{\text{def}}\Hom_{\text{alg}}(H_f({\mathbb C}),{\mathbb
  C}^\times), \qquad X_* =_{\text{def}}\Hom_{\text{alg}}({\mathbb
  C}^\times, H_f({\mathbb C}));
\end{equation}
the automorphism $\theta$ acts on these lattices (as a lattice
automorphism of order two).  Recall from Definition \ref{def:posroots}
the finite subsets (in bijection) of roots and coroots
\begin{equation}
R = R(G,H_f) \subset X^*, \qquad R^\vee = R^\vee(G,H_f)\subset X_*;
\end{equation}
these subsets are preserved by $\theta$.  Because every root has a
nontrivial restriction to $T_f$ (this is equivalent to the assertion
already used that $Z_G(T_f)$ is abelian) we can choose a system $R^+$ of
positive roots so that
\begin{equation}
\theta(R^+) = R^+, \qquad \theta(\Pi) = \Pi;
\end{equation}
here $\Pi$ is the set of simple roots for $R^+$ (Definition
\ref{def:posroots}).  We define
\begin{equation}\label{eq:basedrootdatumaut}
t_{\text{RD}} = \text{action of $\theta$ on $(X^*,\Pi,X_*,\Pi^\vee)$};
\end{equation}
the quadruple on the right is the {\em based root datum} of
$G({\mathbb C})$ (\cite{Spr}*{16.2.1}).
\end{subequations}

\begin{subequations}\label{se:pinning}
We need to lift the automorphism $t_{\text{RD}}$ to $G({\mathbb C})$. Of
course the Cartan involution $\theta$ is such a lift; but it is
convenient to make a simpler and more canonical choice.  For every
simple root $\alpha \in \Pi$, we fix a root vector
\begin{equation} X_\alpha \in {\mathfrak g}_\alpha \qquad (\alpha \in
  \Pi).
\end{equation}
This choice of simple root vectors is called a {\em pinning} for the
based root datum of \eqref{eq:basedrootdatumaut}. We can (and do) make
these choices in such a way that
\begin{equation*} \theta(X_\beta) = X_{\theta\beta} \qquad (\beta
  \ne \theta\beta \in \Pi).\end{equation*}
Since $\theta^2=1$, we must have
\begin{equation} \theta(X_\gamma) = \epsilon(\gamma)X_{\gamma} = \pm
  X_\gamma \qquad (\gamma= \theta\gamma \in \Pi).\end{equation}
These choices then define root vectors $X_{-\alpha}$ by the
requirement
\begin{equation*}
[X_\alpha,X_{-\alpha}] = H_\alpha = d\alpha^\vee(1) =
d\phi_\alpha\begin{pmatrix} 1&0 \\0 & -1\end{pmatrix} \qquad (\alpha
\in \Pi)
\end{equation*}
(notation as in Definition \ref{def:posroots}).   Finally, these
choices of root vectors may be made compatibly with the compact real
form $\sigma_c$ of Theorem \ref{thm:realforms}:
\begin{equation*} \sigma_c(X_\alpha) = -X_{-\alpha} \quad (\alpha \in
  \Pi).\end{equation*}
It is a standard fact (see for example \Cite{Spr}*{Theorem 9.6.2}) that any
automorphism (like $t_{\text{RD}}$) of the based root datum has a unique lift
to an automorphism of algebraic groups 
\begin{equation*}
\delta_f\colon G({\mathbb C}) \rightarrow G({\mathbb C})
\end{equation*}
preserving $H_f$, acting by the automorphism $t_{\text{RD}}$ on
$X^*(H_f({\mathbb C}))$, and satisfying
\begin{equation*} \delta_f(X_\alpha) = X_{t_{\text{RD}}(\alpha)}
  \qquad (\alpha \in  \Pi).\end{equation*}
Because of the uniqueness of $\delta_f$ and the explicit formulas given
above for $\theta$ and $\sigma_c$, we see that
\begin{equation}\label{eq:sigmathetacomm}
\delta_f^2 = 1, \qquad \delta_f\theta = \theta\delta_f, \qquad
\delta_f\sigma_c = \sigma_c\delta_f.
\end{equation}

\end{subequations}

We emphasize that the based root datum automorphism $t_{\text{RD}}$
is determined canonically by $\theta$, and in fact just by the {\em inner
  class} of $\theta$ (that is, by the $\Ad(G({\mathbb C}))$ coset of
$\theta$ in $\Aut(G)$). The lift $\delta_f$ to $G({\mathbb C})$ is
determined by $t_{\text{RD}}$ and the choice of pinning.

\begin{definition}\label{def:extgrp} With notation as in
  \eqref{se:pinning}, the {\em extended group for $G({\mathbb C})$} is
  the semidirect product
$${}^\delta G({\mathbb C}) = G({\mathbb C})\rtimes \{1,\delta_f\}.$$
According to \eqref{se:pinning}, the automorphism $\delta_f$ preserves
the subgroups $G$, $K$, and $K({\mathbb C})$. We can therefore form all of 
the corresponding extended groups, for example
$${}^\delta G = G\rtimes \{1,\delta_f\}.$$

Because of \eqref{eq:sigmathetacomm}, the real forms $\sigma_0$ and
$\sigma_c$ both extend to real forms of ${}^\delta G({\mathbb C})$
acting trivially on $\delta_f$: %LOTS MORE delta to change to delta_f
$${}^\delta G({\mathbb R},\sigma_0) = {}^\delta G, \quad {}^\delta
K({\mathbb R},\sigma_0) = {}^\delta K$$
and so on.   

A {\em strong involution for the real form $G$} is an element
$$x = x_0\delta_f \in G({\mathbb C})\delta_f = {}^\delta G({\mathbb
  C}) \backslash G({\mathbb C})$$
with the property that 
$$\Ad(x)|_{G({\mathbb C})} = \theta.$$
In particular, this implies that
$$x \in {}^\delta K({\mathbb C}), \qquad x^2  = z \in Z(G({\mathbb C})).$$
\end{definition}

By construction of $\delta_f$, strong involutions for $G$ must exist. In
fact, since $\delta_f$ and $\theta$ are automorphisms of $G({\mathbb
  C})$ that agree on the Cartan $H_f({\mathbb C})$, we must have
\begin{subequations}\label{se:stronginv}
\begin{equation}
x_0 \in H_f({\mathbb C}), \quad x_0^2 = x^2 = z \in Z(G({\mathbb C})).
\end{equation}
More precisely, the construction of $\delta_f$ in \eqref{se:pinning}
shows that $x_0$ must satisfy
\begin{equation}\begin{aligned}
\beta(x_0) = 1 \qquad &(\beta\ne\theta(\beta) \in \Pi),\\
\gamma(x_0) = \epsilon(\gamma) = \pm 1 \qquad &(\gamma=\theta(\gamma) \in \Pi).
\end{aligned}\end{equation}
Conversely, any solution $x_0$ of these equations determines a strong
involution for $G$.  Evidently we can find such a solution with the
additional property
\begin{equation*}
x_0 \in (H_f^\theta)_0 \text{\ has finite order}.
\end{equation*}
The finite order hypothesis implies that $\sigma_c(x_0) = x_0$ (since
$\sigma_c$ defines the unique compact real form of $H_f({\mathbb
  C})$); so $x_0$ is also fixed by $\sigma_0 = \theta\circ\sigma_c$,
and therefore
\begin{equation}
x_0 \in K \subset G.
\end{equation}

\end{subequations}

\begin{definition}\label{def:thetaparam} Suppose $\Gamma =
  (\Lambda,\nu)$ is a Langlands parameter, written as in Proposition
  \ref{prop:LCshape}. The {\em $\theta$-twist of $\Gamma$} is
$$ \Gamma^{\theta} = (\Lambda,\theta\nu) = (\Lambda,-\nu).$$
In the case of real infinitesimal character (that is, if $\nu$ is
real-valued) this is the same as the Hermitian dual (Definition
\ref{def:dualparam}):
$$ \Gamma^{\theta} = \Gamma^{h,\sigma_0} \qquad (\nu = \overline\nu).$$
\end{definition}

The next result uses Clifford theory to lift the Langlands
classification of irreducible representations of $G$ (Proposition
\ref{prop:LCshape}) to the extended group ${}^\delta G$.  Just as for
$G$, it says that each irreducible module $J'$ is the unique
irreducible quotient of a ``standard module'' $I'$.

\begin{proposition}\label{prop:Gdeltareps} In the setting of Definition
  \ref{def:extgrp}, fix a strong involution $x=x_0\delta$ for $G$ as
  in \eqref{se:stronginv} (so that $x_0\in K$, and $x^2 = z\in
  Z(G)\cap K$). Fix an irreducible $({\mathfrak g},K)$-module $J$,
  corresponding to a Langlands parameter $\Gamma = (\Lambda,\nu)$
  (Proposition \ref{prop:LCshape}); write $I_{\quo}$ for the
  standard module of which $J$ is the unique irreducible quotient.
\begin{enumerate}
\item The twists $J^{\delta_f}$ and $J^\theta = J^x$ (Definition
  \ref{def:pairaut}) are isomorphic (by the linear map by which
  $x_0\in K$ acts on $J$). Either twist therefore defines the same action of
  $\{1,\delta_f\}$ on the set $\widehat G$ (Definition \ref{def:HCmod}).
\item The twist $J^\theta$ corresponds to the Langlands parameter
  $\Gamma^\theta$.
\item If $J^\theta \simeq J$, then $J$ admits exactly two extensions
  $J_{\pm}$ to an irreducible module for ${}^\delta G$, differing
  by the sign character of ${}^\delta G/G \simeq \{1,\delta_f\}$. In
  this case $J$ or $J_\pm$ is type one (Definition \ref{def:typeext}). 
\item If $J^\theta \simeq J$, then also $I_{\quo}^\theta\simeq
  I_{\quo}$, and this module has exactly two extensions
  $I_{\quo,\pm}$ to a module for
  ${}^\delta G$. Each extension $I_{\quo,\pm}$ has a unique
  irreducible quotient $J_{\pm}$, one of the two extensions of $J$.
\item If $J^\theta \not\simeq J$, then 
$$J_{\text{ind}} = \Ind_{({\mathfrak g},K)}^{({\mathfrak g},{}^\delta K)} (J) = J
\oplus J^\theta$$ 
is the unique irreducible $({\mathfrak g},{}^\delta K)$-module
containing $J$. In
  this case $J$ and $J^\theta$ and $J_{\text{ind}}$ are type two
  (Definition \ref{def:typeext}).
\item If $J^\theta \not\simeq J$, then also $I_{\quo}^\theta\not\simeq
  I_{\quo}$, and
$$I_{\quo,\text{ind}} = \Ind_{({\mathfrak g},K)}^{({\mathfrak g},{}^\delta
  K)} (I_{\quo}) = I_{\quo} \oplus I_{\quo}^\theta$$ 
has $J_{\text{ind}} = J^\theta_{\text{ind}}$ as its unique irreducible quotient.
\end{enumerate}
\end{proposition}
This is essentially a specialization of Proposition
\ref{prop:extpairrep} (Clifford theory) to ${}^\delta G$. The labels
$\pm$ in (4) and (5) do not have any particular meaning; {\em it may not be
the case that one of the two extensions is preferred}.

% yet have any particular meaning; but in
% Section \ref{sec:langlandsext} we will see that there is in fact a
% preferred extension $J_+$.

The modules
$I_{\quo,\pm}$ and $I_{\quo,\text{ind}}$ are called {\it
  standard modules}; of course we can 
in exactly the same way define $I_{\sub,\pm}$ and
$I_{\sub,\text{ind}}$. 

Now Proposition \ref{prop:cdualstd} lifts easily to the extended
group. For brevity we write only the case of real infinitesimal
character.

\begin{proposition}\label{prop:cdualext} Suppose $J'$ is an
  irreducible $({\mathfrak g},{}^\delta K)$-module of real
  infinitesimal character, and  $I'_{\quo}$ is the corresponding
  standard module. 
\begin{enumerate}
\item The module $J'$ admits a nondegenerate $c$-invariant Hermitian
  form $\langle,\rangle^c$ that is unique up to a real scalar
  multiple. It may be chosen to be positive definite on the lowest
  ${}^\delta K$-types of $J'$.
\item The module $I'_{\quo}$ admits a nonzero $c$-invariant form
  that is unique up to a real scalar. This form has radical equal to
  the maximal proper submodule $I'_1$ of $I'_{\quo}$, and
  factors to a nondegenerate form on $J'$.
\item If $J'$ restricts to an irreducible $J$ for $G$ (so that also
  $I'_{\quo} = I_{\quo}$) then the $c$-invariant forms for
  ${}^\delta G$ are the forms for $G$.
\item If $J' = \Ind(J) = J + J^\theta$, then the $c$-invariant form on $J'$
  is equal to the orthogonal sum of the form for $J$ on both $J$ and
  $J^\theta$ (using the identification of the vector spaces for $J$
  and $J^\theta$ (Definition \ref{def:pairaut}). A similar statement
  holds for $I'$.
\end{enumerate}
\end{proposition}

This is immediate from Proposition \ref{prop:cdualstd}.  The result
allows us to pass information about signatures of $c$-invariant forms
back and forth between $G$ and ${}^\delta G$ with no difficulty.

There is one possible confusing point. The outer automorphism $\delta_f$
need not act trivially on ${\mathfrak Z}({\mathfrak g})$, so the
irreducible $J'$ may not be annihilated by a maximal ideal ${\mathfrak
  m}$ but only by a product ${\mathfrak m}{\mathfrak m}^{\delta_f}$ of two
maximal ideals. But the property that a maximal ideal correspond to
real infinitesimal character is preserved by $\delta_f$, so in this case
what we mean by ``real infinitesimal character'' for $J'$ is that
either of these two maximal ideals has real infinitesimal character.

\begin{definition}\label{def:centralstuff} Suppose that $G$ is a real
  reductive algebraic group with extended group ${}^\delta G$ (Definition
  \ref{def:extgrp}). Fix a strong involution $x \in {}^\delta
  K\backslash K$ as in \eqref{se:stronginv}, so that
$$ \theta = \Ad(x), \qquad x^2 = z \in Z({}^\delta G)\cap K$$

For every $\zeta \in {\mathbb C}^\times$, define
$$\widehat{{}^\delta G}_\zeta = \{ V'\in \widehat{{}^\delta G} \mid
z\cdot v = \zeta v \ \ (v\in V')\},$$
the set of irreducible $({\mathfrak g},{}^\delta K)$-modules in which
$z$ acts by $\zeta$. Similarly define $\widehat K_\zeta$.  

Fix a square root $\zeta^{1/2}$ of $\zeta$. On every $\mu'\in
\widehat{{}^\delta K}_\zeta$, $x$ must act by some square root of
$\zeta$; so there is a sign
$$\epsilon(\mu') \in \{\pm 1\}, \qquad \mu'(x) =
\epsilon(\mu')\zeta^{1/2} \qquad (\mu' \in \widehat{{}^\delta K}_\zeta).$$
Of course $\epsilon(\mu')$ depends on the choice of square root
$\zeta^{1/2}$. 
\end{definition}

\begin{theorem}\label{thm:ctoinv} Suppose $G$ is a real reductive
  algebraic group  with extended group ${}^\delta G$ (Definition
  \ref{def:extgrp}). Fix a strong involution $x$ for $G$, an eigenvalue
  $\zeta$ for the central element $z$, and a square root $\zeta^{1/2}$
  of $\zeta$ as in Definition \ref{def:centralstuff}.

Suppose $J' \in \widehat{{}^\delta G}_\zeta$ is an irreducible
representation of real
infinitesimal character.  Write $\langle,\rangle^c_{J'}$ for a
$c$-invariant Hermitian form on $J'$ positive on the lowest ${}^\delta
K$-types (which exists by Proposition \ref{prop:cdualext}). Then $J'$
admits an invariant Hermitian form $\langle,\rangle^0_{J'}$ defined by
the formula 
$$\langle v, w\rangle^0_{J'} = \zeta^{-1/2}\langle x\cdot
v,w\rangle^c_{J'} = \zeta^{1/2}\langle v,x\cdot w\rangle^c_{J'}.$$ 
For each $\mu' \in \widehat{{}^\delta K}_\zeta$, the forms on the $\mu'$
multiplicity spaces (Proposition \ref{prop:sigchar}) are therefore
related by
$$\langle,\rangle_{(J')^{\mu'}}^0 = \epsilon(\mu') \langle,
\rangle_{(J')^{\mu'}}^c;$$ 
the sign $\epsilon(\mu') = \pm 1$ is defined in Definition
\ref{def:centralstuff}. Consequently the signature functions are related by
$$(\POS_{J'}^0(\mu'),\NEG_{J'}^0(\mu')) = \begin{cases}
(\POS_{J'}^c(\mu'),\NEG_{J'}^c(\mu')) & (\epsilon(\mu') = +1)\\
(\NEG_{J'}^c(\mu'),\POS_{J'}^c(\mu')) & (\epsilon(\mu') = -1)\\
\end{cases}$$

\end{theorem}
\begin{proof} The formula for the invariant Hermitian form is a
  special case of Proposition \ref{prop:changeform}; by Cartan's
  original construction of the Cartan involution, the two
  automorphisms $\delta$ and $\epsilon$ of that proposition are both
  $\theta$, and we can choose $\lambda = x$. The statements about
  signatures are immediate.
\end{proof}

\section{Langlands parameters for extended
  groups}\label{sec:langlandsext} 
\setcounter{equation}{0}

Proposition \ref{prop:Gdeltareps} provides a nice abstract description
of the irreducible representations of an extended group. For the
unitarity algorithm which is our goal, we will need something more
concrete and precise. This is of the same nature as what we have
already done with the formulation of the Langlands classification in
Theorem \ref{thm:LC}, which replaced the ``tempered representation of
a Levi subgroup'' appearing in \Cite{LC} with a character of an
appropriate cover of a real maximal torus. Proposition
\ref{prop:Gdeltareps} says that an irreducible representation of $G$
fixed by $\theta$ and corresponding to a Langlands parameter has two
extensions to ${}^\delta G$. We are going to index these two
extensions by two extensions of the parameter. 
% A critical feature of
% this description is that there is always a {\em preferred} extension
% (Lemma \ref{lemma:realexttorus}(3)) characterized by requiring that
% some distinguished generator of the extension act trivially.

\begin{subequations}\label{se:extG}
Throughout this section we fix an extended group 
\begin{equation}
{}^\delta G = G \rtimes \{1,\delta_f\}
\end{equation}
as in Definition \ref{def:extgrp}, and a strong involution
\begin{equation}
x = x_0\delta_f \in {}^\delta K \backslash K, \quad \Ad(x) = \theta,
\quad x^2 = z \in Z(G)\cap K
\end{equation}
as in \eqref{se:stronginv}. We fix also a $\theta$-stable maximal
torus $H \subset G$ (Definition \ref{def:rhoim}). Evidently $H$ is
normalized by $x$, so we can consider the extended group
\begin{equation*}{}^x H = \langle H,x\rangle, \end{equation*}
and try to extend a Langlands parameter to a character of ${}^x
H$. This is almost always impossible when $H$ is not fundamental
(see \eqref{se:delta}). The reason is that what is preserved by
$\theta$ is not the parameter but only its conjugacy class under $K$. 
\end{subequations} %{se:extG}

The point of this section is to find replacements for ${}^x H$ to
which Langlands parameters can usefully extend, and in this way to
parametrize representations of ${}^\delta G$.  We begin with basic
information about the extended Weyl group.

\begin{proposition}\label{prop:extWC} In the setting of Definition
  \ref{def:extgrp} and \eqref{se:extG}, write $H({\mathbb C})$ for
  the complexification of the $\theta$-stable real torus $H$. Fix 
  a maximal torus $H({\mathbb C})$ and a system of positive roots
$$R^+_1 \subset R(G,H),$$
with simple roots $\Pi_1$. Write
$$S_1 = \{s_\alpha \mid \alpha \in \Pi_1\} \subset W(G({\mathbb
  C}),H({\mathbb C})) = W$$
for the corresponding set of simple reflections (Definition
\ref{def:posroots}). Define
$${}^\delta W =_{\text{def}} N_{{}^\delta G({\mathbb C})}(H({\mathbb
  C}))/H({\mathbb C}) \supset W$$
the {\em extended Weyl group of $H({\mathbb C})$ in $G({\mathbb C})$}.
\begin{enumerate}
\item There is a unique nonidentity class 
$$t_1(R^+_1) = t_1 \in {}^\delta W$$
with the property that $t_1\cdot R^+_1 = R^+_1$.
% \item There is a unique isomorphism from the extended Weyl group of
%  the fundamental Cartan (see \eqref{se:delta}) to the extended Weyl
%  group of $H({\mathbb C})$, induced by an inner automorphism from
%  $G({\mathbb C})$ carrying $R^+_f \subset R(G,H_f)$ to $R^+_1$ and
%  $t_{\text{RD}}$ to  $t_1$.
\item Conjugation by $t_1$ permutes the generators $S_1$ of $W$, and
  provides a semidirect product decomposition
$${}^\delta W = W\rtimes \{1,t_1\}.$$
\end{enumerate}
Write $T_1$ for the set of orbits of $\{1,t_1\}$ on $S_1$. We
identify $T_1$ with a collection of elements of order $2$ in $W$, as
follows. Suppose
$$t = \{s,s'\} \qquad s' = t_1(s).$$
Then we associate the orbit $t$ to an element of $W$ by
$$\begin{aligned}
t &\leftrightarrow \text{long element of $\langle s, s' \rangle$}\\
&= {\begin{cases} s & t = \{s = s'\} \\
ss' & t = \{s \ne s'\}, \quad ss' = s's\\
ss's = s'ss' & t = \{s \ne s'\},
\quad ss' \ne s's. \\
\end{cases}}\end{aligned}$$
Each of these elements has order two and is fixed by $t_1$.
\begin{enumerate}[resume] %\setcounter{enumi}{\thesaveenumi}
\item The group $W^{t_1}$ of fixed points of the automorphism
  $t_1$ acts as a reflection group
  (in fact a Weyl group) on the lattice $(X^*)^{t_1}$ of $t_1$-fixed
  characters of $H({\mathbb C})$.  The set $T_1$ described above is a
  set of Coxeter generators.
\item %Suppose $(H'({\mathbb C}),(R')^+)$ is another maximal torus and
  % set of positive roots. 
Recall from \eqref{se:delta} the fundamental maximal torus $H_f$ and
positive root system $R^+_f$ used to construct the extended
group. There is an element $g\in G({\mathbb C})$ with the property that
$$\Ad(g)(H_f({\mathbb C})) = H({\mathbb C}), \qquad \Ad(g)(R^+_f) = R^+_1.$$
The coset $gH_f({\mathbb C}) = H({\mathbb C})g$ is unique.
\item Conjugation by $g$ defines a canonical isomorphism
$${}^\delta W(G({\mathbb C}),H_f({\mathbb C}))\, {\buildrel
  \simeq \over \longrightarrow}\, {}^\delta W(G({\mathbb C}),H({\mathbb
  C})),$$
carrying $t_f$ to $t_1$ and $S_f$ (the simple reflections for
$W(G({\mathbb C}),H_f({\mathbb C}))$) to $S_1$. 
\item Conjugation by $g$ carries the lift $\delta_f \in
  {}^\delta G$ of $t_f$ to a lift
$$\delta_1 \in N_{{}^\delta G({\mathbb C})}(H({\mathbb C}))$$
of $t_1$. The $H({\mathbb C})$ conjugacy class of
$\delta_1$---{\em i.e.}, the coset $\delta_1[H({\mathbb
  C})^{-t_1}]_0$---is unique. We call these the {\em distinguished
  lifts} of $t_1$.
\end{enumerate}
\end{proposition}

This is an immediate consequence of the corresponding facts about the
Weyl group (see for example \Cite{Spr}*{Corollary 6.4.12 and
Proposition 8.24}). 

\begin{definition}\label{def:twinv} In the setting of Proposition
  \ref{prop:extWC}, a 
  {\em twisted involution} in $W$ is an element $w_1\in W$ with the
  property that 
$$w_1 t_1(w_1) = 1.$$
The {\em twisted conjugate of $w_1$ by $y\in W$} is
$$ w_1' = yw_1 [t_1(y)^{-1}].$$
Clearly $w_1$ is a twisted involution if and only if $w_1t_1$ is an
involution in the nonidentity coset $Wt_1 \subset {}^\delta W$. In
this correspondence, $w_1'$ is a twisted conjugate of $w_1$ if and
only if $w_1't_1$ is a $W$-conjugate of $w_1t_1$.
\end{definition}

\begin{lemma}\label{lemma:twinv} Suppose we are in the setting of Proposition
  \ref{prop:extWC} and Definition \ref{def:twinv}.
\begin{enumerate}
\item Each twisted conjugacy class of twisted involutions has a
  representative
$$w_1 = s_{\beta_1} s_{\beta_2}\cdots s_{\beta_m},$$
with $\{\beta_i\}$ an orthogonal collection of $t_1$-fixed roots.
\item The number $m$ is an invariant of the twisted conjugacy class of
  $w_1$. We have
$$m + \dim(\text{$-1$ eigenspace of $t_1$}) = \dim(\text{$-1$
  eigenspace of $w_1t_1$}).$$
\item The following conditions on an involution $t_2 \in Wt_1$ are
  equivalent:
\begin{enumerate}[label=\roman*)]
\item $t_2$ is conjugate by $W$ to $t_1$;
\item $t_2$ preserves some system of positive roots $R^+_2$;
\item $\dim(\text{$-1$ eigenspace of $t_2$}) = \dim(\text{$-1$
    eigenspace of $t_1$})$
\end{enumerate}
\end{enumerate}
\end{lemma}
\begin{proof}  Write
$${\mathfrak h}^*_{\mathbb Q} = X^*\otimes_{\mathbb Z} {\mathbb Q}$$
for the rational part of the dual of the Cartan. Decompose this space
as
$${\mathfrak h}^*_{\mathbb Q} = V_1(w_1t_1) \oplus V_{-1}(w_1t_1)$$
according to the eigenspaces of $w_1t_1$.
We prove (1) by induction on the dimension $d$ of $V_{-1}(w_1t_1)$. There
are no elements for which $d=-1$, so in that case the assertion in
(1) is empty.  

Suppose therefore that $d\ge 0$, and that the result is known for $d-1$. There
are two cases. Suppose first that there is a root $\gamma_1 \in
V_{-1}(w_1t_1)$.  Put $w_2 = s_{\gamma_1}w_2$, so that 
$$V_1(w_2t_1) = V_1(w_1t_1)\oplus {\mathbb Q}\beta, V_{-1}(w_2t_1) =
\ker(\beta^\vee) \cap V_{-1}(w_1t_1).$$
Clearly $w_2t_1$ is an involution with $-1$ eigenspace of dimension
$d-1$, so by induction
$$w_2t_1 = y[s_{\beta_2} s_{\beta_3}\cdots s_{\beta_m}]t_1 y^{-1}.$$ 
If we define $\beta_1 = y^{-1}\cdot \gamma_1$, then it is easy to
check that (1) holds for $w_1$.

Next, suppose that there is no root $\gamma \in
V_{-1}(w_1t_1)$. We may therefore choose a weight $\xi_2 \in
V_{1}(w_1t_1)$ vanishing on {\em no} coroots; for the only coroots
$\gamma^\vee$ on which every element of $V_{1}$ vanishes are those with
$\gamma \in V_{-1}$.  Now define
$$R^+_2 = \{\alpha \in R \mid \xi_2(\alpha^\vee) > 0\}.$$
This is a system of positive roots, and
$$w_1t_1(R^+_2) = R^+_2$$
since $w_1t_1$ fixes $\gamma_2$. Therefore $w_1t_1$ must be the unique
element $t_2\in Wt_1$ fixing $R^+_2$.  If $R_2^+ = yR_1^+$, then
evidently
$$w_1t_1 = t_2 = yt_1y^{-1}.$$
This proves (1) (with $m=0$).  

The remaining assertions in the lemma are elementary consequences of
(1).
\end{proof}
% If we begin with the maximal torus $H_f({\mathbb
%   C})$ used in the construction of ${}^\delta G({\mathbb C})$, then
% the class $\delta \in {}^\delta W$ is represented by the distinguished
% lift $\delta$ described in \eqref{se:pinning}. This fact, and the
% distinguished isomorphisms described in (4), justify the use of the
% unadorned symbol $\delta$.

\begin{definition}\label{def:exttorus} In the setting \eqref{se:extG},
  the {\em extended real Weyl group of $H$} is
$$\begin{aligned} {}^\delta W(G,H) &= N_{{}^\delta G}(H)/H \simeq N_{{}^\delta
K}(H)/(H\cap {}^\delta K) \\ &\subset {}^\delta W(G({\mathbb
C}),H({\mathbb C}))\end{aligned}$$(notation as in Definition \ref{def:rhoim} and Proposition
\ref{prop:extWC}). According to \eqref{se:extG}, the extended real
Weyl group always has an element
$$xH =_{\text{def}} \theta_H,$$
which is characterized by the two properties
$$\theta_H \notin W(G({\mathbb C}),H({\mathbb C})),
\qquad\text{$\theta_H$ acts on $H$ by the Cartan involution.}$$
Consequently ${}^\delta W(G,H)$ inherits from ${}^\delta W(G({\mathbb
  C}),H({\mathbb C}))$ the short exact sequence
$$1 \rightarrow W(G,H) \rightarrow {}^\delta W(G,H) \rightarrow
\{1,\delta\} \rightarrow 1;$$
but the splitting of this sequence by $\theta_H$ is not one of the
nice ones described in Proposition \ref{prop:extWC}(2).

 An {\em extended maximal torus in ${}^\delta G$} is a subgroup
${}^1 H \supset H$, subject to the following conditions.
\begin{enumerate}[label=\alph*)]
\item The group ${}^1 H$ is not contained in $G$, and $[{}^1 H:H] =
  2$. Equivalently, we require that ${}^1 H$ be generated by $H$ and a
  single element
$$t_1 = w_1\theta_H \in {}^\delta W(G,H)\backslash W(G,H),$$
  of order two.
\item There is a set $R^+_1 \subset R = R(G,H)$ of positive roots
 preserved by $t_1$.
\end{enumerate}
We do {\em not} include a particular choice of $R^+_1$ as part of the
definition of ${}^1 H$. Typically we will write something like
$\delta_1\in {}^1 H \cap {}^\delta K$ for a representative of $t_1$.
\end{definition}

The first thing to notice is that the obvious group $\langle
H,\theta_H\rangle$ (generated by $H$ and the strong involution $x$ of
\eqref{se:stronginv}) is {\em not} an extended maximal torus unless
the set $R^+_{\mathbb R}$ of real roots is empty; that is, unless $H$
is fundamental.  The reason we impose the requirement b) above is that
representation-theoretic information is typically encoded not just by
a Cartan subalgebra, but rather by a Borel subalgebra containing
it. (This is the central idea in the theory of highest weights.) 

Nevertheless, every maximal torus {\em is} contained in an extended
maximal torus.

\begin{example} \label{ex:tVZ} Following \eqref{se:stdVZparam}, choose
  ${\mathfrak q} 
  = {\mathfrak l} + {\mathfrak u}$, so that the roots of $H$ in in
  ${\mathfrak l}$ are precisely the real roots $R_{\mathbb R}$ of
  $H$ in ${\mathfrak g}$. Fix a set of positive real roots
  $R^+_{\mathbb R}$, and define
$$R^+_{\text{VZ}} = R^+_{\mathbb R} \cup \{\text{roots of $H$ in $\mathfrak
  u$}\}.$$
Write $w_{0,{\mathbb R}}$ for the long element of $W(R_{\mathbb R})$,
  which carries $R^+_{\mathbb R}$ to $-R^+_{\mathbb R}$, and
  $\sigma_{0,{\mathbb R}}$ for a representative of this Weyl group
  element in $N_K(H)$. Because
  ${\mathfrak u}$ is preserved by $\theta$, we find that
$$w_{0,{\mathbb R}}\theta_H(R^+_1) = R^+_1.$$
Therefore $H$ and $\delta_{\text{VZ}} = \sigma_{0,{\mathbb R}}\cdot x$
  generate an extended maximal torus ${}^{\text{VZ}} H$, with
  $t_{\text{VZ}} = w_{0,{\mathbb R}}\theta_H$. 
\end{example}

\begin{definition}\label{def:extLP} Suppose ${}^1 H$ is an extended
  maximal torus in ${}^\delta G$. 
An {\em extended Langlands parameter} is a triple $\Gamma_1 = ({}^1 H, \gamma_1,
R^+_{i{\mathbb R}})$, subject to the following conditions.

\begin{enumerate}[label=\alph*)]
\item The group ${}^1 H$ is an extended maximal torus for ${}^\delta G$,
 with distinguished generator $t_1 \in {}^\delta W(G,H)$ (Definition
 \ref{def:exttorus}). 
\item The element $t_1$ preserves the positive imaginary roots
  $R^+_{i{\mathbb R}}$ from $\Gamma_1$.  
\item As a consequence of (2), the character $2\rho_{i{\mathbb R}}$ of
  $H$ extends naturally to ${}^1 H$, and so defines the $\rho_{i{\mathbb R}}$
  double cover $\widetilde{{}^1 H}$ (Lemma \ref{lemma:rhoimcover}). We
  require that $\gamma_1$ be a level one irreducible representation of
  this cover (and therefore of dimension one or two).
\item The restriction of $\gamma_1$ to $\widetilde H$ is
  (automatically) a sum of one
  or two irreducible characters $\gamma$ of
  the $\rho_{i{\mathbb R}}$ double cover of $H$; and we require that
  $\Gamma = (H,\gamma, R^+_{i{\mathbb R}})$ be a Langlands
  parameter for $G$ (Theorem 6.1).
\end{enumerate}
We say that $\Gamma_1$ (or $\Gamma$) is {\em type one} if $\gamma_1$ is
one-dimensional; equivalently, if the corresponding parameter $\Gamma$
for $G$ is fixed by $t_1$. We say that $\Gamma_1$ (or $\Gamma$) is
{\em type two} if $\gamma_1 = \gamma_{\text{ind}}$ is two-dimensional
(Definition \ref{def:typeext}). In order to guarantee that a parameter
corresponds to an irreducible representation of ${}^\delta G$, we need
to require also 
% [???I'm not sure what (for the second condition
% f)). The right statement has to be a 
% consequence of the analysis in Proposition \ref{prop:extLP}, but I'm
% pretty sure I haven't found it yet.???]
\begin{enumerate}[resume,label=\alph*)]
\item if $\Gamma_1 = \Gamma_{\text{ind}}$ is type two, then $\Gamma$ is
  {\em not} equivalent to $\Gamma^\theta$.
% \item Fix a positive system $R^+_1 \supset R^+_{i{\mathbb R}}$
%  preserved by $t_1$.  If the positive coroot $\alpha^\vee$ is
%  vanishes on $d\gamma$, $d(\theta\gamma)$, $\rho_{i{\mathbb R}}$, 
%  and $\rho_{\mathbb R}$, then % $\theta\alpha$ is also positive [TOO
% STRONG!].
% $\alpha-\theta\alpha$ is not fixed by $t_1$. [???this is more or less
% just a guess. I think it may be consistent with Example \ref{ex:sl6C}???]
\end{enumerate}

Just as in Definition \ref{def:genparam}, we can define {\em
  continued} and {\em weak} extended parameters; on these we impose
only conditions a)--c), and insert ``continued'' or ``weak'' in d). We
do not require anything like e), just as weak Langlands
parameters are not required to be final (condition (5) of Theorem
\ref{thm:LC}). 
\end{definition}

% The notion of {\em equivalence} of extended parameters is a little
% subtle, so we postpone the definition until \ref{def:equivextLP} below.

\begin{danger} Of course one of the things that we want is that type
  one representations 
  of the extended group (Proposition \ref{prop:Gdeltareps}) correspond
  precisely to type one extended parameters. There is a subtlety here.
  Under hypotheses a)--d), it will be fairly easy to see that a type
  one extended parameter 
  corresponds to a type one representation, and that a type two
  representation corresponds to a type two parameter (see the remarks
  after Lemma \ref{lemma:extLP} below). Proposition
  \ref{prop:extLP} will provide a converse: that any type one
  representation has a type one parameter.  The proof of that
  proposition shows first of all that in order to find this type one
  parameter---that is, to find a one-dimensional extension of the
  character $\gamma$---we must choose the extended torus
  properly. (There is a type two extension, given by induction, for
  any extended torus containing $H$.)  The
  point of condition (e) is to require us to make a proper choice
  of extended torus, so that the type of the representation
  corresponds to the type of the parameter. \end{danger} 
% For singular parameters,
% there can be more than one choice of extended torus to which
% $\gamma$ extends. The point of condition (f) is to prefer just one
% of these possibilities, so that we can get a one-to-one
% parametrization of representations by (equivalence classes of)
% parameters.

We turn next to the definition of equivalence of extended
parameters. The subtlety here arises entirely from the possibility
that a single Langlands parameter may extend to two distinct extended
tori. Before giving the definition, we recall a little about the
stabilizer of a Langlands parameter. 

\begin{lemma}\label{lemma:stabGamma} Suppose $\Gamma = (\Lambda,\nu)$
  is a Langlands parameter decomposed as in Definition \ref{def:dLP},
  with $\Lambda = (\lambda,R^+_{i{\mathbb R}})$. Write
  $\rho_{i{\mathbb R}}$ for half the sum of the roots in
  $R^+_{i{\mathbb R}}$, and $\overline\lambda=d\lambda$, so that
  $[\lambda - \rho_{i{\mathbb R}}]$ is a character of $T$.
Define
$$R^0 = \{\alpha\in R({\mathfrak g},{\mathfrak h})\mid \langle
\alpha^\vee,\overline\lambda \rangle = \langle \alpha^\vee,\nu\rangle
= \langle \alpha^\vee,\rho_{i{\mathbb R}}\rangle = 0\},$$
the set of singular roots for the discrete parameter.
\begin{enumerate}
\item The roots $R^0$ are the root system for a real quasisplit Levi subgroup
$L^0$ of $G$.
\item The real Weyl group $W(L^0,H)$ is equal to the centralizer of
  $\theta_H$ in $W({\mathfrak l}^0,{\mathfrak h})$.
\item If $\alpha$ is a real root in $R^0$, then $s_\alpha\cdot\Lambda
  = \Lambda$, and $[\lambda - \rho_{i{\mathbb R}}](m_\alpha) = 1$
  (notation as in Definition \ref{def:rhoim}). 
\item The stabilizer of $\Gamma$ in $W(G,H)$ is equal to $W(L^0,H)$.
\item The character $[\lambda - \rho_{i{\mathbb R}}]$ of $T$ is trivial
  on the intersection of $T$ with the identity component of the
  derived group $[L^0,L^0]$.
\end{enumerate}
\end{lemma}
% \begin{proof} ???look it up in IC4???\end{proof}

\begin{definition}\label{def:equivextparam}
Two extended parameters of type two are said to be {\em equivalent} if
the underlying Langlands parameters are equivalent.  

In the type one case, we need to be concerned about the
possibility that the extended torus to which $\Gamma$ extends is not
unique. Suppose therefore that $\Gamma$ is a Langlands parameter on a
maximal torus $H$, and that
$${}^1 H = \langle H, t_1\rangle, \qquad {}^2 H = \langle H,
t_2\rangle$$
are two extended tori to which $\Gamma$ extends; that is, that
$$\Gamma^{t_1} = \Gamma^{t_2} = \Gamma.$$
Therefore $t_2=wt_1$, with $w\in W(G,H)^\Gamma = W(L^0,H)$ (Lemma
\ref{lemma:stabGamma}).  Choose a representative $\delta_1$ for $t_1$,
and a representative $\sigma$ for $w$ {\em belonging to the identity
  component of the derived group of $L^0$}. Then
$$\delta_2 = \sigma\delta_1$$
is a representative for $t_2$.  

Suppose $\Gamma_1$ and $\Gamma_2$ are extensions of $\Gamma$ to ${}^1
H$ and ${}^2 H$ respectively. We say that $\Gamma_1$ is {\em
  equivalent} to $\Gamma_2$ if
$$ [\lambda_1 - \rho_{i{\mathbb R}}](\delta_1) = [\lambda_2 -
\rho_{i{\mathbb R}}](\delta_2).$$
According to Lemma \ref{lemma:stabGamma}, this condition is
independent of the choice of representative $\delta_1$ above.
More generally, we say that a type one extended parameter $\Gamma_1$
is {\em equivalent} to $\Gamma_3$ if $\Gamma_3$ is conjugate by $K$
% to an extended parameter 
to some $\Gamma_2$ as above.
\end{definition}

We are going to show (Proposition \ref{prop:extLP} below) that the
parameter for a type one irreducible representation 
in fact extends to an appropriate extended torus.  The next lemma shows what we
need to do. 

\begin{lemma}\label{lemma:extLP}
Suppose $\Gamma = (H,\gamma,R^+_{i{\mathbb R}})$ is a Langlands
parameter (Theorem \ref{thm:LC}). Write $H=TA$ and $\Gamma =
(\Lambda,\nu)$, with $\Lambda = (\lambda,R^+_{i{\mathbb R}})$ as in Definition \ref{def:dLP}; put $\overline\lambda = d\lambda
\in {\mathfrak t}^*$, so that 
$$(\overline\lambda,\nu) \in {\mathfrak t}^* + {\mathfrak a}^* = {\mathfrak
  h}^*$$
represents the infinitesimal character of $J(\Gamma)$. 
Then $\Gamma$ extends to the extended torus ${}^1 H = \langle H,t_1 =
w_1\theta_H\rangle$ if and only if
\begin{enumerate}
\item $w_1\lambda = \lambda$;
\item $w_1(R^+_{i{\mathbb R}}) = R^+_{i{\mathbb R}}$; and
\item $w_1\nu = -\nu$.
\end{enumerate}
These three conditions are equivalent in turn to $w_1\cdot\Gamma =
\Gamma^{\theta_H}$. 
\end{lemma}
This is clear from the definitions (since the real Weyl group element
$w_1$ must commute with the action of $\theta_H$ on $H$).  In
particular, the lemma implies that a Langlands parameter $\Gamma$ on
$H$ can extend to 
some extended torus for $H$ {\em only if} $\Gamma^{\theta_H}$ is
conjugate to $\Gamma$ by the real Weyl group; that is (Proposition
\ref{prop:Gdeltareps}) only if $J(\Gamma)$ is type one.

The converse requires a bit more work, on which we now embark. 
Assume that $J(\Gamma)$ is type one, so that
\begin{subequations}\label{se:extLP} 
\begin{equation}
\Gamma^{\theta_H} = w\cdot \Gamma \qquad \text{(some $w\in W(G,H)$)}.
\end{equation}
Write $H=TA$ and 
\begin{equation}
\Gamma = (\Lambda,\nu), \qquad \Lambda = (\lambda,R^+_{i{\mathbb R}})
\end{equation}
as in Definition \ref{def:dLP}; put $\overline\lambda = d\lambda
\in {\mathfrak t}^*$, so that 
\begin{equation}
(\overline\lambda,\nu) \in {\mathfrak t}^* + {\mathfrak a}^* = {\mathfrak
  h}^*
\end{equation}
represents the infinitesimal character of $J(\Gamma)$.

We want to prove that $\Gamma$ extends to some extended torus ${}^1
H$, and in fact to understand all possible extended tori to which
$\Gamma$ extends.  The easiest way to construct an extended torus
$\langle H,t_1\rangle$ is to construct a system of positive roots
preserved by $t_1$; by Proposition \ref{prop:extWC}(1), such a system
determines $t_1$.  Since we want also that $t_1$ should carry $\Gamma$
to $\Gamma$, it is reasonable to use $\Gamma$ to construct positive
roots.  To simplify the
discussion, and because (in light of Theorem \ref{thm:unitarydual2}) it is
our primary interest, we consider $\Gamma$ of real
infinitesimal character. This means in particular that 
\begin{equation*}
\langle \alpha^\vee,\overline\lambda\rangle \in {\mathbb R},\quad
\langle\alpha^\vee,\nu\rangle \in {\mathbb R} \quad (\alpha \in
R({\mathfrak g},{\mathfrak h})).
\end{equation*}
(The first of these conditions is automatically true for any
$\Gamma$.)  If now $(\overline\lambda,\nu)$ is {\em regular}, then we
could define a positive root system
\begin{equation*}
R^+(\Gamma) = \{\alpha \in R({\mathfrak g},{\mathfrak h}) \mid \langle
\alpha^\vee,(\overline\lambda,\nu)\rangle > 0\}.
\end{equation*}
The only candidate for an extended torus to which $\Gamma$ extends in
this case is the one generated by the unique $t_1\in W\theta_H$
preserving $R^+(\Gamma)$. There are two difficulties with this
approach in general. One is that $(\overline\lambda,\nu)$
need not be regular. The second is that it is not so easy to see
whether the element $t_1$ defined in this way belongs to the extended
real Weyl group $W(G,H)\theta_H$. 

We therefore take a slightly different approach. Begin with what is
(more or less obviously) a ``triangular'' decomposition of the root
system, into the roots of the nil radical of a parabolic subalgebra;
the roots of the Levi factor; and the roots of the opposite nil radical:
\begin{equation}\begin{aligned}
R^+(\overline\lambda) &= \{\alpha\in R({\mathfrak g},{\mathfrak h}) \mid
\langle\alpha^\vee,\overline\lambda \rangle > 0\} \\
&\quad \cup \{\alpha \in R({\mathfrak g},{\mathfrak h}) \mid \langle
\alpha^\vee,\overline\lambda \rangle = 0,\
\langle\alpha^\vee,\rho_{i{\mathbb R}}\rangle > 0\} \\ 
R_0(\overline\lambda) & = \{\alpha \in R({\mathfrak g},{\mathfrak h}) \mid
\langle\alpha^\vee,\overline\lambda \rangle = 0,\
\langle\alpha^\vee,\rho_{i{\mathbb R}}\rangle = 0\}\\
R^-(\overline\lambda) &= -R^+(\overline\lambda); 
\end{aligned}\end{equation}
here $\rho_{i{\mathbb R}}$ is half the sum of the roots in $R^+_{i{\mathbb
  R}}$.  The parabolic subalgebra is 
\begin{equation}\label{eq:plambda}\begin{aligned}
{\mathfrak p}(\overline\lambda) &= {\mathfrak l}(\overline\lambda) + {\mathfrak n}(\overline\lambda),\\
R({\mathfrak n}(\overline\lambda),{\mathfrak h}) &= R^+(\overline\lambda), \quad R({\mathfrak
  l}(\overline\lambda),{\mathfrak h}) = R^0(\overline\lambda).
\end{aligned}\end{equation}

To make a parallel construction for $\nu$, we must fix a system
of positive real roots making $\nu$ weakly dominant:
\begin{equation*}
R^+_{\mathbb R} \supset \{\alpha \in R_{\mathbb R} \mid \langle
\alpha^\vee,\nu \rangle > 0\}.
\end{equation*}
It follows from Lemma \ref{lemma:stabGamma} that such a choice of $R^+_{\mathbb
  R}$ is unique up to the action of $W(G,H)^\Gamma$. With this choice
in hand, we can define
\begin{equation}\begin{aligned}
R^+(\nu) &= \{\alpha\in R({\mathfrak g},{\mathfrak h}) \mid
\langle\alpha^\vee,\nu \rangle > 0\} \\
&\quad \cup \{\alpha \in R({\mathfrak g},{\mathfrak h}) \mid \langle
\alpha^\vee,\nu \rangle = 0,\ \langle \alpha^\vee,\rho_{\mathbb
  R}\rangle > 0\} \\ 
R^0(\nu) &= \{\alpha \in R({\mathfrak g},{\mathfrak h}) \mid
\langle \alpha^\vee,\nu \rangle = 0,\ \langle
\alpha^\vee,\rho_{\mathbb R}\rangle = 0\}\\
R^-(\nu) &= -R^+(\nu); 
\end{aligned}\end{equation}
here $\rho_{\mathbb R}$ is half the sum of the roots in $R^+_{\mathbb
  R}$.  The corresponding parabolic % subalgebra 
is written
\begin{equation}\label{eq:pnu}\begin{aligned}
{\mathfrak p}(\nu) &= {\mathfrak l}(\nu) + {\mathfrak n}(\nu),\\
R({\mathfrak n}(\nu),{\mathfrak h}) &= R^+(\nu), \quad R({\mathfrak
  l}(\nu),{\mathfrak h}) = R^0(\nu).
\end{aligned}\end{equation}
We now combine these two constructions of parabolic subalgebras, defining
\begin{equation}\begin{aligned}
R^+(\lambda,\nu) &= R^+(\lambda) \cup \left(R^0(\lambda)\cap
  R^+(\nu)\right)\\
R^0(\lambda,\nu) &= R^0(\lambda) \cap R^0(\nu)\\
R^-(\lambda,\nu) &= -R^+(\lambda,\nu); 
\end{aligned}\end{equation}
The corresponding parabolic subalgebra is written
\begin{equation}\label{eq:plambdanu}\begin{aligned}
{\mathfrak p}(\overline\lambda,\nu) &= {\mathfrak l}(\overline\lambda,\nu) + {\mathfrak
  n}(\overline\lambda,\nu),\\
R({\mathfrak n}(\overline\lambda,\nu),{\mathfrak h}) &= R^+(\overline\lambda,\nu), \quad
R({\mathfrak l}(\overline\lambda,\nu),{\mathfrak h}) = R^0(\overline\lambda,\nu).
\end{aligned}\end{equation}
We could equally well have reversed the roles of $\overline\lambda$ and
$\nu$, obtaining a different parabolic subalgebra ${\mathfrak
  p}(\nu,\overline\lambda)$. The Levi subalgebra ${\mathfrak
  l}(\overline\lambda,\nu)$, corresponding to the coroots vanishing on
$\overline\lambda$, $\nu$, $\rho_{i{\mathbb R}}$, and $\rho_{\mathbb
  R}$, would be the same.  We will see that these two parabolic
subalgebras correspond to two constructions of the standard
representation for the extended group.  
\end{subequations} %se:extLP

\begin{proposition}\label{prop:extLP} Suppose $\Gamma =
  (H,\gamma,R^+_{i{\mathbb R}})$ is a Langlands parameter (Theorem
  \ref{thm:LC}). Assume that the conjugacy class of $\Gamma$ is fixed
  by twisting by $\theta$ (equivalently, by $x\in {}^\delta G$);
  equivalently, that the corresponding irreducible representation
  $J(\Gamma)$ extends to ${}^\delta G$. Write $H=TA$ and $\Gamma =
  (\Lambda,\nu)$ as in Definition \ref{def:dLP}; put $\overline\lambda
  = d\Lambda \in {\mathfrak t}^*$, so that
$$(\overline\lambda,\nu) \in {\mathfrak t}^* + {\mathfrak a}^* = {\mathfrak
  h}^*$$
represents the infinitesimal character of $J(\Gamma)$. Fix a system of
positive real roots $R^+_{\mathbb R}$ making $\nu$ weakly dominant.

\begin{enumerate}
\item The Levi subgroups $L(\nu)$ and $L(\overline\lambda)$ corresponding to
  $R^0(\nu)$ and $R^0(\overline\lambda)$ are $\theta$-stable and defined over
  ${\mathbb R}$.
\item The parabolic subalgebra ${\mathfrak p}(\nu)$ is defined over
  ${\mathbb R}$, and so corresponds to a real parabolic subgroup
  $P(\nu)=L(\nu)N(\nu)$ of $G$.
\item The parabolic subalgebra ${\mathfrak p}(\overline\lambda)$ is
  $\theta$-stable.  
\item The group $L(\nu)$ has no real roots for $H$, so $H$
  is a fundamental Cartan in $L(\nu)$.
\item The group $L(\overline\lambda)$ has no imaginary roots for $H$, so $H$ is
  a maximally split torus in the quasisplit group $L(\overline\lambda)$.
\end{enumerate}

\begin{enumerate}[resume]
\item The Levi subgroup $L(\overline\lambda,\nu) = L(\overline\lambda)
  \cap L(\nu)$ is real and $\theta$-stable.  
\item There are neither real nor imaginary
  roots of $H$ in $L(\overline\lambda,\nu)$, so
  $L(\overline\lambda,\nu)$ is locally isomorphic to a complex
  reductive group.  
\item The real Weyl group $W(\overline\lambda,\nu)$ for $H$ in
  $L(\overline\lambda,\nu)$ is contained in $W^\Gamma$. 
%the stabilizer of $\Gamma$.
\end{enumerate}
Choose a real Borel subgroup $B = HN \subset L(\overline\lambda)$ contained in
$P(\nu)\cap L(\overline\lambda)$, and write  
$$w_0(\overline\lambda) = \text{long element of
    $W(L(\overline\lambda),H)$.} $$
Because $L(\lambda,\nu)$ is locally isomorphic to a complex reductive
group, we can choose a sum of simple subsystems
$$R^L(\overline\lambda,\nu) \subset R(\overline\lambda,\nu)$$
in such a way that
$$R(\overline\lambda,\nu) = R^L(\overline\lambda,\nu)\ 
{\textstyle\coprod}\ \theta R^L(\overline\lambda,\nu), \qquad
R^R(\overline\lambda,\nu) =_{\text{def}} \theta R^L(\overline\lambda,\nu).$$
Write
$$R^{x,+}(\overline\lambda,\nu) = \text{roots in $N$} \quad (x=L,R).$$
Because the roots in $N$ are permuted by $-\theta$, we have
$$\theta R^{L,+}(\overline\lambda,\nu) = -R^{R,+}(\overline\lambda,\nu)
\qquad \theta R^{R,+}(\overline\lambda,\nu) =
-R^{L,+}(\overline\lambda,\nu).$$

We consider now two different positive root systems for ${\mathfrak
  h}$ in ${\mathfrak g}$:
$$R^{+,1} = R^+(\overline\lambda) \cup [R^0(\overline\lambda) \cap
R^+(\nu)] \cup R^{L,+}(\overline\lambda,\nu) \cup
-R^{R,+}(\overline\lambda,\nu),$$ 
$$R^{+,2} = R^+(\overline\lambda) \cup [R^0(\overline\lambda) \cap
R^+(\nu)] \cup R^{L,+}(\overline\lambda,\nu) \cup
R^{R,+}(\overline\lambda,\nu),$$ 
\begin{enumerate}[resume]
\item The element $w_0(\overline\lambda)$ fixes $\Gamma$.
\item The stabilizer of $(\Gamma,R^+_{\mathbb R})$ in the coset
  $W(G,H)\theta_H$ is $W(\overline\lambda,\nu)w_0(\overline\lambda)\theta_H$.
\item The element $t_2 =_{\text{def}} w_0(\overline\lambda)\theta_H$
  preserves $R^{+,2}$, and so generates an extended torus to which
  $\Gamma$ extends. 
\item The element $t_1 =_{\text{def}}
  w_0(\overline\lambda,\nu)w_0(\overline\lambda)\theta_H$ preserves
  $R^{+,1}$, and so generates an extended torus to which $\Gamma$ extends.
\item The action of $w_0(\overline\lambda)\theta_H$ preserves
  $R(\overline\lambda,\nu)$ and the system of positive roots $R^{L,+}
  \cup R^{R+}$ defined by $N$; so it defines an involutive automorphism of the
  corresponding Coxeter system
  $(W(\overline\lambda,\nu),S(\overline\lambda,\nu))$. The  
extended tori to which $\Gamma$ extends correspond to certain
conjugacy classes of twisted involutions in
this Coxeter system, the correspondence sending a twisted involution
$x$ to $\langle xw_0(\overline\lambda)\theta_H, H\rangle$. 
\end{enumerate}
\end{proposition}

\begin{proof}[Proof of Proposition \ref{prop:extLP}]
% \begin{subequations}\label{se:exttoriproof}
Recall from Proposition \ref{prop:realtorus} that complex conjugation
acts on the roots by $-\theta$. Therefore a Levi subalgebra containing
${\mathfrak h}$ is real if and only if it is $\theta$-stable. That the
sets $R^0(\nu)$ and $R^0(\overline\lambda)$ are $\theta$-stable is
immediate from the fact that $\theta\nu = -\nu$ and
$\theta\overline\lambda = \overline\lambda$. This proves (1).

For (2), the fact that $-\theta\nu=\nu$ implies that $R^+(\nu)$ is
preserved by complex conjugation.  Similarly (3) follows from
$\theta\overline\lambda = \overline\lambda$.

For (4), no real coroot can vanish on $\rho_{\mathbb R}$; and (5) is
similar. Part (6) is immediate from (1).

The first assertion in (7) follows from (4) and (5). That this forces
every simple (real) factor to be a complex group follows either from
the classification of real forms, or from any of the ideas leading to
the proof of that classification.

Part (8) is a consequence of the analysis of the Weyl group action on
Langlands parameters made in \Cite{IC4}.

The hypothesis on $\Gamma$ is that there is a $w\in W(G,H)$ such that
$w\theta_H\Gamma = \Gamma$; so 
$$w\overline \lambda = \overline\lambda, \quad wR^+_{i{\mathbb R}} =
R^+_{i{\mathbb R}}, \quad w\nu =
-\nu.$$ 
The first two conditions imply that $w\in W(L(\overline\lambda),H)$,
and the last that $w\in
W(\overline\lambda,\nu)w_0(\overline\lambda)$.  Now (10) follows, and
the rest of (9)--(13) is  elementary. % consequences.
% \end{subequations} %se:exttoriproof
\end{proof}

Section \ref{sec:stdmods} offered two constructions of standard
modules from Langlands parameters, using appropriately chosen systems
of positive roots. In the same way we will find two
constructions of standard modules for ${}^\delta G$ from extended
parameters. For a type two parameter, the standard representation is
simply induced from a standard representation for $G$; we will say
nothing more about that case. We continue to assume for simplicity
that the infinitesimal character is real, and therefore place
ourselves in the situation of \eqref{se:extLP}. Here is the first
construction, based on Langlands' construction of the standard
representations of $G$.

\begin{definition}\label{def:Lextparam} 
In the setting \eqref{se:extLP}, write
% \addtocounter{equation}{-1}
% \begin{subequations}\label{se:Lextparam}
\begin{equation*}{}^1 H = \langle H,\delta_1\rangle, \qquad t_1 =
  w_1\theta_H \in {}^\delta W(G,H) 
\end{equation*}
for the (first) extended maximal torus constructed in Proposition
\ref{prop:extLP}, and $\Gamma_1= (\Lambda_1,\nu)$ for an extended
Langlands parameter extending $\Gamma$, decomposed as in
Definition \ref{def:dLP}. Write $R^+_{i{\mathbb R}}$ for the system of
positive imaginary roots that is part of the discrete parameter
$\Lambda$, and $R^+_{{\mathbb R}}$ for a system of positive real roots
making $\nu$ weakly dominant and preserved by $t_1$ (such as the one
used in the proof of Proposition \ref{prop:extLP}).  Then (Lemma
\ref{lemma:extLP})
\begin{equation*}\label{eq:extpreserve}
t_1\Lambda = \Lambda, \quad t_1\nu = \nu, \quad t_1 R^+_{i{\mathbb R}}
= R^+_{i{\mathbb R}}, \quad t_1 R^+_{{\mathbb R}} = R^+_{{\mathbb R}}.
\end{equation*}
Define a real parabolic subgroup
\begin{equation*}
P(\nu) = L(\nu)N(\nu),\quad  t_1(P(\nu)) = P(\nu)
\end{equation*}
as in Proposition \ref{prop:extLP}(2). (This parabolic
subgroup is slightly {\em smaller} than the one in 
Langlands' original construction of standard representations, which
used only $\nu$ and not also the choice of positive real roots.) We may
therefore define
\begin{equation*}{}^1 L(\nu) = \langle L(\nu), t_1\rangle,\end{equation*}
the group generated by $L(\nu)$ and (any representative of) $t_1$; we get  
\begin{equation*}{}^1 P(\nu) = {}^1 L(\nu) N(\nu),\end{equation*}
a parabolic subgroup of ${}^\delta G$. Because the imaginary roots are
all in $L(\nu)$, $\Gamma_1$ is still an extended parameter for ${}^1
L(\nu)$. We will define
\begin{equation*}
I_{\quo,G}(\Gamma_1) = \Ind_{{}^1 P(\nu)}^{{}^\delta G}(
I_{\quo,L(\nu)}(\Gamma_1)\otimes 1).
\end{equation*}
(Here induction is normalized as usual.) So we are reduced to the
problem of defining the extended standard 
representation on ${}^1 L(\nu)$. 

The $\theta$-stable parabolic ${\mathfrak p}(\overline\lambda)$ of
Proposition \ref{prop:extLP}(3) clearly meets ${\mathfrak l}(\lambda)$
in a $\theta$-stable parabolic subalgebra
\begin{equation*}
{\mathfrak p}(\lambda)\cap {\mathfrak l}(\nu)=_{\text{def}} {\mathfrak
  p}_{{\mathfrak l}(\nu)}(\lambda) = {\mathfrak l}(\lambda,\nu) + {\mathfrak
  n}_{{\mathfrak l}(\nu)}(\lambda)
\end{equation*}
with Levi factor $L(\lambda,\nu)$ and preserved by $t_1$. By the
particular choice of $t_1$ (among all elements preserving $\Gamma$)
there is a $\theta$-stable Borel subalgebra ${\mathfrak
  b}(\lambda,\nu)$ in ${\mathfrak l}(\lambda,\nu)$, containing
${\mathfrak h}$, that is also
preserved by $t_1$. We therefore have a $t_1$-stable, $\theta$-stable
Borel subalgebra
\begin{equation*}
{\mathfrak b}(\lambda,\nu)+ {\mathfrak
  n}_{{\mathfrak l}(\nu)}(\lambda) =_{\text{def}} {\mathfrak
  b}_{{\mathfrak l}(\nu)} = {\mathfrak h} + {\mathfrak n}_{{\mathfrak l}(\nu)}
\end{equation*}
with Levi subgroup $H$.  Recall that $\Lambda_1$ involves a genuine
character $\lambda_1$ of the $\rho_{i{\mathbb R}}$ cover of ${}^1
T$. Following \eqref{eq:gammaq}, we define a character 
\begin{equation*}
\lambda_{1,{\mathfrak b}_{{\mathfrak l}(\nu)}} = [\lambda_1 -
\rho_{i{\mathbb R}}] - 2\rho({\mathfrak n}_{{\mathfrak l}(\nu)}\cap
{\mathfrak k}) + 2\rho_{c,i{\mathbb R}}
\end{equation*}
of ${}^1 T$. Here the first term in brackets is a character of ${}^1
T$ by Definition \ref{def:ncover}; the second, thought of as the ``sum
of the compact positive roots,'' means the character of ${}^1 T$ on the
top exterior power of the indicated nilpotent subalgebra; and the
third is the character on the top exterior power of the span of the
compact imaginary roots. (We cannot just say ``sum of roots,'' because
the extended torus ${}^1 T$ may permute the root spaces.)

Following Theorem \ref{thm:VZrealiz}, we use cohomological induction
to define
\begin{equation*}
I_{\quo,L(\nu)}(\Gamma_1) = \left({\mathcal L}^{{\mathfrak
      l}(\nu),{}^1 (L(\nu)\cap K)}_{{\mathfrak
      b}^{\text{op}}_{{\mathfrak
        l}(\nu)},{}^1 T}\right)^s(\Gamma_{1,{\mathfrak b}_{{\mathfrak
      l}(\nu)}}) 
\end{equation*}
Here $s=\dim({\mathfrak n}_{{\mathfrak l}(\nu)} \cap {\mathfrak
  k})$. A little more explicitly, the definition means this.
\begin{enumerate}
\item The representation $I_{\quo,L(\nu)}(\Gamma_1)$ contains the
  ${}^1(L(\nu)\cap K)$ representation of highest ${}^1 T$-weight (with
  respect to ${\mathfrak n}_{{\mathfrak l}(\nu)} \cap {\mathfrak k}$)
\begin{equation*}
\mu_1(\Gamma_1) = [\lambda_1 +
\rho_{i{\mathbb R}}] - 2\rho_{c,i{\mathbb R}}.
\end{equation*}
\item Equivalently, the weight
\begin{equation*}
\mu_1(\Gamma_1) - 2\rho({\mathfrak n}_{{\mathfrak l}(\nu)}\cap
{\mathfrak k}) = \lambda_{1,{\mathfrak b}_{{\mathfrak l}(\nu)}} +
2\rho({\mathfrak n}_{{\mathfrak l}(\nu)})
\end{equation*}
appears in $H_s\left({\mathfrak n}^{\text{op}}_{{\mathfrak l}(\nu)} \cap
  {\mathfrak k},I_{\quo,L(\nu)}(\Gamma_1)\right)$.
\item The restriction map
$$H_s\left({\mathfrak n}^{\text{op}}_{{\mathfrak
      l}(\nu)},I_{\quo,L(\nu)}(\Gamma_1)\right) \rightarrow H_s\left({\mathfrak
    n}^{\text{op}}_{{\mathfrak l}(\nu)} \cap
  {\mathfrak k},I_{\quo,L(\nu)}(\Gamma_1)\right)$$
is an isomorphism on the ${}^1 T$-weight space in question; and
\item The extended torus ${}^1 H$ acts on this ${}^1 T$-weight space
  by the character $\Gamma_{1,{\mathfrak b}_{{\mathfrak l}(\nu)}}+
  2\rho({\mathfrak n}_{{\mathfrak l}(\nu)})$.
\end{enumerate}
% \end{subequations}
\end{definition}

Here is the corresponding construction emphasizing % instead
cohomological induction.
\begin{definition}\label{def:VZextparam} 
In the setting \eqref{se:extLP}, write
% \addtocounter{equation}{-1}
% \begin{subequations}\label{se:VZextparam}
\begin{equation*}{}^2 H = \langle H,\delta_2\rangle, \qquad t_2 =
  w_2\theta_H \in {}^\delta W(G,H) 
\end{equation*}
for the (second) extended maximal torus constructed in Proposition
\ref{prop:extLP}, and $\Gamma_2= (\Lambda_2,\nu)$ for an extended
Langlands parameter extending $\Gamma$, decomposed as in
Definition \ref{def:dLP}. Write $R^+_{i{\mathbb R}}$ for the system of
positive imaginary roots that is part of the discrete parameter
$\Lambda$, and $R^+_{{\mathbb R}}$ for a system of positive real roots
making $\nu$ weakly dominant and preserved by $t_2$ (such as the one
used in the proof of Proposition \ref{prop:extLP}).  Then (Lemma
\ref{lemma:extLP})
\begin{equation*}% \label{eq:extpreserve}
t_1\Lambda = \Lambda, \quad t_1\nu = \nu, \quad t_1 R^+_{i{\mathbb R}}
= R^+_{i{\mathbb R}}, \quad t_1 R^+_{{\mathbb R}} = R^+_{{\mathbb R}}.
\end{equation*}
Define a $\theta$-stable parabolic subalgebra
\begin{equation*}
{\mathfrak p}(\overline\lambda) = {\mathfrak l}(\overline\lambda) +
{\mathfrak n}(\overline\lambda),\quad  t_2({\mathfrak
p}(\overline\lambda)) = {\mathfrak p}(\overline\lambda)
\end{equation*}
as in Proposition \ref{prop:extLP}(3). (This parabolic
subalgebra is slightly {\em smaller} than the one used in
\Cite{Vgreen} to construct standard representations; that one
used only $\overline\lambda$ and not also the system of positive
imaginary roots.) We may therefore define
\begin{equation*}{}^2 L(\overline\lambda) = \langle L(\overline\lambda), t_2\rangle,\end{equation*}
the group generated by $L(\overline\lambda)$ and (any representative
of) $t_2$. 

Just as in \eqref{eq:gammaq}, the extended parameter $\Gamma_2$ for
$G$ defines an extended parameter $\Gamma_{2,{\mathfrak p}(\overline\lambda)}$
for ${}^1 L(\overline\lambda)$. The cohomological induction
construction of the extended standard module is
\begin{equation*}\label{eq:VZextind}
I_{\quo,G}(\Gamma_2) = \left({\mathcal L}_{{\mathfrak
          p}^{\text{op}}(\overline\lambda),{}^2 L(\overline\lambda)\cap {}^\delta K}^{{\mathfrak
          g},{}^\delta
        K}\right)^s(I_{\quo,L(\overline\lambda)}(\Gamma_{2,{\mathfrak 
        p}(\overline\lambda)})).  
\end{equation*}
Here $s=\dim({\mathfrak n}(\overline\lambda)\cap{\mathfrak k})$, which
is different from (generally larger than) the $s$ appearing in Definition
\ref{def:Lextparam}. 

What remains is to construct the standard module
$I_{\quo,L(\overline\lambda)}(\Gamma_{2,{\mathfrak
    p}(\overline\lambda)})$ for the group 
${}^2 L(\lambda)$.  But this is easy: we saw in Proposition
\ref{prop:extLP}(5) that $L(\overline\lambda)$ is quasisplit, with a
Borel subgroup $B=HN$ making $\nu$ weakly dominant. The element $t_2$
was constructed to satisfy
\begin{equation*}
t_2(N) = N,
\end{equation*}
so we simply define ${}^2 B = {}^2 H N$, and
\begin{equation*}
I_{\quo,L(\overline\lambda)}(\Gamma_{2,{\mathfrak
    p}(\overline\lambda)}) = \Ind_{{}^2 B}^{{}^2 L
(\overline\lambda)} (\Gamma_{2,{\mathfrak p}(\overline\lambda)}).
\end{equation*}
That this representation extends
$$I_{\quo,L(\overline\lambda)}(\Gamma_{{\mathfrak
    p}(\overline\lambda)} =
  \Ind_B^{L(\overline\lambda)}(\Gamma_{{\mathfrak
      p}(\overline\lambda)})$$
is immediate.
% \end{subequations}
\end{definition}

\begin{theorem}[Langlands classification for extended groups]\label{thm:extLC}
Suppose that $G$ is the group of real points of a connected complex
reductive algebraic group (cf.~\eqref{se:reductive}), and that
${}^\delta G$ is a corresponding extended group
(cf.~\eqref{se:delta}). Then there is a one-to-one correspondence
between infinitesimal equivalence classes of irreducible quasisimple
representations of ${}^\delta G$ (Definition \ref{def:HCmod}) and
equivalence classes % (Definition \ref{def:equivextLP}) 
of extended
Langlands parameters for ${}^\delta G$ (Definition
\ref{def:extLP}). In this correspondence, type one representations
(those restricting irreducibly to $G$) correspond to one-dimensional
parameters; and type two representations (those induced irreducibly
from $G$) correspond to two-dimensional parameters.
\end{theorem}

This is a consequence of the Langlands classification for the
connected group (Theorem \ref{thm:LC}), of Clifford theory
(Proposition \ref{prop:extpairrep}), and of the analysis of extended
parameters made in Proposition \ref{prop:extLP} above.  Of course one
wants to know that the two constructions offered above lead to {\em exactly}
the same standard representations (rather than that the two
extensions to ${}^\delta G$ are sometimes interchanged). This can be
proven in the same way as
the equivalence of the constructions in Theorems \ref{thm:Lrealiz} and
\ref{thm:VZrealiz} above (\cite{KV}*{Theorem 11.129}). We omit the
details.

\section{Jantzen filtrations and Hermitian forms}\label{sec:jantzen} 
\setcounter{equation}{0}
For the rest of this paper we will need to work with extended groups
and their representations, in order to make use ultimately of the
description of invariant Hermitian forms in Theorem
\ref{thm:ctoinv}. In some cases, as in the present section, the
extension from $G$ to ${}^\delta G$ is entirely routine and
obvious. In such cases we may formulate results just for $G$, in order
to keep the notation a little simpler.

Character theory for Harish-Chandra modules is based on
expressing the (complicated) characters of irreducible modules as
(complicated) integer combinations of the (relatively simple)
characters of standard modules. We want to do something parallel for
signatures of invariant forms: to express the invariant forms on
irreducible modules as integer combinations of forms on standard
modules. A fundamental obstruction to this plan is part (3) of
Propositions \ref{prop:dualstd} and \ref{prop:cdualstd}: if the
Langlands quotient is proper, then the standard module $I(\Gamma)$
{\em cannot} admit a nondegenerate invariant Hermitian form.  

Jantzen (in \Cite{Jantzen}*{5.1}) developed tools to deal with this
obstruction; what we need is contained in Definition
\ref{def:jantzenform} below. The idea is that each Langlands parameter $\Gamma =
(\Lambda,\nu)$ (Definition \ref{def:dLP}) is part of a natural family
of parameters
\begin{equation}\label{eq:Gammat}
\Gamma_t =_{\text{def}} (\Lambda,t\nu) \qquad (0 \le t < \infty).
\end{equation}
For $t > 0$, it is evident that $\Gamma_t$ inherits from $\Gamma$ the formal
defining properties of a Langlands parameter.  At $t=0$, {\em all} the
real roots satisfy $\langle d\gamma,\alpha^\vee\rangle = 0$, and
condition 5 of Theorem \ref{thm:LC} can fail; $\Gamma_0$ may be only a
{\em weak} Langlands parameter (Definition \ref{def:genparam}). It is
a classical idea 
(originating in the definition of complementary series representations
for $SL(2)$ by Bargmann and Gelfand-Naimark, and enormously extended
especially by Knapp) that one can study questions of unitarity by
deformation arguments in the parameter $t$.  In order to do that, we
need to know that the family of representations indexed by $\Gamma_t$
is ``nice.'' (The parameter $t$ may be allowed to vary over all of
${\mathbb R}$ or ${\mathbb C}$ for many of the statements to follow;
this causes difficulties only with references to the ``quotient''
realization of a standard module, or to invariant Hermitian forms. In
a similar way, we could allow the continuous parameter to vary
over all of ${\mathfrak a}^*$, rather than just over the line
${\mathbb R}\nu$, with similar caveats.)

\begin{proposition}[Knapp-Stein \Cite{KSII}*{Theorem
  6.6}]\label{prop:meromorphic}  If $\Gamma = (\Lambda,\nu)$ is a Langlands
  parameter, then all of the (weak) standard modules (for $t \ge 0$)
  $I_{\quo}(\Gamma_t)$ and $I_{\sub}(\Gamma_t)$
  (cf.~\eqref{eq:Gammat}) may be realized on a common space 
$$V(\Lambda) =_{\text{def}} \Ind_{M\cap K}^K (D(\Lambda)|_{M\cap
  K})$$
(notation as in \eqref{se:ds}). In this realization 
\begin{enumerate}
\item the action of $K$ (and therefore of the Lie algebra
  ${\mathfrak k}$) is independent of $t$. 
\item The action of the Lie algebra ${\mathfrak g}$ in
  $I_{\quo}(\Gamma_t)$ depends in an affine way on $t$.  That
  is, for each $X$ in ${\mathfrak g}$ there are linear operators
  $Q_0(X)$ and $Q_1(X)$ on $V(\Lambda)$ so that the action of $X$ in
  $I_{\quo}(\Gamma_t)$ is $Q_0(X)+tQ_1(X)$. When $V(\Lambda)$ is
  interpreted as a space of functions on $K$, the operators $Q_0(X)$
  are first-order differential operators, and the operators $Q_1(X)$
  are multiplication operators. 
\item The action of the Lie algebra ${\mathfrak g}$ in
  $I_{\sub}(\Gamma_t)$ depends in an affine way on $t$. That
  is, for each $X$ in ${\mathfrak g}$ there are linear operators
  $S_0(X)$ and $S_1(X)$ on $V(\Lambda)$ so that the action of $X$ in
  $I_{\quo}(\Gamma_t)$ is $S_0(X)+tS_1(X)$. When $V(\Lambda)$ is
  interpreted as a space of functions on $K$, the operators $S_0(X)$
  are first-order differential operators, and the operators $S_1(X)$
  are multiplication operators.
\end{enumerate}
For $t>0$, there is a nonzero interwining operator
$$L_t= L_t(\Gamma)\colon I_{\quo}(\Gamma_t) \rightarrow
I_{\sub}(\Gamma_t),$$  
unique up to a nonzero scalar multiple. Fix a lowest $K$-type $\mu$
of $J(\Gamma)$ (equivalently, a lowest $K$-type of $V(\Lambda)$) and
normalize $L_t$ by the requirement
$$L_t|_\mu = \Id.$$  
\begin{enumerate}[resume]
\item The intertwining operator $L_t$ is a rational function of $t$
  (that is, each matrix entry is a quotient of polynomials) having no
  poles on $[0,\infty)$. 
\item The limit operator $L_0$ is a unitary intertwining operator
  between two unitarily induced representations.
\item Suppose that $J(\Gamma)$ admits a nonzero invariant or $c$-invariant
  Hermitian form $\langle,\rangle_1$. Regard this as a form on
  $I_{\quo}(\Gamma)$ with radical equal to the maximal
  submodule, and therefore as a form on $V(\Lambda)$. Then this form
  has a canonical extension to a rational family of forms
  $\langle,\rangle_t$ on $V(\Lambda)$, characterized by the two
  requirements
\begin{enumerate}
\item for $t > 0$, the form $\langle,\rangle_t$ is invariant for
  $I_{\quo}(\Gamma_t)$; and
\item the form $\langle,\rangle_t$ is constant on the fixed lowest
  $K$-type $\mu$.
\end{enumerate}
The family of forms also satisfies
\begin{enumerate}[resume]
\item the form $\langle,\rangle_t$ has no poles on $[0,\infty)$; and
\item for $t > 0$, the radical of $\langle,\rangle_t$ is the
  maximal proper submodule, so that the form descends to a
  nondegenerate form on the Langlands quotient $J(\Gamma_t)$.
\end{enumerate}
\end{enumerate}
\end{proposition}

\begin{subequations}\label{se:jantzenfilt}
We are going to use the rational family of intertwining operators
$L(t)$ to analyze the submodule structure of
$I_{\quo}(\Gamma)$. We have already seen that defining
\begin{equation}
I_{\quo}(\Gamma)^0 =_{\text{def}} I_{\quo}(\Gamma), \qquad
I_{\quo}(\Gamma)^1 =_{\text{def}} \ker(L(1)) 
\end{equation}
has the consequence
\begin{equation*}
I_{\quo}(\Gamma)^0/I_{\quo}(\Gamma)^1 \simeq J(\Gamma).
\end{equation*}
Once we decide to consider the family of operators $L(t)$, it is
natural to try to define 
$$I_{\quo}(\Gamma)^2 {\buildrel?\over =} \{v \in V(\Lambda) \mid
 \text{$L(t)v$  vanishes to order two at $t=1$}\} \subset
 I_{\quo}(\Gamma)^1.$$ 
The difficulty is that this space need not be a submodule. Here is
why. Suppose that $L(t)v$ vanishes to second order, and $x$ is
in the Hecke algebra $R({\mathfrak g}_0,K)$. Write $\cdot_t$ for the
action in $I_{\quo}(\Gamma_t)$. We would like to show that
$L(t)[x\cdot_1 v]$ vanishes to second order at $t=1$. The natural way
to do that would be to interchange $x$ and $L(t)$. But this we are not
permitted to do: the correct intertwining relation is rather
\begin{equation*}
L(t)[x\cdot_t v] = x\cdot_t[L(t)v].
\end{equation*}
The correct definition of the Jantzen filtration is
\begin{equation}\label{eq:jantzenr}
I_{\quo}(\Gamma)^r = \left\{v\in V(\Lambda) \,\,\bigg|\,\,
\text{\parbox{.55\textwidth}{for some
  rational function $f_v(t)$ with $f_v(1) = v$, $L(t)f_v(t)$
  vanishes to order $r$ at $t=1$.}}\right\}
\end{equation}
It is easy to check that $f_v$ can be taken to be a polynomial of
degree at most $r-1$ without changing the space defined. Another way to
write the definition is
\begin{equation}\label{eq:jantzenrprime}
I_{\quo}(\Gamma)^r = \left\{v\in V(\Lambda) \,\,\bigg|\,\,
\text{\parbox{.55\textwidth}{for some rational function $f_v(t)$ with
    $f_v(1) = v$, $(t-1)^{-r}L(t)f_v(t)$ is finite at $t=1$.}}\right\}
\end{equation}

Now the invariance of $I_{\quo}(\Gamma)^r$ under the Hecke
algebra is easy: if $f_v$ is the rational function of $t$ certifying
the membership of $v$ in $I_{\quo}(\Gamma)^r$, and $x$ belongs
to the Hecke algebra, then we can define
$$f_{x\cdot_1 v}(t) =_{\text{def}} x\cdot_t f_v(t),$$
which is a rational function of $t$ taking the value $x\cdot_1 v$ at
$t=1$. Furthermore
$$L(t)[f_{x\cdot_1 v}(t)] = L(t)[x\cdot_t f_v(t)] = x\cdot_t[L(t)
f_v(t)].$$
The function $[L(t) f_v(t)]$ is assumed to vanish to order
$r$ at $t=1$, so this is true after we apply the operator-valued
polynomial function $x\cdot_t$.  That is, the function $f_{x\cdot_1 v}$
certifies that $x\cdot_1 v$ also belongs to $I_{\quo}(\Gamma)^r$.

The {\em Jantzen filtration} is the decreasing filtration by submodules
\begin{equation}
I_{\quo}(\Gamma) = I_{\quo}(\Gamma)^0 \supset
I_{\quo}(\Gamma)^1 \supset I_{\quo}(\Gamma)^2 \supset \cdots
\end{equation}
Because the image of the intertwining operator is the Langlands
quotient, we get
\begin{equation*}
I_{\quo}(\Gamma)^1 = \ker L(1), \qquad
I_{\quo}(\Gamma)^0/I_{\quo}(\Gamma)^1 \simeq J(\Gamma).
\end{equation*}

Since we are talking about rational functions, \eqref{eq:jantzenr}
could formally be extended to negative values of $k$, using order of
pole rather than order of vanishing. Since $L(t)$ has no poles for $t
> 0$, such a definition adds nothing for $I_{\quo}(\Gamma)$. But
it is perfectly reasonable to consider the inverse intertwining
operator
\begin{equation*}
L(t)^{-1} \colon I_{\sub}(\Gamma_t) \rightarrow
I_{\quo}(\Gamma_t),
\end{equation*}
(always on our fixed space $V(\Lambda)$.) Again the matrix
coefficients are rational functions
of $t$, but now there can be lots of poles in $(0,\infty)$.
For $r\ge 0$, we define
\begin{equation}\label{eq:jantzenminus}
I_{\sub}(\Gamma)^{-r} = \left\{v\in V(\Lambda) \,\,\bigg|\,\,
\text{\parbox{.59\textwidth}{for some
  polynomial function $g_v(t)$ with $g_v(1) = v$, $(t-1)^rL(t)^{-1}g_v(t)$
  is finite at $t=1$.}}\right\}
\end{equation}

On this side the Jantzen filtration is
\begin{equation}
0 = I_{\sub}(\Gamma)^{1} \subset
I_{\sub}(\Gamma)^0 \subset I_{\sub}(\Gamma)^{-1} \subset
\cdots \subset I_{\sub}(\Gamma).
\end{equation}
The Langlands subrepresentation is
\begin{equation*}
I_{\sub}(\Gamma)^0 = \im L(1), \qquad
I_{\sub}(\Gamma)^0/I_{\sub}(\Gamma)^1 \simeq J(\Gamma).
\end{equation*}
\end{subequations}

\begin{lemma}[Jantzen \Cite{Jantzen}*{Lemma 5.1}]\label{lemma:jantzenfilt}
Suppose $E$ and $F$ are vector spaces of finte dimension over a field
$k$, and
$$L\in k(t) \otimes_k \Hom_k(E,F) \simeq \Hom_{k(t)}(k(t)\otimes_kE,
k(t)\otimes_k F).$$
is a rational family of linear maps from $E$ to $F$. Assume that $L$
is invertible over $k(t)$; equivalently, that $\dim_k E = \dim_k F$,
and that in any bases for $E$ and $F$, the representation of $L$ as a
matrix of rational functions has determinant a nonzero rational
function. Define
$$E^r(1) = E^r = \left\{v\in E \,\,\bigg|\,\,
\text{\parbox{.56\textwidth}{for some
  rational function $f_v(t)$ with $f_v(1) = v$, $(t-1)^{-r}L(t)f_v(t)$
  is finite at $t=1$.}}\right\}$$
$$F^q(1) = F^q = \left\{w\in F \,\,\bigg|\,\,
\text{\parbox{.58\textwidth}{for some
  rational function $g_w(t)$ with $g_w(1) = w$, $(t-1)^{-q}L(t)^{-1}g_w(t)$
  is finite at $t=1$.}}\right\}$$
\begin{enumerate}
\item These are finite decreasing filtrations of $E$ and $F$, with
$$E^R = 0, \quad E^{-R} = E, \qquad \text{all $R$ sufficiently
  large,}$$
and similarly for $F$.
\item Taking residues defines natural isomorphisms
$$L^{[r]}(1) = L^{[r]}\colon E^r/E^{r+1} \buildrel\sim\over\rightarrow
F^{-r}/F^{-r+1}, \quad L^{[r]}v = (t-1)^{-r} L(t)f_v(t)$$
with $f_v$ as in the definition of $E^r$.  The inverse is
$(L^{-1})^{[-r]}$.
\item Define an integer $D$ by
$$\det(L) = (t-1)^D(\text{rational function finite and nonzero at $t=1$}).$$
Then
$$D = \sum_{r = -\infty}^\infty r\dim E^r/E^{r+1}.$$
\item Suppose that $L$ is finite at $t=1$. Then $E^{-r} = E$ for all
  $r\ge 0$, and 
$$D = \sum_{r=1}^\infty \dim E^r.$$
\end{enumerate}
\end{lemma}

We will have no occasion to make use of (3), but Jantzen made it a
powerful tool, so we include the statement. In fact Jantzen
looked only at the case when $L$ is finite at $t=1$, and wrote the
result in the form (4); so we include that version as well.

\begin{proof} %[Sketch of proof]
  We can
  multiply $L$ by a power of $(t-1)^c$ to ensure that $L$ is finite at
  $t=1$. (This shifts all the filtrations by $c$, and multiplies $\det
  L$ by $(t-1)^{c\cdot\dim E}$, and so shifts $D$ by $c\cdot\dim E$. All the
  conclusions of the theorem are essentially unchanged.) Now consider
  the local ring $A$ of rational functions finite at $t=1$. The ring
  $A$ is a unique factorization PID, and $t-1$ is the unique prime
  element. Look at the free $A$ modules
$$M = A \otimes_k E, \qquad M' = A\otimes_k F.$$
The operator $L$ amounts to an injective module map
$$L\colon M \hookrightarrow M'.$$
The module $M'/LM$ is finitely generated, so the elementary divisor
theorem provides a basis $\{m'_1,\dots, m'_n\}$ of $M'$ and
nonnegative integers
$$d_1 \le d_2 \le \cdots \le d_n,$$
so that
$$\{(t-1)^{d_1} m'_1,\ldots, (t-1)^{d_n}m'_n\}$$
is a basis of $LM$. Since $L$ is injective, the elements
$$m_i =_{\text{def}} L^{-1}((t-1)^{d_i} m'_i)$$
are automatically a basis of $M$.  Working in these bases, it is easy
to show that
$$E^r = \text{span of values at $1$ of }\{ m_i \mid d_i \ge r\}, $$
and
$$\det L = (t-1)^{\sum_{i=1}^n d_i}(\text{unit in $A$}).$$
The assertions of the lemma are now easy to check. (For example, in
the bases $\{m_i\}$ and $\{m_i'\}$, $L$ is diagonal with the entry $(t-1)^d$
appearing $\dim(E^d/E^{d+1})$ times.)
\end{proof}

\begin{proposition}[\cite{Vu}*{Theorem 3.8}]\label{prop:jantzenfilt}
Suppose $\Gamma = (\Lambda,\nu)$ is a Langlands parameter, and
$\Gamma_t = (\Lambda,t\nu)$ $(0 < t < \infty)$. Write
$$L_t \colon I_{\quo}(\Gamma_t) \rightarrow
I_{\sub}(\Gamma_t)$$
for the rational family of Knapp-Stein intertwining operators constructed in
Theorem \ref{prop:meromorphic}.  Define Jantzen filtrations (of submodules)
$$I_{\quo}(\Gamma) = I_{\quo}(\Gamma)^0 \supset
I_{\quo}(\Gamma)^1 \supset I_{\quo}(\Gamma)^2 \supset
\cdots,$$
and
$$0 =  I_{\sub}(\Gamma)^1 \subset
I_{\sub}(\Gamma)^0 \subset I_{\sub}(\Gamma)^{-1} \subset
\cdots$$
as in \eqref{se:jantzenfilt}.
\begin{enumerate}
\item Each Jantzen filtration is finite: for $R$ sufficiently large,
  $I_{\quo}(\Gamma)^R = 0$, and $I_{\sub}(\Gamma)^{-R}
  = I_{\sub}(\Gamma)$.
\item The Langlands quotient is
\begin{equation*}\begin{aligned}I_{\quo}(\Gamma)/\ker(L_1) &=
I_{\quo}(\Gamma)^0/I_{\quo}(\Gamma)^1 \simeq J(\Gamma)\\
&\simeq I_{\sub}(\Gamma)^0/I_{\sub}(\Gamma)^1 = \im L_1/\{0\}.
\end{aligned}\end{equation*}
\item The residual operators (Lemma \ref{lemma:jantzenfilt}(2)) define
  $({\mathfrak g}_0,K)$-module isomorphisms
$$L^{[r]}\colon I_{\quo}(\Gamma)^r/I_{\quo}(\Gamma)^{r+1}
\buildrel\sim \over\longrightarrow
I_{\sub}(\Gamma)^{-r}/I_{\sub}(\Gamma)^{-r+1}.$$
\item Suppose that $J(\Gamma)$ admits a nonzero invariant or
  $c$-invariant Hermitian form $\langle,\rangle_1$.  Extend this to a
  rational family of invariant forms $\langle,\rangle_t$ as in
  Proposition \ref{prop:meromorphic}. Then there are nondegenerate
  invariant forms
$$\langle, \rangle^{[r]} \text{\ on }
I_{\quo}(\Gamma)^r/I_{\quo}(\Gamma)^{r+1},$$
defined by
$$\langle v, w \rangle^{[r]} = \lim_{s \rightarrow 1}
(s-1)^{-r}\langle f_v(s),f_w(s)\rangle_s \quad
\left(\text{% \parbox{.28\textwidth}
$v,w \in I_{\quo}(\Gamma)^r$,
    $f_v$, $f_w$ as in \eqref{eq:jantzenr}.}\right)$$
  \end{enumerate}
\end{proposition}

\begin{proof} %[Sketch of proof]
  For the finiteness of the filtrations,
  recall that $I(\Gamma)$ has a finite number $N$ of irreducible composition
  factors. Choose some irreducible representation $\mu_i$ of $K$
  ($1\le i \le N$)
  appearing in each of these composition factors, and let $E\subset
  V(\Lambda)$ be the sum of the $\mu_i$-isotypic subspaces, a
  finite-dimensional space.  By the choice of the $\mu_i$,
\begin{equation}\label{eq:finitedim}
\text{%\parbox{.56\textwidth}{
each nonzero $({\mathfrak g}_0,K)$-submodule of
    $I_{\quo}(\Gamma)$ has nonzero intersection with $E$.}
\end{equation}
 Because they respect the
  action of $K$, the intertwining operators $L_t$ must preserve
  $E$. If we apply Lemma \ref{lemma:jantzenfilt} to the restrictions
  of $L_t$ to $E$, we conclude that (for large $R$),
  $I_{\quo}(\Gamma)^R \cap E = 0$. By \eqref{eq:finitedim}, it
  follows that $I_{\quo}(\Gamma)^R = 0$. The rest of (1) is
  similar. Part (2) is part of the Langlands classification Theorem
  \ref{thm:LC}.  Part (3) is immediate from Lemma
  \ref{lemma:jantzenfilt}, applied separately to each
  (finite-dimensional) $K$-isotypic subspace.

For part (4), invariant forms are described in Definition
\ref{def:invherm} in terms of intertwining operators with Hermitian
dual representations.  Propositions \ref{prop:dualstd} and
\ref{prop:cdualstd} compute the Hermitian duals of standard quotient-type
modules as standard sub-type modules.  In the presence of (for
example) a $c$-invariant form on $J(\Gamma)$, the conclusion is that
$$I_{\quo}(\Gamma_t)^{h,\sigma_c} \simeq
I_{\sub}(\Gamma_t).$$
The invariant pairings defining this isomorphism were constructed from
a fixed isomorphism for the underlying spaces
$$J\colon V(\Lambda)^h \simeq V(\Lambda).$$
Now it follows from the definitions that
$$\left([I_{\quo}(\Gamma)^r]^{h,\sigma_c}\right)^\perp \simeq
I_{\sub}(\Gamma)^{-r+1},$$
so that we get isomorphisms
$$[I_{\quo}(\Gamma)^r/I_{\quo}(\Gamma)^{r+1}]^{h,\sigma_c} \simeq
I_{\sub}(\Gamma)^{-r}/I_{\sub}(\Gamma)^{-r+1}.$$
The isomorphism $J\circ L^{[r]}$ from
$I_{\quo}(\Gamma)^r/I_{\quo}(\Gamma)^{r+1}$ to its
Hermitian dual is the invariant Hermitian form in (4).
\end{proof}

\begin{definition}\label{def:jantzenform} Suppose $\Gamma$ is a
  Langlands parameter, and that the Langlands quotient $J(\Gamma)$
  admits a nonzero invariant or $c$-invariant Hermitian form
  $\langle,\rangle$. Define nondegenerate invariant forms 
$$\langle, \rangle^{[r]} \text{\ on }
I_{\quo}(\Gamma)^r/I_{\quo}(\Gamma)^{r+1}$$
as in Proposition \ref{prop:jantzenfilt}.  Write
$$(\POS^{[r]},\NEG^{[r]}) \colon \widehat K \rightarrow {\mathbb N} \times
{\mathbb N}$$
for the signatures of these forms (Proposition
\ref{prop:sigchar}). The {\em Jantzen form for 
  $I_{\quo}(\Gamma)$} is the nondegenerate form
$$\sum_{r=0}^\infty \langle, \rangle^{[r]} \text{\ on }
\gr I_{\quo}(\Gamma) =_{\text{def}} \sum_{r=0}^\infty
I_{\quo}(\Gamma)^r/I_{\quo}(\Gamma)^{r+1}.$$
Write
$$(\POS_{I(\Gamma)}, \NEG_{I(\Gamma)}) = \sum_{r=0}^\infty
(\POS^{[r]},\NEG^{[r]})$$
for the signature of this form.
\end{definition}

This is all that we need for the next few sections.  Eventually,
however, we will be studying these forms by deformation in $t$, and we
will want to know how the signatures change with $t$.

\begin{corollary}[\cite{Vu}*{Theorem 3.8}]\label{cor:formchange}
Suppose $\Gamma$ is a
  Langlands parameter, and that the Langlands quotient $J(\Gamma)$
  admits a nonzero invariant or $c$-invariant Hermitian form
  $\langle,\rangle_1$. Consider
  the family of standard representations $I_{\quo}(\Gamma_t)$
  (for $t\ge 0$) and the family of invariant forms $\langle,\rangle_t$
  extending $\langle,\rangle_1$ as in Proposition
  \ref{prop:jantzenfilt}. For every $t\ge 0$, define forms 
$$\langle,\rangle_t^{[r]} \text{\ on }
I_{\quo}(\Gamma_t)^r/I_{\quo}(\Gamma_t)^{r+1}$$
with signatures
$$(\POS_t^{[r]},\NEG_t^{[r]})\colon \widehat K \rightarrow {\mathbb N}
\times {\mathbb N}$$
as in Definition \ref{def:jantzenform}. Write 
$$(\POS_t,\NEG_t) = \sum_{r=0}^\infty (\POS_t^{[r]}, \NEG_t^{[r]})$$
for the signature of the (nondegenerate) Jantzen form on $\gr
I_{\quo}(\Gamma_t)$.
\begin{enumerate}
\item The subspace $I_{\quo}(\Gamma_t)^1 = \ker L_t$ is zero
  unless $I_{\quo}(\Gamma_t)$ is reducible.  This can happen
  only if either
\begin{enumerate}[label=\alph*)]
\item there is a real root $\alpha$ such that 
$$\langle t\nu,\alpha^\vee\rangle \in {\mathbb Z}\backslash \{0\},
\qquad \Lambda_{\mathfrak q}(m_\alpha) = (-1)^{\langle t\nu,\alpha^\vee\rangle + 1}$$
(notation as in \eqref{eq:gammaq1}); or
\item there is a complex root $\delta$ such that $\langle
  (d{\lambda}, t\nu),\delta^\vee\rangle \in {\mathbb Z}$,
  and 
$$\langle t\nu, \delta^\vee\rangle > |\langle d{\lambda},
  \delta^\vee \rangle|.$$
\end{enumerate}
These conditions define a discrete set of values of $t$---in fact a
subset of a lattice in ${\mathbb R}$.  On the complement of this
discrete set, the form $\langle,\rangle_t$ is nondegenerate and of
locally constant signature
$$(\POS_t,\NEG_t) = (\POS_t^{[0]},\NEG_t^{[0]}).$$
\item Choose $\epsilon$ so small that $I(\Gamma_t)$ is irreducible for
  $t\in [1-\epsilon,1+\epsilon]\backslash\{1\}$. %  and for $t\in
                                %  (1,1+\epsilon]$.  
Then
$$(\POS_{1+\epsilon},\NEG_{1+\epsilon}) = (\POS_1,\NEG_1) = \sum_{r=0}^\infty
(\POS_1^{[r]},\NEG_1^{[r]}),$$
$$(\POS_{1-\epsilon},\NEG_{1-\epsilon}) = \sum_{r \text{\thinspace even}}
(\POS_1^{[r]},\NEG_1^{[r]}) + \sum_{r' \text{\thinspace odd}}
(\NEG_1^{[r']},\POS_1^{[r']}).$$
 Equivalently,
$$\POS_{1+\epsilon} = \POS_{1-\epsilon} + \sum_{r \text{\thinspace
     odd}} \left(\POS_1^{[r]} - \NEG_1^{[r]}\right),\quad
\NEG_{1+\epsilon} = \NEG_{1-\epsilon} + \sum_{r \text{\thinspace odd}}
\left(\NEG_1^{[r]} - \POS_1^{[r]}\right).$$
That is, the signature of the form changes at reducible points
according to the signature on the odd levels of the Jantzen filtration.
\end{enumerate}
\end{corollary}

We will not reproduce the (elementary) proof from \Cite{Vu}. The main
point is the limit formula at the end of  Proposition
\ref{prop:jantzenfilt}; the factor $(s-1)^{-r}$ is positive unless $r$
is odd and $s<1$.

\section{Signature characters for $c$-invariant
  forms}\label{sec:sigcharsc} 
\setcounter{equation}{0}

In this section we will begin to explain what it means to
``calculate'' the signature of a $c$-invariant Hermitian form on an
irreducible representation.  We have already seen in Theorem
\ref{thm:ctoinv} how to relate this to calculating the signature
of a classical 
invariant Hermitian form. Just as in Section \ref{sec:jantzen}, the
extension to ${}^\delta G$ is relatively easy; so to keep the notation
simple we discuss only $G$.

The idea is that the (Jantzen) Hermitian forms on standard modules
(Definition \ref{def:jantzenform}) are taken as building blocks, and
we seek to express the forms on irreducible modules in terms of these
building blocks.  In order to set the stage, we begin with the
classical and simpler case of character theory.  

\begin{subequations}\label{se:charformulas}  
Recall from \eqref{eq:center} the center of the enveloping algebra
${\mathfrak Z}({\mathfrak g})$. Fix an algebra homomorphism
(``infinitesimal character'')
\begin{equation*}
\chi\colon {\mathfrak Z}({\mathfrak g}) \rightarrow {\mathbb C}
\end{equation*}
By Harish-Chandra's theorem, $\chi$ corresponds to a single Weyl
group orbit ${\mathcal O}_{\mathfrak h}(\chi)$ in ${\mathfrak h}^*$, for
every Cartan subalgebra ${\mathfrak h}$ in ${\mathfrak g}$. We will
write
\begin{equation*}
\xi_\lambda\colon {\mathfrak Z}({\mathfrak g}) \rightarrow {\mathbb
  C} \qquad (\lambda \in {\mathfrak h}^*)
\end{equation*}
for the homomorphism defined by $W\cdot\lambda$ (Theorem
\ref{thm:HChom}). Define 
% \begin{equation*}\begin{aligned}
% B(\chi) = \ &\text{equivalence classes of Langlands parameters}\\
% &\ \ \text{of infinitesimal character $\chi$;}
% \end{aligned}\end{equation*}
\begin{equation}
B(\chi) = \text{equiv.~classes of Langlands parameters of
  infl.~character $\chi$;} 
\end{equation}
(cf.~Theorem \ref{thm:LC}). Sometimes it will be convenient to index
the infinitesimal character 
by some particular weight $\lambda \in {\mathfrak h}^*$, and to write
\begin{equation}% \begin{aligned}
B(\lambda) = \text{\ classes of parameters of infinitesimal character
  $\xi_\lambda$} = B(\xi_\lambda).% \end{aligned}
\end{equation}
We will occasionally need
\begin{equation}% \begin{aligned}
B_{\weak}(\lambda) = \text{\ weak parameters of
  infinitesimal character $\xi_\lambda$} = B_{\weak}(\xi_\lambda) %\end{aligned}
\end{equation}
(Definition \ref{def:genparam}).

In the notation of Theorem \ref{thm:LC}, $\Gamma \in
B(\chi)$ if and only if $d\gamma \in {\mathcal O}_{\mathfrak
  h}(\chi)$; equivalently, if and only if $\chi = \xi_{d\gamma}$.  The set
$B(\chi)$ is finite, of 
cardinality bounded by the sum over conjugacy classes of real Cartan
subgroups $H$ of
$$[H:H_0][W({\mathfrak g},{\mathfrak h}):W(G,H)].$$
Define
\begin{equation}\label{eq:multformstd}
\begin{aligned}
m_{\Xi,\Gamma} &= (\text{multiplicity of $J(\Xi)$ as composition
  factor of $I(\Gamma)$})\\
&=_{\text{def}} \ \  m_{I(\Gamma)}(J(\Xi)) \in {\mathbb N}.
\end{aligned}\end{equation}
In the Grothendieck group of virtual $({\mathfrak g}_0,K)$-modules,
this is the same as 
\begin{equation*}
[I(\Gamma)] = \sum_\Xi m_{\Xi,\Gamma} [J(\Xi)].
\end{equation*}
(The brackets denote the image in the Grothendieck group, and serve as
a reminder that this equality need not be true on the level of modules.)
Because the center of the enveloping algebra acts by scalars on the standard
module $I(\Gamma)$, the multiplicity $m_{\Xi,\Gamma}$ can be nonzero
only if $\Xi$ and $\Gamma$ belong to the same set $B(\chi)$. So we can
regard $m$ as a finite matrix of nonnegative integers, with entries
indexed by pairs from the finite set $B(\chi)$. Theorem \ref{thm:LC} says that
the diagonal entries are equal to one:
\begin{equation*}
m_{\Gamma,\Gamma} = 1.
\end{equation*}
We can impose a preorder $\LERP$ on $B(\chi)$ by the length of the
real part of the continuous parameter. (In Definition \ref{def:bruhat}
below, we will introduce the {\em Bruhat order}, which is a partial order on
$B(\chi)$ with fewer relations than this one, but still satisfying the
analogue of \eqref{eq:uppertri}.  The point of using $\LERP$
is simply to get quickly a weak version of upper triangularity, so
that we know certain matrices are invertible.) Langlands' proof of
Theorem \ref{thm:LC} shows that
\begin{equation}\label{eq:uppertri}
m_{\Xi,\Gamma} \ne 0 \implies \Xi \LERP \Gamma;
\end{equation}
and that if ``equality holds'' (that is, if the real parts of the
continuous parameters have the same length) then $\Xi = \Gamma$.
Therefore the matrix $m$ is upper triangular with integer entries and ones
on the diagonal. Accordingly it has an inverse with these same
properties: we can define $M_{\Gamma,\Psi} \in {\mathbb Z}$ by the
requirement
\begin{equation*}
\sum_{\Gamma} m_{\Xi,\Gamma}M_{\Gamma,\Psi} = \delta_{\Xi,\Psi} \qquad
(\Xi, \Psi,\Gamma \in B(\chi)).
\end{equation*}
This is equivalent to an equation in the Grothendieck group of virtual
$({\mathfrak g}_0,K)$-modules, 
\begin{equation}\label{eq:repformirr}
[J(\Psi)] = \sum_\Gamma M_{\Gamma,\Psi} [I(\Gamma)].
\end{equation}
We can also write this as an equation for distribution characters
(Theorem \ref{thm:distnchar}):
\begin{equation*}\label{eq:charformirr}
\Theta_{J(\Psi)} = \sum_\Gamma M_{\Gamma,\Psi} \Theta_{I(\Gamma)}.
\end{equation*}

Since the characters of the standard modules $I(\Gamma)$ were (in
principle) computed by Harish-Chandra, this formula expresses the
character of an irreducible representation $J(\Xi)$ in terms of known
distributions $\Theta_{I(\Gamma)}$, and a matrix of integers
$M_{\Gamma,\Psi}$. The integers $M_{\Gamma,\Psi}$ are computed by the
Kazhdan-Lusztig conjectures (to which we will return in Sections
\ref{sec:klv} and \ref{sec:extklv}). 
\end{subequations}

All of this discussion takes place in the Grothendieck group of finite
length $({\mathfrak g}_0,K)$-modules.  To make a parallel discussion
of invariant forms, we need something like a ``Grothendieck group of
modules with a nondegenerate invariant form.'' The difficulty with this idea is
that the restriction of a nondegenerate form to a submodule may be
degenerate; so the category of modules with a nondegenerate form is
not abelian, and it is not so clear how to define a Grothendieck
group.  This issue is addressed by Lemma 3.9 in \Cite{Vu}; we repeat
the argument (phrasing it a little more formally and generally) in the
next few results, culminating in Proposition \ref{prop:hermgrothgens}.

\begin{lemma}\label{lemma:formsub}
In the setting of Proposition \ref{prop:sigchar}, suppose $M$ is an
admissible $({\mathfrak h},L)$ module of finite length, admitting a
nondegenerate $\sigma$-invariant form $\langle,\rangle_M$. 
\begin{enumerate}
\item The form $\langle, \rangle_M$ defines an isomorphism $M\simeq
  M^{h,\sigma}$ (Definition \ref{def:invherm}).
\item Forming orthogonal complement
$$N^\perp =_{\text{def}} \{m \in M \mid \langle m, n\rangle = 0,
\text{\ all $n\in N$}\}$$
is inclusion-reversing on submodules of $M$, and $(N^\perp)^\perp
= N$.
\item The isomorphism of (1) restricts to
$$N^{h,\sigma} \simeq M/N^\perp.$$
\item The radical of the invariant form 
$$\langle,\rangle_N =_{\text{def}} \text{\ restriction of
  $\langle,\rangle_M$ to $N$}$$
is equal to $N\cap N^\perp$.
\item If $N_1$ and $N_2$ are submodules, then
$$(N_1 \cap N_2)^\perp = N_1^\perp + N_2^\perp, \qquad (N_1 +
N_2)^\perp = N_1^\perp \cap N_2^\perp.$$
\end{enumerate}
\end{lemma}
This is entirely elementary; the admissibility hypothesis allows us to
focus on the finite-dimensional multiplicity spaces $M^\delta$ defined
in \eqref{se:multspace}.

\begin{proposition}\label{prop:grothform}
In the setting of Proposition \ref{prop:sigchar}, suppose $M$ is an
admissible $({\mathfrak h},L)$ module of finite length, with a
nondegenerate $\sigma$-invariant form $\langle,\rangle_M$. Suppose $S$
is an $({\mathfrak h},L)$-submodule of $M$.  Use the notation of Lemma
\ref{lemma:formsub}. Define
$$R = S\cap S^\perp = \text{\ radical of form restricted to $S$}.$$
\begin{enumerate}
\item The form $\langle,\rangle_{S+S^\perp}$ has radical equal to $R$,
  and so descends to a nondegenerate form on $(S+S^\perp)/R$.
\item There is an orthogonal direct sum decomposition
$$(S+S^\perp)/R = (S/R) \oplus (S^\perp/R).$$
\item The form $\langle,\rangle_M$ provides a natural identification
$$R^{h,\sigma} \simeq M/R^\perp = M/(S+S^\perp).$$
\item The signature 
$$(\POS_M,\NEG_M)\colon \widehat{L} \rightarrow {\mathbb N} \times {\mathbb
  N}$$
is given by
$$\POS_M = \POS_{S/R} + \POS_{S^\perp/R} + \mult_R,$$
$$\NEG_M = \NEG_{S/R} + \NEG_{S^\perp/R} + \mult_R.$$
Here $\mult_R$ is the multiplicity function for the module $R$.
\end{enumerate}
\end{proposition}

The proof is immediate from the lemma.

If $A$ is any admissible $({\mathfrak h},L)$-module, then the direct
sum $A\oplus A^{h,\sigma}$ carries a natural nondegenerate form
vanishing on both summands; the formula is
\begin{equation}\label{eq:hyperbolic}
\langle (a,\xi), (a',\xi')\rangle_{\text{hyp}} = \xi(a') + \overline{\xi'(a)}.
\end{equation}
This is called the {\em hyperbolic form}.  Its signature is easily
computed to be
\begin{equation*}
\POS_{A\oplus A^{h,\sigma}} = \mult_A, \qquad \NEG_{A\oplus
  A^{h,\sigma}} = \mult_A.
\end{equation*}
The proposition says that, in terms of signatures of invariant forms,
the module $M$ looks like
$$(R\oplus R^{h,\sigma}) \oplus (S/R) \oplus (S^\perp/R).$$

It {\em need not} be the case that $M$ has this decomposition as a
module. For us such modules will arise from
``wall-crossing'' translation functors. For example, there is a
representation of $SL(2,{\mathbb R})$ carrying a nondegenerate invariant
Hermitian form and three irreducible composition factors: a discrete
series representation with $SO(2)$-types indexed by $\{2, 4,
6,\ldots\}$, appearing twice, and the trivial representation (having
only the $SO(2)$-type indexed by zero). The form has signature $(1,1)$
on each of the nontrivial $SO(2)$-types. There is a submodule $S$
containing the trivial representation and one copy of the discrete
series. On this submodule the radical $R$ is the discrete series
representation, $S^\perp$ is equal to $R$, and $S+S^\perp = S$. 

\begin{definition}\label{def:hermgroth}
 Suppose (as in Proposition \ref{prop:sigchar}) that
  $({\mathfrak h},L({\mathbb C}))$ is a pair with $L({\mathbb C})$
  reductive, and that $\sigma$ is a real structure defining a compact
  form $L$ (Definition \ref{def:pair}). The {\em Grothendieck group of
    finite length admissible $({\mathfrak h},L({\mathbb C}))$-modules
    with nondegenerate $\sigma$-invariant Hermitian forms} (briefly,
  the {\em Hermitian Grothendieck group}) is the
  abelian group ${\mathcal G}({\mathfrak h},L({\mathbb
    C}))^\sigma$ generated by such $(M,\langle,\rangle_M)$, subject to the
  following relations. Write $[M,\langle,\rangle_M]$ for the class
  in the Grothendieck group.  Whenever $S$ is a submodule of
  $M$, and $R$ is the radical of the restricted form
  $\langle,\rangle_S$, then we impose the relation 
$$[M,\langle,\rangle_M] = [R\oplus R^{h,\sigma},
\langle,\rangle_{\text{hyp}}] + [S/R, \langle,\rangle_{S/R}] +
[S^\perp/R,\langle,\rangle_{S^\perp/R}].$$
\end{definition}

\begin{subequations}\label{se:hermgroth}
The motivation for the definition is Proposition
\ref{prop:grothform}, which gives exact sequences
\begin{equation*}
0 \rightarrow R^\perp \rightarrow M \rightarrow R^{h,\sigma}
\rightarrow 0,\qquad 
0 \rightarrow R \rightarrow R^\perp \rightarrow S/R \oplus S^\perp/R
\rightarrow 0,
\end{equation*}
and therefore an equality
\begin{equation}
[M] = [R\oplus R^{h,\sigma}] + [S/R] + [S^\perp/R]
\end{equation}
in the ordinary Grothendieck group ${\mathcal G}({\mathfrak
  h},L({\mathbb C}))$ of admissible finite-length $({\mathfrak
  h},L({\mathbb C}))$-modules.
From this we conclude that there is a natural homomorphism of abelian groups
\begin{equation}
{\mathcal G}({\mathfrak h},L({\mathbb C}))^\sigma \rightarrow
{\mathcal G}({\mathfrak h},L({\mathbb C})), \qquad
[M,\langle,\rangle_M] \mapsto [M]
\end{equation}
defined by forgetting the form.  Proposition \ref{prop:grothform} also
implies that there is a well-defined signature homomorphism
\begin{equation}% \begin{split}
(\POS_{\bullet},\NEG_{\bullet})\colon {\mathcal G}({\mathfrak h},L({\mathbb
  C}))^\sigma \rightarrow \Map({\widehat L},{\mathbb Z} \times
{\mathbb Z}), \quad [M,\langle,\rangle_M] \mapsto (\POS_M,\NEG_M). %\end{split}
\end{equation}
\end{subequations}

The fundamental fact about the ordinary Grothendieck group is
\begin{proposition}\label{prop:grothgens}
Suppose that $({\mathfrak h},L({\mathbb C}))$ is a pair with $L({\mathbb C})$
  reductive (Definition \ref{def:pair}). The Grothendieck group of
  admissible $({\mathfrak h},L({\mathbb C}))$-modules of finite length
  is a free abelian group (that is, a free ${\mathbb Z}$-module) with
  generators the (equivalence classes of) 
  irreducible admissible $({\mathfrak h},L({\mathbb C}))$-modules.
  There is a well-defined multiplicity homomorphism
$$\mult_{\bullet}\colon {\mathcal G}({\mathfrak h},L({\mathbb
  C})) \rightarrow \Map({\widehat L},{\mathbb Z}), \quad [M] \mapsto
(\mult_M).$$
\end{proposition}

We want a corresponding statement about the Hermitian Grothendieck
group. In the preceding proposition, the ring of ordinary integers
appears as the Grothendieck group of complex vector spaces.  Roughly
speaking, the integers appearing are dimensions of ``multiplicity
spaces'' $\Hom_{{\mathfrak h},L({\mathbb C})}(J,M)$, with $J$ an
irreducible module.  (Of course this statement is not precisely
correct: the indicated $\Hom$ is too small to capture the full
multiplicity of $J$ in $M$.) For the Hermitian case, the role of
${\mathbb Z}$ is played by the Grothendieck group of vector spaces
with nondegenerate forms.

\begin{definition}\label{def:witt} The {\em signature ring} is the
  Hermitian Grothendieck group ${\mathbb W}$ of finite-dimensional vector spaces
  with nondegenerate Hermitian forms. (This is the Hermitian
  Grothendieck group 
$${\mathbb W} = {\mathcal G}(0,\{1\})^\sigma$$
of Definition \ref{def:hermgroth} in the case when the Lie algebra
${\mathfrak h}$ is zero and the group $L({\mathbb C})$ is trivial.)
The ring structure is defined by tensor product of Hermitian forms.
The ordinary Grothendieck group of this category is ${\mathbb Z}$, so
we get as in \eqref{se:hermgroth} an algebra homomorphism
$$\for\colon {\mathbb W} \rightarrow {\mathbb Z}$$
by forgetting the form.

The identity element of ${\mathbb W}$ is the class of
the one-dimensional space ${\mathbb C}$ with its standard (positive) Hermitian
form $\langle z, w\rangle = z\overline w$:
$$1 = [{\mathbb C},\langle,\rangle].$$
We also use the element
$$s = [{\mathbb C}, -\langle,\rangle].$$
Taking the tensor square of this form eliminates the minus sign, so
$$s^2 = 1.$$

Using Sylvester's law of inertia for Hermitian forms, we find that any
finite-dimensional vector space $V$ with a 
nondegenerate form $\langle,\rangle_V$ is isomorphic to a sum of
copies of these two cases:
$$[V,\langle,\rangle_V] = p\cdot 1 + q\cdot s, \qquad (p, q\in {\mathbb
  N}).$$ 
Furthermore
\begin{equation*}\begin{aligned}
{\mathbb W} &= \{p\cdot 1 + q\cdot s \mid p, q \in {\mathbb Z}\}\\
&\simeq {\mathbb Z}[s]/\langle s^2 - 1\rangle.
\end{aligned}
\end{equation*}
The reduction map (forgetting the form) is
$$\for\colon{\mathbb W} \rightarrow {\mathbb Z}, \qquad \for(p + qs) = p+q.$$
\end{definition}

The ring ${\mathbb W}$ allows us to make a small consolidation of
notation. In the setting of Proposition \ref{prop:sigchar}, suppose
$(V,\langle,\rangle_V)$ is an admissible $({\mathfrak h},L({\mathbb
  C}))$-module with a $\sigma$-invariant Hermitian form.  We can
define the {\em signature character of $V$} to be the function
\begin{equation}\label{eq:sigchar} % \begin{split}
\sig_V\colon \widehat{K} \rightarrow {\mathbb W}, \quad  \sig_V(\delta) =
[V^\delta/(\text{radical}),\langle,\rangle_V^\delta] =
\POS_V(\delta) + \NEG_V(\delta)s %\end{split}
\end{equation}
(notation as in Proposition \ref{prop:sigchar}(4).)

\begin{proposition}\label{prop:hermgrothgens}
In the setting of Definition \ref{def:hermgroth} and Definition
\ref{def:witt}, forming the tensor product of vector spaces with
$({\mathfrak h},L({\mathbb C}))$-modules (each endowed with a
nondegenerate Hermitian form) defines the structure of a ${\mathbb
  W}$-module on the Hermitian Grothendieck group ${\mathcal
  G}({\mathfrak h},L({\mathbb C}))^{\sigma}.$
\begin{enumerate}
\item This Grothendieck group
has the following set of generators as a ${\mathbb W}$-module:
\begin{enumerate}
\item For each irreducible admissible $({\mathfrak h},L({\mathbb
    C}))$-module $J$ such that $J\simeq J^{h,\sigma}$, fix a ``base'' choice of
  nondegenerate $\sigma$-invariant form $\langle,\rangle_{J,b}$.  Then the
  generator is $[J,\langle,\rangle_{J,b}]$; the subscript $b$ stands
  for ``base.''
\item For each inequivalent (unordered) pair of irreducible admissible
  $({\mathfrak h},L({\mathbb C}))$-modules $J'$ and $J'^{h,\sigma}$, the
  generator is $[J'\oplus J'^{h,\sigma}, \langle,\rangle_{\text{hyp}}]$
  (cf. \eqref{eq:hyperbolic}). This generator satisfies the relation
$$s\cdot [J'\oplus J'^{h,\sigma},\langle,\rangle_{\text{hyp}}] =
[J'\oplus J'^{h,\sigma},\langle,\rangle_{\text{hyp}}]. $$
\end{enumerate}
\item The Hermitian Grothendieck group is a free ${\mathbb W}$-module
  over these generators, except for the relations indicated in $1(b)$.
\item Any admissible $({\mathfrak h},L({\mathbb C}))$-module $M$ with
  a nondegenerate $\sigma$-invariant form $\langle,\rangle_M$ may be
  written uniquely as
\begin{equation*} \begin{aligned}
\ [M,\langle,\rangle_M] &= \sum_{J \simeq J^{h,\sigma}} (p_J(M) +
q_J(M)s)[J, \langle,\rangle_{J,b}] % \\ &\qquad
+  \sum_{J'\not\simeq J'^{h,\sigma}} m_{J'}(M)[J'\oplus J'^{h,\sigma}, \langle,\rangle_{\text{hyp}}]\\
&= \sum_{J \simeq J^{h,\sigma}} w_J(M)[J, \langle,\rangle_{J,b}] % \\ &\qquad
+ \sum_{J'\not\simeq J'^{h,\sigma}} m_{J'}(M)[J'\oplus J'^{h,\sigma},
\langle,\rangle_{\text{hyp}}].
\end{aligned}\end{equation*}
\end{enumerate}
Here all the $p_J(M)$, $q_J(M)$, and $m_{J'}(M)$ are nonnegative
integers; $m_{J'}(M)$ is the multiplicity of $J'$ as a composition
factor of $M$; and
$$p_J(M) + q_J(M) = m_J(M),$$
the multiplicity of $J$ as a composition factor of $M$.  The elements
$$w_J(M) =_{\text{def}} p_J(M) + q_J(M)s$$
are in  ${\mathbb W}$; in the formula $m_{J'}(M)$ could be replaced by
any element $w'$ of ${\mathbb W}$ with $\for(w') = m_{J'}(M)$.
\begin{enumerate}[resume]
\item The signature character of $M$ may be computed from those of the
  irreducible composition factors and the coefficients in (3):
$$\sig_M = \sum_{J \simeq J^{h,\sigma}} w_J(M)\sig_{J,b} +
\sum_{J'\not\simeq J'^{h,\sigma}} m_{J'}(M)\mult_{J'}(1+s).$$
\end{enumerate}
\end{proposition}

All of these statements follow from the decomposition in (3), which is
easily established by induction on the length of $M$.

\begin{subequations}\label{se:cinvtsigformulas}

Using the description of the Hermitian Grothendieck group in
Proposition \ref{prop:hermgrothgens}, we can extend the discussion in
\eqref{se:charformulas} to invariant forms. Partly to avoid small
technical difficulties (like the possible nonexistence of invariant
forms), and partly because it is what we need first, we will consider
a {\em real} infinitesimal character (Definition
\ref{def:realinf}) % (\cite{Vgreen}*{Definition 5.4.11}) 
\begin{equation*}
\chi\colon {\mathfrak Z}({\mathfrak g}) \rightarrow {\mathbb C}.
\end{equation*}
This means that the parameter set
\begin{equation*}
B(\chi) = \text{\ Langlands parameters of infinitesimal character $\chi$}
\end{equation*}
consists of parameters with {\em real} continuous part (see the
discussion after Theorem \ref{thm:unitarydual2}).  We take the real
structure $\sigma_c$ as in \eqref{se:gKrealstr}, and consider the
Grothendieck group
\begin{equation}
{\mathcal G}({\mathfrak g},K)^c
\end{equation}
of finite-length $({\mathfrak g},K)$ modules with nondegenerate
$c$-invariant Hermitian forms.
By Proposition \ref{prop:cdualstd}, there is for each $\Gamma\in
B(\chi)$ a {\em canonical} nonzero $c$-invariant form
\begin{equation}\label{eq:cancform}
\langle,\rangle^c_{J(\Gamma),b}
\end{equation}
characterized (up to a positive multiple) by being positive on every
lowest $K$-type.  According to Proposition \ref{prop:hermgrothgens},
the elements
\begin{equation*}
[J(\Gamma),\langle,\rangle^c_{J(\Gamma),b}]
\end{equation*}
constitute a basis for (the infinitesimal character $\chi$ part of)
${\mathcal G}({\mathfrak g},K)^c$.

Definition \ref{def:jantzenform} explains how to pass from
$\langle,\rangle^c_{J(\Gamma),b}$ to a nondegenerate form % (Jantzen) form
\begin{equation*}
\langle,\rangle^c_{I(\Gamma)} \ \  \text{on}\ \  \gr I_{\quo}(\Gamma).
\end{equation*}
Then Proposition \ref{prop:hermgrothgens} allows us to write this
Jantzen form in the Hermitian Grothendieck group as a ${\mathbb
  W}$-linear combination of the basis of irreducibles:
\begin{equation}\label{eq:formformstd}
[I(\Gamma), \langle,\rangle^c_{I(\Gamma)}] = \sum_\Xi w^c_{\Xi,\Gamma}
[J(\Xi),\langle,\rangle^c_{J(\Xi),b}]. 
\end{equation}
Each coefficient
\begin{equation}
w^c_{\Xi,\Gamma} = p^c_{\Xi,\Gamma} + q^c_{\Xi,\Gamma}s \in {\mathbb W}
\end{equation}
has nonnegative $p^c$ and $q^c$, and
\begin{equation*}
\for(w^c_{\Xi,\Gamma}) =_{\text{def}} p^c_{\Xi,\Gamma} +
q^c_{\Xi,\Gamma} =  m_{\Xi,\Gamma}.
\end{equation*}
These facts imply that 
\begin{equation}\label{eq:multboundssig}
w^c_{\Xi,\Gamma}\ne 0 \quad \text{if and only if} \quad  m_{\Xi,\Gamma} \ne 0.
\end{equation}
It follows that, just as for the multiplicity matrix, the coefficient
$w^c_{\Xi,\Gamma}$ can be nonzero only if $\Xi$ and $\Gamma$ belong to
the same set $B(\chi)$.  So we can
regard $w^c$ as a finite matrix of elements of the commutative ring
${\mathbb W}$, with entries
indexed by the finite set $B(\chi)$. Proposition
\ref{prop:jantzenfilt}(2) says that the diagonal entries are equal to one:
\begin{equation*}
w^c_{\Gamma,\Gamma} = 1.
\end{equation*}
Since the multiplicity matrix is upper triangular
(\eqref{eq:uppertri}),  \eqref{eq:multboundssig} implies
\begin{equation*}
w^c_{\Xi,\Gamma} \ne 0 \implies \Xi \LERP \Gamma;
\end{equation*}
and that if ``equality holds'' (that is, if the real parts of the
continuous parameters have the same length) then $\Xi = \Gamma$.
Therefore the matrix $w^c$ is upper triangular with entries in
${\mathbb W}$ and ones
on the diagonal. Accordingly it has an inverse with these same
properties: we can define $W^c_{\Gamma,\Psi} \in {\mathbb W}$ by the
requirement
\begin{equation*}
\sum_{\Gamma} w^c_{\Xi,\Gamma}W^c_{\Gamma,\Psi} = \delta_{\Xi,\Psi} \qquad
(\Xi, \Psi,\Gamma \in B(\chi)).
\end{equation*}
This is equivalent to an equation in the Hermitian Grothendieck group
% of virtual $({\mathfrak g}_0,K)$-modules, 
\begin{equation}\label{eq:formformirr}
\left[J(\Psi),\langle,\rangle^c_{J(\Psi),b}\right] = \sum_\Gamma W^c_{\Gamma,\Psi}
\left[I(\Gamma),\langle,\rangle^c_{I(\Gamma)}\right]. 
\end{equation}
We can also write this as an equation for signature functions (see
\eqref{eq:sigchar})
\begin{equation}\label{eq:sigformirr}
\sig^c_{J(\Psi)} = \sum_\Gamma W^c_{\Gamma,\Psi} \sig^c_{I(\Gamma)}.
\end{equation}

Formally this equation appears to calculate signatures (for
$c$-invariant forms) for the irreducible modules $J(\Psi)$ in terms of
the {\em finite} matrix $W^c$, each entry of which is a pair of
integers.  This sounds very good for the problem of computing
signatures explicitly. Nevertheless this formula is less satisfactory
than \eqref{eq:charformirr} for two reasons. First, there is no
Kazhdan-Lusztig conjecture to compute the coefficient matrix
$W^c_{\Gamma,\Psi}$. We will address this in Section
\ref{sec:klvform}. Second, we do not understand % so well 
the signature functions $\sig^c_{I(\Gamma)}$ for the standard modules
$I(\Gamma)$. We will address this in Section
\ref{sec:defto0}.
\end{subequations}

We conclude this section by rewriting Corollary \ref{cor:formchange} in our
new notation.

\begin{corollary}[\cite{Vu}*{Theorem 3.8}] \label{cor:sigchange} Suppose
  $\Gamma$ is a Langlands parameter, and that the Langlands quotient
  $J(\Gamma)$ admits a nonzero invariant or $c$-invariant Hermitian
  form $\langle,\rangle_1$. Consider 
  the family of standard representations $I_{\quo}(\Gamma_t)$
  (for $t\ge 0$) defined in \eqref{eq:Gammat}, and the family of
  invariant or $c$-invariant forms $\langle,\rangle_t$ 
  extending $\langle,\rangle_1$ as in Proposition
  \ref{prop:jantzenfilt}. For every $t\ge 0$, define forms 
$$\langle,\rangle_t^{[r]} \text{\ on }
I_{\quo}(\Gamma_t)^r/I_{\quo}(\Gamma_t)^{r+1}$$
with signatures
$$\sig^{[r]}_{I(\Gamma_t)} = (\POS_t^{[r]} + s\NEG_t^{[r]})\colon
\widehat K \rightarrow {\mathbb W}$$ 
as in Definition \ref{def:witt}. Write 
$$\sig_t = \sig_{I(\Gamma_t)} = \sum_{r=0}^\infty \sig^{[r]}_{I(\Gamma_t)}$$
for the signature of the (nondegenerate) Jantzen form on $\gr
I_{\quo}(\Gamma_t)$. Consider a finite subset
$${0 < t_r < t_{r-1} < \cdots <t_1 \le 1} \subset (0,1].$$
so that $I(\Gamma_t)$ is irreducible for $t\in (0,1]
\setminus\{t_i\}$.  
\begin{enumerate}
\item On the complement of $\{t_i\}$, $\langle,\rangle_t$ is
  nondegenerate of locally constant signature
$$ \sig_t = \sig_t^{[0]}.$$
\item In terms of the ``signature matrix'' defined in \eqref{eq:formformstd},
$$\sig_1 = \sum_{\Xi\in B(\chi)} w^c_{\Xi,\Gamma} \sig_{J(\Xi)}.$$
\item Choose $\epsilon$ so small that $I(\Gamma_t)$ is irreducible for
  $t\in [1-\epsilon,1+\epsilon]\backslash\{1\}$. % and for $t\in
                                % (1,1+\epsilon]$. 
Then
$$\sig_{1+\epsilon} = \sig_1 = \sum_{r=0}^\infty
\sig_1^{[r]},\qquad % $$ $$
\sig_{1-\epsilon} = \sum_{r \text{\thinspace even}}
\sig_1^{[r]} + \sum_{r' \text{\thinspace odd}}
s\cdot \sig_1^{[r']}.$$
 Equivalently,
$$\sig_{1+\epsilon} = \sig_{1-\epsilon} + (1-s)\sum_{r  % OLD WRONG(s-1)\sum_{r
   \text{\thinspace odd}} \sig_1^{[r]}.$$ 
That is, the signature of the form changes at reducible points
according to the signature on the odd levels of the Jantzen filtration.
\end{enumerate}
\end{corollary}

\section{Translation functors: first facts}\label{sec:translI}
\setcounter{equation}{0}
Our next serious goal is to introduce the Kazhdan-Lusztig
``$q$-analogues'' of the multiplicities and signatures defined in
Section \ref{sec:sigcharsc}. The powers of $q$ appearing will come
from a grading on the set $B(\chi)$ of Langlands parameters with a
fixed infinitesimal character (see \eqref{se:charformulas}). This
grading is combinatorially understandable only in the case of {\em
  regular} infinitesimal character $\chi$. We will define it for {\em
  singular} $\chi$ by some kind of ``collapsing'' from the regular
case, which makes sense because of the Jantzen-Zuckerman translation
principle. Our purpose in this section is to say enough about the
translation functors to explain these definitions.

Translation functors are defined using finite-dimensional
representations of $G$, and in particular by thoroughly understanding
their weights. These matters extend to ${}^\delta G$ only with care
and effort. We will therefore present the theory for $G$ in this
section, and discuss the extension to ${}^\delta G$ in Section
\ref{sec:translext}. 

Recall from \eqref{se:reductive} that $G$ is assumed to be the group
of real points of a connected complex reductive algebraic group
$G({\mathbb C})$.  So far we have made serious use of only two aspects
of this hypothesis. The first is that $G$ is in the Harish-Chandra
class, meaning that the automorphism $\Ad(g)$ of the complexified Lie
algebra ${\mathfrak g}$ is inner.  This requirement means that
$\Ad(G)$ acts trivially on the center ${\mathfrak Z}({\mathfrak g})$
of $U({\mathfrak g})$, which in turn implies that ${\mathfrak
  Z}({\mathfrak g})$ must act by scalars on any irreducible
$({\mathfrak g},K)$-module.

The second aspect is that the Cartan subgroups of $G$ (the
centralizers in $G$ of Cartan subalgebras of ${\mathfrak g}$ that are
defined over ${\mathbb R}$) are all abelian.  One effect of this is
to simplify slightly the statement of the Langlands classification,
because the character (of a double cover of a Cartan subgroup)
appearing in a Langlands parameter is necessarily one-dimensional. A
more subtle consequence is that the lowest $K$-types have multiplicity
one (Proposition \ref{prop:LCshape}).

A third aspect of our category of groups is that $G$ is isomorphic to
a group of {\em matrices}.  So far we have not made serious use of
this hypothesis. In the theory of translation functors it is necessary
for the simple formulation below of Theorem \ref{thm:cohfam}(2) (the
behavior of Langlands parameters under translation functors).

Even to formulate a theory of translation functors we will make use of
the first two properties again, through the following two
propositions.  (The reason the propositions are not trivial is
that $G$ may be disconnected.)

\begin{proposition}\label{prop:findimlwts}
Suppose $G$ is a real reductive group as in \eqref{se:reductive}. Fix
a real Cartan subgroup $H$ of $G$, and a set $R^+ \subset R(G,H)$ of
positive roots (Definition \ref{def:posroots}).

\begin{enumerate}
\item Suppose that $F$ is a finite-dimensional irreducible representation
  of $G$.  Then the $R^+$-highest weight space of $F$ is a
  one-dimensional irreducible representation $\phi = \phi(F) \in
  \widehat{H}$.
\item Write $d\psi \in {\mathfrak h}^*$ for the differential
  of any % (one-dimensional) 
character $\psi$ of $H$.  The set 
$$\Delta(F,H) \subset \widehat{H}$$
of weights of the finite-dimensional irreducible representation $F$ is
(if we write $\langle S \rangle$ for the convex hull of a subset $S$
of a vector space)
$$\Delta(F,H) = \{\psi \in \widehat{H} \mid \phi - \psi \in
{\mathbb Z}R(G,H),\ \  d\psi \in \left\langle W({\mathfrak
    g},{\mathfrak h})\cdot d\phi \right\rangle \}.$$
\end{enumerate}
Define $\Lambda_{\text{fin}}(G,H) \subset \widehat{H}$ to be
the group of weights of finite-dimensional representations of $G$.
\begin{enumerate}[resume]
\item The dominant weights in $\Lambda_{\text{fin}}(G,H)$ are the
  highest weights of finite-dimensional representations of $G$.
\item A weight $\psi \in \widehat{H}$ belongs to
  $\Lambda_{\text{fin}}(G,H)$ if and only if we have both
\begin{enumerate}[label=\alph*)]
\item $d\psi(\beta^\vee) \in {\mathbb Z}$ for every coroot
  $\beta^\vee \in R^\vee({\mathfrak g},{\mathfrak h})$, and
\item $\psi(m_\alpha) = (-1)^{\psi(\alpha^\vee)}$ for every real
  coroot $\alpha^\vee$ (Definition \ref{def:rhoim}).
\end{enumerate} 
\end{enumerate}
\end{proposition}

\begin{proposition}\label{prop:findimlHswts}
Suppose $G$ is a real reductive group as in \eqref{se:reductive}. Fix
a maximally split real Cartan subgroup $H_s$ of $G$, preserved by the
Cartan involution $\theta$.  Write $H_s = T_s A_s$ for the Cartan
decomposition of $H_s$ (Proposition 4.3).  That $H_s$ is maximally
split is equivalent to either of the conditions
\begin{enumerate}[label=\alph*)]
\item ${\mathfrak a}_{s,0}$ is a maximal abelian subalgebra of the $-1$
  eigenspace ${\mathfrak s}_0$ of $\theta$; or
\item every imaginary root of ${\mathfrak h}_s$ in ${\mathfrak g}$ is
  compact (Definition \ref{def:rhoim}).
\end{enumerate}
Fix a set of positive roots $R_s^+$ for $R({\mathfrak g},{\mathfrak
  h}_s)$.
\begin{enumerate}
\item The group $H_s$ meets every component of $G$: {\em i.e.}, $G = G_0\cdot
  H_s$.
\item Suppose $F$ is a finite-dimensional irreducible representation
  of $G$, and $\phi_s = \phi_s(F) \in
  \widehat{H_s}$ is the $R_s^+$-highest weight of $F$.  Then
  $\phi_s(F)$ determines $F$ up to isomorphism.
\item Suppose $H$ is any other real Cartan subgroup of $G$, and $R^+$
  is any set of positive roots for $R({\mathfrak g}, {\mathfrak
    h})$. Write 
$$i(R^+_s,R^+)\colon {\mathfrak h}_s^* \rightarrow {\mathfrak h}^*$$
for the unique isomorphism implemented by an element of $G({\mathbb
  C})$ and carrying $R_s^+$ to $R^+$.  Then there is surjective group
homomorphism 
$$I(R^+_s,R^+)\colon \Lambda_{\text{fin}}(H_s) \rightarrow
\Lambda_{\text{fin}}(H)$$ 
characterized by
\begin{enumerate}[label=\alph*)]
\item for every finite-dimensional representation $F$ of $G$,
  $I(R_s^+,R^+)$ carries $\Delta(F,H_s)$ bijectively to $\Delta(F,H)$,
  and
\item the differential of $I(R_s^+,R^+)$ is $i(R_s^+,R^+)$.
\end{enumerate}
\item The kernel of the homomorphism $I(R_s^+,R^+)$ consists of the
  characters of $H_s$ trivial on $H_s\cap [G_0\cdot H]$. 
\end{enumerate}
\end{proposition}

For the proofs, we refer to \Cite{Vgreen}*{\S 0.4}. (The
description in Proposition \ref{prop:findimlwts}(4) of
$\Lambda_{\text{fin}}(G,H)$ is not proved in \Cite{Vgreen}, but we
will make no use of it.) 

\begin{definition}[\cite{Vgreen}*{Definition 7.2.5}]\label{def:cohfam}
  Suppose $H$ is a Cartan subgroup of $G$, 
  $\Lambda_{\text{fin}}(G,H) \subset \widehat{H}$ is the group of
  weights of finite-dimensional representations, and $\lambda \in
  {\mathfrak h}^*$ is any weight. Write 
$$\lambda + \Lambda_{\text{fin}}(G,H) = \{ \lambda +
\psi \mid \psi \in \Lambda_{\text{fin}}(G,H) \},$$
(a set of formal symbols), called {\em the translate of
  $\Lambda_{\text{fin}}(G,H)$ by $\lambda$}.  A {\em coherent
  family of virtual $({\mathfrak g},K)$-modules based on
  $\lambda + \Lambda_{\text{fin}}(G,H)$} is a map
$$\Phi\colon \lambda + \Lambda_{\text{fin}}(G,H) \rightarrow
{\mathcal G}({\mathfrak g},K)$$
(notation as in Proposition \ref{prop:grothgens}) satisfying
\begin{enumerate}[label=\alph*)]
\item $\Phi(\lambda + \psi)$ has infinitesimal character
  $\lambda + d\psi \in {\mathfrak h}^*$; and
\item for every finite-dimensional representation $F$ of $G$,
$$F\otimes \Phi(\lambda + \psi) = \sum_{\mu \in \Delta(F,H)}
\Phi(\lambda + (\psi + \mu)).$$
\end{enumerate}
In this last formula we regard $\Delta(F,H)$ as a multiset, in which
$\mu$ occurs with multiplicity equal to the dimension of the $\mu$
weight space of $F$.
\end{definition}

\begin{definition}\label{def:introots} Suppose ${\mathfrak h} \subset
  {\mathfrak g}$ is a 
  Cartan subalgebra of a complex reductive Lie algebra, 
\addtocounter{equation}{-1}
\begin{subequations}\label{se:introots}
and $\xi \in
  {\mathfrak h}^*$. The set of {\em integral roots for $\xi$} is
\begin{equation}
R(\xi) = \{\alpha \in R({\mathfrak g},{\mathfrak h}) \mid
\xi(\alpha^\vee) \in {\mathbb Z}\}.\end{equation}
Fix a positive root system $R^+(\xi) \subset R(\xi)$. We say that
$\xi$ is {\em integrally dominant} if
\begin{equation}\label{eq:intdom}
\xi(\alpha^\vee) \ge 0, \qquad (\alpha \in R^+(\xi)).
\end{equation}
If $R^+(\xi)$ is extended to a positive system $R^+$ for $R({\mathfrak
  g},{\mathfrak h})$, then integral dominance is equivalent to
\begin{equation*}
\xi(\beta^\vee) \text{\ is not a strictly negative integer}\quad
(\beta\in R^+).\end{equation*}

Occasionally we will need to make use of a stronger condition. We say
that $\xi$ is {\em real dominant} for $R^+$ if
\begin{equation}\label{eq:realdom}
\xi(\beta^\vee) \in {\mathbb R}^{\ge 0} \quad
(\beta\in R^+).\end{equation}

The {\em integral Weyl group for $\xi$} is
\begin{equation}
W(\xi) =_{\text{def}} W(R(\xi)) \subset W({\mathfrak g},{\mathfrak
  h}).
\end{equation}
An equivalent definition (by Chevalley's theorem for the affine Weyl
group) is
\begin{equation*}
W(\xi) = \{w \in W({\mathfrak g},{\mathfrak h}) \mid w\xi - \xi \in
{\mathbb Z}R({\mathfrak g},{\mathfrak h}) \}\end{equation*}
(see \Cite{Bour}*{Exercice VI.2.1, page 227} or \Cite{Vgreen}*{Lemma
  7.2.17}).  % give reference to Bourbaki

It is useful also to define the set of {\em singular roots for $\xi$}
\begin{equation}
R^\xi = \{\alpha \in R({\mathfrak g},{\mathfrak h}) \mid
\xi(\alpha^\vee) = 0\} \subset R(\xi).\end{equation}
Notice that the choice of a system of positive integral roots making
$\xi$ integrally dominant is precisely the same as the choice of an
(arbitrary) system of positive singular roots $R^{\xi,+}$.  The
correspondence is
\begin{equation*}\begin{aligned}
R^+(\xi) &= \{\alpha \in R(\xi) \mid \xi(\alpha^\vee) > 0\} \cup
R^{\xi,+},\\
R^{\xi,+} &= R^+(\xi) \cap R^\xi.\end{aligned}\end{equation*}

% In the same way, if $\xi$ is real

The {\em singular Weyl group for $\xi$} is
$$W^\xi =_{\text{def}} W(R^\xi) \subset W(R(\xi)) \subset W({\mathfrak
  g},{\mathfrak h}).$$
An equivalent definition (by Chevalley's theorem for $W$) is
$$W^\xi = \{w \in W({\mathfrak g},{\mathfrak h}) \mid w\xi = \xi\}.$$
\end{subequations} %se:introots
\end{definition}

\begin{subequations}\label{se:movelambda0}

In terms of the notion of integral roots, we can recast Proposition
\ref{prop:findimlHswts} in a way that is often useful. So fix two
Cartan subalgebras and weights
\begin{equation}
{\mathfrak h}_i \subset {\mathfrak g}, \qquad 
% \end{equation}
% and weights
% \begin{equation}
\lambda_i \in {\mathfrak h}_i^* \qquad (i= 1, 2).
\end{equation}
Assume that they define the same infinitesimal character
\begin{equation*}
\xi_{\lambda_1} = \xi_{\lambda_2}.
\end{equation*}
Equivalently, assume that there is an isomorphism ${\mathfrak h}_1^* \simeq
{\mathfrak h}_2^*$ implemented by an element of $G({\mathbb C})$ and
carrying $\lambda_1$ to $\lambda_2$:
\begin{equation*}
i\colon {\mathfrak h}_1^* \rightarrow {\mathfrak
  h_2}^*,\qquad i(\lambda_1) = \lambda_2.
\end{equation*}
This isomorphism is unique only up to (multiplication on the right by)
the stabilizer $W^{\lambda_1}$ of $\lambda_1$.  But if we fix also
systems of positive singular roots
\begin{equation*}
R^{\lambda_1,+}, \qquad R^{\lambda_2,+},
\end{equation*}
then we can specify a {\em unique} isomorphism by requiring
\begin{equation}
i(\lambda_1,R^{\lambda_1,+},\lambda_2, R^{\lambda_2,+}) \colon
{\mathfrak h}_1^* \rightarrow {\mathfrak h_2}^*,\quad \lambda_1
\mapsto \lambda_2,\ \  R^{\lambda_1,+} \mapsto R^{\lambda_2,+}.
\end{equation}
Equivalently, if we require that $\lambda_i$ is integrally dominant
for $R^+(\lambda_i)$, then we can specify a unique isomorphism
\begin{equation} %\begin{aligned}
i(\lambda_1,R^+(\lambda_1),\lambda_2, R^+(\lambda_2)) \colon
{\mathfrak h}_1^* \rightarrow {\mathfrak h_2}^*, \quad % \\ 
\lambda_1 \mapsto \lambda_2,\quad  R^+(\lambda_1) \mapsto
R^+(\lambda_2). %\end{aligned}
\end{equation}
These isomorphisms can take the place of $i(R^+_s,R^+)$ in Proposition
\ref{prop:findimlHswts}.

\end{subequations}

\begin{theorem}[\cite{Vgreen}*{Corollary 7.3.23}]\label{thm:cohfam}
  Suppose $\Gamma = (H,\gamma,R^+_{i{\mathbb R}})$ is a Langlands
  parameter for $G$, and $J(\Gamma)$ the corresponding irreducible
  representation (Theorem \ref{thm:LC}).  Write $d\gamma\in
  {\mathfrak h}^*$ for the differential of $\gamma$, which is the
  infinitesimal character of $J(\Gamma)$.  Write $R(d\gamma)$
  for the set of integral roots (Definition \ref{def:introots}). Then
  $R(d\gamma)$ is automatically preserved by the action of the Cartan
  involution $\theta$.

Fix a set $R^+(d\gamma)$ of positive integral roots, subject to  the
following requirements: 
\begin{enumerate}[label=\alph*)]
\item if $d\gamma(\alpha^\vee)$ is a positive integer, then
  $\alpha \in R^+(d\gamma)$;
\item if $d\gamma(\alpha^\vee) = 0$, but
  $d\gamma(\theta\alpha^\vee) > 0$, then $\alpha \in
  R^+(d\gamma)$; and
\item $R^+(d\gamma) \supset R^+_{i{\mathbb R}}$.
\end{enumerate}
Fix a set $R^+$ of positive roots for $H$ in $G$ containing
$R^+(d\gamma)$. Suppose finally that $H_s$ is a maximally split
Cartan, with any positive root system $R^+_s$, as in Proposition
\ref{prop:findimlHswts}; we will use the maps $i(R^+_s,R^+)$ and
$I(R^+_s,R^+)$ defined in that proposition. Set
$$d\gamma_s = i(R^+_s,R^+)^{-1}(d\gamma) \in {\mathfrak h}_s^*, \quad
R^+(d\gamma_s) =  i(R^+_s,R^+)^{-1}(R^+(d\gamma)). $$
Then there is a unique
  coherent family $\Phi$ of virtual $({\mathfrak g},K)$-modules based on
  $d\gamma_s + \Lambda_{\text{fin}}(G,H_s)$, with the following
  characteristic properties:
\begin{enumerate}[label=\alph*)]
\item $\Phi(d\gamma_s) = [J(\Gamma)]$; and
\item whenever $d\gamma_s + d\psi_s$ is integrally
  dominant (with respect to $R^+(d\gamma_s)$), $\Phi(d\gamma + \psi)$
  is (the class in the Grothendieck group of) an irreducible
  representation or zero. 
\end{enumerate}
 Write $\Pi(d \gamma)$
for the simple roots of $R^+(d\gamma)$, and define
$\tau(\Gamma)\subset\Pi(d\gamma)$ by 
\begin{enumerate}[label=\alph*)]
\item if $\alpha \in \Pi(d\gamma)$ is real, (Definition
  \ref{def:rhoim}), then $\alpha\in
  \tau(\Gamma)$ if and only if $\gamma_{\mathfrak q}(m_\alpha) =
  (-1)^{d\gamma(\alpha^\vee) + 1}$ (notation \eqref{eq:gammaq1});
\item if  $\beta \in \Pi(d\gamma)$ is imaginary, (Definition
  \ref{def:rhoim}), then $\beta\in \tau(\Gamma)$ if
and only if $\beta$ is compact; and
\item if $\delta \in \Pi(d\gamma)$ is complex, (Definition
  \ref{def:rhoim}), then $\delta\in\tau(\Gamma)$ if and only if
  $\theta\delta$ is a negative root.
\end{enumerate}
Write
$$\tau_s(\Gamma) =  i(R^+_s,R^+)^{-1}(\tau(\Gamma)) \subset \Pi(d\gamma_s)$$
for the corresponding integral simple roots for $H_s$. The coherent
family $\Phi$ has the following additional properties. 
\begin{enumerate}
\item If $d\gamma_s + d\psi_s$ is integrally dominant, then $\Phi(d\gamma_s
  + \psi_s) = 0$ if and only if there is a simple root $\alpha_s \in
  \tau_s(\Gamma)$ with $(d\gamma + d\psi)(\alpha_s^\vee) = 0$.
\item Suppose that $d\gamma_s + d\psi_s$ is integrally dominant, and does
  not vanish on any root in $\tau_s(\Gamma)$.  Write 
$$\psi = I(R^+_s,R^+)(\psi_s) \in \Lambda_{\text{fin}}(H).$$
 Then
$$\Gamma + \psi =_{\text{def}} (H,\gamma + \psi,R^+_{i{\mathbb R}})$$
is a Langlands parameter; and
$$\Phi(d\gamma_s + \psi_s) = [J(\Gamma + \psi)].$$
\item The set of all the irreducible constituents of the various virtual
  representations $\Phi(d\gamma_s + \psi_s)$ is equal to the set of
  irreducible constituents of all the tensor products $F\otimes
  J(\Gamma)$, with $F$ a finite-dimensional representation of $G$.
\end{enumerate}
\end{theorem}

\begin{definition}\label{def:tauinv}
In the setting of Theorem \ref{thm:cohfam}, the subset 
$$\tau_s(\Gamma) \subset \Pi(d\gamma_s)$$
of the simple integral roots is called the {\em $\tau$-invariant of
  the irreducible representation $J(\Gamma)$}, or the {\em
  $\tau$-invariant of $\Gamma$}.  (Because of the canonical
isomorphism $i(R^+_s,R^+)$, we may also refer to $\tau(\Gamma) \subset
\Pi(d\gamma)$ as the $\tau$-invariant.)
\end{definition}

\begin{subequations}\label{se:transfunc}
Return for a moment to the abstract language of
\eqref{se:charformulas}. So we fix a maximally split Cartan $H_s$ of 
$G$, and a weight $\lambda_0\in {\mathfrak h}_s^*$ corresponding to a
(possibly singular) infinitesimal character 
\begin{equation*}
\xi_{\lambda_0}\colon {\mathfrak Z}({\mathfrak g}) \rightarrow
{\mathbb C}.
\end{equation*}  
Write $R(\lambda_0)$ for the integral roots (Definition
\ref{def:introots}), and choose a set of positive roots
\begin{equation}
R^+(\lambda_0) \subset R(\lambda_0) \subset R({\mathfrak g},{\mathfrak h}_s)
\end{equation}
making $\lambda_0$ integrally dominant. When we study signatures of
invariant forms, we will assume that the infinitesimal character is
real (Definition \ref{def:realinf}). In that case we can extend
$R^+(\lambda_0)$ uniquely to a positive system
\begin{equation}\label{eq:lambda0realdom}
R^+_s \supset R^+(\lambda_0)
\end{equation}
making $\lambda_0$ real dominant \eqref{eq:realdom}.

It will be useful to have notation for the simple roots
\begin{equation*}
\Pi(\lambda_0) \subset R^+(\lambda_0),
\end{equation*}
and for the subset
\begin{equation*}
\{\alpha \in \Pi(\lambda_0) \mid \lambda_0(\alpha^\vee) = 0
\}=_{\text{def}} \Pi^{\lambda_0}
\end{equation*}
of simple roots orthogonal to $\lambda_0$.  Finally, we fix a weight
\begin{equation*}
\phi \in \Lambda_{\text{fin}}(G,H_s)
\end{equation*}
of a finite-dimensional representation of $G$, subject to the
additional condition
\begin{equation}
\lambda_0 + d\phi \text{\ is dominant and regular for $R^+(\lambda_0)$.}
\end{equation}
In connection with signatures of invariant forms, we will want also
to assume that 
\begin{equation}\label{eq:transrealdom}
\lambda_0 + d\phi \text{\ is dominant and regular for $R^+_s$.}
\end{equation}
One way to achieve this is to take $\phi$ to be the sum of the roots
in $R^+_s$. We need also
\begin{equation*}
F_\phi = \text{\ finite-dimensional irreducible of extremal weight $\phi$.}
\end{equation*}

In order to define the Jantzen-Zuckerman translation functors, we need
the functor (on ${\mathfrak g}$-modules)
\begin{equation*}
P_{\lambda_0}(M) =\  \text{\parbox{.5\textwidth}{largest submodule of
    $M$ on which $z - \xi_{\lambda_0}(z)$ acts nilpotently, all 
  $z \in {\mathfrak Z}({\mathfrak g})$}} 
\end{equation*}
If $M$ is a $({\mathfrak g},K)$ module, so is $P_{\lambda_0}(M)$.  On
the category of ${\mathfrak Z}({\mathfrak g})$-finite modules
(for example, on $({\mathfrak g},K)$-modules of finite length) the
functor $P_{\lambda_0}$ is exact.  The translation functors we want are
\begin{equation}
T_{\lambda_0}^{\lambda_0 + \phi}(M) =_{\text{def}} P_{\lambda_0+
  d\phi}\left(F_{\phi} \otimes P_{\lambda_0}(M)\right),
\end{equation}
(``translation away from the wall'') and
\begin{equation}
T_{\lambda_0+\phi}^{\lambda_0}(N) =_{\text{def}} P_{\lambda_0}\left(F^*_\phi
  \otimes P_{\lambda_0+d\phi}(N)\right)
\end{equation}
(``translation to the wall.'')  These are clearly exact functors on
the category of finite-length $({\mathfrak g},K)$-modules, so they
also define group endomorphisms of the Grothendieck group ${\mathcal
  G}({\mathfrak g},K)$. 
\end{subequations}

\begin{corollary}[Jantzen \Cite{Jantzen}, Zuckerman
  \Cite{Zuck}; see \Cite{Vgreen}*{Proposition
    7.2.22}]\label{cor:transfunc} 
Suppose we are in the setting of \eqref{se:transfunc}.
\begin{enumerate}
\item The functors $T_{\lambda_0+\phi}^{\lambda_0}$ and
  $T_{\lambda_0}^{\lambda_0+\phi}$ are left and right adjoint to each
  other:
$$\Hom_{{\mathfrak g},K}(N, T_{\lambda_0}^{\lambda_0+\phi}(M))
\simeq \Hom_{{\mathfrak g},K}(T_{\lambda_0+\phi}^{\lambda_0}(N),M),$$
and similarly with the two functors interchanged.
\item For a coherent family $\Phi$ of virtual $({\mathfrak
    g},K)$-modules on $\lambda_0+\Lambda_{\text{fin}}(G,H_s)$,
$$T_{\lambda_0+\phi}^{\lambda_0}(\Phi(\lambda_0 + \phi)) = \Phi(\lambda_0).$$
\item The translation functor $T_{\lambda_0+\phi}^{\lambda_0}$
  (translation to the wall) carries
  every irreducible 
$({\mathfrak g},K)$-module of infinitesimal character $\lambda_0 +
d\phi$ to an irreducible of infinitesimal character $\lambda_0$, or to
zero. The irreducible representations mapped to zero are exactly those
containing some root of $\Pi^{\lambda_0}$ in the $\tau$-invariant
(Definition \ref{def:tauinv}).
\item The (inverse of the) translation functor of (3) defines an
  injective correspondence of Langlands parameters (and therefore of
  irreducible representations)  
$$t_{\lambda_0}^{\lambda_0+\phi}\colon B({\lambda_0}) \hookrightarrow
B({\lambda_0 + d\phi}).$$ 
The image consists of all irreducible representations (of
infinitesimal character $\lambda_0 + d\phi$) having no root
of $\Pi^{\lambda_0}$ in the $\tau$-invariant.
\item Suppose $\Gamma_0 = (H,\gamma_0,R^+_{i{\mathbb R}}) \in
  B(\lambda_0)$ is a Langlands parameter (at the possibly singular
  infinitesimal character $\lambda_0$). Choose positive integral
  roots $R^+(d\gamma_0)$ as in Theorem \ref{thm:cohfam}, so that we
  have an isomorphism
$$i(\lambda_0,R^+(\lambda_0),d\gamma_0,R^+(d\gamma_0))\colon {\mathfrak
  h}_s^* \rightarrow {\mathfrak h}^*$$
as in \eqref{se:movelambda0}.  Define
$$\phi(\Gamma_0) =
I(\lambda_0,R^+(\lambda_0),d\gamma_0,R^+(d\gamma_0))(\phi),$$
an $H$-weight of a finite-dimensional representation of $G$
(Proposition \ref{prop:findimlHswts}). Then the correspondence of
Langlands parameters in (4) is
$$t_{\lambda_0}^{\lambda_0 + \phi}(\Gamma_0) = \Gamma_0 + \phi(\Gamma_0)$$
(notation as in Theorem \ref{thm:cohfam}(2)).
\item Suppose $\Gamma_1 = (H,\gamma_1,R^+_{i{\mathbb R}}) \in
  B(\lambda_0+d\phi)$ is a Langlands parameter (at the regular
  infinitesimal character $\lambda_0 + d\phi$). Write $R^+(d\gamma_1)$
  for the unique set of positive integral roots making $d\gamma_1$
  integrally dominant, so that we have an isomorphism
$$i(\lambda_0+d\phi,R^+(\lambda_0),d\gamma_1,R^+(d\gamma_1))\colon {\mathfrak
  h}_s^* \rightarrow {\mathfrak h}^*$$
as in \eqref{se:movelambda0}.  Define
$$\phi(\Gamma_1) = I(\lambda_0,R^+(\lambda_0),d\gamma_0,R^+(d\gamma_0))(\phi),$$
an $H$-weight of a finite-dimensional representation %of $G$
(Proposition \ref{prop:findimlHswts}). Then 
$$t_{\lambda_0+\phi}^{\lambda_0}(\Gamma_1) =_{\text{def}} (H,\gamma_1
- \phi(\Gamma_1), R^+_{i{\mathbb R}}) =_{\text{def}}\Gamma_1 - \phi(\Gamma_1)$$
is a weak Langlands parameter (Definition \ref{def:genparam}).  This
correspondence is a left inverse to the injection of (4):
$$t_{\lambda_0+\phi}^{\lambda_0}\circ t_{\lambda_0}^{\lambda_0 +
  \phi}(\Gamma_0) = \Gamma_0.$$
\item Translation to the wall carries
  standard representations to weak standard representations:
$$T_{\lambda_0 + \phi}^{\lambda_0}(I_{\quo}(\Gamma_1)) =
I_{\quo}(\Gamma_1 - \phi(\Gamma_1)).$$
In particular,
$$T_{\lambda_0 + \phi}^{\lambda_0}(I_{\quo}(\Gamma_0+\phi(\Gamma_0)) =
I_{\quo}(\Gamma_0).$$
\item Translation to the wall carries irreducibles to
  irreducibles or to zero:
$$T_{\lambda_0 + \phi}^{\lambda_0}(J(\Gamma_1)) = \begin{cases}
J(\Gamma_1 - \phi(\Gamma_1)) & (\tau(J(\Gamma_1)) \cap \Pi^{\lambda_0}
= \emptyset),\\
0 & (\tau(J(\Gamma_1)) \cap \Pi^{\lambda_0}
\ne \emptyset). \end{cases} $$
In particular,
$$T_{\lambda_0 + \phi}^{\lambda_0}(J(\Gamma_0+\phi(\Gamma_0)) =
J(\Gamma_0).$$
\end{enumerate}
\end{corollary}

% \begin{subequations}\label{se:reducetoreg}
Corollary \ref{cor:transfunc} allows us to ``reduce'' the computation
of the multiplicity matrix $m_{\Xi,\Gamma}$ defined in
\eqref{se:charformulas} to the case of regular infinitesimal
character.  Given a possibly singular infinitesimal character
$\lambda_0$, we choose a weight of a finite-dimensional representation
$\phi$ as in \eqref{se:transfunc}, so that $\lambda_0+d\phi$ is
regular. The (possibly singular) Langlands parameters $\Xi$ and
$\Gamma$ in $B(\lambda_0)$ correspond to regular Langlands parameters
$\Xi+\phi(\Xi)$ and $\Gamma+\phi(\Gamma)$ in $B(\lambda_0 + d\phi)$,
and
\begin{equation}\label{se:reducetoreg}
m_{\Xi,\Gamma} = m_{\Xi+\phi(\Xi),\Gamma+ \phi(\Gamma)}.
\end{equation}
A little more precisely, we can start with the decomposition into irreducibles
\begin{equation*}
\left[I(\Gamma + \phi(\Gamma))\right] = \sum_{\Xi_1 \in B(\lambda_0+d\phi)}
m_{\Xi_1,\Gamma+\phi(\Gamma)}\left[J(\Xi_1)\right]
\end{equation*}
and apply the exact translation functor $T_{\lambda_0 +
  \phi}^{\lambda_0}$. On the left we get $\left[I(\Gamma)\right]$, by
Corollary \ref{cor:transfunc}(7). On the right we get
\begin{equation*}
\sum_{\substack{\Xi_1 \in B(\lambda_0+d\phi)\\ \tau(\Xi_1) \cap
    \Pi^{\lambda_0} = \emptyset}} m_{\Xi_1,\Gamma+\phi(\Gamma)}\left[J(\Xi_1
- \phi(\Xi_1))\right].
\end{equation*}
The irreducibles appearing on the right are all distinct (this is part
of the uniqueness assertion for the coherent family in Theorem
\ref{thm:cohfam}), so the coefficients are the multiplicities of
irreducibles in $I(\Gamma)$.

To say this another way: the index set of Langlands parameters at
the possibly singular infinitesimal character $\lambda_0$ is naturally
a {\em subset} of the index set at the regular infinitesimal character
$\lambda_0 + d\phi$. The multiplicity matrix at $\lambda_0$ is the
corresponding submatrix of the one at $\lambda_0+d\phi$.

We call this fact a ``reduction'' in quotes, because the multiplicity
problem at regular infinitesimal character is much more complicated
than at singular infinitesimal character: we have ``reduced'' a less
complicated problem to a more complicated one.  But the gain is
real. At regular infinitesimal character we have the Kazhdan-Lusztig
algorithm (see Section \ref{sec:klv}) to solve the problem, and this
algorithm cannot (as far as we now understand) be applied directly to
the case of singular infinitesimal character.
% \end{subequations}

Finally, let us consider the problem of relating the character matrix
$M_{\Xi,\Gamma}$ to regular infinitesimal character.  Since this
problem concerns the inverse of the multiplicity matrix, we might
suspect that there are difficulties: there is not usually a nice
way to find the inverse of a submatrix from the inverse of the larger
matrix. Since our matrices are upper triangular, one could hope that
the submatrix we are considering is a corner; but that turns out not
to be the case.  In order to see what happens, we begin with a careful
statement of what happens to standard modules under translation to
singular infinitesimal character.

\begin{theorem}\label{thm:transstd} 
Suppose we are in the setting of \eqref{se:transfunc}. Write
$$t_{\lambda_0+\phi}^{\lambda_0}\colon B(\lambda_0+d\phi) \rightarrow
B_{\weak}(\lambda_0)$$ 
(notation \eqref{se:charformulas}) as in Corollary
\ref{cor:transfunc}(7). Fix a Langlands parameter $\Xi_1 \in
B(\lambda_0 + d\phi)$. Write
$$\Xi_0 = t_{\lambda_0+\phi}^{\lambda_0}(\Xi_1) \in B_{\weak}(\lambda_0), $$
so that
$$T_{\lambda_0+\phi}^{\lambda_0}(I_{\quo}(\Xi_1)) =
I_{\quo}(\Xi_0).$$ 
\begin{enumerate}
\item There is a unique subset $C_{\text{sing}}(\Xi_0) \subset B(\lambda_0)$
  with the property that
$$I_{\quo}(\Xi_0) = \sum_{\Xi' \in C_{\text{sing}}(\Xi_0)} I_{\quo}(\Xi').$$
\item The set $C_{\text{sing}}(\Xi_0)$ has cardinality either % equal to
  zero (if some simple compact imaginary root for $\Xi_0$ vanishes on
  $\lambda_0$, so that $I(\Xi_0) = 0$) or a power of two. 
\end{enumerate}
For each Langlands parameter $\Xi\in B(\lambda_0)$, define
$$C_{\text{reg}}(\Xi) = \{\Xi_1\in B(\lambda_0 + d\phi) \mid \Xi \in
C_{\text{sing}}(\Xi_0)\} \subset B(\lambda_0 + d\phi);$$
here $\Xi_0 = t_{\lambda_0+d\phi}^{\lambda_0}(\Xi_1)$ as above.
\begin{enumerate}[resume]
\item Suppose the Langlands parameter $\Gamma\in B(\lambda_0)$
  corresponds to 
$$\Gamma+\phi(\Gamma) \in B(\lambda_0+d\phi).$$
Then the character formula \eqref{eq:charformirr} is
$$J(\Gamma) = \sum_{\Xi \in B(\lambda_0)}\ \ \sum_{\Xi_1\in C_{\text{reg}}(\Xi)}
M_{\Xi_1,\Gamma+\phi(\Gamma)}[I(\Xi)].$$ 
Equivalently,
$$M_{\Xi,\Gamma} = \sum_{\Xi_1\in C_{\text{reg}}(\Xi)}
M_{\Xi_1,\Gamma+\phi(\Gamma)}.$$
% NOT TRUE! the number of terms in the sum is zero or a power of two.
\end{enumerate}
\end{theorem}
\begin{proof} %[Sketch of proof]
% \begin{subequations}\label{se:reducechartoreg}
That $\Xi_0$ must be a weak Langlands
  parameter is Corollary \ref{cor:transfunc}(6), and the statement
  about translation of standard modules is Corollary
  \ref{cor:transfunc}(7). As is explained in \eqref{se:ds}, the weak
  standard module $I(\Xi_0)$ of infinitesimal character $\lambda_0$
is either equal to zero (in case the parameter vanishes on some simple
compact imaginary root) or can be written (by a simple algorithm) as a
direct sum
\begin{equation*}
I(\Xi_0) = \sum_{\Xi' \in C_{\text{sing}}(\Xi_0) \subset B(\lambda_0)} I(\Xi').
\end{equation*}
The letter $C$ stands for ``Cayley.'' The fundamental example is
writing a nonspherical principal series for 
$SL(2,{\mathbb R})$ (with continuous parameter zero) as a direct sum
of two limits of discrete series; 
the general case is just a repeated application of this one.) The
construction of this decomposition takes place inside the subgroup of
$G$ generated by the Cartan subgroup $H_1$ for the parameter $\Xi_0$,
and the real roots in $\Pi^{\xi_0}$. It turns out that the parameters
appearing in $C_{\text{sing}}(\Xi_0)$ are all attached to a single 
Cartan subgroup, and that the number of them is a power of $2$. Since
we will not use these last two facts, we omit the proof.  This is (2).

For the character formula in (3), we begin with the character formula
at infinitesimal character $\lambda_0 + d\phi$:
\begin{equation*}
\left[J(\Gamma + \phi(\Gamma))\right] = \sum_{\Xi_1 \in B(\lambda_0+d\phi)}
M_{\Xi_1,\Gamma+\phi(\Gamma)}\left[I(\Xi_1)\right]
\end{equation*}
Now apply the exact translation functor $T_{\lambda_0 +
  \phi}^{\lambda_0}$. On the left we get $\left[J(\Gamma)\right]$, by
Corollary \ref{cor:transfunc}(8). On the right we get a sum of weak
standard modules
\begin{equation*}
\sum_{\Xi_1 \in B(\lambda_0+d\phi)}
  M_{\Xi_1,\Gamma+\phi(\Gamma)}\left[I(\Xi_1 - \phi(\Xi_1))\right] =
  \sum_{\Xi_1 \in B(\lambda_0+d\phi)} 
  M_{\Xi_1,\Gamma+\phi(\Gamma)}\left[I(\Xi_0)\right].
\end{equation*}
Using (2), we may therefore rewrite our character formula as 
\begin{equation*}\begin{aligned}
\left[J(\Gamma)\right] &= \sum_{\Xi \in B(\lambda_0)}\ \  \sum_{\substack{\Xi_1 \in
  B(\lambda_0 + d\phi)\\ \Xi \in C_{\text{sing}}(\Xi_0)}}
M_{\Xi_1,\Gamma+ \phi(\Gamma)}\left[I(\Xi)\right]\\
&= \sum_{\Xi \in B(\lambda_0)} \ \ \sum_{\Xi_1 \in  C_{\text{reg}}(\Xi)}
M_{\Xi_1,\Gamma+ \phi(\Gamma)}\left[I(\Xi)\right].\end{aligned}
\end{equation*}
Equivalently,
\begin{equation*}
M_{\Xi,\Gamma} =  \sum_{\Xi_1 \in C_{\text{reg}}(\Xi)} M_{\Xi_1,\Gamma+ \phi(\Gamma)}.
\end{equation*}
% \end{subequations}
\end{proof}

It is natural to ask about a corresponding use of the translation
functors to relate the signature formulas \eqref{se:cinvtsigformulas} at two
different infinitesimal characters. More or less this is exactly
parallel to the discussion above, but there is a subtlety about the
choice of invariant form: the translation functor will carry one
nondegenerate invariant form to another, but the new form may not
share the characteristic positivity property of the old one.  We
therefore postpone a careful analysis to Section
\ref{sec:klvform}.

\section{Translation functors for extended groups}\label{sec:translext}
\setcounter{equation}{0}

In this section we fix an extended group
\begin{subequations}\label{se:extfindiml}
\begin{equation}
{}^\delta G = G\rtimes \{1,\delta_f\}
\end{equation}
as in Definition \ref{def:extgrp}.  Suppose 
\begin{equation}
{}^1 H = \langle H,\delta_1 \rangle, \qquad t_1(R^+_1) = R^+_1
\end{equation}
is an extended maximal torus preserving a positive root system
$R^+_1$, with $\delta_1 \in K$ % one of the distinguished
a representative % described in Lemma \ref{lemma:realexttorus} 
for the
element $t_1$ of the extended Weyl group.  According to Clifford
theory (Proposition \ref{prop:extpairrep}),
\begin{equation}\label{eq:exttorusreps}\begin{aligned}
&\bullet\parbox[t]{4in}{each  $\gamma = \gamma^{t_1}\in \widehat{H}$
  has two one-dimensional extensions
  $\gamma_{\pm}$ in $\widehat{{}^1 H}$, on which the scalar actions of
  $\delta_1$ differ by sign;}\\
&\bullet\parbox[t]{4in}{each $\gamma\ne \gamma^{t_1} \in \widehat{H}$
  induces to  % a two-dimensional irreducible 
$\gamma_{\text{ind}} = \gamma^{t_1}_{\text{ind}} \in \widehat{{}^1 H}$; and}\\ 
&\bullet\text{these types exhaust $\widehat{{}^1 H}$.} 
\end{aligned}\end{equation}
Of course it also makes sense to induce a $t_1$-fixed character
$\gamma$; in that case we get
\begin{equation*}
\gamma_{\text{ind}} = \gamma_+ + \gamma_- \qquad (\gamma =
\gamma^{t_1}).
\end{equation*}

%DV
Despite the notation, there is no way to prefer one of the two
extensions $\gamma_+$ and $\gamma_-$ without making a further
choice: that of a square root of the scalar by which $\delta_1^2 \in
H$ acts in $\gamma$.

There is a parallel description of the finite-dimensional
representations of ${}^\delta G$:
\begin{equation}\label{eq:extGfindiml}\begin{aligned}
&\bullet\parbox[t]{4in}{each $t_1$-fixed (equivalently,
  $\theta$-fixed) 
  finite-dimensional irreducible representation $F$ of $G$ has two
  extensions $F_{\pm}$ to ${}^\delta G$, distinguished by the action
  of $\delta_1$ on the $R^+_1$-highest weight space;}\\
&\bullet\parbox[t]{4in}{each member of a two-element orbit
      $\{F,F^{t_1}\}$ of irreducible  
  finite-dimensional representations of $G$ induces to an irreducible
  representation $F_{\text{ind}} = F^{t_1}_{\text{ind}}$ of ${}^\delta G$; and}\\ 
&\bullet\parbox[t]{4in}{these two types exhaust the irreducible
  finite-dimensional representations of ${}^\delta G$.} 
\end{aligned}\end{equation}
Inducing a $\theta$-fixed finite-dimensional irreducible to ${}^\delta
G$ gives
\begin{equation*}
F_{\text{ind}} = F_+ + F_- \qquad (F =
F^{t_1}).
\end{equation*}

%DV
Again the notation obscures that fact that there is no way to prefer
$F_+$ over $F_-$ without choosing a square root of the action of
$\delta_1^2$ on the highest weight space.

\begin{danger}
There is a possibly unexpected subtlety about highest weight
theory for ${}^\delta G$: the highest weight of an irreducible
finite-dimensional representation {\em need not} be an irreducible
  representation of ${}^1 H$, and {\em need not} determine the
  representation of ${}^\delta G$. For us, the most important aspect
  of this problem is that the highest weight of $F$ may be $t_1$-fixed
  even if $F$ is not $\theta$-fixed. Examples appear in Example
  \ref{ex:gl2ext} below. \end{danger}
 
As in Proposition \ref{prop:findimlwts}, for each
finite-dimensional representation $F_1$ of ${}^\delta G$, define
\begin{equation*}
\Delta(F_1,{}^1 H) \subset \widehat{{}^1 H}
\end{equation*}
to be the collection of irreducible representations of ${}^1 H$
appearing in the restriction of $F_1$.  Sometimes it will be
convenient to regard $\Delta(F_1,{}^1 H)$ as a multiset.  We define
\begin{equation*}
\Lambda_{\text{fin}}({}^\delta G,{}^1 H) \subset \widehat{{}^1 H}
\end{equation*}
to be the set of irreducible representations of ${}^1 H$ appearing in
the restrictions of finite-dimensional representations of ${}^\delta
G$. Using finite-dimensional representations induced from $G$ to
${}^\delta G$, it is very easy to see that
\begin{equation*}
\Lambda_{\text{fin}}({}^\delta G,{}^1 H) = \{\phi_1 \in \widehat{{}^1
  H} \mid \phi_1|_H \subset \Lambda_{\text{fin}}(G,H)\};
\end{equation*}
here we write $\subset$ (instead of $\in$) to indicate that the
restriction is a sum of characters which are weights of
finite-dimensional representations of $G$.  We write
\begin{equation}\label{eq:extfingroth}
{\mathcal G}_{\text{fin}}({}^1 H) = \text{subgroup of ${\mathcal
    G}({}^1 H)$ generated by $\Lambda_{\text{fin}}({}^\delta G,{}^1 H)$.}
\end{equation}

Finally, for any finite-dimensional representation of $F_1$ of
${}^\delta G$, define
\begin{equation*} \phi_1(F_1) = \text{representation of ${}^1 H$ on
    $R^+_1$-highest weight vectors.}
\end{equation*}
Even if $F_1$ is irreducible, it turns out that $\phi_1(F_1)$ need not
be irreducible; this is related again to the phenomenon in Example
\ref{ex:gl2ext}.
\end{subequations} %{se:extfindiml}  

\begin{proposition}\label{prop:findimlextHswts} In the setting of
  Proposition \ref{prop:findimlHswts}, suppose that $R^+_s$ is chosen
  compatible with an extended maximal torus ${}^1 H_s$, and that $t_s$ is
  the corresponding automorphism of $H_s$.
\begin{enumerate}
\item The surjection 
$$I(R^+_s,R^+_1)\colon \Lambda_{\text{fin}}(H_s) \twoheadrightarrow
\Lambda_{\text{fin}}(H)$$
carries the action of $t_s$ to that of $t_1$.
% \item The surjection restricts to a canonical isomorphism
% $$I(R^+_s,R^+_1)\colon X^*(H_s) \buildrel\sim \over\longrightarrow
% X^*(H).$$
\item The surjection induces a homomorphism of Grothen\-dieck groups
$${}^\delta I(R^+_s,R^+_1)\colon {\mathcal G}_{\text{fin}}({}^1 H_s) \rightarrow 
{\mathcal G}_{\text{fin}}({}^1 H)$$ 
preserving dimensions and satisfying
\begin{enumerate}
\item For every finite-dimensional representation $F_1$ of ${}^\delta
  G$, ${}^\delta I(R^+_s,R^+_1)$ carries the restriction of $F_1$ to
  ${}^1 H_s$ to the restriction of $F_1$ to ${}^1 H$.
\item For every weight $\lambda_s \in
  \Lambda_{\text{fin}}({}^\delta G,{}^1 H_s)$, write $\lambda =
  I(R^+_s,R^+_1)(\mu_s)$. Then
$${}^\delta I(R^+_s,R^+_1)(\lambda_{s,\text{ind}}) = \lambda_{\text{ind}}.$$
\item If in addition $\lambda_s \in
  \Lambda^{t_s}_{\text{fin}}({}^\delta G,{}^1 H_s)$ is fixed by $t_s$,
  then automatically $\lambda$ is fixed by $t_1$; and
$${}^\delta I(R^+_s,R^+_1)(\lambda_{s,\pm}) = \lambda_\pm.$$
\end{enumerate}
\item A $t_s$-fixed dominant weight $\lambda_s \in
  \Lambda^{t_s}_{\text{fin}}(G,H_s)$ is automatically the highest
  weight of unique a finite-dimensional irreducible
  representation $F(\lambda_s) \simeq F(\lambda_s)^\theta$.
\item A $t_1$-fixed dominant weight $\lambda\in
  \Lambda^{t_1}_{\text{fin}}(G,H)$ is the highest weight of a
  $\theta$-fixed finite-dimensional representation of $G$ if and only if 
$$\lambda = I(R^+_s,R^+_1)(\lambda_s) \qquad (\lambda_s \in
\Lambda^{t_s}_{\text{fin}}(G,H_s)).$$ 
% \item A $t_1$-fixed algebraic weight $\lambda \in X^*(H)$ is the
%   highest weight of a unique $\theta$-fixed algebraic representation
%   $F(\lambda)$ for $G$.
\end{enumerate}
\end{proposition}

It is important to note in (4) that there can exist $t_1$-fixed weights of
finite-dimensional representations that are {\em not} highest weights
of $\theta$-fixed finite-dimensional representations of $G$. 

\begin{example}\label{ex:gl2ext} Suppose $G=GL(2,{\mathbb R})$. We
  take for $H$ the Cartan subgroup
$$H\simeq {\mathbb C}^\times = \{re^{i\phi} \mid r >0, \phi \in
{\mathbb R}\},$$
coming from the multiplicative action on ${\mathbb C} \simeq {\mathbb R}^2$.
The Cartan involution (inverse transpose) acts by $\theta(z) =
\overline z^{-1}$, so this is also the action of $t_1$:
$$t_1(re^{i\phi}) = r^{-1}e^{i\phi}.$$
The characters of $H$ appearing in finite-dimensional representations
are
$$\lambda_{\nu,m}(re^{i\phi}) = r^{2\nu} e^{im\phi} \qquad (m\in {\mathbb Z},
\nu\in {\mathbb C}).$$
That is,
$$\Lambda_{\text{fin}}(G,H) \simeq {\mathbb C}\times {\mathbb Z},
\qquad t_1(\nu,m) = (-\nu,m).$$
The $t_1$-fixed weights are 
$$\Lambda^{t_1}_{\text{fin}}(G,H) = \{(0,m)\mid m\in {\mathbb Z}\}.$$

For a maximally split Cartan we choose the diagonal subgroup
$$H_s \simeq {\mathbb R}^\times \times {\mathbb R}^\times =
\{(d_1,d_2) \mid d_i \in {\mathbb R}^\times.$$
The action of $t_s$ is
$$t_s(d_1,d_2) = (d_2^{-1},d_1^{-1}).$$
The characters appearing in finite-dimensional representations are
$$% \begin{aligned}
\lambda_{\nu,m,\delta}(d_1,d_2) = |d_1|^{\nu +m/2}|d_2|^{\nu-m/2}
(\sgn d_1)^\delta (\sgn d_2)^{m+\delta} \quad % \\
 (\nu \in {\mathbb C}, m\in {\mathbb Z}, \delta \in {\mathbb
   Z}/2{\mathbb Z}).$$ %\end{aligned}$$ 
Therefore
$$\Lambda_{\text{fin}}(G,H_s) \simeq {\mathbb C}\times {\mathbb
  Z}\times {\mathbb Z}/2{\mathbb Z},
\qquad t_s(\nu,m,\delta) = (-\nu,m,m+\delta).$$
The $t_s$-fixed weights are
$$\Lambda^{t_s}_{\text{fin}}(G,H_s) = \{(0,m,\delta)\mid m\in
2{\mathbb Z}\}.$$
The map of the proposition is
$$I(R^+_s,R^+_1)(\nu,m,\delta) = (\nu,m).$$
This is indeed surjective (it is exactly two to one), and it
intertwines the actions of $t_s$ and $t_1$; but it does {\em not}
carry $t_s$-fixed weights onto $t_1$-fixed weights.

If for example we take $m=1$, we see that the character $re^{i\phi}
\mapsto e^{i\phi}$ of $H$, which is fixed by $t_1$, is not the highest
weight of any $\theta$-fixed finite-dimensional representation of
$G$. It is actually the highest weight of two different
representations of $G$, namely
$${\mathbb C}^2\otimes |\det|^{-1/2}, \qquad {\mathbb C}^2 \otimes
|\det|^{-1/2} \otimes (\sgn(\det)).$$
These two representations are interchanged by $\theta$ (since they are
Hermitian dual to each other).
\end{example}

In order to avoid difficulties like those appearing in this example,
we will use for extended groups only coherent families based on a
maximally split Cartan. The definition makes sense more generally,
however.

\begin{definition}[\cite{Vgreen}*{Definition
  7.2.5}]\label{def:extcohfam}
% \addtocounter{equation}{-1}
% \begin{subequations}\label{se:extcohfam}
Suppose we are in the setting of \eqref{se:extfindiml}, and $\lambda \in
  {\mathfrak h}^{*,t_1}$ is any $t_1$-fixed weight. Write 
\begin{equation*}\lambda + \Lambda_{\text{fin}}({}^\delta G,{}^1 H) = \{ \lambda +
\psi \mid \psi \in \Lambda_{\text{fin}}({}^\delta G,{}^1 H) \},\end{equation*}
(a set of formal symbols), {\em the translate of
  $\Lambda_{\text{fin}}({}^\delta G,{}^1 H)$ by $\lambda$}.  Similarly, write
\begin{equation*}\lambda + {\mathcal G}_{\text{fin}}({}^1 H) = \{ \lambda +
\psi \mid \psi \in {\mathcal G}_{\text{fin}}({}^1 H) \},\end{equation*}
(notation as in \eqref{eq:extfingroth}---another set of formal
symbols, but with the abelian group structure from the
Grothendieck group) called {\em the translate of 
  ${\mathcal G}_{\text{fin}}({}^1 H)$ by $\lambda$}.  A {\em coherent
  family of virtual $({\mathfrak g},{}^\delta K)$-modules based on
  $\lambda + \Lambda_{\text{fin}}({}^\delta G,{}^1 H)$} is a map
\begin{equation*}\label{eq:classicalextcoh} \Phi_1\colon \lambda +
 \Lambda_{\text{fin}}({}^\delta G,{}^1 H) \rightarrow 
{\mathcal G}({\mathfrak g},{}^\delta K)\end{equation*}
(notation as in Proposition \ref{prop:grothgens}) satisfying two
conditions we will write below. Since the irreducible modules are a
${\mathbb Z}$-basis of the Grothendieck group, it is equivalent to
give a ${\mathbb Z}$-linear map
\begin{equation}\label{eq:grothextcoh} \Phi_1\colon \lambda +
  {\mathcal G}_{\text{fin}}({}^1 H) \rightarrow 
{\mathcal G}({\mathfrak g},{}^\delta K).\end{equation}
% \end{subequations}
Here are the conditions.
\begin{enumerate}[label=\alph*)]
\item $\Phi_1(\lambda + \psi)$ has infinitesimal character
  $\lambda + d\psi \in {\mathfrak h}^*$; and
\item for every finite-dimensional algebraic representation $F_1$ of
  ${}^\delta G$,
$$F_1\otimes \Phi_1(\lambda + \psi_1) = % \sum_{\mu_1 \in \Delta(F_1,{}^1 H)}
\Phi_1(\lambda + [F\otimes \psi_1]).$$
In (a), if $\psi$ is two-dimensional, then the statement means that
each irreducible constituent of the restriction to $G$ of this virtual
representation has 
infinitesimal character given by one of the two summands of the Lie
algebra representation $\lambda + d\psi$. The brackets in (b) denote
the class in ${\mathcal G}_{\text{fin}}({}^1 H)$, and we are using
there the second formulation \eqref{eq:grothextcoh} of the notion of
coherent family.  
\end{enumerate}
The reader should verify that this formulation in the case of $G$
is precisely equivalent to the classical formulation in
Definition \ref{def:cohfam}. The language here avoids a tiresome
enumeration of the various possible decompositions of tensor products
of irreducible representations of ${}^1 H$.
\end{definition}

There is an easy construction of some coherent families for ${}^\delta G$.
\begin{example}\label{ex:indcoh} In the setting of Definition
  \ref{def:extcohfam}, suppose $\Phi$ is a coherent family based on
  $\lambda+\Lambda_{\text{fin}}(G,H)$. Extend $\Phi$ to a ${\mathbb
    Z}$-linear map
$$\Phi\colon \lambda + {\mathcal G}_{\text{fin}}(H) \rightarrow
{\mathcal G}({\mathfrak g},K)$$
as in \eqref{eq:grothextcoh}.  Define 
$$\Phi_{\text{ind}}(\lambda + \psi_1) = \Ind_{({\mathfrak g},K)}^{({\mathfrak
    g},{}^\delta K)}(\Phi(\lambda + \psi_1|_H)).$$
That is, we restrict $\psi_1$ to $H$, getting one or two characters of
$H$; we add the values of the coherent family for $G$ at those
characters; then we induce from $G$ to ${}^\delta G$.

That this construction gives a coherent family for ${}^\delta G$ is
elementary. The reason it is not so interesting is this. We saw in
Theorem \ref{thm:cohfam} that interesting information is encoded by a
coherent family many of whose values are irreducible.  The values of
the coherent family $\Phi_{\text{ind}}$ are all induced; so the values
can never include 
the type one irreducible modules for ${}^\delta G$ (Definition
\ref{def:typeext}), arising by extension of $\theta$-fixed irreducible
modules for $G$.   
\end{example}

Here is a description of more interesting coherent families.

\begin{theorem}\label{thm:extcohfam}
  Suppose $\Gamma_1 = ({}^1H,\gamma_1,R^+_{i{\mathbb R}})$ is a Langlands
  parameter for an irreducible representation $J(\Gamma_1)$ of
  ${}^\delta G$, with $t_1$-fixed differential 
$$d\gamma_1 = (d\gamma_1)^{t_1} \in {\mathfrak h}^*$$
representing the infinitesimal character of $J(\Gamma_1)$.
Just as in Theorem \ref{thm:cohfam}, we write $R(d\gamma_1)$
  for the set of integral roots (Definition \ref{def:introots}), and
  fix a set $R^+(d\gamma_1)$ of positive integral roots, subject to  the
  requirements:  
\begin{enumerate}[label=\alph*)]
\item if $d\gamma_1(\alpha^\vee)$ is a positive integer, then
  $\alpha \in R^+(d\gamma_1)$;
\item if $d\gamma_1(\alpha^\vee) = 0$, but
  $d\gamma_1(\theta\alpha^\vee) > 0$, then $\alpha \in
  R^+(d\gamma_1)$;
\item $R^+(d\gamma_1) \supset R^+_{i{\mathbb R}}$; and
\item $t_1(R^+(d\gamma_1)) = R^+(d\gamma_1)$.
\end{enumerate}
(Such a system always exists.) Fix a set $R^+= t_1(R^+)$ of positive
roots for $H$ in $G$ containing 
$R^+(d\gamma_1)$. Suppose finally that $H_s$ is a maximally split
Cartan, with any positive root system $R^+_s$, as in Proposition
\ref{prop:findimlHswts}; we will use the maps $i(R^+_s,R^+)$ and
$I(R^+_s,R^+)$ defined in that proposition. Set
$$d\gamma_{1,s} = i(R^+_s,R^+)^{-1}(d\gamma_1) \in {\mathfrak h}_s^*, \quad
R^+(d\gamma_{1,s}) =  i(R^+_s,R^+)^{-1}(R^+(d\gamma_1)). $$
There is a unique coherent family $\Phi_1$ of virtual $({\mathfrak g},K)$-modules based on 
  $d\gamma_{1,s} + \Lambda_{\text{fin}}(G,{}^1 H_s)$, with the following
  characteristic properties:
\begin{enumerate}[label=\alph*)]
\item $\Phi_1(d\gamma_{1,s}) = [J(\Gamma_1)]$; and
\item whenever 
\begin{enumerate}[label=\roman*)]
\item $d\gamma_{1,s} + d\psi_s$ is integrally
  dominant (with respect to $R^+(d\gamma_{1,s})$), and
\item $\Gamma_1$ is type one, and
\item  $\psi_s$ is type one,
\end{enumerate}  
then $\Phi_1(d\gamma_{1,s} + \psi_s)$  is (the class % in the
Grothendieck group 
of) an irreducible representation or zero. 
% or two times an irreducible representation or zero (if $\Gamma_1$
%  and $\psi_1$ are type two). 
\end{enumerate}
Let $\Gamma$ be an irreducible constituent of $\Gamma_1|_H$ (of which
there are one or two, depending on the type of  $\Gamma_1$).  Write $\Phi$
for the coherent family for $G$ based on $d\gamma_{1,s} + 
\Lambda_{\text{fin}}(G,H_s)$ constructed in Theorem \ref{thm:cohfam}. 
\begin{enumerate}
\item If $\Gamma_1 = \Gamma_{\text{ind}}$ is type two (equivalently, if
  $\Gamma$ is not fixed by $t_1$) then $\Psi_1 = \Psi_{\text{ind}}$
  (Example \ref{ex:indcoh}). In particular,
$$\Phi_1(d\gamma_{1,s}+\psi_1)|_G = \sum_{\psi \subset \psi_1|H}
\Phi(d\gamma_{1,s} + \psi) + \Phi(d\gamma_{1,s} + \psi)^{t_1}.$$
\end{enumerate}
Henceforth {\em assume that $\Gamma_1$ is type one}; equivalently, that
$\Gamma$ is fixed by $t_1$.
 Write $\Pi(d \gamma_1)$ for the set of simple roots of $R^+(d\gamma_1)$,
 and define the {\em $\tau$-invariant of $\Gamma_1$} $\tau_s(\Gamma_1)
 =_{\text{def}} \tau_s(\Gamma)\subset\Pi_s(d\gamma_1)$ as in Theorem
 \ref{thm:cohfam}.  
\begin{enumerate}[resume]
\item $$\Phi_1(d\gamma_{1,s}+\psi_1)|_G = \sum_{\psi \subset \psi_1|H}
\Phi(d\gamma_{1,s} + \psi).$$
\item The $\tau$-invariant is preserved by $t_s$.
\item If $d\gamma_s + d\psi_s$ is integrally dominant, then $\Phi(d\gamma_s
  + \psi_s) = 0$ if and only if there is a simple root $\alpha_s \in
  \tau_s(\Gamma)$ with $(d\gamma + d\psi)(\alpha_s^\vee) = 0$.
\item Suppose that $d\gamma_s + d\psi_s$ is integrally dominant, and does
  not vanish on any root in $\tau_s(\Gamma_1)$.  Write 
$$\psi = I(R^+_s,R^+)(\psi_s) \in \Lambda_{\text{fin}}({}^1 H).$$
 Then
$$\Gamma_1 + \psi =_{\text{def}} ({}^1 H,\gamma + \psi,R^+_{i{\mathbb R}})$$
is a Langlands parameter; and
$$\Phi(d\gamma_s + \psi_s) = [J(\Gamma + \psi)].$$
\item The set of all the irreducible constituents of the various virtual
  representations $\Psi_1(d\gamma_s + \psi_s)$ is equal to the set of
  irreducible constituents of all the tensor products $F_1\otimes
  J(\Gamma_1)$, with $F_1$ a finite-dimensional representation of
  ${}^\delta G$.
\end{enumerate}
\end{theorem}
\begin{proof}
% \begin{subequations}\label{se:extcohproof}
The first issue is the existence of a $t_1$-stable set of positive
integral roots. Recall from \eqref{se:extfindiml} the $t_1$-stable set
of positive roots $R^+_1$ containing $R^+_{i{\mathbb R}}$. Using
these, we can construct the desired set of positive integral roots by
\begin{equation*}\begin{aligned}
R^+(d\gamma_1) = \{\alpha \in &R(G,H) \mid d\gamma_1(\alpha^\vee) \in
{\mathbb Z}^{>0},\ \text{or} \\
&d\gamma_1(\alpha^\vee) = 0,\ {and}\ d\gamma_1(\theta\alpha^\vee) >
0,\ \text{or} \\
&d\gamma_1(\alpha^\vee) = d\gamma_1(\theta\alpha^\vee) = 0,\ {and}\
\alpha \in R^+_1 \}. 
\end{aligned}\end{equation*}
Conditions (a) and (b) from the theorem are obviously
satisfied. Condition (c) follows from the requirement that $d\gamma_1$
is weakly dominant for $R^+_{i{\mathbb R}}$ (Theorem \ref{thm:LC}),
and the hypothesis that $R^+_{i{\mathbb R}} \subset R^+_1$ (Definition
\ref{def:extLP}. For (d), since $t_1$ fixes both $d\gamma_1$ and
$R^+_1$, and $t_1$ commutes with the action of $\theta$ on $H$, it
follows that $t_1$ preserves $R^+(d\gamma_1)$.

For the coherent family, suppose first that $\Gamma_1 =
\Gamma_{\text{ind}}$ is of type two, so that
\begin{equation*}
J(\Gamma_1) = J(\Gamma_{\text{ind}}) = J(\Gamma)_{\text{ind}}.
\end{equation*}
Let $\Phi$ be the coherent family for $J(\Gamma)$ based on $d\gamma_s
+ \Lambda_{\text{fin}}(G,H_s)$ (Theorem \ref{thm:cohfam}), and form
the induced coherent family $\Phi_1 = \Phi_{\text{ind}}$ as in Example
\ref{ex:indcoh}. Then the characteristic properties (a) and (b) are
immediate, assertion (1) is true by definition, and the rest of the
theorem does not apply.

Next suppose that $\Gamma_1$ is type one, so that it is one-dimensional
and extends the Langlands parameter $\Gamma$ for $G$. There are two of
these extensions $\Gamma_+$ and $\Gamma_-$, to each of which we would
like to attach a coherent family $\Phi_{\pm}$.  What is easy is that
\begin{equation*}\Phi_+ +\Phi_- = \Phi_{\text{ind}},\end{equation*}
in the sense that the induced coherent family behaves as the theorem
says that the sum ought to; what is harder is to find the two summands
individually. For this we return to the proof of
Theorem \ref{thm:cohfam} in \Cite{Vgreen}. Because parabolic and
cohomological induction behave well with respect to translation
functors, we can use either Definition \ref{def:Lextparam} or
Definition \ref{def:VZextparam} to construct a 
coherent family $\Phi^{\text{std}}_1$ for
 ${}^\delta G$ (based on ${}^1 H$) satisfying
$$\Phi^{\text{std}}_1(d\gamma_1) =
I_{\quo}(\Gamma_1).$$
Using these building
blocks (and taking ${\mathbb Z}$-linear combinations) we
can use the proof of Corollary 7.3.23 in \Cite{Vgreen} to construct
the coherent family $\Phi_1$ that we seek. 

The construction makes clear the description in (2) of the restriction
to $G$.  The statement in (3) means that the corresponding set of
simple roots for $R^+(d\gamma_1)$ is preserved by $t_1$. This is
obvious from the description of $\tau$ in Theorem
\ref{thm:cohfam}. The statement in (4) about vanishing reduces at once
to $G$, where it is part of Theorem \ref{thm:cohfam}.  The
calculation in (5) of Langlands parameters follows from the construction
of $\Phi_1$. Part (6) is a formal consequence of the
definition of coherent family. 
%\end{subequations} %{se:extcohproof}
\end{proof}

\begin{subequations}\label{se:exttransfunc}
In order to discuss translation functors for extended groups, we
need the language of \eqref{se:charformulas} and \eqref{se:transfunc}
explicitly extended to ${}^\delta G$. So we fix a maximally split
Cartan and a minimal parabolic subgroup
\begin{equation}
H_s = T_sA_s, \qquad M_s = K^{A_s}, \qquad P_s = M_sA_sN_s
\end{equation}
of $G$, and positive imaginary roots $R^+_{s,i{\mathbb R}}$
for $M$. These define a positive system
\begin{equation*}
R^+_s = R^+_{s,i{\mathbb R}} \cup \text{(roots of $H_s$ in $N_s$)} =
R^+_{s,i{\mathbb R}} \cup R(N_s,H_s)
\end{equation*}
for $G$.  Define a Weyl group element $w_s$ by the requirement
\begin{equation*}
\theta(R^+_s) = R^+_{s,i{\mathbb R}} \cup -R(N_s,H_s) = w_s^{-1}(R^+_s).
\end{equation*}
This is the long element of the ``little Weyl group'' of $A_s$ in $G$,
so in particular it belongs to $W(G,H_s)$. Setting
\begin{equation*}
t_s = w_s\theta_{H_s}, \qquad t_s(R^+_s) = R^+_s
\end{equation*}
we find that a representative $\delta_s$
for $t_s$ generates with $H_s$ an extended torus ${}^1 H_s$.

Fix also a weight
\begin{equation*}
\lambda_0 \in {\mathfrak h}_s^*
\end{equation*}
corresponding to a (possibly singular) infinitesimal character 
\begin{equation*}
\xi_{\lambda_0}\colon {\mathfrak Z}({\mathfrak g}) \rightarrow
{\mathbb C}.
\end{equation*}
To simplify the discussion slightly (although it is not necessary) we
assume that $\lambda_0$ is real (\cite{Vgreen}*{Definition 5.4.11}). We
need (for the abstract theory of translation functors) to assume that
$\lambda_0$ is integrally dominant for $R^+_s$; and we assume also
that $\lambda_0$ is real dominant (\eqref{eq:realdom}):
\begin{equation}\label{eq:sdominant}
\lambda_0(\alpha^\vee) \ge 0, \qquad (\alpha \in R^+_s).
\end{equation}
This requirement determines $\lambda_0$ uniquely in its Weyl group
orbit. Because $t_s$ preserves $R^+_s$, the weight $t_s(\lambda_0)$
shares the dominance property \eqref{eq:sdominant}.  As a consequence,
\begin{equation}\label{eq:lambda0infl}\begin{aligned}
t_s(\lambda_0) = \lambda_0 &\iff t_s(\lambda_0) \in W\cdot
\lambda_0\\ 
&\iff \text{infinitesimal character $\xi_{\lambda_0}$
  fixed by $\theta$}
\end{aligned}
\end{equation}

Define
\begin{equation}\label{eq:setofextLP}
% \begin{aligned}
{}^\delta B(\lambda_0) = \ \text{equiv. classes of extended parameters
 % Langlands parameters %} \\
% &\ \ \text{
of infl. character $\lambda_0$;}
% \end{aligned}
\end{equation}
as usual for the extended group, ``infinitesimal character
$\lambda_0$'' means ``annihilated by
$[\ker\xi_{\lambda_0}][\ker\xi_{t_s\lambda_0}]$.''

If $\lambda_0 \ne t_s(\lambda_0)$, then by \eqref{eq:lambda0infl} the
infinitesimal character $\xi_{\lambda_0}$ is not fixed by $\theta$,
so no representation of $G$ of infinitesimal character
$\xi_{\lambda_0}$ is fixed by $\theta$. By Clifford theory
(Proposition \ref{prop:extpairrep}), in this case all representations
of infinitesimal character $\lambda_0$ are induced from $G$. Once all
the representations are induced, then any question about character
theory can be reduced at once to $G$.  (We will recall in detail how
this works in Section \ref{sec:extklv}.)

We therefore only need to analyze the case
\begin{equation*}
t_s(\lambda_0) = \lambda_0.
\end{equation*}
Write $R(\lambda_0)$ for the integral roots (Definition
\ref{def:introots}), and write
\begin{equation*}
R^+(\lambda_0) =_{\text{def}} R(\lambda_0) \cap R^+_s,
\end{equation*}
which we have assumed makes $\lambda_0$ integrally dominant.  Write
\begin{equation*}
\Pi(\lambda_0) \subset R^+(\lambda_0),
\end{equation*}
and 
\begin{equation*}
\{\alpha \in \Pi(\lambda_0) \mid \lambda_0(\alpha^\vee) = 0
\}=_{\text{def}} \Pi^{\lambda_0}
\end{equation*}
for the subset of simple roots orthogonal to $\lambda_0$. Since
$t_s(\lambda_0) = \lambda_0$, the sets $\Pi^{\lambda_0} \subset
\Pi(\lambda_0)$ are preserved by $t_s$. If $X_1\in {}^\delta
B(\lambda_0)$ is a type one irreducible, extending a type one
irreducible $X$ for $G$, we have defined the $\tau$-invariant 
\begin{equation}\label{eq:type1tau}
\tau(X_1) = \tau(X) \subset \Pi(\lambda_0)
\end{equation}
in Theorem \ref{thm:extcohfam}.  If $Y_1 = Y_{\text{ind}}$ is type two,
then we define
\begin{equation}\label{eq:type2tau}
\tau(Y_{\text{ind}}) =_{\text{def}} (\tau(Y),\tau(Y^{t_s})) = (\tau(Y),t_s(\tau(Y)),
\end{equation}
an unordered pair of subsets of $\Pi(\lambda_0)$.  (In order to unify
the notation, we may sometimes regard $\tau(X_1)$ as the unordered
pair $(\tau(X),\tau(X))$.)

Finally, we fix a type one weight
\begin{equation}\label{eq:regularizingphi}
\phi \in \Lambda_{\text{fin}}(G,H_s)^{t_s}
\end{equation}
of a finite-dimensional representation of $G$, dominant for $R^+_s$,
and such that 
\begin{equation*}
\lambda_0 + d\phi \text{\ is dominant and regular for $R^+(\lambda_0)$.}
\end{equation*}
Just as in \eqref{eq:transrealdom}, for the purpose of studying
signatures we will want to assume in addition
\begin{equation}\label{eq:extrealdom}
\lambda_0 + d\phi \text{\ is dominant and regular for $R^+_s$.}
\end{equation}
(One way to achieve this is to take $\phi$ to be the sum of the roots
in $R^+_s$.) Write
\begin{equation*}
\phi_{\pm} \in \Lambda_{\text{fin}}({}^\delta G,{}^1 H_s) %, \quad
% \phi_{\pm}(\delta_s) = \pm 1
\end{equation*}
for the two extensions of $\phi$ to ${}^1 H_s$. %; here $\delta_s$ is one
% of the distinguished lifts of $t_s$ described in Lemma
% \ref{lemma:realexttorus}). 
Now we can write 
% (with notation as in Lemma \ref{lemma:realexttorus})
\begin{equation*}
F_{\phi_{\pm}} = \text{\ finite-dimensional irreducible for ${}^\delta
  G$ of highest weight $\phi_{\pm}$.}
\end{equation*}

In order to define the Jantzen-Zuckerman translation functors, we need
the functor (on ${\mathfrak g}$-modules)
\begin{equation*}
P_{\lambda_0}(M) =\  \text{\parbox{.6\textwidth}{largest submodule of
    $M$ on which $z - \xi_{\lambda_0}(z)$ acts nilpotently, all 
  $z \in {\mathfrak Z}({\mathfrak g})$}} 
\end{equation*}
If $M$ is a $({\mathfrak g},{}^\delta K)$ module, so is $P_{\lambda_0}(M)$.  On
the category of ${\mathfrak Z}({\mathfrak g})$-finite modules
(for example, on $({\mathfrak g},{}^\delta K)$-modules of finite length) the
functor $P_{\lambda_0}$ is exact.  The translation functors we want are
\begin{equation}\label{eq:exttransfunc}
T_{\lambda_0}^{\lambda_0 + \phi_+}(M) =_{\text{def}} P_{\lambda_0+
  d\phi}\left(F_{\phi_+} \otimes P_{\lambda_0}(M)\right),
\end{equation}
(``translation away from the wall'') and
\begin{equation}
T_{\lambda_0+\phi_+}^{\lambda_0}(N) =_{\text{def}} P_{\lambda_0}\left(F^*_{\phi_+}
  \otimes P_{\lambda_0+d\phi}(N)\right)
\end{equation}
(``translation to the wall.'')  These are clearly exact functors on
the category of finite-length $({\mathfrak g},{}^\delta K)$-modules, so they
also define group endomorphisms of the Grothendieck group ${\mathcal
  G}({\mathfrak g},{}^\delta K)$. 

These translation functors are quite special, in the sense that each
% both the domain and the range
infinitesimal character is assumed to be
represented by a $t_s$-fixed weight.  There are some reasonable
statements to be made in more generality; but most of those statements
concern type two representations, which are induced from $G$; so they
can be deduced easily from (for example) Corollary
\ref{cor:transfunc}.  We are concerned here only with the more
delicate situation for type one representations.
\end{subequations} %se:exttransfunc

\begin{corollary}[Jantzen \Cite{Jantzen}, Zuckerman
  \Cite{Zuck}; see \Cite{Vgreen}*{Proposition 7.2.22}]\label{cor:exttransfunc} 
Suppose we are in the setting of \eqref{se:exttransfunc}, so that in
particular $\xi_{\lambda_0}$ (possibly singular) and
$\xi_{\lambda_0+\phi}$ (regular) are infinitesimal characters fixed by
twisting by $\theta$.
\begin{enumerate}
\item The functors $T_{\lambda_0+\phi_+}^{\lambda_0}$ and
  $T_{\lambda_0}^{\lambda_0+\phi_+}$ are left and right adjoint to each
  other:
$$\Hom_{{\mathfrak g},{}^\delta K}(N, T_{\lambda_0}^{\lambda_0+\phi_+}(M))
\simeq \Hom_{{\mathfrak g},{}^\delta K}(T_{\lambda_0+\phi_+}^{\lambda_0}(N),M),$$
and similarly with the two functors interchanged.
\item If $\Phi_1$ is a coherent family (Definition
  \ref{def:extcohfam}) on $\lambda_0+\Lambda_{\text{fin}}(G,{}^1
  H_s)$,
$$T_{\lambda_0+\phi_+}^{\lambda_0}(\Phi_1(\lambda_0 + \phi_+)) =
\Phi_1(\lambda_0).$$ 
\item The translation functor $T_{\lambda_0+\phi_+}^{\lambda_0}$
  (translation to the wall) carries every irreducible 
$({\mathfrak g},{}^\delta K)$-module of infinitesimal character $\lambda_0 +
d\phi$ to an irreducible of infinitesimal character $\lambda_0$, or to
zero. The irreducible representations mapped to zero are those
with some root of $\Pi^{\lambda_0}$ in the $\tau$-invariant
(Definition \ref{def:tauinv}, \eqref{eq:type2tau}).
\item The (inverse of the) translation functor of (3) defines an
  injective correspondence of Langlands parameters (and therefore of
  irreducible representations) 
$$t_{\lambda_0}^{\lambda_0+\phi_+}\colon {}^\delta B({\lambda_0}) \hookrightarrow
{}^\delta B({\lambda_0 + d\phi}).$$ 
The image consists of all irreducible representations (of
infinitesimal character $\lambda_0 + d\phi$) having no root
of $\Pi^{\lambda_0}$ in the $\tau$-invariant.
\item Suppose $\Gamma_2 = ({}^2 H,\gamma_2,R^+_{i{\mathbb R}}) \in 
  {}^\delta B(\lambda_0)$ is a Langlands parameter (at the possibly singular
  infinitesimal character $\lambda_0$) for an irreducible
  representation $J(\Gamma_2)$ of ${}^\delta G$. Choose positive integral
  roots $R^+(d\gamma_2)$ as in Theorem \ref{thm:extcohfam}, so that 
$$i(\lambda_0,R^+(\lambda_0),d\gamma_2,R^+(d\gamma_2))\colon {\mathfrak
  h}_s^* \rightarrow {\mathfrak h}^*$$
is an isomorphism as in \eqref{se:movelambda0}, and a map of virtual
representations 
$${}^\delta I(\lambda_0,R^+(\lambda_0),d\gamma_2,R^+(d\gamma_2))\colon
{\mathcal G}_{\text{fin}}({}^1 H_s) \rightarrow {\mathcal
  G}_{\text{fin}}({}^1 H)$$  
as in Proposition \ref{prop:findimlextHswts}(2). Define 
$$\phi(\Gamma_2)_+ =
{}^\delta I(\lambda_0,R^+(\lambda_0),d\gamma_2,R^+(d\gamma_2))(\phi_+),$$
a (type one) ${}^2 H$-weight of a finite-dimensional representation of
${}^\delta G$ (Proposition \ref{prop:findimlextHswts}). Then the
correspondence of Langlands parameters in (4) is
$$t_{\lambda_0}^{\lambda_0 + \phi_+}(\Gamma_2) = \Gamma_2 + \phi(\Gamma_2)_+$$
(notation as in Theorem \ref{thm:extcohfam}(5)).
\item Suppose $\Gamma_1 = ({}^2 H,\gamma_1,R^+_{i{\mathbb R}}) \in
  {}^\delta B(\lambda_0+d\phi)$ is a Langlands parameter (at the regular
  infinitesimal character $\lambda_0 + d\phi$). Write $R^+(d\gamma_1)$
  for the unique set of positive integral roots making $d\gamma_1$
  integrally dominant, so that 
$$i(\lambda_0+d\phi,R^+(\lambda_0),d\gamma_1,R^+(d\gamma_1))\colon {\mathfrak
  h}_s^* \rightarrow {\mathfrak h}^*$$
is an isomorphism as in \eqref{se:movelambda0}.  Define
$$\phi(\Gamma_1)_+ = {}^\delta
I(\lambda_0,R^+(\lambda_0),d\gamma_0,R^+(d\gamma_0))(\phi_+),$$ 
a (type one) ${}^2 H$-weight of a finite-dimensional representation of
${}^\delta G$ (Proposition \ref{prop:findimlextHswts}). Then 
$$t_{\lambda_0+\phi_+}^{\lambda_0}(\Gamma_1) =_{\text{def}} (H,\gamma_1
- \phi(\Gamma_1)_+, R^+_{i{\mathbb R}}) =_{\text{def}}\Gamma_1 -
\phi(\Gamma_1)_+$$ 
is a weak extended Langlands parameter (Definition \ref{def:extLP}).  This
correspondence is a left inverse to the injection of (4):
$$t_{\lambda_0+\phi_+}^{\lambda_0}\circ t_{\lambda_0}^{\lambda_0 +
  \phi_+}(\Gamma_0) = \Gamma_0.$$
\item Translation to the wall carries
  standard representations to weak standard representations:
$$T_{\lambda_0 + \phi_+}^{\lambda_0}(I_{\quo}(\Gamma_1)) =
I_{\quo}(\Gamma_1 - \phi(\Gamma_1)_+).$$
In particular,
$$T_{\lambda_0 + \phi_+}^{\lambda_0}(I_{\quo}(\Gamma_0+\phi(\Gamma_0)_+) =
I_{\quo}(\Gamma_0).$$
\item Translation to the wall carries irreducibles to
  irreducibles or to zero:
$$T_{\lambda_0 + \phi_+}^{\lambda_0}(J(\Gamma_1)) = \begin{cases}
J(\Gamma_1 - \phi(\Gamma_1)_+) & (\tau(J(\Gamma_1)) \cap \Pi^{\lambda_0}
= \emptyset),\\
0 & (\tau(J(\Gamma_1)) \cap \Pi^{\lambda_0}
\ne \emptyset). \end{cases} $$
In particular,
$$T_{\lambda_0 + \phi_+}^{\lambda_0}(J(\Gamma_0+\phi(\Gamma_0)_+) =
J(\Gamma_0).$$
\end{enumerate}
\end{corollary}

% \begin{subequations}\label{se:extreducetoreg}
Corollary \ref{cor:exttransfunc} allows us to ``reduce'' the computation
of the multiplicity matrix $m_{\Xi,\Gamma}$ defined in
\eqref{se:charformulas} to the case of regular infinitesimal
character.  Given a possibly singular infinitesimal character
$\lambda_0$, we choose a type one weight of a finite-dimensional representation
$\phi_+$ as in \eqref{se:exttransfunc}, so that $\lambda_0+d\phi$ is
regular. The (possibly singular) Langlands parameters $\Xi$ and
$\Gamma$ in ${}^\delta B(\lambda_0)$ correspond to regular Langlands parameters
$\Xi+\phi(\Xi)_+$ and $\Gamma+\phi(\Gamma)_+$ in ${}^\delta
B(\lambda_0 + d\phi)$, and
\begin{equation}\label{se:extreducetoreg}
m_{\Xi,\Gamma} = m_{\Xi+\phi(\Xi)_+,\Gamma+ \phi(\Gamma)_+}.
\end{equation}
A little more precisely, we can start with the decomposition into irreducibles
\begin{equation*}
\left[I(\Gamma + \phi(\Gamma)_+)\right] = \sum_{\Xi_1 \in {}^\delta
  B(\lambda_0+d\phi)} 
m_{\Xi_1,\Gamma+\phi(\Gamma)_+}\left[J(\Xi_1)\right]
\end{equation*}
and apply the exact translation functor $T_{\lambda_0 +
  \phi}^{\lambda_0}$. On the left we get $\left[I(\Gamma)\right]$, by
Corollary \ref{cor:exttransfunc}(7). On the right we get
\begin{equation*}
\sum_{\substack{\Xi_1 \in B(\lambda_0+d\phi)\\ \tau(\Xi_1) \cap
    \Pi^{\lambda_0} = \emptyset}} m_{\Xi_1,\Gamma+\phi(\Gamma)_+}\left[J(\Xi_1
- \phi(\Xi_1)_+)\right].
\end{equation*}
The irreducibles appearing on the right are all distinct (this is part
of the uniqueness assertion for the coherent family in Theorem
\ref{thm:cohfam}), so the coefficients are the multiplicities of
irreducibles in $I(\Gamma)$.

To say this another way: the index set of extended Langlands parameters at
the possibly singular infinitesimal character $\lambda_0$ is naturally
a {\em subset} of the index set at the regular infinitesimal character
$\lambda_0 + d\phi$. The multiplicity matrix at $\lambda_0$ is the
corresponding submatrix of the one at $\lambda_0+d\phi$.
% \end{subequations}

Finally, we record the reduction to regular infinitesimal character of
the character matrix
$M_{\Xi,\Gamma}$.

\begin{theorem}\label{thm:exttransstd} 
Suppose we are in the setting of \eqref{se:exttransfunc}; in
particular, we assume that the infinitesimal characters $\lambda_0$
and $\lambda_0+d\phi$ are both fixed by $\theta$. Write
$$t_{\lambda_0+\phi_+}^{\lambda_0}\colon {}^\delta B(\lambda_0+d\phi)
\rightarrow 
{}^\delta B_{\weak}(\lambda_0)$$ 
(notation \eqref{se:charformulas}) for the map of Corollary
\ref{cor:exttransfunc}(6). Fix an extended Langlands parameter $\Xi_1 \in
{}^\delta B(\lambda_0 + d\phi)$. Write
$$\Xi_0 = t_{\lambda_0+\phi_+}^{\lambda_0}(\Xi_1) \in {}^\delta
B_{\weak}(\lambda_0), $$ 
so that
$$T_{\lambda_0+\phi_+}^{\lambda_0}(I_{\quo}(\Xi_1)) =
I_{\quo}(\Xi_0).$$ 
\begin{enumerate}
\item There is a unique subset $C_{\text{sing}}(\Xi_0) \subset
  {}^\delta B(\lambda_0)$ %%% OK TO HERE? cut following proof.
  with the property that
$$I_{\quo}(\Xi_0) = \sum_{\Xi' \in C_{\text{sing}}(\Xi_0)} I_{\quo}(\Xi').$$
\item The set $C_{\text{sing}}(\Xi_0)$ has cardinality either equal to
  zero (in case some simple compact imaginary root for $\Xi_0$ vanishes on
  $\lambda_0$, so that $I(\Xi_0) = 0$) or a power of two. 
\end{enumerate}
For each Langlands parameter $\Xi\in {}^\delta B(\lambda_0)$, define
$$C_{\text{reg}}(\Xi) = \{\Xi_1\in {}^\delta B(\lambda_0 + d\phi) \mid \Xi \in
C_{\text{sing}}(\Xi_0)\} \subset {}^\delta B(\lambda_0 + d\phi);$$
here $\Xi_0 = t_{\lambda_0+\phi_+}^{\lambda_0}(\Xi_1)$ as above.
\begin{enumerate}[resume]
\item Suppose $\Gamma\in {}^\delta B(\lambda_0)$ is any Langlands parameter,
  corresponding to $\Gamma+\phi(\Gamma) \in {}^\delta B(\lambda_0)$. Then the
  character formula \eqref{eq:charformirr} is
$$J(\Gamma) = \sum_{\Xi \in {}^\delta B(\lambda_0)}\ \  \sum_{\Xi_1\in
  C_{\text{reg}}(\Xi)} M_{\Xi_1,\Gamma+\phi(\Gamma)}[I(\Xi)].$$ 
Equivalently,
$$M_{\Xi,\Gamma} = \sum_{\Xi_1\in C_{\text{reg}}(\Xi)}
M_{\Xi_1,\Gamma+\phi(\Gamma)}.$$ 
\end{enumerate}
\end{theorem}
The proof is identical to that of Theorem \ref{thm:transstd}.

\section{Kazhdan-Lusztig polynomials for characters} \label{sec:klv}  
% \section{KLV theory}\label{sec:klv}
\setcounter{equation}{0}
We have seen in \eqref{se:charformulas} how to define finite
matrices over ${\mathbb Z}$ 
\begin{equation*}
\left(M_{\Gamma,\Psi}\right), \quad \left(m_{\Xi,\Gamma}\right) \quad
({\Xi,\Gamma,\Psi \in B(\chi)})
\end{equation*}
describing the characters of irreducible
representations. In the present section we will recall how
Kazhdan and Lusztig computed the character matrices.
Carrying this out for the extended group ${}^\delta G$ requires
significant effort (most of which is carried out in \Cite{LV12}); we
will therefore discuss that case separately in Section
\ref{sec:extklv}.  However, a number of the preliminary definitions
given here extend easily to the extended group case; we will try to
note that explicitly as they arise.

A {\em $q$-analogue of an integer $n$} is a polynomial $N \in {\mathbb
  Z}[q]$ with the property that $N(1) = n$. To be an interesting
$q$-analogue, the coefficients of $N$ should carry refined information
about some question to which $n$ is an answer.

The first step in the Kazhdan-Lusztig idea is to form $q$-analogues of
the multiplicities and signatures. The matrices over ${\mathbb Z}$ are
then replaced by matrices over ${\mathbb Z}[q]$. The polynomial
entries satisfy recursion relations that are not visible at $q=1$, so
they can be computed. 

Kazhdan and Lusztig in \cites{kl79,kl80} introduced the
$q$-analogues in a 
way that is quite subtle and complicated (transforming the problem
first into algebraic geometry, then into finite characteristic, and
finally studying eigenvalues of a Frobenius operator).  The advantage
of their approach was that Deligne and others (\cite{BBD}) had proven
deep facts about these eigenvalues. We use an approach that is much
simpler to explain, but apparently much harder to prove anything
about.  Fortunately Beilinson and Bernstein in \Cite{BB} have already
proven the equivalence of the two approaches.

In \cites{kl79,kl80}, the definitions of $q$-analogues are mostly
geometric, but with modifications by powers of $q$ given by lengths of
Weyl group elements (which played the role of Langlands parameters).
We need to make such modifications in our setting as well, so we need
a notion of length for Langlands parameters.

\begin{definition}[\cite{Vgreen}*{Definition 8.1.4}]\label{def:length}
Suppose $\Gamma = (H,\gamma,R^+_{i{\mathbb R}})$ is a Langlands
parameter for $G$ (Theorem \ref{thm:LC}). Choose a system of positive
integral roots $R^+(d\gamma)$ subject to the requirements in Theorem
\ref{thm:cohfam}.  Let $L_{\text{real}}$ be a split semisimple real group
with root system the set $R_{\mathbb R}(d\gamma)$ of integral real
roots, and $K_{\text{real}}$ a maximal compact subgroup of
$L_{\text{real}}$. Define 
$$\begin{aligned}
c_{\text{real}} &=_{\text{def}} \#\{\text{positive roots for
  $L_{\text{real}}$}\} - 
\#\{\text{positive roots for $K_{\text{real}}$}\}.\\
&= \text{dimension of a Borel subgroup of $K_{\text{real}}({\mathbb
    C})$.}
\end{aligned}
$$
(The reason for the equality is that the Iwasawa decomposition for the
split group $L_{\text{real}}$ implies that
$$\dim K_{\text{real}} = \#\{\text{positive roots for $L_{\text{real}}$}\}.)$$
The {\em (integral) length of $\Gamma$} is
$$\ell(\Gamma) = \#\{\text{pairs $(\alpha,-\theta(\alpha))$ of complex
  roots in $R^+(d\gamma)$}\} + c_{\text{real}}.$$

Because this definition of length is natural, it is unchanged by
twisting $\Gamma$ by any automorphism commuting with $\theta$.  In
particular, it extends immediately to parameters for the extended
group: the length of a parameter for ${}^\delta G$ is just the length
of any constituent of its restriction to $G$.
\end{definition}

\begin{proposition}[\cite{Vgreen}*{Proposition 8.6.19}]\label{prop:length}
Suppose that $\Xi$ and $\Gamma$ are Langlands parameters, and that
$J(\Xi)$ occurs as a composition factor in the standard module
$I(\Gamma)$. Then 
$$\ell(\Xi) \le \ell(\Gamma).$$
Equality holds if and only if $\Xi$ is equivalent to $\Gamma$ (so that
$J(\Xi)$ is equal to the Langlands subquotient $J(\Gamma)$).

Exactly the same statement holds for the extended group ${}^\delta G$.
\end{proposition}

This Proposition provides another proof of the upper triangularity of
the multiplicity matrix proved in \eqref{eq:uppertri}:

\begin{corollary}\label{cor:uppertri}
Suppose that $m_{\Xi,\Gamma} \ne 0$. Then 
$$\ell(\Xi) \le \ell(\Gamma).$$
Equality holds if and only if $\Xi$ is equivalent to $\Gamma$. Consequently
the multiplicity matrix $m_{\Xi,\Gamma}$ is upper triangular with
respect to any ordering of the Langlands parameters consistent with
length.
\end{corollary}

This Corollary is all that we really need about the upper triangular
nature of the multiplicity matrix; but having come so far, we will
pause to include a more precise statement.

\begin{definition}\label{def:bruhat} Suppose $H$ is a Cartan subgroup 
  of $G$, and $\lambda_1 \in {\mathfrak h}^*$ is a {\em regular}
  weight. Recall from \eqref{se:charformulas} the finite set
  $B(\lambda_1)$ of (equivalence classes of) Langlands parameters of
  infinitesimal character $\lambda_1$.  The {\em Bruhat order} on
  $B(\lambda_1)$ is the transitive closure of the relation
$$\Xi < \Gamma\quad \text{if} \quad \text{$J(\Xi)$ is a composition
  factor of $I(\Gamma)$.}$$

This definition makes sense for ${}^\delta G$.
\end{definition}

\begin{danger}
This definition makes perfectly good sense for singular infinitesimal
character, {\em but we choose not to use it in that case.} The reason
is that the order with this definition is easy to compute (for $G$) in the case
of regular infinitesimal character (see the remarks after Proposition
\ref{prop:bruhat} below), but would {\em not} be easy to compute for singular
infinitesimal character. Instead we use
\end{danger}

\begin{definition}\label{def:singbruhat} Suppose $H_s$ is a maximally
  split Cartan subgroup of $G$, and $\lambda_0\in
  {\mathfrak h}_s^*$ is a a possibly singular weight. Choose a weight
  $\phi\in \Lambda_{\text{fin}}(G,H_s)$ as 
  in \eqref{se:transfunc}, so that (among other things) 
$$\lambda_1 = \lambda_0 + d\phi$$
is a regular weight, and there is an inclusion of Langlands parameters
$$t_{\lambda_0}^{\lambda_0 + \phi} \colon B(\lambda_0) \hookrightarrow
B(\lambda_1)$$
(Corollary \ref{cor:transfunc}). The {\em Bruhat order} on
$B(\lambda_0)$ is the Bruhat order on $B(\lambda_1)$, pulled back by
$t_{\lambda_0}^{\lambda_0 + \phi}$.  
\end{definition}
The preservation of multiplicities explained in \eqref{se:reducetoreg}
shows first of all that if $\lambda_0$ is regular, then this
definition agrees with Definition \ref{def:bruhat} (even if we use a
nonzero $\phi$). It also follows that
\begin{equation*}
J(\Xi_0) \text{\ a composition factor of } I(\Gamma_0) \implies
\Xi_0 \le \Gamma_0.
\end{equation*}
What is different in the singular case is that there may be additional
relations in the Bruhat order not generated by these ``composition
factor'' relations.
 
%% give an example! Maybe SU?? No example in SU(2,2), I think none in
%% SU(3,1). In SU(2,2) the singular Bruhat order is not graded: making
%% middle alpha_2 singular makes 2 (discrete series) an immediate
%% predecessor of 18 (length 3). But still agrees with inherited from
%% regular.

\begin{proposition}\label{prop:bruhat}
Fix a regular weight $\lambda_1\in {\mathfrak h}^*$. The Bruhat order
on $B(\lambda_1)$ (Definition \ref{def:bruhat}) is graded by the length
function of Definition \ref{def:length}: if $\ell(\Gamma) = p$, then
every immediate predecessor $\Xi$ of $\Gamma$ has $\ell(\Xi) = p-1$.
\end{proposition}

The statement and proof of this proposition carry over to the extended
group ${}^\delta G$ with little change. In the case of $G$, the proof
also provides a recursion (on $\ell(\Gamma)$) for computing the Bruhat
order; but this recursion does not carry over in a simple fashion to
${}^\delta G$. A short version of the difficulty is that ``complex
cross actions for the extended group may be double-valued.'')  We do
not need this computation, so we have not pursued the matter.

The recursive computation of the Bruhat order is implemented by the
command {\tt blockorder} in the software {\tt atlas} written by Fokko
du Cloux and Marc van Leeuwen (\cite{atlas}).

\begin{proof} This is nearly a consequence of the proof of Prop 8.6.19
  of \Cite{Vgreen}. The argument given there for ``Case I'' (or in
  fact also for ``Case II'') serves perfectly as an inductive step in
  the proof of the stronger statement given in this proposition. The
  difficulty arises only in ``Case III'' of that proof. So we may
  assume that Case I does not arise for any simple root: that is, that  
% \begin{subequations}\label{se:realsplitcase}
\begin{equation*}
\theta\alpha \in R^+(d\gamma), \qquad (\alpha \in \Pi(d\gamma) \text{\
  complex}).
\end{equation*}
According to \Cite{Vgreen}*{Lemma 8.6.1}, this is equivalent to
\begin{equation*}
\theta(R^+(d\gamma)) - R^+_{\mathbb R}(d\gamma)) = R^+(d\gamma) -
R^+_{\mathbb R}(d\gamma):
\end{equation*}
the Cartan involution $\theta$ preserves the set of non-real positive integral
roots. Just as in \eqref{se:stdVZparam}, we define
\begin{equation*}
L = \text{centralizer in $G$ of $T_0$,}
\end{equation*}
the reductive subgroup generated by $H$ and its real roots. Define a
Langlands parameter $\Gamma_{\mathfrak q}$ for $L$ as in
\eqref{se:stdVZparam}; the corresponding standard representation
$I(\Gamma_{\mathfrak q})$ is an ordinary principal series
representation of $L$. What is important here is that if $\Xi_{\mathfrak
  q}$ is any Langlands parameter for $L$ of infinitesimal character
$d\gamma_{\mathfrak q}$, then Zuckerman's cohomological induction
functor carries the irreducible for $L$ to the irreducible for $G$
\begin{equation*}
\left({\mathcal L}_{\overline{\mathfrak q},L\cap K}^{{\mathfrak g}, K}
\right)^s(J(\Xi_{\mathfrak q})) = J(\Xi)
\end{equation*}  
and correspondingly for standard modules
\begin{equation*}
\left({\mathcal L}_{\overline{\mathfrak q},L\cap K}^{{\mathfrak g}, K}
\right)^s(I(\Xi_{\mathfrak q})) = I(\Xi).
\end{equation*}  
% \end{subequations}

These facts show that the inclusion of Langlands parameters
\begin{equation*}
B_L(d\gamma_{\mathfrak q}) \hookrightarrow B_G(d\gamma), \qquad
\Xi_{\mathfrak q} \mapsto \Xi
\end{equation*}
respects Bruhat orders, and in fact carries the Bruhat interval below
$\Gamma_{\mathfrak q}$ isomorphically onto the Bruhat interval below
$\Gamma$. It is easy to see that the inclusion preserves lengths.

\begin{subequations}\label{se:PSbruhat}
So the inductive step at $\Gamma$ in the proof of our proposition may
be reduced to the case $G=L$: that is, to the case that $I(\Gamma)$ is
an ordinary principal series representation for the split group
$L$. We assume this for the rest of the proof. To simplify the
notation, we will assume also that all the odd roots are ``type
II''. (If this assumption fails and there is a ``type I'' root, then
as remarked at the beginning of the proof, we can just apply the
argument in \Cite{Vgreen}.)
In this setting there is a ${\mathbb Z}/2{\mathbb Z}$ grading 
\begin{equation}
R(d\gamma)^\vee = R(d\gamma)^\vee_{\text{even}} \cup R(d\gamma)^\vee_{\text{odd}}
\end{equation}
on the
integral coroots $R(d\gamma)^\vee$ (\cite{Vgreen}*{Lemma 8.6.3}); the
odd coroots in this grading (and also the corresponding roots) are
said to {\em satisfy the parity condition}. Attached to every odd
positive root $\beta$ there is a ``Cayley transformed 
parameter''
\begin{equation}
\Gamma_\beta = (H_\beta,\Gamma_\beta,\{\tilde\beta\}) \qquad (\beta
\in R^+(d\gamma)_{\text{odd}}).
\end{equation}
Here $H_\beta \subset \phi_\beta(SL(2,{\mathbb R}))\cdot H$
(Definition \ref{def:posroots}) corresponds to the Cartan subgroup
$SO(2)$ of $SL(2,{\mathbb R})$; its only imaginary roots are the two
$\pm\tilde\beta$ in $\phi_\beta({\mathfrak s}{\mathfrak l}(2,{\mathbb
  C}))$.  This is essentially the Langlands parameter of the discrete
series representation of $SL(2,{\mathbb R})$ appearing inside a
reducible principal series representation.  (The ``type II''
assumption, which concerns the disconnectedness of $H$, means that
this discrete series representation is unique, rather than being one
of a pair; that is the ``simplified notation'' of which we are taking
advantage.)

Because $H$ is assumed to be split, the Cartan involution $\theta$ of
${\mathfrak h}$ is $-1$ on the span of the roots. The
Cartan involution $\theta_\beta$ on ${\mathfrak h}_\beta$ is therefore
$-s_{\tilde \beta}$ on the span of the roots.

The long intertwining operator for $I(\Gamma)$ has a factorization in
which each factor corresponds to a real root.  The factors for the
nonintegral roots, and for the integral roots of even grading, are all
isomorphisms, and so do not contribute to the kernel of the long
intertwining operator. The conclusion is that the composition factors
of the kernel---that is, the composition factors of $I(\Gamma)$ other
than $J(\Gamma)$---all appear in $I(\Gamma_\beta)$ for some 
(odd positive integral real) $\beta$. 

In terms of the Bruhat order, this means
\begin{equation*}
\Xi < \Gamma \iff \Xi \le \Gamma_\beta \qquad \text{(some
  $\beta\in R^+(d\gamma)_{\text{odd}}$)}
\end{equation*}
It is easy to compute
\begin{equation*}
\ell(\Gamma_\beta) = \ell(\Gamma) - [\ell(s_\beta) + 1]/2,
\end{equation*}
where the length function in brackets is the one for the integral Weyl
group $W(R(d\gamma))$.  
The only elements of a Weyl group of length one are the simple
reflections; so
\begin{equation*}
\ell(\Gamma_\beta) \le \ell(\Gamma) - 1,
\end{equation*}
with equality if and only if $\beta$ is simple in $R^+(d\gamma)$.

The conclusion is that the only possible
immediate predecessors of $\Gamma$ are the various $\Gamma_\beta$,
with $\beta$ odd positive integral.  To complete the proof of the
proposition, we must show that if $\beta$ is not simple, then
$\Gamma_\beta$ is not an {\em immediate} predecessor.  To prove this,
we will show
\begin{equation}\label{eq:bruhatheight}
\text{if $\beta$ is odd positive, there is a simple odd $\alpha$
  with $\Gamma_\beta \le \Gamma_\alpha$.}
\end{equation}
This we prove by induction on the height of $\beta$.  If $\beta$ is
simple, then we can take $\alpha=\beta$; so suppose $\beta$ is not
simple. What we need to prove is
\begin{equation*}% \begin{gathered}
\text{if $\beta$ is odd positive not simple, there is % }  \\ \text{
an odd  positive $\beta'$ with $\Gamma_\beta < \Gamma_{\beta'}$.}
% \end{gathered}
\end{equation*}
It is convenient to prove the stronger statement
\begin{equation}\label{eq:bruhatheightind} \begin{gathered}
\text{if $\beta$ is odd positive not simple, there is an}\\
\text{odd positive $\beta'$ with $\Gamma_\beta \not\simeq
  \Gamma_{\beta'}$, $J(\Gamma_\beta) \subset I(\Gamma_{\beta'})$.}
\end{gathered}
\end{equation}
Here $\subset$ is abused to mean ``is a composition factor of.''
Choose a simple root $\delta$ so % that 
\begin{equation*}
\langle\beta,\delta^\vee \rangle = m > 0.
\end{equation*}
Since $\beta$ is not simple, $\delta \ne \beta$; so 
\begin{equation*}
s_\delta(\beta) = \beta - m\delta = \beta''
\end{equation*}
is a positive root of lower height than $\beta$. By a similar
calculation, the root $\tilde\delta$ for $H_\beta$ corresponding to
$\delta$ is complex, and $\theta_\beta(\tilde\delta)$ is
negative. Consequently $\delta \notin \tau(\Gamma_\beta)$ (see Theorem
\ref{thm:cohfam}). The Kazhdan-Lusztig algorithm (in fact the more
elementary results proved in \Cite{IC1}) shows that
$J(\Gamma_\beta)$ is a composition factor of $I(s_\delta \times
\Gamma_\beta)$, so that
\begin{equation}\label{eq:bruhatind}
J(\Gamma_\beta) \subset I(s_\delta \times \Gamma_\beta), \qquad \ell(\Gamma_\beta)
= \ell(s_\delta \times \Gamma_\beta) - 1.
\end{equation}
On the other hand, an elementary calculation shows that
\begin{equation*}
s_\delta \times \Gamma_\beta \simeq (s_\delta \times \Gamma)_{\beta''}.
\end{equation*}
(The notation is slightly imprecise. As a representative for the
equivalence class $s_\delta\times\Gamma$, we choose the one having the
same differential as $\Gamma$; that is, we apply some shift by a
multiple of $\delta$, then conjugate by the real Weyl group reflection
in $\delta$.)

We now consider two cases. If $\delta$ is even, then
$s_\delta\times \Gamma = \Gamma$, and so we have shown
 \begin{equation*}
J(\Gamma_\beta) \subset I(s_\delta \times \Gamma_\beta) = I(\Gamma_{\beta''}).
\end{equation*}
This is \eqref{eq:bruhatheightind}, with $\beta'= \beta''$.

Next suppose $\delta$ is odd.  
In order to prove \eqref{eq:bruhatheightind}, we can consider the
(semisimple) rank two reductive subgroup $L$ of $G$ generated by the
rational span of the roots $\delta$ and $\beta$ and the split Cartan
$H$.  This is a Levi subgroup of a real parabolic (not necessarily
containing our fixed Borel). By induction by stages, one sees that it
suffices to prove \eqref{eq:bruhatheightind} for $G=L$. The setting is
therefore that $G$ is split and simple of rank two, $\Gamma$ is
attached to the split Cartan, the simple root $\delta$ is odd, and the
non-simple root $\beta$ is also odd.

In types $A_2$ and $G_2$, and in the block for $B_2$ with no
finite-dimensional representations, every non-discrete series block element of
length at most $\ell(\Gamma)-2$ (like $J(\Gamma_\beta)$) is a
composition factor of {\em every} $I(\Xi)$ with $\ell(\Xi) =
\ell(\Gamma) - 1$ (like $I(\Gamma_\delta)$). So
\eqref{eq:bruhatheightind} is true with $\beta' = \delta$.

We may therefore assume that $L$ is of type $B_2$, and that the block
contains finite-dimensional representations. Since $\delta$ is assumed
to be odd, there are just two possibilities.
\begin{enumerate}[label=\alph*)]
\item The simple root $\delta$ is long, the other simple root
  $\delta'$ is short, $\beta = \delta + \delta'$ 
  is short, and $\beta^\vee = 2\delta^\vee + (\delta')^\vee$.
\item The simple root $\delta$ is short, $\beta =
  2\delta + \delta'$ is long, and $\beta^\vee = \delta^\vee + (\delta')^\vee$.
\end{enumerate}

In case (a), since $\beta$ is assumed to be odd, $\delta'$ must also
be odd; so both simple roots are odd for $\Gamma$. In this case it
turns out that in fact $s_\delta\times\Gamma_\beta$ may {\em not}
precede $\Gamma$ in the Bruhat order; so \eqref{eq:bruhatind} is of no
help to us. Instead one can compute that $J(\Gamma_\beta)$ is a
composition factor of $I(\Gamma_{\delta})$; so we get
\eqref{eq:bruhatheightind} with $\beta' = \delta$. 

In case (b), since $\beta$ and $\delta$ are assumed to be odd,
$\delta'$ must be even. Again it turns out that
$s_\delta\times\Gamma_\beta$ may {\em not} precede $\Gamma$ in the
Bruhat order, but that $J(\Gamma_\beta)$ is a composition factor of
$I(\Gamma_\delta)$. 

By \eqref{eq:bruhatheight}, the only candidates for
immediate predecessors of $\Gamma$ are the $\Gamma_\alpha$, with
$\alpha$ odd and simple. These have length one less, as the
Proposition requires.
\end{subequations}
\end{proof}

What we are looking for is a ``$q$-analogue'' of the multiplicity of
$J(\Xi)$ in $I(\Gamma)$; that is, a collection of
(meaningful!)~integers that sum to the multiplicity. The Jantzen filtration of
$I(\Gamma)$ is perfectly designed for this.

\begin{definition}\label{def:multpoly} Suppose $\Gamma$ and $\Xi$
  are Langlands parameters. Recall from \eqref{eq:jantzenr} the finite
  decreasing filtration
$$I_{\quo}(\Gamma) = I_{\quo}(\Gamma)^0 \supset
    I_{\quo}(\Gamma)^1 \supset I_{\quo}(\Gamma)^2 \cdots$$
We define the {\em multiplicity polynomial} so that the coefficients
record the multiplicities in the subquotients of this filtration:
by
$$Q_{\Xi,\Gamma} = \sum_{r=0}^\infty
m_{I_{\quo}(\Gamma)^r/I_{\quo}(\Gamma)^{r+1}} (J(\Xi))
    q^{(\ell(\Gamma) - \ell(\Xi) - r)/2}.$$
This is a finite Laurent polynomial in the formal variable
$q^{1/2}$. The value at $q^{1/2} = 1$ of this Laurent
polynomial is the multiplicity of $J(\Xi)$ in $I(\Gamma)$:
$$Q_{\Xi,\Gamma}(1) = m_{I_{\quo}(\Gamma)}(J(\Xi)) =
m_{\Xi,\Gamma}$$ 
(see \eqref{eq:multformstd}).

This definition is normalized to make $Q_{\Xi,\Gamma}$ a polynomial in
$q$. It would be more natural to consider
$$(q^{-1/2})^{(\ell(\Gamma) - \ell(\Xi))}Q_{\Xi,\Gamma} = \sum_{r=0}^\infty
m_{I_{\quo}(\Gamma)^r/I_{\quo}(\Gamma)^{r+1}} (J(\Xi))
    q^{-r/2}.$$
% The object on the left here is what appears in the fundamental
% inversion formula of Theorem \ref{thm:BBpol}(3).
This definition carries over to the extended group ${}^\delta G$
without change.
\end{definition}

Although the multiplicity polynomial is the one we are most
interested in, it seems aesthetically satisfactory to include the dual
definition.

\begin{definition}\label{def:charpoly} Suppose $\Gamma$ and $\Psi$
  are Langlands parameters of regular infinitesimal character. Write
  $H = TA$ for the Cartan subgroup attached to $\Gamma$. Choose
  as in \eqref{se:stdVZparam} a $\theta$-stable parabolic subalgebra 
$${\mathfrak q} = {\mathfrak l} + {\mathfrak u}$$
of ${\mathfrak g}$, with $L$ the centralizer in $G$ of $T_0$, of type
VZ with respect to $\Gamma$ (see \eqref{eq:typeVZ}).  Finally recall
(still from \eqref{se:stdVZparam}) the Langlands parameter
$\Gamma_{\mathfrak q}$ for $L$. We define the {\em character polynomial} by
$$P_{\Gamma,\Psi} = \sum_{r=0}^\infty
[\text{multiplicity of $J(\Gamma_{\mathfrak q})$ in
  $H^{s+r}(\overline{\mathfrak u},(J(\Psi))$}](-q^{1/2})^{(\ell(\Psi) -
  \ell(\Gamma) - r)},$$ 
a finite Laurent polynomial in the formal variable $q^{1/2}$.

The preceding definition is normalized to make $P_{\Gamma,\Psi}$ a
polynomial with nonnegative coefficients. It would be more natural to
consider 
$$(-q^{-1/2})^{\ell(\Psi) - \ell(\Gamma)}P_{\Gamma,\Psi} =
\sum_{r=0}^\infty (-1)^r[\text{mult.~$J(\Gamma_{\mathfrak q})$ in
  $H^{s+r}(\overline{\mathfrak u},(J(\Psi))$}]q^{-r/2}.$$ 
% It is the object on the left here that actually appears in the
% inversion formula of Theorem \ref{thm:BBpol}(4).

This definition carries over unchanged to the extended group ${}^\delta G$.
\end{definition}

We turn now to the theorems about {\em computing} the polynomials $P$
and $Q$. These do {\em not} extend painlessly to ${}^\delta G$.

\begin{theorem}[Lusztig-Vogan \Cite{IC3}] \label{thm:KLpol} Suppose
  $\Gamma$ and $\Psi$ are 
  Langlands parameters of regular infinitesimal character. Then the
  character polynomial $P_{\Gamma,\Psi}$ of Definition
  \ref{def:charpoly} is equal to the Kazhdan-Lusztig character
  polynomial (see \Cite{IC3}). In particular, it is a polynomial in
  $q$. Therefore
\begin{enumerate}
\item $J(\Gamma_{\mathfrak q})$ can occur as a composition factor in
  $H^{s+r}(\overline{\mathfrak u}, J(\Psi))$ only if
\begin{enumerate}
\item $\Gamma \le \Psi$ in the Bruhat order (Definition
  \ref{def:bruhat}); and
\item $\ell(\Psi) - \ell(\Gamma)$ is congruent to $r$ modulo $2$; and
\item $0 \le [(\ell(\Psi) - \ell(\Gamma)) - r]/2 \le [\ell(\Psi) -
  \ell(\Gamma)]/2$.
\end{enumerate}
\item Equality can hold in the second inequality of 1(c) only if
  $\Psi=\Gamma$. Equivalently, 
\begin{equation*}\deg P_{\Gamma,\Psi} \le [\ell(\Psi) - \ell(\Gamma) -
  1]/2, \qquad  \text{all $\Gamma < \Psi$.} \end{equation*}
\end{enumerate}
Here is the corresponding result about the cohomology of standard
modules. Write a subscript $\Gamma$ to denote projection on the
summand of infinitesimal character equal to that of
$J(\Gamma_{\mathfrak q})$. % with irr quo?
% should only describe the part of the cohom with some restr on infl char???
\begin{enumerate}[resume]
\item 
\begin{equation*} H^{s+r}(\overline{\mathfrak u},
  I_{\sub}(\Psi))_\Gamma \simeq \begin{cases} 
    I_{\sub}(\Gamma_{\mathfrak q}) & \Gamma = \Psi,\quad r=0, \\
0 & \text{otherwise.} \end{cases}\end{equation*}
\item In the Grothendieck group of virtual $({\mathfrak
    g}_0,K)$-modules (cf.~\eqref{se:charformulas}),
\begin{equation*}[J(\Psi)] = \sum_{\Gamma} (-1)^{\ell(\Psi) - \ell(\Gamma)}
P_{\Gamma,\Psi}(1) [I(\Gamma)].\end{equation*}
\end{enumerate}
\end{theorem}

\begin{theorem}[Beilinson-Bernstein \Cite{BB}] \label{thm:BBpol}
  Suppose $\Gamma$ and $\Xi$ are 
  Langlands parameters of regular infinitesimal character. Then the
  multiplicity polynomial $Q_{\Xi,\Gamma}$ of Definition
  \ref{def:multpoly} is equal to the Kazhdan-Lusztig ``multiplicity
  polynomial'' obtained by inverting the Kazhdan-Lusztig matrix
  $(P_{\Gamma,\Psi})$ (\cite{IC4}*{\S\S 1 and 12}). In
  particular, it is a polynomial in $q$. Therefore
\begin{enumerate}
\item $J(\Xi)$ can occur as a composition factor in
  $I_{\quo}(\Gamma)^r/I_{\quo}(\Gamma)^{r+1}$ only if
\begin{enumerate}
\item $\Xi \le \Gamma$ in the Bruhat order (Definition
  \ref{def:bruhat}); and
\item $\ell(\Gamma)-\ell(\Xi)$ is congruent to $r$ modulo $2$; and
\item $0 \le [(\ell(\Gamma)-\ell(\Xi)) - r]/2 \le [\ell(\Gamma) -
  \ell(\Xi)]/2$. 
\end{enumerate}
\item Equality can hold in the second inequality of 1(c) only if
  $\Gamma = \Xi$. Equivalently, 
$$\deg Q_{\Xi,\Gamma} \le [\ell(\Gamma) - \ell(\Xi)- 1]/2, \qquad
\text{all $\Xi < \Gamma$.}$$
\item In the Grothendieck group of virtual $({\mathfrak
    g}_0,K)$-modules (cf.~\eqref{se:charformulas}),
$$[I(\Gamma)] = \sum_{\Xi} Q_{\Xi,\Gamma}(1) J(\Xi).$$
\end{enumerate}
\end{theorem}

We conclude this section with the corresponding results for singular
infinitesimal character.

\begin{corollary}\label{cor:singBBpol}
Suppose $\lambda_0$ is a possibly singular infinitesimal character,
and $\phi$ is a weight of a finite-dimensional representation of $G$
chosen as in \eqref{se:transfunc}, so that in particular the
infinitesimal character $\lambda_1 = \lambda_0+\phi$ is regular.
Write
$$t_{\lambda_0}^{\lambda_0+\phi}\colon B({\lambda_0}) \hookrightarrow
B({\lambda_0 + d\phi}).$$ 
as in Corollary \ref{cor:transfunc} for the corresponding inclusion of
Langlands parameters.  Fix Langlands parameters
$$\Gamma_0, \Xi_0 \in B(\lambda_0),$$
and write
$$\Gamma = t_{\lambda_0}^{\lambda_0+\phi}(\Gamma_0), \quad \Xi =
t_{\lambda_0}^{\lambda_0+\phi}(\Xi_0)$$
for the corresponding regular parameters in $B(\lambda_0+\phi)$. Then
the multiplicity polynomials of Definition \ref{def:multpoly} satisfy
$$Q_{\Xi_0,\Gamma_0} = Q_{\Xi,\Gamma}.$$  
That is, the multiplicity polynomials at singular infinitesimal
character are a subset of the (Kazhdan-Lusztig) multiplicity
polynomials at regular infinitesimal character; the subset is where
both indices have no root of $\Pi^{\lambda_0}$ (notation
\eqref{se:transfunc}; that is, no root singular on $\lambda_0$) in the
$\tau$-invariant (Definition \ref{def:tauinv}).
\end{corollary}

\section{Kazhdan-Lusztig polynomials for extended groups}\label{sec:extklv}
\setcounter{equation}{0}

In this section we fix again an extended group 
\begin{subequations}\label{se:extklv}
\begin{equation*}
{}^\delta G = G \rtimes \{1,\delta_f\}
\end{equation*}
as in Definition \ref{def:extgrp}.  Fix a maximally split extended
torus ${}^1 H_s$ as in \eqref{se:exttransfunc}, and a weight
$\lambda_0\in {\mathfrak h}_s^*$ satisfying the dominance condition
\eqref{eq:sdominant}.  We want information about representations of
${}^\delta G$---precisely, about irreducible
$({\mathfrak g},{}^\delta K)$-modules---annihilated by $[\ker
\xi_{\lambda_0}][\ker \xi_{t_s\lambda_0}]$; these are parametrized
by the extended Langlands parameters from \eqref{eq:setofextLP}
\begin{equation}
{}^\delta B(\lambda_0) = {}^\delta B(\lambda_0)^1\
{\textstyle\coprod}\ {}^\delta B(\lambda_0)^2, 
\end{equation}
where the (coproduct) notation indicates the disjoint union into type
one and type two parameters (Definition \ref{def:extLP}).  

A word about notation: we prefer to write a type two parameter for
${}^\delta G$ as $\Gamma_{\text{ind}}$ (with $\Gamma$ a type two
parameter for $G$, and a type one parameter as $\Gamma_+$ or
$\Gamma_-$. When we are discussing ``generic'' parameters for
${}^\delta G$ (not specified to be type one or type two) we will use
Roman letters $x, y, z\ldots$. 

We will use the sets of Langlands parameters for $G$
corresponding to ${}^\delta B(\lambda_0)^i$, namely
\begin{equation}
B(\lambda_0) = B(\lambda_0)^1\ {\textstyle\coprod}\  B(\lambda_0)^2;
\end{equation}
recall that a type one representation of $G$ is one fixed by
twisting by $\theta$, and a type one Langlands parameter is one
equivalent (that is, conjugate by a real Weyl group element) to its
twist by $\theta$.  

Definition \ref{def:multpoly} for multiplicity polynomials makes it
clear how tensoring with the character $\xi$ (Definition
\ref{def:extpair}) (or by any other one-dimensional character of
$({\mathfrak g},{}^\delta K)$) affects them:
\begin{equation} \label{eq:Qxi}
Q_{\xi\otimes x,\xi\otimes y} = Q_{x,y} \qquad (x,y \in {}^\delta
B(\lambda_0)). 
\end{equation}
For the character polynomials of Definition \ref{def:charpoly} the
result is
\begin{equation}\label{eq:Pxi}
P_{\xi\otimes y,\xi\otimes z} = P_{y,z} \qquad(y,z \in
{}^\delta B(\lambda_0), \lambda_0 \text{\ regular}).
\end{equation}

There are corresponding assertions about twisting by $\theta$ (or by
any other automorphism of $({\mathfrak g},K)$) on $G$:
\begin{equation}\label{eq:PQtheta}\begin{aligned}
Q_{\Xi^\theta,\Gamma^\theta} &= Q_{\Xi,\Gamma} \qquad (\Xi,\Gamma \in B(\lambda_0)),\\
% \end{equation}
% \begin{equation}\label{eq:Ptheta}
P_{\Gamma^\theta,\Psi^\theta} &= P_{\Gamma,\Psi} \qquad(\Gamma,\Psi \in
B(\lambda_0), \lambda_0 \text{\ regular}).\end{aligned}
\end{equation}

We first say a few words about the (essentially trivial) case
\begin{equation*}
t_s(\lambda_0) \ne \lambda_0,
\end{equation*}
or equivalently $\xi_{\lambda_0} \ne \xi_{\lambda_0}^\theta$. In
this case twisting by $\theta$ defines a bijection 
\begin{equation*}
\theta \colon B(\lambda_0) \buildrel\sim \over\longrightarrow B(t_s\lambda_0)
\end{equation*}
between the disjoint sets $B(\lambda_0)$ and $B(t_s\lambda_0)$.
Therefore induction of parameters is a bijection
\begin{equation*}
\text{ind}\colon B(\lambda_0) \buildrel\sim \over\longrightarrow
{}^\delta B(\lambda_0)^2 = {}^\delta B(\lambda_0), \quad \Gamma\mapsto
\Gamma_{\text{ind}}. 
\end{equation*}
In this case (by Clifford theory) all standard representations and all
irreducible representations are induced from $G$ to ${}^\delta G$.  It
is very easy to check from the definition of the Jantzen filtration
that the filtration is also induced:
\begin{equation*}
I_{\quo}(\Gamma_{\text{ind}})^r = \Ind_{({\mathfrak
    g},K)}^{({\mathfrak g},{}^\delta K)}(I_{\quo}(\Gamma)^r).
\end{equation*}
From this it follows immediately that
\begin{equation*}
Q_{\Xi_{\text{ind}},\Gamma_{\text{ind}}} = Q_{\Xi,\Gamma} +
Q_{\Xi^\theta,\Gamma} \qquad (\Xi,\Gamma \in B(\lambda_0))
\end{equation*}
Since $\Xi^\theta$ has the wrong infinitesimal character to appear in
$I(\Gamma)$, the second term on the right is zero, and we get
\begin{equation}\label{eq:Qall2}
Q_{\Xi_{\text{ind}},\Gamma_{\text{ind}}} = Q_{\Xi,\Gamma} \qquad
(\Xi,\Gamma \in B(\lambda_0), \quad \lambda_0 \ne t_s\lambda_0).
\end{equation}
By a slightly more technical but equally easy computation with Lie
algebra cohomology, we find
\begin{equation}\label{eq:Pall2}
P_{\Gamma_{\text{ind}},\Psi_{\text{ind}}} = P_{\Gamma,\Psi} \qquad
(\Gamma,\Psi \in B(\lambda_0), \quad \lambda_0 \ne t_s\lambda_0 \
\text{regular}).
\end{equation}

Having disposed of this easy case, we will assume for the rest of this
section that
\begin{equation*}
t_s(\lambda_0) = \lambda_0.
\end{equation*}
What Clifford theory says is that 
\begin{equation*}
\text{ind}\colon B(\lambda_0)^2 \rightarrow {}^\delta B(\lambda_0)^2
\ \text{is a two-to-one surjection,}
\end{equation*}
and
\begin{equation*}
\text{res}\colon {}^\delta B(\lambda_0)^1 \rightarrow B(\lambda_0)^1
\ \text{is a two-to-one surjection.}
\end{equation*}
The fibers of the first surjection are the orbits of the twisting
action by $\theta$, and the fibers of the second surjection are the
orbits of tensoring with the nontrivial character $\xi$ of ${}^\delta
G/G$ (Definition \ref{def:extpair}).

We will also make use of a type one weight as in
\eqref{eq:regularizingphi}:
\begin{equation*}
\phi \in \Lambda_{\text{fin}}(G,H_s)^{t_s}
\end{equation*}
of a finite-dimensional representation of $G$, dominant for $R^+_s$,
and such that 
\begin{equation*}
\lambda_0 + d\phi \text{\ is dominant and regular for $R^+(\lambda_0)$.}
\end{equation*}
Parallel to Corollary \ref{cor:transfunc}(4), we get from Corollary
\ref{cor:exttransfunc}(4) an inclusion
\begin{equation*}
t_{\lambda_0}^{\lambda_0+\phi_+}\colon {}^\delta B(\lambda_0)
\hookrightarrow {}^\delta B(\lambda_0 + d\phi),
\end{equation*}
allowing us (parallel to Theorem \ref{thm:BBpol}) to compute the
Kazhdan-Lusztig multiplicity polynomials $Q$ at the (possibly
singular) infinitesimal character $\lambda_0$
as a subset of those at the (regular) infinitesimal character
$\lambda_0+d\phi$.
\end{subequations}%{se:klv}

With notation established, we turn to computation of the
Kazhdan-Lusztig character polynomials $P_{x,y}$ for ${}^\delta G$ at regular
infinitesimal character. Recall (Definition \ref{def:charpoly}) that
such a polynomial describes certain constituents of the Lie algebra
\begin{subequations}\label{se:indextP}
cohomology of the irreducible representation $J(y)$; the
constituents are related to the standard representation
$I(x)$. We consider four cases separately, according to the types of
$x$ and $y$.  First, assume that
$y=\Psi_{\text{ind}}$ is type two, so that the irreducible representation
$J(y)$ is induced from $J(\Psi)$ on $G$. The Lie algebra
cohomology of this induced representation is easy to compute, and we
get just as in \eqref{eq:Qall2} above (always assuming $\lambda_0$ regular)
\begin{equation}\label{eq:PGammaPsi212} \begin{aligned}
P_{\Gamma_{\text{ind}},\Psi_{\text{ind}}} &= P_{\Gamma,\Psi} +
P_{\Gamma,\Psi^\theta} % \\
= P_{\Gamma^\theta,\Psi} +
P_{\Gamma^\theta,\Psi^\theta}
\quad (\Gamma \in B(\lambda_0)^2,\Psi \in
B(\lambda_0)^2),\\ %\lambda_0 \text{\ regular}),\\
% \end{aligned}
% \end{equation}
% \begin{equation}\label{eq:PGamma1Psi2} %\begin{aligned}
P_{\Gamma_{\pm},\Psi_{\text{ind}}} &= P_{\Gamma,\Psi} % \\ 
= P_{\Gamma,\Psi^\theta} \qquad (\Gamma \in B(\lambda_0)^1, \Psi \in
B(\lambda_0)^2),% \lambda_0 \text{\ regular});
\end{aligned}
\end{equation}
the last equality is a consequence of \eqref{eq:PQtheta}, and the
hypothesis $\Gamma^\theta\simeq \Gamma$.
Next, assume that $y=\Psi_{\pm}$ is type one, so that $J(\Psi_{\pm})$
is an extension to ${}^\delta G$ of the irreducible representation
$J(\Psi)$. If $x=\Gamma_{\text{ind}}$ is type two, then
$P_{\Psi_{\pm},\Gamma_{\text{ind}}}$ counts multiplicities of a
certain induced representation in the Lie algebra cohomology of
$J(\Psi_{\pm})$. By Frobenius reciprocity, such a multiplicity can be
computed after restriction to $G$ (still assuming $\lambda_0$ regular):
\begin{equation}\label{eq:PPsi1Gamma2} % \begin{aligned}
P_{\Gamma_{\text{ind}},\Psi_{\pm}} = P_{\Gamma,\Psi} % \\
=P_{\Gamma^\theta,\Psi} \quad (\Gamma \in B(\lambda_0)^2, \Psi \in
B(\lambda_0)^1). %, \lambda_0 \text{\ regular}).
% \end{aligned}
\end{equation}

We record also the more elementary analogues for the multiplicity
polynomials $Q$, which do not require the hypothesis of regular
infinitesimal character:
\begin{equation*}% \label{eq:QXi2Gamma2}% \begin{aligned}
Q_{\Xi_{\text{ind}},\Gamma_{\text{ind}}} = Q_{\Xi,\Gamma} +
Q_{\Xi,\Gamma^\theta} % \\
= Q_{\Xi^\theta,\Gamma} + Q_{\Xi^\theta,\Gamma^\theta} \quad (\Xi
\in B(\lambda_0)^2,\Gamma \in B(\lambda_0)^2); 
% \end{aligned}
\end{equation*}
\begin{equation*}% \label{eq:QXi1Gamma2}
Q_{\Xi_{\pm},\Gamma_{\text{ind}}} = Q_{\Xi,\Gamma} =
Q_{\Xi,\Gamma^\theta} \qquad (\Xi\in
B(\lambda_0)^1, \Gamma \in B(\lambda_0)^2);
\end{equation*}
\begin{equation*}% \label{eq:QXi2Gamma1}
Q_{\Xi_{\text{ind}},\Gamma_{\pm}} = Q_{\Xi,\Gamma} =
Q_{\Xi^\theta,\Gamma} \qquad (\Xi \in
B(\lambda_0)^2, \Gamma \in B(\lambda_0)^1). 
\end{equation*}
\end{subequations}% se:indextP

In all of these cases---that is, whenever one of the indices is type
two---the Kazhdan-Lusztig polynomials for ${}^\delta G$ are given by
extremely simple formulas in terms of the 
Kazhdan-Lusztig polynomials for $G$. The remaining case is when both
indices are type one. 
\begin{subequations}\label{se:typeoneextLP}
We will therefore be considering parameters
\begin{equation}
\Xi, \Gamma, \Psi \in B(\lambda_0)^1,
\end{equation}
and their extensions
\begin{equation}
\Xi_\pm, \Gamma_\pm, \Psi_\pm  \in {}^\delta B(\lambda_0)^1.
\end{equation}
These pairs (and the corresponding representations) are interchanged
by tensoring with $\xi$ (Definition \ref{def:extpair}): 
\begin{equation*}
\xi\otimes J(\Xi_+) = J(\Xi_-), \qquad \xi\otimes J(\Xi_-) = J(\Xi_+).
\end{equation*}
Using \eqref{eq:Pxi}, we deduce
\begin{equation*}
P_{\Gamma_{\pm},\Psi_-} = P_{\Gamma_{\mp},\Psi_+} \qquad (\Gamma,\Psi
\in B(\lambda_0)^1, \lambda_0\ \text{regular}),
\end{equation*}
and similarly
\begin{equation*}
Q_{\Xi_\pm,\Gamma_-} = Q_{\Xi_{\mp},\Gamma_+} \qquad (\Xi,\Gamma
\in B(\lambda_0)^1).
\end{equation*}
So it is enough to compute all of the polynomials $P$ and $Q$ with the second
index labeled $+$. (The same statement would be true with all four variations on
``first'' and ``second,'' but this is the one we prefer.)

The polynomials $P_{\Gamma_{\pm},\Psi_+}$ describe the occurrence of
certain representations 
of an extended group ${}^\delta L$ (depending on $\Gamma$) in the Lie
algebra cohomology of $J(\Psi_+)$. By restriction to $G$, we see that
\begin{equation*}
P_{\Gamma_+,\Psi_+} + P_{\Gamma_-,\Psi_+} = P_{\Gamma,\Psi} \qquad (\Gamma,\Psi
\in B(\lambda_0)^1, \lambda_0\ \text{regular}),
\end{equation*}
That is, the sum of these two character polynomials for ${}^\delta G$
that we wish to compute is a (known) character polynomial for $G$.  It
is therefore natural to consider the difference
\begin{equation}\label{eq:twcharpoly}
P^\delta_{\Gamma,\Psi} =_{\text{def}} P_{\Gamma_+,\Psi_+} -
  P_{\Gamma_-,\Psi_+}  \qquad (\Gamma,\Psi \in B(\lambda_0)^1,
  \lambda_0\ \text{regular}).
\end{equation}
The polynomial $P^\delta_{\Gamma,\Psi}$ is called a {\em twisted
  character polynomial}. What the preceding formulas say is that {\em
  if we can compute the twisted character polynomials
  $P^\delta_{\Gamma,\Psi}$, we will know the character polynomials
  for ${}^\delta G$.} 

In exactly the same way, the polynomials $Q_{\Xi_{\pm},\Gamma_+}$
describe the occurrence of the the irreducible representations
$J(\Xi_{\pm})$ in the Jantzen filtration of of $I(\Gamma_+)$. By
restriction to $G$, we see that 
\begin{equation*}
Q_{\Xi_+,\Gamma+} + Q_{\Xi_-\Gamma_+} = Q_{\Xi,\Gamma} \qquad (\Xi,\Gamma
\in B(\lambda_0)^1),
\end{equation*}
That is, the sum of these two multiplicity polynomials for ${}^\delta G$
that we wish to compute is a (known) multiplicity polynomial for $G$.  It
is therefore natural to consider the difference
\begin{equation}\label{eq:twmultpoly}
Q^\delta_{\Xi,\Gamma} =_{\text{def}} Q_{\Xi_+\Gamma_+} -
  Q_{\Xi_-\Gamma_+}  \qquad (\Gamma,\Psi \in B(\lambda_0)^1),
\end{equation}
a {\em twisted multiplicity polynomial}. Again what is true is
{\em if we can compute the twisted multiplicity polynomials
  $Q^\delta_{\Gamma,\Psi}$, we will know the multiplicity
  polynomials for ${}^\delta G$.} 
\end{subequations}% se:typeoneextP

Here is the main theorem of this section.

\begin{theorem}[Lusztig-Vogan \Cite{LV12}] \label{thm:twKLpol} Suppose
  $\Gamma, \Psi\in B(\lambda_0)^1$ are type one
  Langlands parameters of regular $\theta$-stable infinitesimal
  character $\lambda_0$. Then the twisted character polynomial
  $P^\delta_{\Gamma,\Psi}$ of \eqref{eq:twcharpoly} is equal to the
  polynomial defined in \Cite{LV12}. In particular, it is a polynomial in
  $q$ with integer coefficients. Each coefficient is bounded in
  absolute value by the corresponding coefficient of
  $P_{\Gamma,\Psi}$, and congruent to that coefficient modulo two.
\end{theorem}

What is actually written in \Cite{LV12} is a statement about
equivariant perverse sheaves on a flag variety in finite
characteristic. By standard base change results, this is equivalent to
a statement about ${}^\delta K({\mathbb C})$-equivariant perverse
sheaves on the flag variety of Borel subalgebras of ${\mathfrak
  g}$. By the Beilinson-Bernstein localization theory, with
appropriate minor modifications as in \Cite{IC3}, this becomes a
statement about twisted character polynomials for $({\mathfrak
  g},{}^\delta K({\mathbb C}))$-modules, which is precisely the
theorem above in case $\lambda_0 = \rho$, the half sum of a full set
of positive roots. 

The case of non-integral infinitesimal character can be
treated just as for connected groups, using the ideas of Beilinson,
Bernstein, Brylinski, Kashiwara, and Lusztig as sketched in Chapter 17
of \Cite{ABV}.  (As far as we know, the ideas have been written down by
their originators only in \Cite{LRed}*{Chapter 1}, which
explains the case of Verma modules.)
\section{Kazhdan-Lusztig polynomials for $c$-invariant
  forms}\label{sec:klvform} 
\setcounter{equation}{0}

In this section we fix a weight $\lambda_0 \in {\mathfrak h}_s^*$
representing a {\em real} infinitesimal character (that is, belonging
to the real span of the algebraic weight lattice $X^*$) as in
% \begin{subequations}\label{se:klvforms}
Definition \ref{def:tauinv}.  In \eqref{se:cinvtsigformulas}, we wrote
finite matrices over ${\mathbb W}$  (Definition \ref{def:witt})
\begin{equation} \label{se:klvforms}
\left(W^c_{\Gamma,\Psi}\right), \quad \left(w^c_{\Xi,\Gamma}\right) \quad
({\Gamma,\Psi \in B(\lambda_0)})
\end{equation} 
describing the signatures of $c$-invariant forms on irreducible
representations in terms of those on standard representations.  The
ring ${\mathbb W} = {\mathbb Z} + s{\mathbb Z}$ comes equipped with a
homomorphism (Definition \ref{def:witt}
\begin{equation*}
\for\colon {\mathbb W} \rightarrow {\mathbb Z},\qquad p+qs \mapsto p+q,
\end{equation*}
and these signature matrices are related to the multiplicity matrices
by
\begin{equation*}
\for(w^c_{\Xi,\Gamma}) = m_{\Xi,\Gamma}, \qquad \for(W^c_{\Gamma,\Psi}) =
M_{\Gamma,\Psi}. 
\end{equation*}

In sections \ref{sec:klv} and \ref{sec:extklv} we recalled the ideas
of Kazhdan and Lusztig that led to computations of the multiplicity
matrices $m$ and $M$: introducing appropriate
$q$-analogues $Q$ and $P$, and then computing those $q$-analogues
using recursion relations that do not make sense at $q=1$. In the
present section we will extend this to a computation of the signature
matrices. We will need all this for extended groups. We will write
only about $G$ to keep the notation under control, making comments as
we go about the case of ${}^\delta G$.
% \end{subequations} %se:klvforms

The first issue is to find $q$-analogues of $w^c$ and $W^c$. The first is
easy.
\begin{definition}\label{def:wpoly} In the setting of
  \eqref{se:klvforms}, suppose $\Gamma$ and $\Xi$
  are Langlands parameters in $B(\lambda_0)$. Recall from
  \eqref{eq:jantzenr} the finite decreasing filtration
$$I_{\quo}(\Gamma) = I_{\quo}(\Gamma)^0 \supset
    I_{\quo}(\Gamma)^1 \supset I_{\quo}(\Gamma)^2
    \cdots,$$
and the nondegenerate Jantzen forms $\langle,\rangle^{[r]}$ on the
subquotients of the filtration. Using Proposition
\ref{prop:hermgrothgens}, write in the Hermitian Grothendieck group
$$[I_{\quo}(\Gamma)^r/I_{\quo}(\Gamma)^{r+1},\langle,\rangle^{[r]}]
= \sum_{\Xi \in B(\lambda_0)} w^{c,r}_{\Xi,\Gamma}
[J(\Xi),\langle,\rangle^c_{J(\Xi),b}],$$
with $w^{c,r}_{\Xi,\Gamma} \in {\mathbb W}$. Here the ``base''
$c$-invariant forms $\langle,\rangle^c_{J(\Xi),b}$ are the canonical
ones described in \eqref{eq:cancform}. We define the {\em
  signature multiplicity polynomial}
$$Q^c_{\Xi,\Gamma} = q^{(\ell(\Gamma) - \ell(\Xi)/2)}\sum_{r=0}^\infty
  w^{c,r}_{\Xi,\Gamma} q^{-r/2} \in {\mathbb W}[q^{1/2},q^{-1/2}].$$
The value at $q^{1/2} = 1$ of this Laurent
polynomial is the signature multiplicity of $J(\Xi)$ in $I(\Gamma)$
defined in \eqref{se:cinvtsigformulas}:
$$Q^c_{\Xi,\Gamma}(1) = w^c_{\Xi,\Gamma}.$$
Applying the natural forgetful homomorphism
$$\for\colon {\mathbb W}[q^{1/2},q^{-1/2}] \rightarrow {\mathbb
  Z}[q^{1/2},q^{-1/2}]$$
clearly (by Definition \ref{def:multpoly}) gives
$$\for(Q^c_{\Xi,\Gamma}) = Q_{\Xi,\Gamma}.$$

This definition applies to the extended group ${}^\delta G$
without change.
\end{definition}

\begin{proposition}\label{prop:wpoly} Suppose we are in the setting of
  Definition \ref{def:wpoly}.
\begin{enumerate}
\item The signature multiplicity
  polynomial $Q^c_{\Xi,\Gamma}$ is a polynomial in ${\mathbb W}[q]$ of
  degree less than or equal to $[\ell(\Gamma) - \ell(\Xi)]/2$. Its
  coefficients are nonnegative, in the sense that if $p+qs$ is a
  coefficient, then $p$ and $q$ belong to ${\mathbb N}$.
\item The diagonal entries of the matrix $Q^c$ are equal to one:
  $Q^c_{\Gamma,\Gamma}=1$.
\item The polynomial $Q^c_{\Xi,\Gamma}$ can be nonzero only if $\Xi
  \le \Gamma$ in the Bruhat order of Definition \ref{def:singbruhat}.
\item Equality can hold in the degree estimate of (1) only if
  $\Gamma=\Xi$. Equivalently,
$$\deg Q^c_{\Xi,\Gamma} \le [\ell(\Gamma) - \ell(\Xi) -1]/2, \qquad
\text{all $\Xi < \Gamma$.}$$
\item The matrix $Q^c$ is upper triangular with diagonal entries $1$,
  and so is invertible.
\end{enumerate}
Parallel results hold for ${}^\delta G$.
\end{proposition}
\begin{proof}
The nonnegativity of the coefficients is clear from the definition.
Because of the nonnegativity and the last formula of Definition
\ref{def:wpoly}, $Q^c_{\Xi,\Gamma}$ and $Q_{\Xi,\Gamma}$ have exactly
the same set of nonvanishing coefficients. Because $Q_{\Xi,\Gamma}$ is
known (by Beilinson-Bernstein's Theorem \ref{thm:BBpol}) to be a
polynomial with a degree bound, the same conclusion follows for
$Q^c_{\Xi,\Gamma}$.
\end{proof}

\begin{definition}\label{def:Wpoly} In the setting of
  \eqref{se:klvforms} and Definition \ref{def:wpoly}, consider the
  matrix $Q^c$ (with entries indexed by $B(\lambda_0)$) of polynomials
  in ${\mathbb W}[q]$. The {\em signature character polynomial} is
$$P^c_{\Gamma,\Psi} = (-1)^{\ell(\Psi) -
  \ell(\Gamma)}[\text{$(\Gamma,\Psi)$ entry of $(Q^c)^{-1}$}].$$
By Proposition \ref{prop:wpoly}, $P^c_{\Gamma,\Psi}$ is a polynomial
in $q$ (with coefficients in ${\mathbb W}$) of degree at most
$[\ell(\Psi) - \ell(\Gamma)]/2$, with 
equality only if $\Psi=\Gamma$.  The value at $q = 1$ of this
polynomial is the signature coefficient of $I(\Gamma)$ in the
irreducible $J(\Psi)$ defined in \eqref{eq:formformirr}:
$$P^c_{\Gamma,\Psi}(1) = (-1)^{\ell(\Psi) - \ell(\Gamma)} W^c_{\Gamma,\Psi}.$$ 
Applying the natural forgetful homomorphism
$$\for\colon {\mathbb W}[q] \rightarrow {\mathbb
  Z}[q]$$
gives
$$\for(P^c_{\Gamma,\Psi}) = P_{\Gamma,\Psi}.$$ 

This definition applies to the extended group ${}^\delta G$
without change.
\end{definition}

It would certainly be good to find a direct definition of
$P^c_{\Gamma,\Psi}$ parallel to the definition of $P_{\Gamma,\Psi}$ in
Definition \ref{def:charpoly}, perhaps in terms of some natural
invariant form on Lie algebra cohomology. But we do not know how to do
that.

With $q$-analogues of the signature matrices in hand, we can ask how
the ideas of Kazhdan and Lusztig help in computing
them. Kazhdan-Lusztig theory involves in a fundamental way the notion
of length recalled in Definition \ref{def:length}, which involves only
the integral roots.  The behavior of invariant forms is also affected
by roots which ``were'' integral at some deformed value of a
continuous parameter. The following version of length looks at those
``previously integral'' nonintegral roots.

\begin{definition}\label{def:ornumber} Suppose $\Gamma =
  (\Lambda,\nu)$ is a Langlands parameter (\ref{def:dLP}). Write
  $\gamma$ for the differential of $\Gamma$, a weight defining the
  infinitesimal character. The {\em
    orientation number $\ell_o(\Gamma)$} is equal to the sum of
\begin{enumerate}[label=\alph*)]
\item the number of pairs $(\alpha,-\theta\alpha)$ of complex
  nonintegral roots such that
$$\langle\gamma,\alpha^\vee\rangle > 0, \qquad
\langle\gamma,-\theta\alpha^\vee\rangle > 0;$$
and
\item the number of real nonintegral roots $\beta$ such that
  $\langle\gamma,\beta^\vee\rangle > 0$, and the integral part
  $[\langle\gamma,\beta^\vee\rangle]$ is {\em even} if
  $\Lambda_{\mathfrak q}(m_\alpha) = +1$, or {\em odd} if
  $\Lambda_{\mathfrak q}(m_\alpha) = -1$ (notation \eqref{eq:gammaq1}).
\end{enumerate}
% For principal series for $SL(2,{\mathbb R})$, the two cases in
\end{definition}

\begin{theorem}\label{thm:KLsigpolQ} Suppose $\Gamma$ and $\Xi$ are
  Langlands parameters of (real) infinitesimal character
  $\chi$. Then the signature multiplicity polynomial (Definition
  \ref{def:wpoly}) is
$$Q^c_{\Xi,\Gamma} = s^{(\ell_o(\Xi) - \ell_o(\Gamma))/2}
Q_{\Xi,\Gamma}(sq).$$
In particular
\begin{enumerate}
\item if $\Xi$ and $\Gamma$ belong to a common block, then
  $\ell_o(\Xi) \equiv \ell_o(\Gamma) \pmod{2}$;
\item if $J(\Xi)$ appears in level $r$ of the Jantzen filtration of
  $I_{\quo}(\Gamma)$, then
\begin{enumerate}
\item $r \equiv \ell(\Gamma) - \ell(\Xi) \pmod{2}$; and
\item in terms of the canonical form \eqref{eq:cancform}, the $r$th
  Jantzen form restricted to $J(\Xi)$ is  
  $$(-1)^{[(\ell(\Gamma) - \ell(\Xi)-r)+(\ell_o(\Gamma) -
    % PREVIOUSLY FORGOT -r
   \ell_o(\Xi))]/2}\langle,\rangle^c_{J(\Xi),b};$$
\end{enumerate}
\item if $\ell_o(\Gamma) - \ell_o(\Xi) \equiv 0 \pmod{4}$, then the
  signature multiplicity coefficient of \eqref{eq:formformstd} is
$$w^c_{\Xi,\Gamma} = (\text{sum of even coeffs of
  $Q_{\Xi,\Gamma}$}) + s(\text{sum of odd coeffs of
  $Q_{\Xi,\Gamma}$});$$
\item if $\ell_o(\Gamma) - \ell_o(\Xi) \equiv 2 \pmod{4}$, then the
  signature multiplicity coefficient is
$$w^c_{\Xi,\Gamma} = (\text{sum of odd coeffs of
  $Q_{\Xi,\Gamma}$}) + s(\text{sum of even coeffs of
  $Q_{\Xi,\Gamma}$}).$$
\end{enumerate}
Exactly the same results hold for the extended group ${}^\delta G$.
\end{theorem}

%%DV SKETCH PROOF HERE!
\begin{proof}%[Sketch of proof]
Of course the proof is modeled on the proof of the Kazhdan-Lusztig
conjectures in \Cite{IC3} and \Cite{LV83}. Part (1) is a fairly easy
combinatorial statement, for which we omit the proof. (The proof of 2(b) will implicitly prove (1) at least for the case that $\Xi$ and $\Gamma$ are comparable in the Bruhat order.) Part 2(a) is a restatement of Theorem
\ref{thm:BBpol} 1(b), recorded here only as a reminder that the
powers of $-1$ appearing below are integer powers.  Parts 3 and 4
are just restatements of 2(b); so what requires proof is 2(b).

The first step is to reduce to the case of regular infinitesimal
character. For this, fix a weight $\lambda_0 \in {\mathfrak h}_s^*$
representing $\xi$ as in \eqref{se:exttransfunc}.  Choose $\phi_\pm$
as in \eqref{se:exttransfunc}. Here is the statement about translation
functors that we need. % , ensuring that $\lambda_0$ and 
% $\lambda_0+d\phi_\pm$ are real dominant as in \eqref{eq:extrealdom}.
% NO don't assume that; need the first part of the lemma to cross
% nonintegral walls.

\begin{lemma}\label{lemma:transforms} In the setting of
  \eqref{se:exttransfunc}, the finite-dimensional representation
  $F_{\phi_+}$ admits a positive-definite $c$-invariant Hermitian
    form that is unique up to positive scalar multiple. Consequently
    the translation functors
$$ T^{\lambda_0 +\phi_+}_{\lambda_0}, \qquad T_{\lambda_0
    +\phi_+}^{\lambda_0} $$
of \eqref{eq:exttransfunc} carry (nondegenerate) $c$-invariant Hermitian
forms to (nondegenerate) $c$-invariant Hermitian forms.

Suppose in addition that the real dominance hypotheses of \eqref{eq:sdominant}
and \eqref{eq:extrealdom} are satisfied.  Suppose $\Gamma_0 \in {}^\delta
B(\lambda_0)$ is any (singular) Langlands parameter, and
$$\Gamma_1 = t_{\lambda_0}^{\lambda_0+d\phi}(\Gamma_0) {}^\delta
B(\lambda_0 + d\phi)$$
is the corresponding (regular) parameter (Theorem
\ref{cor:exttransfunc}(4)). Then
\begin{enumerate}
\item $T_{\lambda_0 +\phi_+}^{\lambda_0}$ carries the canonical
  $c$-invariant Hermitian form on $I_{\quo}(\Gamma_1)$ to a positive
  multiple of the canonical $c$-invariant Hermitian form on
  $I_{\quo}(\Gamma_0)$;
\item $T_{\lambda_0 +\phi_+}^{\lambda_0}$ carries the canonical
  $c$-invariant Hermitian form on $J(\Gamma_1)$ to a positive
  multiple of the canonical $c$-invariant Hermitian form on
  $J(\Gamma_0)$;
\item $T_{\lambda_0 +\phi_+}^{\lambda_0}$ carries the Jantzen
  filtration on $I_{\quo}(\Gamma_1)$ to that on % the Jantzen filtration on
  $I_{\quo}(\Gamma_0)$; and
\item  $T_{\lambda_0 +\phi_+}^{\lambda_0}$ carries the Jantzen
  form on each level of $I_{\quo}(\Gamma_1)$ (Definition
  \ref{def:jantzenform} to a positive multiple of the Jantzen  form on
    the corresponding level of $I_{\quo}(\Gamma_0)$.

\end{enumerate}
\end{lemma}
\begin{proof} %[Sketch of proof]
The first assertion of the lemma is just Weyl's ``unitary trick:'' the
$c$-invariant form on $F_{\phi_+}$ is by definition preserved by a
compact form of $G$, and therefore must be definite. The computation
of the sign in (1) trivially reduces from $^\delta G$ to $G$. There (in
light of Proposition \ref{prop:cdualstd}) it
depends on looking carefully at how lowest $K$-types are affected by
translation functors. (This is where the ``real dominance'' hypotheses
are needed.) Part (2) follows immediately.

For (3), recall how the Jantzen filtration is defined in Proposition
\ref{prop:jantzenfilt}. When the regular parameter
$\Gamma_1=(\Lambda_1,\nu_1)$ is modified by a small dilation of the
continuous parameter, all the real dominance hypotheses are
preserved. Write $\nu_{1,s} \in {\mathfrak h}_s^*$ for the weight
corresponding to $\nu_1$ under the identification $i(R^+_s,R^+)$ of
Proposition \ref{prop:findimlHswts}. The one-parameter family of
translation functors
\begin{equation*}
T_{(\lambda_0 +\epsilon\nu_{1,s}) +\phi_+}^{(\lambda_0+\epsilon\nu_{1,s})}
\end{equation*}
(all of which use the same finite-dimensional representation; only the
infinitesimal character projection changes) all carry irreducibles to
irreducibles or zero for small $\epsilon$.  These functors carry
$I_{\quo}(\Lambda_1,(1+\epsilon)\nu_1)$---the one-parameter family
defining the Jantzen filtration of $I_{\quo}(\Gamma_1)$---to
$I_{\quo}(\Lambda_0,\nu_0 +\epsilon\nu_1)$. Although this second
family is not precisely the one used to define the Jantzen filtration
for $I_{\quo}(\Gamma_0)$, it is close enough that the filtration is
the same. (This independence-of-deformation is stated by Beilinson and
Bernstein in \Cite{BB}*{Section 4.3}, and attributed to Barbasch
\Cite{BFilt}. We have not been able to locate the result in
\Cite{BFilt}. but in any case the proof is not difficult.) This proves
(3).

Parts (1) and (3) immediately imply (4).
%% PROBLEM: Gamma_1 = even integral principal series for SL(2); corr
%% dilation of Gamma_0?? Seems OK?
\end{proof}

% \begin{subequations}\label{se:proofKLsigpolQ}
We return now to the proof of Theorem \ref{thm:KLsigpolQ}.
The lemma allows us to deduce all of the assertions of the theorem at
possibly singular infinitesimal character $\lambda_0$ by applying the
translation functor $T_{\lambda_0 +\phi_+}^{\lambda_0}$ to the
assertions at the regular infinitesimal character
$\lambda_0+d\phi$. Just as in the proof of (1) of the lemma, all the
assertions about signature reduce immediately from the extended group
to the case of $G$; so we will make no further mention of ${}^\delta G$.

So we may assume henceforth that
\begin{equation*} \text{the infinitesimal character $\lambda_0$ is regular.}
\end{equation*}
Writing $H$ for the maximal torus underlying the Langlands parameter
$\Gamma = (\Lambda,\nu)$, we get a unique positive system
\begin{equation*} R^+_{\text{full}}(\Gamma) = R^+({\mathfrak g}, {\mathfrak
    h})_{\text{full}}(\Gamma) \end{equation*}
making $(d\Lambda,\nu)$ real dominant, and therefore a corresponding
Borel subalgebra ${\mathfrak b}(\Gamma)$ (Definition \ref{def:HChom}).
We proceed by increasing induction on the dimension of the $K({\mathbb
  C})$ orbit of ${\mathfrak b}(\Gamma)$. (The original Kazhdan-Lusztig
conjecture Theorem \ref{thm:KLpol}, in the case of integral
infinitesimal character, concerns the intersection cohomology complex
on the closure of $K({\mathbb C})\cdot {\mathfrak b}(\Gamma)$, coming
from a local system corresponding to $\Gamma$. The proof in that case
proceeds by induction on the dimension of this orbit, along exactly
the lines we are now following.)

Suppose first that there is a complex simple root $\alpha_\Gamma$ for
$R^+_{\text{full}}(\Gamma)$, such that
\begin{equation}\label{eq:cplxnonint}
\theta\alpha_\Gamma\notin R^+_{\text{full}}(\Gamma), \qquad
\alpha_\Gamma \notin R(\Gamma).
\end{equation}
That is, $\alpha_\Gamma$ is assumed to be {\em nonintegral}. Write $\alpha$
for the corresponding simple root in $R^+_s$ 
(Proposition \ref{prop:findimlHswts}).  In the setting of
\eqref{se:transfunc}, we choose the weight $\phi$ {\em not} to make
$\lambda_0 +d\phi$ real dominant as in \eqref{eq:transrealdom}, but
instead to put it in the adjacent Weyl chamber:
\begin{equation}\label{eq:noninttrans}
\lambda_0+d\phi \text{\ is regular and real dominant for
  $s_{\alpha}(R^+_s)$.} \end{equation}
Translation across this nonintegral wall is an equivalence of
categories by Corollary \ref{cor:transfunc}. The translated Langlands parameter
$\Gamma' = t_{\lambda_0}^{\lambda_0+\phi}(\Gamma)$ is real dominant for
$s_{\alpha_\Gamma}(R^+_{\text{full}}(\Gamma)$, and therefore (by
\eqref{eq:cplxnonint}) the corresponding orbit of Borel subalgebras
has dimension one less; so the statements of 2(b) are available by
inductive hypothesis for $\Gamma'$. In order to use them, we need to
know how the translation functor $T_{\lambda_0}^{\lambda_0+\phi}$
affects orientation numbers and canonical Hermitian forms.

The effect on the orientation number of $\Xi$ (Definition
\ref{def:ornumber}) is fairly easy. If the
simple root $\alpha_\Xi$ is complex, then $\ell_o(\Xi)$
increases by one if $\theta\alpha_\Xi >0$, and decreases by one if
$\theta\alpha_\Xi < 0$. If $\alpha_\Xi$ is real, then $\ell_0(\Xi)$
increases by one if 
$$(-1)^{[\langle\lambda_0,\alpha_s^\vee\rangle]} = \Xi_{\mathfrak
  q}(m_{\alpha_\Xi}),$$
and decreases by one otherwise. (The point for this second fact is
that if $t\in {\mathbb R} - {\mathbb Z}$, then the integer parts
$[t]$ and $[-t]$ (meaning ``greatest integer not exceeding'') have
opposite parity.  In order to use the inductive
hypothesis, we need to verify (still assuming that we are crossing a
single nonintegral wall as in \eqref{eq:noninttrans}) that the
translation functor $T_{\lambda_0}^{\lambda_0+\phi}$ {\em does not} change
the sign of the canonical $c$-invariant form in case the orientation
number increases by one, and {\em does} change the sign in case the
orientation number decreases by one. This is another version of the
lowest $K$-type calculation used for Lemma
\ref{lemma:transforms}(1). The most difficult case is that of real
$\alpha_\Xi$, but that case can be reduced fairly easily to
$SL(2,{\mathbb R})$.

Suppose next that there is a complex simple root $\alpha_\Gamma$ for
$R^+_{\text{full}}(\Gamma)$, such that
\begin{equation}\label{eq:cplxint}
\theta\alpha \notin R^+_{\text{full}}(\Gamma), \qquad \alpha\in R(\Gamma).
\end{equation}
That is, $\alpha_\Gamma$ is assumed to be {\em integral}. Write $\alpha$
for the corresponding simple root in $R^+_s$. (This case is the most
serious one.) Write
\begin{equation*}
m = \langle d\gamma,\alpha_\Gamma^\vee \rangle,
\end{equation*}
a strictly positive integer.  There is a Langlands parameter
\begin{equation*}
\Gamma' = \Gamma-m\alpha_\Gamma
\end{equation*}
corresponding to the positive root system
$$s_{\alpha_\Gamma}(R^+_{\text{full}}(\Gamma)) =
R^+_{\text{full}}(\Gamma) - \{\alpha_\Gamma\} \cup
\{-\alpha_\Gamma\}.$$
The $K({\mathbb C})$-orbit of the corresponding Borel subalgebra has
dimension one less than that for $\Gamma$, so the inductive hypothesis
applies to $\Gamma'$. 

We will not recall all of the machinery of wall-crossing translation
functors developed in \cites{IC1, IC2}.  The central idea is to use
these functors to construct from $J(\Gamma')$ a new module
\begin{equation*}
U_\alpha(J(\Gamma')).
\end{equation*}
The composition factors of $U_\alpha(J(\Gamma'))$ consist of
$J(\Gamma)$ (with multiplicity one) and certain composition factors of
$I_{\quo}(\Gamma')$.  The construction provides a nondegenerate
$c$-invariant form on $U_\alpha(J(\Gamma'))$, and what we need to do
is understand that form.  The key calculation is that when we apply
translation across the wall to the standard module
$I_{\quo}(\Gamma')$, obtaining a module $M$ with a short exact
sequence
$$0 \rightarrow I_{\quo}(\Gamma) \rightarrow M \rightarrow
I_{\quo}(\Gamma') \rightarrow 0,$$
then the induced form on $M$ restricts to a {\em positive} multiple of
the canonical form on $I_{\quo}(\Gamma)$.  (Just as for Lemma
\ref{lemma:transforms}(1), this amounts to a careful examination of
lowest $K$-types.) From this it follows that the nondegenerate form on
$U_\alpha(J(\Gamma'))$ restricts to a positive multiple of the
canonical form on $J(\Gamma)$.

Essentially \Cite{LV83} uses the decomposition theorem for perverse
sheaves to deduce that
$U_\alpha(J(\Gamma'))$ is completely reducible. Using this complete
reducibility, everything else about the form on $U_\alpha(J(\Gamma'))$
follows by relating $U_\alpha$ to the first level of the Jantzen
filtration of $I_{\quo}(\Gamma')$, and applying the inductive
hypothesis.
%% DV should say more about this!

The remaining possibility is that for {\em every} complex simple root
$\alpha_\Gamma$, $\theta\alpha_\Gamma$ is also positive.  From this it
follows that
\begin{equation} \label{eq:thetastable} \text{the set of nonreal roots in
    $R^+_{\text{full}}(\Gamma)$ is $\theta$-stable.}
\end{equation}
Writing ${\mathfrak u} = {\mathfrak u}(\Gamma)$ for the span of the
root spaces of these nonreal positive roots, we get as in
\eqref{se:stdVZparam} a $\theta$-stable parabolic subalgebra
\begin{equation*}
{\mathfrak q} = {\mathfrak l} + {\mathfrak u} = {\mathfrak l}(\Gamma)
+ {\mathfrak u}(\Gamma)
\end{equation*}
satisfying the (very strong) positivity condition
\begin{equation*}
\langle d\gamma,\beta^\vee \rangle > 0 \qquad (\beta \in
\Delta({\mathfrak u},{\mathfrak h})).
\end{equation*}
This condition guarantees that cohomological induction from $L$ by
${\mathfrak q}$ is an exact functor carrying irreducibles to
irreducibles, and respecting $c$-invariant Hermitian forms (Theorem
10.9). Because Langlands parameters behave in a simple way under this
induction, we see that length and orientation number are preserved; so
the assertions in 2(b) of the theorem are reduced to the subgroup
$L$. That is, we may assume that 
\begin{equation*}
\text{the Cartan subgroup $H$ for $\Gamma$ is split;}
\end{equation*}
that is, that $I_{\quo}(\Gamma)$ is a minimal principal series
representation for a split group.  If there are no real roots
satisfying the parity condition \eqref{se:PSbruhat} then $I(\Gamma)$
is irreducible, and the assertions are easy. So suppose that there are
real roots satisfying the parity condition. We proceed by induction on
the smallest height of such a root.  If the height is equal to one,
there is a simple root (necessarily integral) satisfying the parity
condition. In this case we argue essentially as in the case
\eqref{eq:cplxint} above. If the height is greater than one, we can choose
a nonintegral simple root $\alpha_\Gamma$ such that the simple
reflection $s_{\alpha_\Gamma}$ decreases the height.  In this case we
can translate across the nonintegral wall for $\alpha_\Gamma$ as in
the case \eqref{eq:cplxnonint} above.

% \end{subequations} %se:proofKLsigpolQ
\end{proof}

\begin{corollary}\label{cor:KLsigpolP} Suppose $\Gamma$ and $\Psi$ are
  Langlands parameters of (real) infinitesimal character
  $\chi$. Then the signature character polynomial (Definition
  \ref{def:Wpoly}) is
$$P^c_{\Gamma,\Psi} = s^{(\ell_o(\Psi) - \ell_o(\Gamma))/2}
P_{\Gamma,\Psi}(sq).$$
In particular
\begin{enumerate}
\item if $\ell_o(\Psi) - \ell_o(\Gamma) \equiv 0 \pmod{4}$, the
  signature character coefficient of \eqref{eq:formformirr} is
$$(-1)^{\ell(\Psi) - \ell(\Gamma)}W^c_{\Gamma,\Psi} = (\text{sum of even coeffs of
  $P_{\Gamma,\Psi}$}) + s(\text{sum of odd coeffs of
  $P_{\Gamma,\Psi}$});$$
\item if $\ell_o(\Psi) - \ell_o(\Gamma) \equiv 2 \pmod{4}$, the
  signature character coefficient is
$$(-1)^{\ell(\Psi) - \ell(\Gamma)}W^c_{\Gamma,\Psi} = (\text{sum of odd coeffs of
  $P_{\Gamma,\Psi}$}) + s(\text{sum of even coeffs of
  $P_{\Gamma,\Psi}$}).$$
\end{enumerate}
The same result holds for the extended group ${}^\delta G$.
\end{corollary}

Here are representation-theoretic translations of these results.

\begin{corollary}\label{cor:KLsig} Suppose $\Gamma$ and $\Psi$ are
  Langlands parameters of (real) infinitesimal character
  $\chi$.

\begin{enumerate}
\item The signature function for the $c$-invariant form
  on the irreducible representation $J(\Psi)$ (see \eqref{eq:sigchar})
  is
$$\sig^c_{J(\Psi)} = \sum_{\Gamma\in B(\chi)} s^{(\ell_o(\Psi) -
    \ell_o(\Gamma))/2} (-1)^{\ell(\Psi) - \ell(\Gamma)}
  P_{\Gamma,\Psi}(s) \sig^c_{I(\Gamma)}.$$
\item The signature function for the
  Jantzen $c$-invariant form (Proposition \ref{prop:jantzenfilt}) on
  the standard representation  $I(\Gamma)$ is
$$\sig^c_{I(\Gamma)} = \sum_{\Xi\in B(\chi)}s^{(\ell_o(\Xi) -
    \ell_o(\Gamma))/2} Q_{\Xi,\Gamma}(s) \sig^c_{J(\Xi)}.$$
  % ADDED /2 in exponent of s per JDA 11/24/16
\end{enumerate}
The same results hold for the extended group ${}^\delta G$.
\end{corollary}
\begin{proof}
Part (1) is \eqref{eq:sigformirr}, together with the formula in
Corollary \ref{cor:KLsigpolP} for the signature character matrix. Part
(2) is Corollary \ref{cor:sigchange}(2), together with the formula in
Theorem \ref{thm:KLsigpolQ}.
\end{proof}

\section{Deformation to $\nu=0$}\label{sec:defto0}
\setcounter{equation}{0}
Fix a (possibly singular) real infinitesimal character $\chi$, and the
finite set $B(\chi)$ of Langlands parameters of infinitesimal
character $\chi$ as in \eqref{se:charformulas}. What we have achieved
in Corollary \ref{cor:KLsig} is the
computability of the  formulas 
\eqref{eq:sigformirr} expressing the signature character of each
irreducible representation $J(\Psi)$ as a ${\mathbb
  W}$-combination of signature characters of standard representations:
\begin{equation*}
\sig^c_{J(\Psi)} = \sum_{\Gamma\in B(\chi)} W^c_{\Gamma,\Psi} \sig^c_{I(\Gamma)}.
\end{equation*}
That is, we have computed the coefficients $W^c_{\Gamma,\Psi}$ in
terms of the (known) Kazhdan-Lusztig polynomials for $G$.

In this section we will describe how to compute the signature of each
standard representation $I(\Gamma)$ appearing on the right side. As in
Section \ref{sec:langlands}, we write
\begin{subequations}\label{se:defparams}
\begin{equation}
H = TA, \qquad G^A = MA
\end{equation}
for the Langlands decomposition of the Cartan subgroup and for the
centralizer of its vector part, and
\begin{equation}
\Gamma = (\Lambda,\nu)
\end{equation}
for the corresponding decomposition of the parameter; $\Lambda$
determines a limit of discrete series $I^M(\Lambda)$ for $M$, and
$\nu$ may be identified with a character $e^\nu$ of $A$.  We also write
\begin{equation}
d\gamma = (d\Lambda,\nu) = (\lambda,\nu) \in {\mathfrak t}^* +
{\mathfrak a}^* = {\mathfrak h}^*
\end{equation}
as in Theorem \ref{thm:cohfam}; the weight $d\gamma$ is a
representative of the infinitesimal character $\chi$. We will also
need the parameters $\Gamma_t$ introduced in \eqref{eq:Gammat}.
\end{subequations}

What makes the signature of $I(\Gamma)$ complicated is the fact that
$\nu$ is not zero.  Here is how to fix that.
\begin{theorem}\label{thm:deftozero}
For each Langlands parameter $\Gamma = (\Lambda,\nu)$ as above, there
is a computable and unique finite formula
$$\sig^c_{I(\Gamma)} = \sig^c_{I(\Lambda,0)}\, +\;
(s-1)\mspace{-9mu}\sum_{\Lambda' \in
  \Pi_{\text{fin,disc}}(G)} \mspace{-9mu} M_{\Lambda',\Gamma}
\sig^c_{I(\Lambda',0)} \qquad  (M_{\Lambda',\Gamma} \in {\mathbb
  Z}).$$
Here final discrete parameters are defined in Definition
\ref{def:dLP}. Every $\Lambda'$ appearing in the sum satisfies the
bound
$$|d\Lambda'|^2 < |(d\Lambda,\nu)|^2.$$
This bound determines a finite subset of $\Pi_{\text{fin,disc}}$.

An exactly parallel result holds for the extended group ${}^\delta G$.
\end{theorem}
The term $I(\Lambda,0)$ on the right need not be
final, but one can always use Hecht-Schmid character identities to
write it in terms of final parameters.  If the set of lowest $K$-types
is
$${\mathcal L}(\Lambda) = \{\mu_1,\ldots,\mu_r\} \subset
\widehat{K},$$
and
$$\mu_i = \mu(\Lambda_i) \qquad (\Lambda_i \in
\Pi_{\text{fin,disc}}(G))$$
(Proposition \ref{prop:LCshape}), then
$$I(\Lambda,0) = \sum_{i=1}^r I(\Lambda_i,0).$$
This identity is true also for the natural $c$-invariant Hermitian
forms (which were all normalized to be positive on the lowest $K$-types).
In this way we get a signature formula entirely in terms of final
discrete parameters. 

Before proving Theorem \ref{thm:deftozero}, we record what it tells us
about signature formulas.

\begin{corollary}\label{cor:csigform} Suppose $\Psi = (\Lambda,\nu)$
  is a Langlands parameter of real infinitesimal character as in
  \eqref{se:defparams}. Then there is a computable and unique finite
  formula
$$\sig^c_{J(\Psi)} = \sig^c_{I(\Lambda,0)} + \sum_{\Lambda' \in
  \Pi_{\text{fin,disc}}(G)} A_{\Lambda',\Psi}\sig^c_{I(\Lambda',0)}.$$
The coefficients $A_{\Lambda',\Psi}$ belong to the signature ring
${\mathbb W}$ of Definition \ref{def:witt}. Final discrete parameters
are defined in Definition \ref{def:dLP}. Every $\Lambda'$ appearing in
the sum satisfies the bound
$$|d\Lambda'|^2 < |(d\Lambda,\nu)|^2.$$
This bound determines a finite subset of $\Pi_{\text{fin,disc}}$.

An exactly parallel result holds for the extended group ${}^\delta G$.
\end{corollary}
\begin{proof} We begin with the formula of Corollary
\ref{cor:KLsig}(1) (one large Kazhdan-Lusztig computation), and then
apply Theorem \ref{thm:deftozero} to each 
term on the right (a long chain of nested Kazhdan-Lusztig
computations). \end{proof}

Now we can apply Theorem \ref{thm:ctoinv} and get a statement about
signatures of invariant Hermitian forms.

\begin{corollary}\label{cor:sigform} Suppose $G$ is a real reductive
  group with extended group ${}^\delta G$ (Definition
  \ref{def:extgrp}). Fix a strong involution $x$ for $G$, an
  eigenvalue $\zeta$ for the central element $z=x^2$, and a square
  root $\zeta^{1/2}$ for $\zeta$ as in Definition
  \ref{def:centralstuff}.

Suppose $\Psi = (\Lambda,\nu)$ is a Langlands parameter for ${}^\delta
G$ of real infinitesimal character, in which the central element $z$
acts by the scalar $\zeta$.
\begin{enumerate}
\item There is a computable and unique finite
  formula
$$\sig^c_{J(\Psi)} = \sum_{\Lambda' \in
  \Pi_{\text{fin,disc}}({}^\delta G)}
B_{\Lambda',\Psi}\sig^c_{I(\Lambda',0)},$$ 
with coefficients $B_{\Lambda',\Psi} \in {\mathbb W}$ (Definition
\ref{def:witt}).
\item The representation $J(\Psi)$ admits an invariant Hermitian form
$$\langle v, w\rangle^0_{J(\Psi)} = \zeta^{-1/2}\langle x\cdot
v,w\rangle^c_{J(\Psi)} = \zeta^{1/2}\langle v,x\cdot w\rangle^c_{J(\Psi)}.$$ 
\item For each $\Lambda'$ appearing in the formula of (1), the central
  element $z$ must act by $\zeta$. Consequently the central element
  $x\in {}^\delta K$ must act by a scalar
$$(-1)^{\epsilon(\Lambda')}\zeta^{1/2}$$
on the unique lowest ${}^\delta K$-type $\mu(\Lambda')$ of
$I(\Lambda',0)$; here the parity 
$$\epsilon(\Lambda') \in {\mathbb  Z}/2{\mathbb Z}$$
changes if we change our (fixed global) choice of
square root $\zeta^{1/2}$.
\item There is a computable and unique finite formula
$$\sig^0_{J(\Psi)} = \sum_{\Lambda' \in
  \Pi_{\text{fin,disc}}({}^\delta G)}
s^{\epsilon(\Lambda')}B_{\Lambda',\Psi}\sig^0_{I(\Lambda',0)}.$$ 
Here the signatures on the right are the positive-definite invariant
Hermitian forms on the irreducible tempered representations
$I(\Lambda',0)$.
\end{enumerate}
As a consequence, $J(\Psi)$ is unitary if and only if one of the
following conditions is satisfied:
\begin{enumerate}[label=\alph*)]
\item For every $\Lambda'$, the coefficient $B_{\Lambda',\Psi}$ is an
  integer multiple of $s^{\epsilon(\Lambda')}$; or
\item For every $\Lambda'$, the coefficient $B_{\Lambda',\Psi}$ is an
  integer multiple of $s^{\epsilon(\Lambda')+1}$.
\end{enumerate}
\end{corollary}
The corollary says that to test whether $J(\Psi)$ is unitary, we must
first write down a formula as in Corollary \ref{cor:csigform} for the
$c$-invariant form {\em on the extended group ${}^\delta G$}. (The
first term $\sig^c_{I(\Lambda,0)}$ must be  
rewritten as a sum over the lowest ${}^\delta K$-types as explained
after Theorem \ref{thm:deftozero}.)  If any coefficient in that
expression is of the form $a+bs$ with $a$ and $b$ both nonzero, then
$J(\Psi)$ cannot be unitary. If every coefficient is either an integer
or an integer multiple of $s$, then we must compare this behavior with
the action of the element $x$ on the lowest ${}^\delta K$-type of the
corresponding standard representation. If they either agree everywhere
or disagree everywhere, then $J(\Psi)$ is unitary; otherwise it is not.

The proof of Theorem \ref{thm:deftozero} will occupy the rest of this
section. The idea is to deform the continuous parameter $\nu$ to zero,
and to compute the changes in the signature along the
way. Mathematically the simplest way to perform this deformation is in
a straight line. (We believe that deformation along a more complicated
path could be preferable computationally, but we do not pursue that
possibility here.) The main step is Corollary
\ref{cor:sigchange}. Here is a reformulation of that result using the
detailed information about Jantzen filtrations and forms contained in
Theorem \ref{thm:KLsigpolQ}.

% We will first describe how to do that, and then
% return to a more circuitous but (perhaps) computationally preferable
% path.

\begin{theorem}\label{thm:sigchange} Suppose
  $\Gamma$ is a Langlands parameter of real infinitesimal character,
  so that the Langlands quotient $J(\Gamma)$ admits a nonzero
  $c$-invariant Hermitian form $\langle,\rangle^c_1$. Consider 
  the family of standard representations $I_{\quo}(\Gamma_t)$
  (for $t\ge 0$) defined in \eqref{eq:Gammat}, and the family of
  $c$-invariant forms $\langle,\rangle^c_t$ 
  extending $\langle,\rangle^c_1$ as in Proposition
  \ref{prop:jantzenfilt}. For every $t\ge 0$, define forms 
$$\langle,\rangle_t^{c,[r]} \text{\ on }
I_{\quo}(\Gamma_t)^r/I_{\quo}(\Gamma_t)^{r+1}$$
with signatures
$$\sig^{c,[r]}_{I(\Gamma_t)} = \POS_t^{[r]} + s\NEG_t^{[r]}\colon
\widehat K \rightarrow {\mathbb W}$$ 
as in Definition \ref{def:witt}. Write 
$$\sig^c_t = \sig^c_{I(\Gamma_t)} = \sum_{r=0}^\infty
\sig^{c,[r]}_{I(\Gamma_t)}$$ 
for the signature of the (nondegenerate) Jantzen form on $\gr
I_{\quo}(\Gamma_t)$. Consider a finite subset
$${0 < t_r < t_{r-1} < \cdots <t_1 \le 1} \subset (0,1].$$
so that $I(\Gamma_t)$ is irreducible for $t\in (0,1]
\setminus\{t_i\}$.  
\begin{enumerate}
\item On the complement of $\{t_i\}$, the form $\langle,\rangle^c_t$ is
  nondegenerate and of locally constant signature
$$ \sig^c_t = \sig_t^{c,[0]}.$$
% \item In terms of the signature multiplicity polynomials of Definition
%  \ref{def:wpoly}, 
%$$\sig^c_1 =\sig^c_{I(\Gamma)} = \sum_{\Xi\in B(\chi)}s^{\ell_o(\Xi) - \ell_o(%\Gamma)}
%Q_{\Xi,\Gamma}(s) \sig^c_{J(\Xi)}.$$ 
%\item In terms of the signature character polynomials of Definition
%  \ref{def:Wpoly}, 
%$$\sig^c_{J(\Psi)} = \sum_{\Gamma\in B(\chi)}s^{\ell_o(\Psi) - \ell_o(\Gamma)}
%P_{\Gamma,\Psi}(s) \sig^c_{I(\Gamma)}.$$ 
\item Choose $\epsilon$ so small that $I(\Gamma_t)$ is irreducible for
  $t\in [1-\epsilon,1+\epsilon]\backslash\{1\}$. Then
$$\begin{aligned}
\sig^c_{1+\epsilon} &= \sig^c_1 = \sig^c_{1-\epsilon}\\
&+ (1-s) % WRONG!(s-1) 
\sum_{\substack{\Xi\in
    B(\chi),\ \Xi < \Gamma \\ \ell(\Xi) - \ell(\Gamma)
    \text{\thinspace odd}}} 
s^{(\ell_o(\Xi) - \ell_o(\Gamma))/2}
Q_{\Xi,\Gamma}(s)\sig^c_{J(\Xi)}. \end{aligned}$$ 
\item  Under the same hypotheses as in (2),
$$\begin{aligned}
 \sig^c_{I(\Gamma_{1+\epsilon})} &= \sig^c_{I(\Gamma_1)} = \sig^c_{I(\Gamma_{(1-\epsilon)})}\\
&+ (1-s) % WRONG!(s-1) 
\sum_{\substack{\Phi,\Xi\in
    B(\chi)\\\Phi \le \Xi < \Gamma \\ \ell(\Xi) - \ell(\Gamma)
    \text{\thinspace odd}}}
% \sum_{\substack{\Phi\in B(\chi)\\ \Phi \le \Xi}} 
s^{(\ell_o(\Phi) - \ell_o(\Gamma))/2}(-1)^{\ell(\Xi) - \ell(\Phi)}
P_{\Phi,\Xi}(s)Q_{\Xi,\Gamma}(s) \sig^c_{I(\Phi)}\\
 &= \sig^c_{I(\Gamma_1)} = \sig^c_{I(\Gamma_{(1-\epsilon)})}\\
&+ (1-s) % \sum_{\substack{\Phi,\Xi\in
    % B(\chi)\\\Phi \le \Xi < \Gamma \\ \ell(\Xi) - \ell(\Gamma)
    % \text{\thinspace odd}}}
\sum (-1)^{(\ell_o(\Phi) - \ell_o(\Gamma))/2}
(-1)^{\ell(\Xi) - \ell(\Phi)}P_{\Phi,\Xi}(-1)Q_{\Xi,\Gamma}(-1) \sig^c_{I(\Phi)}.
\end{aligned}$$ 
\end{enumerate}
\end{theorem}
\begin{proof}
Part (1) is exactly Corollary \ref{cor:sigchange}(1). %Part (2) is
%Corollary \ref{cor:sigchange}(2), together with the formula from
%Theorem \ref{thm:KLsigpolQ} for the signature multiplicity matrix. Part
%(3) is \eqref{eq:sigformirr}, together with the formula from
%Corollary \ref{cor:KLsig} for the signature character matrix.  
Part (2) is Corollary \ref{cor:sigchange}(3), with the explicit formulas from
Theorem \ref{thm:KLsigpolQ}.  The first formula in (3) plugs Corollary
\ref{cor:KLsig} into (2). The second is a consequence of the first
because of the identity in ${\mathbb W}$
$$(s-1)s = (s-1)(-1).$$
\end{proof}

\begin{proof}[Proof of Theorem \ref{thm:deftozero}] We fix an
% \begin{subequations}\label{se:proofdeftozero}
  invariant symmetric bilinear form on ${\mathfrak g}$ that is
  negative definite on the compact form ${\mathfrak g}({\mathbb
    R},\sigma_c)$ (Theorem \ref{thm:realforms}).  This form gives rise
  to a $W({\mathfrak g},{\mathfrak h})$-invariant form
  $\langle,\rangle$ on the dual ${\mathfrak h}^*$ of any Cartan
  subalgebra, which is positive definite on the canonical real weights
  (the real span of the differentials of algebraic characters).  We
  can arrange for this form to take rational values on the lattice
  $X^*(H({\mathbb C}))$; the denominators appearing are then
  necessarily bounded by some positive integer $N$. In this case the
  form also takes rational values on the differentials of discrete
  Langlands parameters, this time with denominators bounded by $4N$.
  (We may gain one $2$ from the $\rho$ shift in the parameter, and one
  from the operation $(1+\theta)/2$ of restricting to the compact
  part.)

To prove the theorem, we fix a bound $M$ on the size of infinitesimal
characters, and prove the theorem for all Langlands parameters
\begin{equation*}
\Gamma' = (\Lambda',\nu'), \qquad \langle
(d\Lambda',\nu'),(d\Lambda,\nu') \rangle \le M.
\end{equation*}
The proof proceeds by {\em downward} induction on $\langle
d\Lambda',d\Lambda'\rangle = k$. For $k>M$ there are no such Langlands
parameters, and the theorem is vacuously true; so we assume that the
theorem is known for $\Gamma'$ with
\begin{equation*}
\langle d\Lambda',d\Lambda' \rangle > \langle
d\Lambda,d\Lambda\rangle.
\end{equation*}
We apply the deformation formulas of Theorem \ref{thm:sigchange} $r$
times, at the points $t_1$, $t_2$,\dots, $t_r$. We get in the end a
(computable) formula
\begin{equation}\label{eq:rdefform}
\sig^c_{I(\Gamma)} = \sig^c_{I(\Lambda,0)} + (s-1)\sum_{\Phi \in
  B(\chi_{t_j})} M_{\Phi,\Gamma} \sig^c_{I(\Phi)}.
\end{equation}
Here every Langlands parameter $\Phi$ appearing has infinitesimal character
of the form
\begin{equation*}
\chi_{t_j} \leftrightarrow (d\Lambda, t_j\nu),
\end{equation*}
which has size less than or equal to the size of $\chi = \chi_1$, and so is
bounded by $M$. The condition in Theorem \ref{thm:sigchange}(5) $\Phi
< \Gamma_{t_j}$ implies in particular that the discrete part
$\Lambda(\Phi)$ satisfies
\begin{equation*}
\langle d\Lambda(\Phi),d\Lambda(\Phi)\rangle > \langle d\Lambda,
d\Lambda\rangle.
\end{equation*}
By the inductive hypothesis, each $\sig^c_{I(\Phi)}$ has a formula
\begin{equation}\label{eq:indhypdefform}
\sig^c_{I(\Phi)} = \sig^c_{I(\Lambda(\Phi),0)} +
(s-1)\sum_{\Lambda' \in \Pi_{\text{fin,disc}}(G)} M_{\Lambda',\Phi}
\sig^c_{I(\Lambda',0)} \quad  (M_{\Lambda',\Phi} \in {\mathbb
  Z}).
\end{equation}
By the remarks after the statement of Theorem \ref{thm:deftozero}, the
first term on the right can be rewritten as a sum (over the lowest
$K$-types of $I(\Phi)$) of standard final discrete
representations. Now if we insert all these inductively known formulas
\eqref{eq:indhypdefform} in \eqref{eq:rdefform}, we get the conclusion
of the theorem.
% \end{subequations} % \label{se:proofdeftozero}   
\end{proof}

\section{Example: Hermitian forms for $SL(2,{\mathbb
    C})$}\label{sec:exsl2C} 
\setcounter{equation}{0}

In this section we will carry out the algorithm of this paper to
calculate the signatures of the invariant Hermitian forms on the
irreducible representations of $G = SL(2,{\mathbb C})$, and in particular
to find the unitary representations of $SL(2,{\mathbb C})$. Of course
these results go back to Gelfand and Naimark in the 1940s; the point
is just to indicate as explicitly as possible the structure of our
algorithm. 

\begin{subequations}\label{se:SL2C}
We follow the notation for complex groups (regarded as real groups)
explained in \Cite{Vgreen}*{\S 7.1}. The complex Lie group with which
we begin is
\begin{equation}
G({\mathbb C}) = SL(2,{\mathbb C}) \times SL(2,{\mathbb C}) = G^L \times
G^R;
\end{equation}
the labels $L$ and $R$ stand for ``left'' and ``right.'' Here is some
of the notation introduced in Section \ref{sec:realred}. The standard
compact real form of $G({\mathbb C})$ is
\begin{equation}
\sigma_c(g,h) = ((g^*)^{-1}, (h^*)^{-1}), \qquad G({\mathbb
  R},\sigma_c) = SU(2) \times SU(2);
\end{equation}
here $g^* = {}^t \overline g$ is the usual Hermitian transpose on
matrices. The real form we are considering is
\begin{equation}
\sigma(g,h) = ((h^*)^{-1},(g^*)^{-1});
\end{equation}
the corresponding real group is
\begin{equation*}
 \quad G({\mathbb R},\sigma) =
\{(g,(g^*)^{-1}) \mid g\in SL(2,{\mathbb C})\} \simeq SL(2,{\mathbb C}).
\end{equation*}
Then the Cartan involution and complexified maximal compact subgroup
are
\begin{equation} \theta = \sigma\sigma_c, \quad \theta(g,h) = (h,g),
  \quad K({\mathbb C}) = SL(2,{\mathbb C})_\Delta \subset
  SL(2,{\mathbb C}) \times SL(2,{\mathbb C}).
\end{equation}
If we choose the same pinning (see \eqref{se:pinning}) on each of the
two simple factors of $G({\mathbb C})$, then $\theta$ preserves the
pinning; so the distinguished automorphism $\delta_f$ is
\begin{equation*}
\delta_f(g,h) = \theta(g,h) = (h,g).
\end{equation*}
The extended group is therefore
\begin{equation}
{}^\delta G({\mathbb C}) = G({\mathbb C}) \rtimes \{1,\delta_f\},
\end{equation}
and we can choose $x=\delta_f$ as a strong involution representing
$G$. (There are four legal choices of $x$: we can multiply $\delta_f$
by any of the central elements of $G$.)

The maximal compact subgroup is
\begin{equation*}
K = G^\theta = SU(2).
\end{equation*}
We will index irreducible representations $\mu_n\in \widehat K$ by nonnegative
integers $n$ (highest weights). Of course $\dim \mu_n = n+1$.
\end{subequations} %{se:SL2C}

% \begin{equation}

\begin{subequations}\label{se:SL2Ctorus}
Up to conjugation by $G$, there is a unique maximal torus defined over
${\mathbb R}$. We choose as a representative the diagonal one
\begin{equation}
H({\mathbb C}) = \left\{\left(\begin{pmatrix} z & 0 \\ 0 & z^{-1}
    \end{pmatrix}, \begin{pmatrix} w & 0 \\ 0 &
      w^{-1}\end{pmatrix}\right) \mid z, w\in {\mathbb
    C}^\times\right\}. 
\end{equation}
We will write $h(z,w)$ for this element of the torus. The algebraic
characters of $H({\mathbb C})$ send $h(z,w)$ to $z^pw^q$ (for $p$ and
$q$ integers. The four roots are
\begin{equation*}
\pm \alpha^L(h(z,w)) = z^{\pm 2}, \qquad \pm\alpha^R(h(z,w)) = w^{\pm 2},
\end{equation*}
and the corresponding coroots are
\begin{equation*}
\pm (\alpha^L)^\vee(s) = h(s^{\pm 1},1), \qquad \pm (\alpha^R)^\vee(t)
= h(1,t^{\pm 1}).
\end{equation*}
The complex Weyl group is $W(G({\mathbb C}),H({\mathbb C})) = \{\pm 1\}
\times \{\pm 1\}$, acting by inversion on each coordinate
separately. Identifying the Lie algebra of ${\mathbb C}^\times$ with
${\mathbb C}$ identifies
% \begin{equation*}
\begin{alignat*}{2}
{\mathfrak h}({\mathbb C}) &\simeq {\mathbb C}^2,& \qquad
\pm (\alpha^L)^\vee &= (\pm 1,0),\\
{\mathfrak h}({\mathbb C})^* &\simeq {\mathbb C}^2,& % \qquad
\pm\alpha^L &= (\pm 2,0).\end{alignat*}
% \end{equation*} 

The real torus is
\begin{equation}\begin{aligned}
H &= \{h(z,\overline z^{-1})\mid z\in {\mathbb C}^\times
\\
&= \{h(re^{i\theta}, r^{-1}e^{i\theta}) \mid \theta \in {\mathbb R},
r> 0\}.
\end{aligned}
\end{equation}
Evidently the roots are all complex (Definition \ref{def:rhoim}), with
$\theta(\alpha^L) = \alpha^R$. 
The characters of this torus are 
\begin{equation}
\gamma_{n,\nu}(h(re^{i\theta}, r^{-1}e^{i\theta})) = r^\nu
e^{in\theta} \qquad (n\in {\mathbb Z}, \nu \in {\mathbb C}).
\end{equation}
The real Weyl group is the diagonal
subgroup
\begin{equation*}
W(G,H) = \{\pm 1\}_\Delta \subset \{\pm 1\} \times \{\pm 1\},
\end{equation*}
acting by inversion on each coordinate. (This is consistent with
Knapp's general result in Proposition \ref{prop:realweyl}, but of
course this case is much older and entirely elementary.) The action of
the real Weyl group on characters sends $\gamma_{n,\nu}$ to
$\gamma_{-n,-\nu}$. The Hermitian dual and $c$-Hermitian dual
parameters (Definitions \ref{def:dualparam} and \ref{def:cdualparam})
are
\begin{equation}
\gamma_{n,\nu}^{h,\sigma_0} = \gamma_{n,-\overline\nu}, \qquad
\gamma_{n,\nu}^{h,\sigma_c} = \gamma_{n,\overline\nu}.
\end{equation}

The differential of the character $\gamma_{n,\nu}$ is
\begin{equation}
d\gamma_{n,\nu} = \left( (n+\nu)/2, (n-\nu)/2\right) \in {\mathfrak h}^*.
\end{equation}
A root $\alpha$ is integral for $\gamma$ if and only if $d\gamma$
takes an integer value on $\alpha^\vee$ (Definition
\ref{def:introots}). Therefore
\begin{equation}
R(\gamma_{n,\nu}) = \begin{cases} R & (\nu \in 2{\mathbb Z} + n)\\
\emptyset & (\nu \notin 2{\mathbb Z} + n).\end{cases}
\end{equation}
That is, the only ``integrality condition'' is that the continuous
parameter $\nu$ should be an integer congruent to $n$ modulo $2$. The
``integral length'' (Definition \ref{def:length}) is
\begin{equation}
\ell(\gamma_{n,\nu}) = \begin{cases} 1 & (\nu \in 2{\mathbb Z} + n,
  |\nu| > |n|)\\
0 &\text{(otherwise).} \end{cases}
\end{equation}
Similarly, the ``orientation number''  (Definition \ref{def:ornumber})
is
\begin{equation}
\ell_o(\gamma_{n,\nu}) = \begin{cases} 1 & (\nu \notin 2{\mathbb Z} + n,
  |\nu| > |n|)\\
0 &\text{(otherwise).} \end{cases}
\end{equation}
\end{subequations} %{se:SL2Ctorus}

A Langlands parameter for $G$ (Definition \ref{def:langlandsparam}) is
\begin{subequations} \label{se:SL2CLanglands}
the same thing (since there are no imaginary roots) as a character of
$H$. Equivalence classes of 
parameters under conjugation by $G$ are the same thing as $W({\mathbb
  R})$ orbits of characters. The standard representations are
principal series (see Theorem \ref{thm:LC})
\begin{equation}\label{e:SL2Cps}
I_{\quo}(\gamma_{n,\nu}) =_{\text{def}} I_{\quo}(n,\nu), \qquad
I(n,\nu)|_K = \sum_{p=0}^\infty \mu_{|n| + 2p}.
\end{equation}
Each of these has a Langlands quotient representation (Theorem \ref{thm:LC})
\begin{equation}\label{e:SL2Cirr}
J(\gamma_{n,\nu}) =_{\text{def}} J(n,\nu), \qquad
J(n,\nu) \simeq J(n',\nu') \iff (n',\nu') = \pm(n,\nu).
\end{equation}
According to Theorem \ref{thm:unitarydual2}, 
\begin{equation*}
\text{$J(n,\nu)$ is Hermitian if and only if either $\nu\in i{\mathbb
    R}$, or $n=0$ and $\nu\in {\mathbb R}$.}
\end{equation*}
The first class of representations $J(n,i\sigma)=I(n,is)\simeq
I(-n,-is)$ is the unitary principal series; these representations are
all unitary. It is the second family $J(0,t) \simeq J(0,-t)$ that
requires study. We do this by comparing the invariant Hermitian form
to the invariant $c$-Hermitian form. According to Proposition
\ref{prop:cdualstd},  
\begin{equation*}
\text{$J(n,\nu)$ is $c$-Hermitian if and only if either $\nu\in {\mathbb
    R}$, or $n=0$ and $\nu\in i{\mathbb R}$.}
\end{equation*}
\end{subequations} %{se:SL2CLanglands}

To continue, we need to discuss parameters and representations of the
\begin{subequations}\label{se:SL2Cextparam}
extended group ${}^\delta G$. 
We have
\begin{equation}\begin{aligned}
\gamma_{n,\nu}^\theta &= \gamma_{n,-\nu},\\
\gamma_{n,\nu}^\theta \simeq \gamma_{n,\nu} &\iff \nu = 0 \text{\ or\
} n=0.
\end{aligned}\end{equation}
The second assertion amounts to
\begin{equation*}
\text{$\gamma_{n,\nu}$ is type one if and only if $\nu=0$ or $n=0$.}
\end{equation*}
Up to conjugation by $G$, there are exactly two classes of extended
maximal tori (Definition \ref{def:exttorus}):
\begin{equation}
{}^f H = \langle H,\delta_f\rangle, \qquad {}^s H = \langle H,
\delta_s\rangle;
\end{equation}
here $f$ and $s$ stand for ``fundamental'' and ``split.'' The second
generator for ${}^f H$ is the involution $\delta_f$ exchanging
the two factors of $G({\mathbb C})$. The generator for ${}^s H$ is
\begin{equation}
\delta_s = \left(\begin{pmatrix}0&1\\-1 &0\end{pmatrix},
  \begin{pmatrix} 0&-1 \\ 1 & 0 \end{pmatrix} \right)\delta_f
\end{equation}
The element $\delta_s$ is the conjugate of $\delta_f$ by
$\left(\begin{pmatrix} 0&1\\-1&0 \end{pmatrix},I\right)$.  That is one
of many ways to see that $\delta_s^2 = \delta_f^2 = 1$.  

The type one extended Langlands parameters attached to ${}^f H$ are
the two extensions of the characters $\gamma_{n,0}$:
\begin{equation}
{}^f\gamma_{n,0}^{\epsilon} \in \widehat{{}^f H}, \quad
{}^f\gamma_{n,0}^{\epsilon}(\delta_f) = \epsilon \qquad (\epsilon = \pm
1).
\end{equation}
The character ${}^f\gamma_{n,0}^{\epsilon}$ is conjugate (by the
normalizer in $K$ of ${}^f H$) to ${}^f\gamma_{-n,0}^{\epsilon}$ (no
change of sign in the superscript).

In the same way, the type one extended Langlands parameters attached
to ${}^s H$ are the two extensions of the characters $\gamma_{0,\nu}$:
\begin{equation}
{}^s\gamma_{0,\nu}^{\epsilon} \in \widehat{{}^s H}, \quad
{}^s\gamma_{0,\nu}^{\epsilon}(\delta_s) = \epsilon \qquad (\epsilon = \pm 1).
\end{equation}
The character ${}^s\gamma_{0,\nu}^{\epsilon}$ is conjugate (by the
normalizer in $K$ 
of ${}^s H$) to ${}^s\gamma_{0,-\nu}^{\epsilon}$ (no change of sign in the
superscript). % is this 
% right?? I think that the rep of W in K commutes with \delta_s. 
The only other equivalence for extended parameters is
\begin{equation}\label{eq:sfequiv}
{}^f\gamma_{0,0}^{\epsilon} \simeq {}^s\gamma_{0,0}^{\epsilon}
\end{equation}
(no change of sign in the superscript); this is explained in
Definition \ref{def:equivextparam}. 
\end {subequations} %{se:SL2Cextparam}

In order to make the deformation in $\nu$ described in Section
\ref{sec:defto0}, we need to % find the hyperplanes described in
% Definition \ref{def:hyperplanearr}.
control possible reducibility of standard representations
$I({}^s\gamma_{n,\nu}^{+1}$. We will do this only for $n=0$,
since that is the  only case arising in the analysis of unitary
representations. That is, we are considering the one-dimensional real
vector space
\begin{equation*}
\{\nu \in {\mathbb R}\} \leftrightarrow \{{}^s\gamma_{0,\nu}^{+1}\}.
\end{equation*}
(It is enough to consider only the extended parameters with
superscript $+1$
because signatures of forms on irreducibles are unchanged by switching
$+1$ and $-1$.) % The ``potential reducibility hyperplanes'' (see also
Reducibility is possible (Proposition \ref{prop:LCshape}) only if
\begin{equation*}
\nu = \pm 2, \pm 4,\ldots
\end{equation*}
% There are {\em no} reorienting hyperplanes.
The signature of the
$c$-invariant form on $I(0,\nu)$ can change only when $\nu = 2p$ is a
nonzero even integer. For definiteness we take $p>0$. The
infinitesimal character $\xi$ is represented by the Weyl group orbit $(\pm
p, \pm p)$. Up to equivalence, the extended Langlands parameters
with this infinitesimal character are
\begin{equation*}
B(\xi) = \{{}^f\gamma_{2p,0}^{\pm 1}, {}^s\gamma_{0,2p}^{\pm 1}\},
\end{equation*}
having length $0$ and $1$ respectively. The non-diagonal Kazhdan-Lusztig
multiplicity polynomials are
\begin{equation*}
Q_{{}^f\gamma_{2p,0}^{(-1)^p},{}^s\gamma_{0,2p}^+} = 1.
\end{equation*}
(Deciding what sign to put in the superscript of
${}^s\gamma_{0,2p}$ is not trivial. There is an explanation in
\Cite{AV}*{\S 9}; it is one of the main points of that paper.)
According to Theorem \ref{thm:sigchange}, it follows that
\begin{equation*}
\sig^c_{I({}^s\gamma^{+1}_{0,2p+\epsilon})} =
\sig^c_{I({}^s\gamma^{+1}_{0,2p})}
    = \sig^c_{I({}^s\gamma^{+1}_{0,2p+\epsilon})} +
      (1-s)\sig^c_{I({}^f\gamma^{(-1)^p}_{2p,0})}.
\end{equation*}
Applying such formulas at each of the reducibility points, as in the
proof of Theorem \ref{thm:deftozero}, we find
\begin{equation*}
\sig^c_{I({}^s\gamma^+_{0,t})} = \sig^c_{I({}^s\gamma^+_{0,0})} + (1-s)
\sum_{j=1}^p \sig^c_{I({}^f\gamma^{(-1)^j}_{2j,0})} \qquad (0 \le 2p\le t <
2p+2).
\end{equation*}
For $t$ not a positive integer, this principal series representation
is irreducible, so we have computed the signature of the $c$-invariant
form on the Langlands quotient. For $t=2p$, the calculation in
Corollary \ref{cor:sigform} is very short:
\begin{equation*}
\sig^c_{J({}^s\gamma^+_{0,2p})} = \sig^c_{I({}^s\gamma^+_{0,0})} + (1-s)
\left(\sum_{j=1}^{p-1} \sig^c_{I({}^f\gamma^{(-1)^j}_{2j,0})}\right)  - s 
\sig^c_{I({}^f\gamma^{(-1)^p}_{2p,0})}.
\end{equation*}

Now we convert these formulas to formulas for the ordinary invariant
Hermitian forms $\langle,\rangle^0$ as in Corollary
\ref{cor:sigform}. Recall that we chose $x=\delta_f$ in
\eqref{se:SL2C} for our strong involution; so $z=x^2=1$ acts by $\zeta
= 1$ in all of our representations.  We choose the square root
$\zeta^{1/2} =1$. Almost by definition of the Langlands parameters,
$x$ acts by the scalar $\epsilon = \pm 1$
% $$(-1)^{\epsilon({}^s\gamma^\pm_{0,0})} = \pm 1, \qquad
% (-1)^{\epsilon({}^f\gamma^\pm_{2p,0})} = \pm 1$$
on the lowest $K$-type of the standard
\begin{subequations}\label{se:SL2Csig0}
representation $I({}^f\gamma^\epsilon_{2p,0})$. In light of
\eqref{eq:sfequiv}, $x$ also acts by $\epsilon$ on the lowest $K$-type
of $I({}^s\gamma^\epsilon_{0,0})$. By Corollary
\ref{cor:sigform}  
\begin{equation}%\begin{aligned}
\sig^0_{I({}^s\gamma^{+1}_{0,t})} = \sig^0_{I({}^s\gamma^{+1}_{0,0})} + (1-s)
\sum_{j=1}^p s^j\sig^0_{I({}^f\gamma^{(-1)^j}_{2j,0})} \qquad (0 \le 2p\le t <
2p+2), % \\
% &= \sig^0_{I({}^s\gamma^+_{0,0})} + (s-1)
% \sum_{j=1}^p \sig^0_{I({}^f\gamma^-_{2j,0})} \qquad (0 \le 2p\le t <
% 2p+2),
% \end{aligned}
\end{equation}
\begin{equation}\label{e:SL2Cfd} %\begin{aligned}
\sig^0_{J({}^s\gamma^{+1}_{0,2p})} = \sig^0_{I({}^s\gamma^{+1}_{0,0})} + (1-s)
\left(\sum_{j=1}^{p-1}
  s^j\sig^0_{I({}^f\gamma^{(-1)^j}_{2j,0})}\right)  -s^{p+1} 
\sig^0_{I({}^f\gamma^{(-1)^p}_{2p,0})}. % \\
% &= \sig^0_{I({}^s\gamma^+_{0,0})} + (s-1)
% \left(\sum_{j=1}^{p-1} \sig^0_{I({}^f\gamma^-_{2j,0})}\right)  -
% \sig^0_{I({}^f\gamma^-_{2p,0})}.
% \end{aligned}
\end{equation}
As explained in Corollary \ref{cor:sigform}(4), such a formula is the
signature of a unitary representation if and only if $s$ never
appears. We therefore find (for $t\ge 0$)
\begin{equation}
\text{$J({}^s\gamma^+_{0,t})$  is unitary} \iff 0\le t \le 2.
\end{equation}
The reason is that for $t>2$, the term $(s-1)I({}^f\gamma^{-1}_{2,0})$ in
the signature formula 
guarantees that the three-dimensional representation of $K$
contributes to the negative part of the invariant Hermitian form. 

A little more precisely, the formula \eqref{e:SL2Cfd} says that in the
spherical finite-dimensional representation of $SL(2,{\mathbb C})$ of
infinitesimal character $(p,p)$ (which has dimension $p^2$), the $K$
representations $\mu_0, \mu_2,\ldots,\mu_{2p-2}$ appear with
alternating signs in the invariant Hermitian form.
\end{subequations} %{se:SL2Csig0}

Of course these facts about $SL(2,{\mathbb C})$ are classical (see
\Cite{GN}) and easy to prove directly.
% by elementary means; we include them just as an illustration
% of the structure of the argument in general.

\end{document}